\documentclass[11pt,letterpaper]{amsart}

\usepackage{mathrsfs}

\usepackage{lipsum}

\usepackage{amsthm,amsfonts,amssymb,amsmath,amsxtra}
\usepackage[all]{xy}
\SelectTips{cm}{}
\usepackage{tikz}
\usepackage{tikz-cd}
\usepackage{verbatim}

\RequirePackage{xspace}
\RequirePackage{etoolbox}
\RequirePackage{varwidth}
\RequirePackage{enumitem}
\RequirePackage{tensor}
\RequirePackage{mathtools}
\RequirePackage{longtable}
\RequirePackage{multirow}
\RequirePackage{makecell}
\RequirePackage{bigints}
\RequirePackage{xcolor}

\usepackage[backref=page]{hyperref} %

\newcommand{\BA}{\ensuremath{\mathbb{A}}\xspace}

\newcommand{\BC}{\ensuremath{\mathbb{C}}\xspace}
\newcommand{\BD}{\ensuremath{\mathbb{D}}\xspace}
\newcommand{\BE}{\ensuremath{\mathbb{E}}\xspace}
\newcommand{\BF}{\ensuremath{\mathbb{F}}\xspace}
\newcommand{\BG}{\ensuremath{\mathbb{G}}\xspace}

\newcommand{\BJ}{\ensuremath{\mathbb{J}}\xspace}

\newcommand{\BL}{\ensuremath{\mathbb{L}}\xspace}
\newcommand{\BM}{\ensuremath{\mathbb{M}}\xspace}
\newcommand{\BN}{\ensuremath{\mathbb{N}}\xspace}

\newcommand{\BP}{\ensuremath{\mathbb{P}}\xspace}
\newcommand{\BQ}{\ensuremath{\mathbb{Q}}\xspace}
\newcommand{\BR}{\ensuremath{\mathbb{R}}\xspace}

\newcommand{\BT}{\ensuremath{\mathbb{T}}\xspace}

\newcommand{\BV}{\ensuremath{\mathbb{V}}\xspace}

\newcommand{\BX}{\ensuremath{\mathbb{X}}\xspace}

\newcommand{\BZ}{\ensuremath{\mathbb{Z}}\xspace}

\newcommand{\CA}{\ensuremath{\mathcal{A}}\xspace}

\newcommand{\CC}{\ensuremath{\mathcal{C}}\xspace}

\newcommand{\CE}{\ensuremath{\mathcal{E}}\xspace}
\newcommand{\CF}{\ensuremath{\mathcal{F}}\xspace}
\newcommand{\CG}{\ensuremath{\mathcal{G}}\xspace}

\newcommand{\CM}{\ensuremath{\mathcal{M}}\xspace}
\newcommand{\CN}{\ensuremath{\mathcal{N}}\xspace}
\newcommand{\CO}{\ensuremath{\mathcal{O}}\xspace}

\newcommand{\CS}{\ensuremath{\mathcal{S}}\xspace}

\newcommand{\CV}{\ensuremath{\mathcal{V}}\xspace}

\newcommand{\CX}{\ensuremath{\mathcal{X}}\xspace}
\newcommand{\CY}{\ensuremath{\mathcal{Y}}\xspace}
\newcommand{\CZ}{\ensuremath{\mathcal{Z}}\xspace}

\newcommand{\ab}{\mathrm{ab}}

\DeclareMathOperator{\Aut}{Aut}

\newcommand{\Ch}{{\mathrm{Ch}}}
\DeclareMathOperator{\charac}{char}

\newcommand{\Hk}{{\mathrm{Hk}}}
\newcommand{\Fix}{{\mathrm{Fix}}}
\newcommand{\Art}{{\mathrm{Art}}}
\newcommand{\adm}{{\mathrm{adm}}}

\newcommand{\del}{\operatorname{\partial Orb}}

\DeclareMathOperator{\diag}{diag}

\renewcommand{\div}{{\mathrm{div}}}

\DeclareMathOperator{\End}{End}

\DeclareMathOperator{\Gal}{Gal}

\newcommand{\GL}{\mathrm{GL}}

\newcommand{\GU}{\mathrm{GU}}

\DeclareMathOperator{\Hom}{Hom}

\newcommand{\id}{\ensuremath{\mathrm{id}}\xspace}
\let\Im\relax

\DeclareMathOperator{\Im}{Im}

\DeclareMathOperator{\Int}{\ensuremath{\mathrm{Int}}\xspace}

\DeclareMathOperator{\length}{length}
\DeclareMathOperator{\Lie}{Lie}

\newcommand{\new}{{\mathrm{new}}}

\DeclareMathOperator{\Nm}{Nm}

\DeclareMathOperator{\Orb}{Orb}

\DeclareMathOperator{\Ros}{Ros}

\newcommand{\red}{\ensuremath{\mathrm{red}}\xspace}

\DeclareMathOperator{\Res}{Res}

\newcommand{\rs}{\ensuremath{\mathrm{rs}}\xspace}

\newcommand{\Sh}{\mathrm{Sh}}

\newcommand{\SL}{{\mathrm{SL}}}

\DeclareMathOperator{\Spec}{Spec}
\DeclareMathOperator{\Spf}{Spf}

\newcommand{\SU}{{\mathrm{SU}}}

\newcommand{\surj}{\twoheadrightarrow}

\DeclareMathOperator{\tr}{tr}

\newcommand{\U}{\mathrm{U}}

\DeclareMathOperator{\vol}{vol}

\newcommand{\wt}{\widetilde}
\newcommand{\wh}{\widehat}

\newcommand{\ov}{\overline}

\newcommand{\lra}{\longrightarrow}

\newcommand{\basic}{\text{basic}}
\newcommand{\Ker}{\text{Ker}}
\newcommand{\Bl}{\mathrm {Bl}}

\newtheorem{theorem}{Theorem}
\newtheorem{proposition}[theorem]{Proposition}
\newtheorem{lemma}[theorem]{Lemma}

\newtheorem{conjecture}[theorem]{Conjecture}
\newtheorem{`conjecture'}[theorem]{``Conjecture''}
\newtheorem{corollary}[theorem]{Corollary}

\theoremstyle{definition}
\newtheorem{definition}[theorem]{Definition}
\newtheorem{example}[theorem]{Example}

\newtheorem{remark}[theorem]{Remark}

\newenvironment{altenumerate}
{\begin{list}
		{(\theenumi) }
		{\usecounter{enumi}
			\setlength{\labelwidth}{0pt}
			\setlength{\labelsep}{0pt}
			\setlength{\leftmargin}{0pt}
			\setlength{\itemsep}{\the\smallskipamount}
			\renewcommand{\theenumi}{\roman{enumi}}
	}}
	{\end{list}}

\newenvironment{altitemize}
{\begin{list}
		{$\bullet$}
		{\setlength{\labelwidth}{0pt}
			\setlength{\itemindent}{5pt}
			\setlength{\labelsep}{5pt}
			\setlength{\leftmargin}{0pt}
			\setlength{\itemsep}{\the\smallskipamount}
	}}
	{\end{list}}

\numberwithin{equation}{section}
\numberwithin{theorem}{section}

\setitemize[0]{leftmargin=*,itemsep=\the\smallskipamount}
\setenumerate[0]{leftmargin=*,itemsep=\the\smallskipamount}

\renewcommand{\to}{%
	\ifbool{@display}{\longrightarrow}{\rightarrow}%
}
\let\shortmapsto\mapsto
\renewcommand{\mapsto}{%
	\ifbool{@display}{\longmapsto}{\shortmapsto}%
}
\newcommand{\hooklongrightarrow}{\mathrel{\mkern 0.5mu\lhook\mkern -3.5mu\relbar\mkern -3mu \rightarrow }}
\newcommand{\inj}{%
	\ifbool{@display}{\hooklongrightarrow}{\hookrightarrow}
}

\newlength{\olen}
\newlength{\ulen}
\newlength{\xlen}
\newcommand{\xra}[2][]{%
	\ifbool{@display}%
	{\settowidth{\olen}{$\overset{#2}{\longrightarrow}$}%
		\settowidth{\ulen}{$\underset{#1}{\longrightarrow}$}%
		\settowidth{\xlen}{$\xrightarrow[#1]{#2}$}%
		\ifdimgreater{\olen}{\xlen}%
		{\underset{#1}{\overset{#2}{\longrightarrow}}}%
		{\ifdimgreater{\ulen}{\xlen}%
			{\underset{#1}{\overset{#2}{\longrightarrow}}}
			{\xrightarrow[#1]{#2}}}}%
	{\xrightarrow[#1]{#2}}
}
\makeatother
\newcommand{\xyra}[2][]{%
	\settowidth{\xlen}{$\xrightarrow[#1]{#2}$}%
	\ifbool{@display}%
	{\settowidth{\olen}{$\overset{#2}{\longrightarrow}$}%
		\settowidth{\ulen}{$\underset{#1}{\longrightarrow}$}%
		\ifdimgreater{\olen}{\xlen}%
		{\mathrel{\xymatrix@M=.12ex@C=3.2ex{\ar[r]^-{#2}_-{#1} &}}}%
		{\ifdimgreater{\ulen}{\xlen}%
			{\mathrel{\xymatrix@M=.12ex@C=3.2ex{\ar[r]^-{#2}_-{#1} &}}}
			{\mathrel{\xymatrix@M=.12ex@C=\the\xlen{\ar[r]^-{#2}_-{#1} &}}}}}%
	{\mathrel{\xymatrix@M=.12ex@C=\the\xlen{\ar[r]^-{#2}_-{#1} &}}}%
}
\makeatletter
\newcommand{\xla}[2][]{%
	\ifbool{@display}%
	{\settowidth{\olen}{$\overset{#2}{\longleftarrow}$}%
		\settowidth{\ulen}{$\underset{#1}{\longleftarrow}$}%
		\settowidth{\xlen}{$\xleftarrow[#1]{#2}$}%
		\ifdimgreater{\olen}{\xlen}%
		{\underset{#1}{\overset{#2}{\longleftarrow}}}%
		{\ifdimgreater{\ulen}{\xlen}%
			{\underset{#1}{\overset{#2}{\longleftarrow}}}
			{\xleftarrow[#1]{#2}}}}%
	{\xleftarrow[#1]{#2}}
}
\renewcommand{\lra}{%
	\ifbool{@display}{\longleftrightarrow}{\leftrightarrow}%
}

\begin{document}
	
	\title[Maximal parahoric arithmetic transfers]{Maximal parahoric arithmetic transfers, resolutions and modularity}
	\author{Zhiyu Zhang}
	\address{Stanford University, Department of Mathematics, 450 Jane Stanford Way, Stanford, CA 94305}
	\email{zyuzhang@stanford.edu}
	
	\date{\today}
	
	\begin{abstract}
		For any unramified quadratic extension of $p$-adic local fields $F/F_0$ ($p$ odd), we formulate several arithmetic transfer conjectures at any maximal parahoric level, in the context of Zhang's relative trace formula approach to the arithmetic Gan--Gross--Prasad conjecture. The formulation involves a way to resolve the singularity of relevant moduli spaces via natural stratifications and modify derived fixed points. By a local--global method and double induction, we prove these conjectures when $F_0/\BQ_p$ is unramified and the arithmetic fundamental lemma for any $F_0$. We introduce the relative Cayley map and also establish explicit Jacquet--Rallis transfers at maximal parahoric levels.  Moreover, we prove new modularity results for arithmetic theta series with levels via a method of modification over fibers. Along the way, we study the complex and mod $p$ geometry of Shimura varieties and special cycles. 
	\end{abstract}
	
	\maketitle
	
	\tableofcontents

	\section{Introduction}\label{section: intro}

	The Gross--Zagier formula \cite{GZ} relates N\'eron--Tate heights of Heegner points on modular curves to central derivatives of $L$-functions of modular forms, which has fundamental applications to the Birch and Swinnerton--Dyer (BSD) conjectures for elliptic curves. As a higher dimensional generalization, the \emph{arithmetic Gan--Gross--Prasad  (AGGP) conjecture} \cite{GGP, RSZ-Diagonalcycle} relates heights of special diagonal cycles on unitary (resp. orthogonal) Shimura varieties to central derivatives of automorphic $L$-functions for unitary groups (resp. orthogonal groups), which has fundamental applications to the Bloch--Kato conjectures for automorphic motives.
	
	In the \emph{relative trace formula} (RTF) approach proposed by W. Zhang \cite{AFL-Invent} (inspired by the RTF approach of Jacquet--Rallis \cite{JR-GGP, Zhang-GGP, Zhang-RS} to the Gan--Gross--Prasad conjecture \cite{GGP}), the AGGP conjecture for \emph{unitary groups} is reduced to following local conjectures:
	\begin{itemize}
		\item The \emph{arithmetic fundamental lemma} (AFL) at hyperspecial places formulated by W. Zhang \cite{AFL-Invent}.
		\item Suitable \emph{arithmetic transfer} (AT) conjectures at places with levels. See the works of Rapoport--Smithling--Zhang \cite{RSZ-exotic, RSZ-regular} for some formulations and known low-dimensional cases.
	\end{itemize} 
Over a prime $p \geq 2$, these conjectures relate central derivatives of orbital integrals of suitable test functions to intersection numbers of special diagonal cycles on unitary Rapoport--Zink spaces :
	\[
	\del\bigl(\gamma, f') 
	= - \Int(g)\cdot\log q.
	\]
	See the ICM report of W. Zhang \cite{Zhang-ICM} for details and some applications. Significant progress has been made in recent years: the AFL over $\BQ_p$ (resp.  any $p$-adic field $F_0$) for $p \geq n$ (resp. $q \geq n$) is proved via a global method in the work of W. Zhang \cite{AFL} (resp. Mihatsch--Zhang \cite{AFL2021}).  
	
	For arithmetic applications, we could not avoid local ramifications. Therefore, it is important and necessary to establish AT identities at these places. Compared to AFL, it is more difficult to formulate AT conjectures: the relevant local moduli spaces (and global integral models of Shimura varieties) are no longer regular, and there may be no natural test functions on the analytic side. The relevant reductive group may not be quasi-split, and the geometry is also more complicated. New ideas and techniques are needed for both the formulation and proof of AT conjectures.

	The first goal of this paper is to formulate and prove arithmetic transfer conjectures for \emph{all maximal parahoric levels} at \emph{unramified places}.  Our formulation is coordinate free. On the geometric side, we introduce a general method to resolve the singularity via natural stratifications and  blowing up along \emph{the local model diagram}. We do intersections and modify \emph{derived fixed points} on the resolution. On the analytic side, we construct explicit test functions and study their orbital integrals. Moreover, we formulate and prove explicit transfer conjectures at these levels, generalizing the work of Beuzart--Plessis \cite{RBP19} on Jacquet--Rallis fundamental lemmas.
	
     We also use similar resolutions to construct regular integral models for products of unitary Shimura varieties at these levels appearing in the AGGP conjecture, which is necessary to compute Beilinson--Bloch heights of cycles via arithmetic intersection theory.

	The second goal of this paper is to prove new modularity results for arithmetic theta series on unitary Shimura varieties with parahoric levels, in the context of the Kudla program. For hermitian (resp. quadratic) lattices $L$ with suitable signatures over a totally real number field $F_0$, \emph{Kudla's generating series of special cycles} on unitary (resp. orthogonal) Shimura varieties \cite{Kudla-Eisenstein} for $L$ could be thought as arithmetic and geometric analogs of classical theta series, which are conjectured to be modular. The conjectural modularity on unitary Shimura varieties (in different set ups) has several arithmetic applications.
	\begin{itemize}
		\item In the work of Liu and Li--Liu \cite{Liu-thesis, Li-Liu2020}, the conjectural modularity over the generic fiber is used to construct arithmetic theta liftings; 
		\item In the work of Bruinier--Howard--Kudla--Rapoport--Yang \cite{BHKRY}, the modularity in the arithmetic divisor case is proved when $F_0=\mathbb Q$ and $L$ is self-dual, which applies to the \emph{averaged Colmez's conjecture} on the Faltings heights of abelian varieties with complex multiplication; 
		\item In the work of Mihatsch--Zhang \cite{AFL2021},  an almost modularity result is obtained in the arithmetic divisor case when $F_0 \not =\mathbb Q$ and $L$ is self-dual (away from $\Delta$), which is used in their proof of AFL when $q \geq n$. 
		\item In the work of Howard--Madapusi \cite[Section 8-9]{HMP2020}, the modularity on orthogonal Shimura varieties in the arithmetic divisor case is proved for maximal quadratic lattices, which has applications to exceptional jumps of Picard ranks of reductions of K3 surfaces over number fields \cite{SSTT19}. 
	\end{itemize}  
	
	Now we go to more details on our main results.
	
	\subsection{Arithmetic transfers at maximal parahoric levels}
	
	Let $p$ be an odd prime. Let $F/F_0$ be an unramified quadratic extension of $p$-adic fields with residue fields $\BF_{q^2}/\BF_q$. Denote by $\breve{F}$ the completion of the maximal unramified extension of $F_0$.
	
\subsubsection{Geometric side}	Let $V$ be a $F/F_0$ hermitian space of dimension $n \geq 1$. Let $L \subseteq V$ be a vertex lattice of type $t$, i.e. $L \subseteq  L^\vee$ and $L^\vee/L$ is a $\BF_{q^2}$-vector space of dimension $t$.  Consider the unitary group $G:=\U(V)$ over $F_0$, which may be non-quasi-split. Stabilizers of vertex lattices give all maximal parahoric subgroups of $G(F_0)$ \cite[Section 4]{Tits}. Associated to $L$ is \emph{the relative Rapoport--Zink space} \cite{RZ96} (a formal scheme) parameterizing certain deformations of $p$-divisible groups with $O_F$-action:
	$$
	\CN=\CN_{\U(L)}=\CN_n^{[t]} \to \Spf O_{\breve{F}}.
	$$ 
It can be regarded as an integral model of local unitary Shimura varieties \cite{SW20} with maximal parahoric levels.  
	
The morphism $\CN \to \Spf O_{\breve F}$ is formally smooth if $t=0, n$, and is of \emph{strictly semi-stable reduction} if $0<t<n$ by \cite{Cho18, Gortz01}. For example, $\CN$ is isomorphic to $\Spf O_{\breve F}$ if $n=1$, to a height $2$ Lubin--Tate space if $n=2, \, t=0,2$, and to the Drinfeld half plane $\wh{\Omega}^1_{F_0}$ \cite{KR-Drinfeld plane} if $n=2, t=1$.

Let $\BV$ be \emph{the hermitian space of special quasi--homomorphisms} (\ref{eq: space of special quasi-homos}) which is of dimension $n$ and \emph{not isomorphic} to $V$. The group $J_b(F_0):=\U(\BV)(F_0)$ acts on $\CN$ naturally by changing the framing. 	
	
The space $\CN \times \CN$ (product over $O_{\breve{F}}$) is singular if $0<t<n$.  A question is how to resolve the singularity in a natural way. Inspired by recent work on Beilinson--Bloch--Kato conjectures \cite[Section 5]{LTXZZ}, we introduce the $J_b(F_0)$-stable \emph{formal Balloon--Ground stratification} on the formal special fiber $\CN_\BF$:
	\[
	\CN_\BF = \CN^\circ \cup \CN^\bullet.
	\]
	Our definition (\ref{defn: formal Balloon-Ground}) is purely Lie algebra theoretic and indicates these strata are examples of (closed) formal Kottwitz--Rapoport strata. We blow up along the product $\CN^\circ \times_\BF \CN^\circ$ (a Weil divisor of $\CN \times \CN$) and resolve the singularity:
	\[
	\wt{\CN \times \CN} \to \CN \times \CN
	\]
which may be thought as an arithmetic analog of the Atiyah flop \cite{Atiyah}.

Using the action of $J_b(F_0)$ on $\CN$, now we formulate an arithmetic transfer conjecture for any orthogonal decomposition 
$$L=L^\flat \oplus O_Fe $$ 
where $(e,e)$ has valuation $i \in \{0,1\}$. Firstly, the decomposition gives a natural embedding 
	$$\CN^\flat:=\CN_{\U(L^\flat)} \to \CN:=\CN_{\U(L)}$$
	by adding a universal factor of relative height $2$ and dimension $1$. The induced embedding $\CN^\flat \to \CN^\flat \times \CN$ lifts to an embedding of regular formal schemes via strict transforms inside $\wt {\CN \times \CN}$:
	$$\CN^\flat \to \wt{\CN^\flat \times \CN}.$$
	
	Consider the conjugacy action of $\U(\BV^\flat)$ on $\U(\BV)$. For a regular semi-simple element $g' \in \U(\BV)(F_0)_\rs$, we consider its derived intersection number as the Euler characteristic 
	$$\wt{\Int}(g')=\chi(\wt{\CN^\flat \times \CN}, \, \CO_{\CN^\flat} \otimes^\BL \CO_{(1 \times g') \CN^\flat}) \in \BQ.$$

\subsubsection{Analytic side}	

Fix an orthogonal basis of $L^\flat$ and let $L_0$ be the $O_{F_0}$-span of this basis.  Set 
$$
L_0:=L_0^\flat \oplus O_{F_0}e, \, V_0:=L_0 \otimes \BQ, \, V_0^\flat:=L_0^\flat \otimes\BQ. 
$$
We have an induced Galois involution $\ov{(-)}$ on $V$ with fixed subspace $V_0$. Consider the symmetric space over $F_0$:
$$ 
S(V_0)=\{ \,\gamma \in \GL(V) \, | \, \gamma \ov \gamma =id \,\}.
$$

Consider the conjugacy action of $\GL(V_0^\flat)$ on $S(V_0)$.  For a regular semi-simple element $\gamma' \in S(V_0)(F_0)_\rs$ and $f' \in \CS(S(V_0)(F_0))$,  we consider the twisted orbital integral ($s \in \BC$)
\[ 
\Orb(\gamma', f', s):= \omega(\gamma)  \int_{h \in \GL(V_0^{\flat})} f'(h^{-1}\gamma h) \eta(h) |h|^{s} d h,
\]
\[
\Orb(\gamma', f')=\Orb(\gamma', f', 0), \, \, \, \del\bigl (\gamma', f')= \frac{d}{ds}\Big|_{s=0}\Orb(\gamma', f', s).
\]
Here we insert the transfer factor $\omega(\gamma') \in \{\pm 1\}$ (\ref{transfer factor: gp}).  There is a matching relation (\ref{defn: gp match elements}) between regular semi-simple orbits
\[
[S(V_0) /\!/ \GL(V_0^\flat)]_\rs \cong [\U(V) /\!/ \U(V^\flat)]_\rs \coprod [\U(\BV) /\!/ \U(\BV^\flat)]_\rs.
\]

\subsubsection{Explicit transfers and arithmetic transfer conjectures}

Write $S(L,L^{\vee})$ as the stabilizer of the lattice chain $L \subseteq L^\vee$ inside $S(V_0)(F_0)$. Firstly, we prove the following explicit Jacuqet--Rallis transfer result. 

\begin{theorem}[Theorem \ref{proof of TC}]
For any $\gamma' \in S(V_0)(F_0)_\rs$, we have 
\[
\Orb(\gamma, 1_{S(L,L^\vee)})= \begin{cases}
	\Orb(g, \, 1_{\U(L)}) & \text{if $\gamma$ matches some $g \in \U(V)(F_0)_\rs$,} \\
	0 & \text{if $\gamma$ matches some $g \in \U(\BV)(F_0)_\rs$.}
\end{cases}
\]
\end{theorem}

In fact, we also establish the theorem when $p=2$. The proof is via a pure local method and double induction, generalizing the work of Beuzart-Plessis \cite{RBP19}. Our coordinate-free formulation allows us to establish the general reduction and induction of orbital integrals.

Now it is natural to study the derived orbital integrals. We formulate the following arithmetic transfer conjecture for $(L, e)$.
 
\begin{conjecture}[Group arithmetic transfer for $(L, e)$, see also \protect{Conjecture \ref{conj: gp version ATC}} ]\label{Mainconj: gp}
		For any $\gamma' \in S(V_0)(F_0)_\rs$ matching $g' \in \U(\BV)(F_0)_\rs$ , we have an equality in $\BQ \log q$
		\begin{equation}
			\del\bigl (\gamma', 1_{S(L,L^{\vee}) } )= -\wt {\Int}(g') \log q.
		\end{equation}	
\end{conjecture}
	
It recovers the AFL conjecture \cite{AFL-Invent} when $t=0, \,  i=0$ and a version of AT conjecture in the work of Rapoport--Smithling--Zhang \cite[Section 10]{RSZ-regular} when  $t=0, \, i=1$.
	
	\subsubsection{Semi-Lie version arithmetic transfers}
	
	We also formulate two semi-Lie version arithmetic transfer conjectures for any vertex lattice $L$, in the context of the Fourier--Jacobi AGGP conjecture studied by Yifeng Liu \cite{Liu-Fourier-Jacobi}. 
	
	For $g \in \U(\BV)(F_0)$, we introduce the \emph{modified derived fixed point locus} $\wt {\Fix}^\BL  (g)$ as the \emph{derived fiber product} of strict transforms of the graph of $g$ and the diagonal inside $\wt{ \CN \times \CN}$, which is a derived $1$-cycle. This modification procedure is general and explains the embedding points phenomenon in the work of Kudla--Rapoport on orthogonal Shimura curves \cite{KR-height}.   
	
	On the geometric side, for any $u \in \BV - \{0\}$ we have \emph{Kudla--Rapoport divisors} $\CZ(u)$ and $\CY(u)$ on $\CN$ studied by Kudla--Rapooprt and Cho \cite{Cho18, KR-local}. For $(g, u) \in (\U(\BV) \times \BV)(F_0)_\rs$, we define $\wt {\Int}^{\CZ}(g, u)$ (resp. $\wt {\Int}^{\CY}(g, u)$) as the derived intersection number between $\CZ(u)$ (resp. $\CY(u)$) and $\wt {\Fix}^\BL(g)$ on $\CN$.
	
	On the analytic side, similarly we consider the derived orbital integral  $\del\bigl(-,-)$ on $(S(V_0) \times V_0 \times V_0^* )(F_0)$ with respect to the natural action of $\GL(V_0)(F_0)$. 
	There is also a matching relation (\ref{defn: semi-Lie match elements}) between regular--semisimple orbits
	\[
	[S(V_0) \times V_0 \times V_0^* /\!/ \GL(V_0)]_\rs \cong [\U(V) \times V /\!/ \U(V) ]_\rs \coprod [\U(\BV) \times \BV /\!/ \U(\BV)]_\rs.
	\]
	
	Define two standard functions on $(S(V_0) \times V_0 \times V_0^* )(F_0)$:
\begin{equation}\label{std function}
	f_{\text{std}}:=1_{S(L, L^{\vee})} \times 1_{L_0} \times 1_{(L^{\vee}_0)^*} , \quad f_{\text{std}}' =1_{S(L, L^{\vee})} \times 1_{L_0^\vee} \times 1_{(L_0)^*}.
\end{equation}
	
	\begin{conjecture} [Semi-Lie arithmetic transfers for $L$, also see \protect{Conjecture \ref{conj: semi-Lie version ATC} }] \label{Mainconj: semiLie}
		For any $(g, u) \in (\U(\BV) \times \BV)(F_0)_\rs$ matching $(\gamma, u_1, u_2) \in ( S(V_0) \times V_0 \times V_0^* )(F_0)_\rs$, we have equalities in  $\BQ \log q$:
		\begin{enumerate}
			\item  $\del\bigl ((\gamma,u_1,u_2),  f_{\text{std}})
			=- \wt {\Int}^{\CZ}(g, u) \log q$.
			\item  $\del\bigl ((\gamma,u_1,u_2),  f_{\text{std}}')
			=- (-1)^{t} \wt {\Int}^{\CY}(g, u) \log q$.
		\end{enumerate}
	\end{conjecture}

Our first main theorem in this paper is a proof of these group and semi-Lie arithmetic transfer conjectures (under a mild assumption on $F_0$) in arbitrary higher dimension.
	
	\begin{theorem}\label{Intro thm: ATC}
		Assume $F_0$ is unramified over $\BQ_p$ (e.g. $F_0=\BQ_p$) if $0< t< n$. Then the above semi-Lie (resp. group) version arithmetic transfer conjecture (\ref{Mainconj: semiLie}) (resp. \ref{Mainconj: gp})hold for any vertex lattice $L$ (resp. any orthogonal decomposition $L=L^\flat \oplus O_Fe$). 
	\end{theorem}
	
	In particular, we prove the AFL over any $p$-adic field $F_0$ for $p>2$, including small primes which is not covered by the previous works of Wei Zhang and Mihatsch--Zhang \cite{AFL2021, AFL}. The proof is based on a local-global method and double induction, using new modularity results of \emph{arithmetic theta series} at maximal parahoric levels that we now discuss.

	\subsection{Modularity of arithmetic theta series}
	
	Let $F/ F_0$ be a CM quadratic extension of a totally real number field. Fix a CM type $\Phi \subseteq \Hom (F, \ov{\BQ})$ and a distinguished element $\varphi_0 \in \Phi$.  
	
	Let $V$ be a $F/F_0$-hermitian space of dimension $n \geq 1$ and signature $\{ (n-1,1)_{\varphi_0}, (n,0)_{\varphi \in \Phi -  \{ \varphi_0 \} }  \} $. Choose a hermitian lattice $L$ in $V$ and a finite collection $\Delta$ of finite places for $F_0$. Assume that
	\begin{itemize}
		\item $F/F_0$ is unramified outside $\Delta$, and all $2$-adic places are in $\Delta$.
		\item If $v \not \in \Delta$ is a place of $F_0$ such that $L_v$ is not self-dual, then $v$ is inert in $F$, $L_v$ is a vertex lattice, the prime underlying $v$ is unramifed in $F_0$.
	\end{itemize}
	Consider $G=\U(V)$ over $F_0$. Choose a level $K \subseteq G(\BA_{0, f})$ associated to $L$ and $\Delta$, which agrees with the stabilizer $\U(\wh{L})$ away from $\Delta$, and could be chosen to be sufficiently small at $\Delta$. We consider unitary Shimura variety $M_{G, K}$ over $F$. By the work of Rapoport--Smithling--Zhang \cite{RSZ-Diagonalcycle, RSZ-Shimura}, we have a natural \emph{regular} integral model
	$$\CM=\CM_{\wt{G}, \wt{K}} \to \Spec O_E[\Delta^{-1}]$$ 
	of (the RSZ-variant of) $M_{G,K}$. Here the number field $E:=\varphi_0(F)F^\Phi \subseteq \ov{\BQ}$. If either $F$ is Galois over $\BQ$ or $F$ contains an imaginary quadratic field, then $E=F$ \cite[Remark 3.1]{RSZ-Diagonalcycle}.  An important advantage is that the RSZ-variant $\CM$ admits a PEL type moduli interpretation (\ref{moduli int of RSZ}). 
	
	As in the Kudla program, for any function $\phi=\phi_L \in \CS(V(\BA_{0,f}))^K$ that agrees with the function $1_{\wh{L}}$ away from $\Delta$, we can consider the \emph{arithmetic theta series} for $L$ as Kudla's generating series 
	$
	\wh{\CZ}(\phi)
	$
	of \emph{arithmetic Kudla--Rapoport divisors} $\wh \CZ^{\bf B} (\xi, \phi)$ (\ref{eq: Fourier coeff of arithmetic generating series})  \cite{KR-local, KR-global} equipped with \emph{admissible automorphic Green functions}. It has coefficients in the \emph{arithmetic Chow group} $\wh{\Ch}^1(\CM)$ and lifts the generating series $Z(\phi)$ (\ref{eq: generating series generic fiber}) \cite{Liu-thesis} of Kudla--Rapoport cycles on the generic fiber $M_{G, K}$. 
	
Assume that $F_0 \not =\BQ$  so that $\CM$ is projective. Consider the arithmetic intersection pairing (\refeq{eq: Truncated Arithmetic int pairing}) between arithmetic divisors and proper $1$-cycles on $\CM$:
	$$
	(-,-): \quad \wh\Ch^1(\CM)\times \CC_{1}(\CM)  \to \BR_\Delta
	$$
	As $O_E[\Delta^{-1}]$ is not ``compact'' in the sense of Arakelov geometry, the intersection pairing takes values in the quotient $\BQ$-vector space 
	$$
	\BR_\Delta:=\BR \, / \sum_{v|\ell, v \in \Delta}{ \BQ \log \ell}.
	$$ If $\mathcal{C} \in \CC_{1}(\CM)  $ is a vertical $1$-cycle at $v \not\in \Delta$, we may take the value $(-,\CC)$ in $\mathbb R$. Our second main theorem in this paper is the modularity of arithmetic theta series when intersecting with $1$-cycles.
	
	\begin{theorem}[Global modularity, also see \protect{Theorem \ref{thm: global modularity}}] \label{Intro thm: mod}
		Assume that $F_0 \not =\BQ$ and $E=F$. 
		\begin{enumerate}
			\item  Assume that $\Delta$ is large enough (see Theorem \ref{thm: global modularity} for precise conditions). For any proper $1$-cycle $\mathcal{C} \to \CM$ and $\phi=\phi_L$, the generating series $(\wh{\CZ}(\phi), \mathcal{C})$ is modular of weight $n$, up to adding a constant in $\BR$.
			\item For any linear functional $\ell: \wh{\Ch}^{1}(\CM) \to \BR_{\Delta}$ that has degree $0$ over $\mathbb C$ and trivial on irreducible components of special fibers of $\CM$ over $v \in \Delta$, $\ell(\wh{\CZ}(\phi))$ is modular of weight $n$.
		\end{enumerate}	
	\end{theorem}
	
 The modularity in $(1)$ shows that there exists a Hilbert modular form on $\SL_2(F_0)$ whose Fourier expansion agrees with $(\wh{\CZ}(\phi), \mathcal{C})$ under the projection $\mathbb R \to \BR_\Delta$ (up to the constant terms).  We avoid the use of elementary CM cycles in Mihatsch--Zhang \cite{AFL2021} hence give a different proof of the almost modularity result (in which case $\CM$ has good reduction away from $\Delta$) in \emph{loc. cit.}.  The related results on mod $p$ geometry of Shimura varieties may be of independent interest.  We also note the recent result of Qiu \cite{Qiu-mod} prove some general arithmetic modularity results on admissible extensions (rather than moduli extensions here) of arithmetic special divisors on unitary Shimura varieties with no levels at inert places (rather than maximal parahoric levels here) and Drinfeld levels at split paces.
	
	Now we briefly describe the proofs of our main theorems.
	
	\subsection{Modularity by modification and geometry of Shimura varieties}
	
	We introduce a modification method to reduce the conjectured modularity of arithmetic theta series on integral models to the modularity over the generic fiber  (known by the work of Liu \cite{Liu-thesis}) and over the basic locus (proved in \S \ref{section:Lmod}). The method relies on the complex and mod $p$ geometry of Shimura varieties, and may apply to more general levels. 
	
     Our arithmetic theta series $\wh{\CZ}(\phi)$ lies in the admissible arithmetic Chow group $\wh{\Ch}^{1,\adm}(\CM)$ (\ref{section: adm arithmetic int}).  We have a short exact sequence (\ref{factorization of Chow})
	\[
	0 \to \wh{\Ch}^{1}_{\mathrm{Vert}}(\CM) \to \wh{\Ch}^{1,\adm}(\CM) \to \Ch^1(\CM_E) \to 0.
	\]
Here the subgroup $\wh{\Ch}^{1}_{\mathrm{Vert}}(\CM)$ is generated by
	\begin{itemize}
		\item $(0, c_\nu)$, where $c_\nu$ is a locally constant function on $\CM_{E, \nu}(\BC)$ for an embedding $\nu: E \to \BC$;
		\item $(X_s, 0)$, where $X_s$ is an irreducible component of the special fiber $\CM_{k_\nu}$ at a finite place $\nu$ of $E$ (with residue field $k_\nu$) such that $\CM_{O_{E_\nu}}$ is non-smooth over $O_{E_\nu}$. In our cases, when $v \not \in \Delta$ the non--smoothness means that $v$ is inert and $L_v$ is a vertex lattice of type $0<t_v<n$.  
	\end{itemize}  
	
Consider any linear functional $\ell: \wh{\Ch}^{1}(\CM) \to \BC$, e.g. $\ell=(-, \mathcal{C})$ for a proper $1$-cycle $\CC$ on $\CM$. We make three observations of the arithmetic modularity for $\ell$:
\begin{enumerate}
	\item  If $\ell$ is trivial on $\wh{\Ch}^{1}_{\mathrm{Vert}}(\CM)$, then Theorem \ref{Intro thm: mod} holds for $\ell$, i.e. $\ell(\wh{\CZ}(\phi))$ is modular by the known modularity on the generic fiber $\CM_E$. If $\ell=(-, \CC)$, this means $\CC$ has degree $0$ over $\BC$ and the intersection $(X_s, \CC)=0$ for any irreducible component $X_s$ of a non-smooth special fiber of $\CM$.
	\item (Comparison to automorphic generating functions in specific cases) For specific $\ell': \wh{\Ch}^{1}(\CM) \to \BC$, we may prove Theorem \ref{Intro thm: mod} for $\ell'$ directly. 
	
	Over $\BC$, we consider $\ell'=(-, \CC')$ for specific Hecke CM cycle $\CC'=\mathcal{CM}(\alpha_0, \mu_\Delta)$ (\S \ref{section: global KR}) where $\alpha_0$ has maximal order. Using simple cases of AFLs and ATs (proved in \S \ref{section: ATC unram max order}), we show that $\ell'(\wh{\CZ}(\phi))$ corresponds to $\alpha_0$-part of the analytic generating function $\partial \BJ(h, \Phi')$  in \S \ref{section: global analytic and geometric side}, which is known to be modular by Poisson summation formula.  
	
	Over a non-smooth special fiber $\CM_{k_{\nu}}$ (over a place $v \not \in \Delta$ with residue field $\BF_q$), we consider $\ell'=(-, \CC')$ for any $1$-cycle $\CC'_\nu$ in the basic locus of $\CM_{k_{\nu}}$. Assuming the local modularity Conjecture \ref{local mod conj} on relevant Rapoport--Zink spaces, we show that under the basic uniformization (\S \ref{section: basic uni local-global}), $\ell'(\wh{\CZ}(\phi))$ corresponds to a classical theta series on the nearby totally definite hermitian space $V^{(v)}$ of $V$ at $v$, which is known to be modular by Poisson summation formula. We prove the local modularity conjecture for \emph{very special} (\ref{very special}) $1$-cycles (Theorem \ref{thm: partial local modularity}) and use them as specific functionals $\ell'$.
	
	The local modularity Conjecture \ref{local mod conj} is quite interesting as these Deligne--Lusztig varieties appearing in the reduced locus are not of Coxeter type in general \cite[Section 6]{ADLV-Coexter} and mysterious.

	\item  (Modification via complex and mod $p$ geometry) To prove Theorem \ref{Intro thm: mod} in general, under our assumptions we show that for any given $\ell$, we may find specific $\ell'$ such that $\ell-\ell'$ is trivial on $\wh{\Ch}^{1}_{\mathrm{Vert}}(\CM)$. Hence the modularity for $\ell$ is proved by the known modularity of $\ell-\ell'$ and $\ell'$.  
	
	To do the modification i.e. show the existence of such specific $\ell'$ for given $\ell$, we use \emph{global geometry} of Shimura varieties. Over $\BC$, we show the unitary Shimura variety $M_{G,K}$ is connected over $E$ (assuming $E=F$) by studying the Galois and Hecke action on connected components of complex Shimura varieties.
	
	 Over a non-smooth special fiber $\CM_{k_{\nu}}$, we introduce the \emph{Balloon--Ground stratification} (\S \ref{section: global KR strata})
	$$
	\CM_{k_{\nu}}= \CM_{k_{\nu}}^\bullet \cup \CM_{k_{\nu}}^\circ.
	$$
	
		Our construction may be thought as a moduli definition of top dimensional Kottwitz--Rapoport strata, which is not considered before in the unitary case.  Irreducible components of  $\CM_{k_{\nu}}$ are given by irreducible components of $\CM_{k_{\nu}}^\bullet$ and $\CM_{k_{\nu}}^\circ$.  If $t_v=1$, then irreducible components of the Balloon stratum $\CM_{k_{\nu}}^\circ$ are isomorphic to projective spaces $\mathbb P^{n-1}_{k_\nu}$ and contained in the basic locus of $\CM_{k_\nu}$ (\cite[Section 5]{LTXZZ}). And the Ground stratum $\CM_{k_{\nu}}^\bullet$ is irreducible in any given connected component of $\CM_{k_\nu}$, see the Figure below. Therefore, we can choose $\CC'$ to be a finite linear combination of $\BP^1$ in different irreducible components $\BP^{n-1}$ of the Balloon stratum. 

\begin{figure}[h]\label{Figure 1}
	\centering
	
	\tikzset{every picture/.style={line width=0.75pt}}
	
	\begin{tikzpicture}[x=0.75pt,y=0.75pt,yscale=-0.65,xscale=0.65]
		
		\draw    (200,4750) -- (650,4750) ;

		\draw [fill={rgb, 255:red, 244; green, 244; blue, 244 }  ,fill opacity=1 ]   (283,4750) .. controls (277,4700) and (248, 4640) .. (284,4640) ;
		\draw [fill={rgb, 255:red, 244; green, 244; blue, 244 }  ,fill opacity=1 ]   (283,4750) .. controls (293,4700) and (320,4640) .. (284,4640) ;
		
		\draw [fill={rgb, 255:red, 244; green, 244; blue, 244 }  ,fill opacity=1 ]   (383,4750) .. controls (377,4700) and (348, 4640) .. (384,4640) ;
		\draw [fill={rgb, 255:red, 244; green, 244; blue, 244 }  ,fill opacity=1 ]   (383,4750) .. controls (393,4700) and (420,4640) .. (384,4640) ;

		\draw [fill={rgb, 255:red, 244; green, 244; blue, 244 }  ,fill opacity=1 ]   (483,4750) .. controls (477,4700) and (448, 4640) .. (484,4640) ;
		\draw [fill={rgb, 255:red, 244; green, 244; blue, 244 }  ,fill opacity=1 ]   (483,4750) .. controls (493,4700) and (520,4640) .. (484,4640) ;

		\draw [fill={rgb, 255:red, 244; green, 244; blue, 244 }  ,fill opacity=1 ]   (583,4750) .. controls (577,4700) and (548, 4640) .. (584,4640) ;
		\draw [fill={rgb, 255:red, 244; green, 244; blue, 244 }  ,fill opacity=1 ]   (583,4750) .. controls (593,4700) and (620,4640) .. (584,4640) ;
		
	\end{tikzpicture}
	
	\caption{``Balloons'' (with grey colors) on the ``ground'' (the bottom line) for almost self dual $L_v$}
	\label{fig:Balloon}
\end{figure}
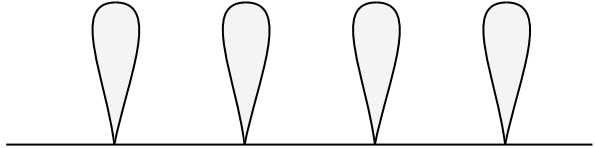

The case $t_v=n-1$ is similar to the case $t_v=1$ by duality. If $1<t_v<n-1$, the level $K_v$ is a non-special parahoric. For dimensional reasons, the Balloon and Ground strata are not completely in the basic locus. Nevertheless, we show $\CM_{k_{\nu}}^\circ$ and $\CM_{k_{\nu}}^\bullet$ are both irreducible in any given connected component of $\CM_{k_\nu}$, which is related to the irreducibility of non-basic Kottwitz--Rapoport strata at parahoric levels (Conjecture \ref{Important conj: irred} (2)) studied in the work of van Hoften \cite{IrredKR}. Then we choose a very special $1$-cycle $\CC'$ lying in the basic locus but not in $\CM_{k_{\nu}}^\bullet \cap \CM_{k_{\nu}}^\circ$ to do the job.

\end{enumerate}

	In fact we obtain a double modularity result (i.e., for both special divisors $\CZ(\phi_L)$ and $\CZ(\phi_L^\vee)$) as ``vector-valued'' modular forms. Their generating series may be related via the action of local Weil representation on hermitian spaces $\{ V_v \}_v$, see $\S$ \ref{section:Gmod}. As in the work on AFL \cite{AFL}, the modularity of arithmetic theta series plays a similar role in the proof of arithmetic transfer identities as the ``perverse continuation principle'' in the geometric approach (see \cite{Ngo-FL}, \cite{Yun-Duke} ) to fundamental lemmas over function fields.
	
	\subsection{Arithmetic transfer identities via modularity and double induction}
	
	Our approach to arithmetic transfers is local-global.
	
	\begin{itemize}
		\item (Local aspects)  
		We introduce the \emph{relative Cayley map} as a more natural reduction tool (\ref{relative Cayley map}) for a decomposition of $F$-vector spaces $V=V^\flat  \oplus Fe$. The unitary (\ref{UnitaryCayley}) and symmetric (\ref{SymCayley}) variants behave well for all primes $p \geq 2$. They give reductions of orbits and (derived) orbital integrals \S \ref{section: red via relative Cayley}. 
		
		We establish the reduction of intersection numbers in \S \ref{section: reduction Int}. A key point is the smallness of the resolution $\wt{\CN \times \CN} \to \CN \times \CN$ similar to the smallness of the Atiyah flop. Given an orthogonal decomposition of vertex lattices $L=L^\flat \oplus O_Fe$, consider \emph{the unitary relative Cayley map} for the decomposition $\BV=\BV^\flat \oplus Fe^{[i]}$
		\[
		\begin{gathered}
			\xymatrix@R=0ex{
				\mathfrak{c}_\U:  \U(\BV) \ar[r]  &  \U(\BV^{\flat}) \times \BV^{\flat}  \\
				g'={\begin{pmatrix}
						t & u \\ w  & d 
				\end{pmatrix}}  \ar@{|->}[r]  &  (g=t +\frac{uw}{1-d}, \, u_1=\frac{u}{1-d}). 
			}
		\end{gathered}
		\] 
		For example when $i=0$, for $g' \in \U(\BV)(F_0)_\rs$ with $1-d \in O_F^\times$ and $u_1 \not =0$ we have \emph{the reduction equality}:
		$$
		\wt {\Int}(g') = \wt {\Int}^{\CZ}(g, u_1).
		$$
		In conclusion,	we establish a reduction from group version AT conjectures (\ref{conj: gp version ATC}) of rank $n$ to semi-Lie version AT conjectures (\ref{conj: semi-Lie version ATC}) of rank $n-1$.  Moreover, we show semi-Lie AT conjectures for $\CZ$-cycles on $\CN_n^{[t]}$ and $\CY$-cycles on $\CN_n^{[n-t]}$ are equivalent using the dual isomorphism $\lambda: \CN_n^{[t]} \to \CN_n^{[n-t]}$ (\ref{dual lambda}).  Hence we may focus on semi-Lie version AT conjectures for $\CZ$-cycles.
		
		\item (Global aspects) We globalize the local data to cycles on suitable global RSZ unitary Shimura varieties as above. We follow the strategy of \cite{AFL}, but the original induction in \cite{AFL} for the proof of the AFL doesn't apply directly at parahoric levels. We introduce a method of double induction at parahoric levels in \S \ref{section: double induction} to prove one AT conjecture of rank $n$ using two AT conjectures of rank $n-1$ (hence the name of ``double induction''). Using the new (double) modularity result above as an input (we may enlarge  $\Delta$ to be sufficiently large), and local reductions and equivalences above as base cases, we finish the proof of AT conjectures. 
		
	\end{itemize}

	We also study the geometry of relevant Rapoport--Zink spaces and reconstruct the Bruhat--Tits stratification (\cite{Cho18, VW11}) on the reduced locus via special cycles. The phenomenon that there are two kinds of Bruhat--Tits strata at our parahoric levels is related to the formal Balloon--Ground stratification, see Remark $\ref{rek: Ballloon-Ground Bruat-Tits}$. We also use Bruhat--Tits strata to study fixed points of $g$ on $\CN$ in simple cases, see Example  \ref{Exa: Fix(g) on Pn}. A new pheonomenon in parahoric levels is that the behavior of $\Fix(g)$ in general depends on the embedding $O_F[g] \to \End(\BV)$ rather than just $O_F[g]$.  
	
	\subsection{Further directions}
	
	Firstly we comment on the assumptions. The assumption that $F_0$ is unramified over $\BQ_p$ is used to show the relevant Rapoport--Zink spaces appear in the basic uniformization of relevant unitary Shimura varieties. We hope to establish a comparison theorem and basic uniformization at general parahoric levels as in \cite{M-Thesis}. Applications include the proof of AT conjectures (including $p=2$) and understanding the cohomology of local Shimura varieties (e.g. the work of Hansen \cite{Hansen}) for any $p$-adic field $F_0$ via globalization.  Moreover, the modification method for arithmetic modularity could be applied to general levels.
	
	Our method may be further developed to formulate and prove arithmetic transfer identities for more general levels (e.g. the Iwahori level) and other set ups (e.g. \cite{HSY20, AFL-linear, RZ-Drinfeld space, AFL-More}). In general, the singularity of suitable integral model is controlled by certain sections of universal vector bundles (or ``Shtukas'' as in \cite{Pappas-Rapoport}), and we may blow up along the local model diagram (using Kottwitz--Rapoport strata) to resolve the singularity. They will be used in the proof of AGGP conjectures in higher dimensions over number fields with ramifications.  
	
	Finally, we give an overview of the paper, which is divided into $4$ parts. We formulate the set up of Jacquet--Rallis transfers in \S \ref{section: local set up transfer}. In \S \ref{section: red via relative Cayley}, we introduce the relative Cayley map and do reductions for orbits and (derived) orbital integrals. In \S \ref{section: TC proof}, we establish the Jacuqet--Rallis transfer for vertex lattices using a double induction and local uncertainty principles. 
	
	In Part $2$, we go into the local geometric story. In \S \ref{section:local RZ}, we introduce relevant moduli spaces and cycles, and use the formal Balloon--Ground stratification to resolve the singularity hence modify derived fixed points . In \S \ref{section: conj ATC}, we formulate semi-Lie and group version arithmetic transfer conjectures, and use relative Cayley map and derived intersection theory to show the reduction of intersection numbers. In \S \ref{section: local BT strata}, we introduce the Bruhat--Tits stratification. In \S \ref{section: local modularity}, we formulate the local modularity conjecture and prove the local modularity for very special $1$-cycles. In \S \ref{section: ATC unram max order}, using moduli interpretations and Bruhat--Tits strata we prove our ATCs in the unramified maximal order case.
	
	In Part $3$ and $4$, we go into the global story. In \S \ref{section: global Shimura var}, we introduce RSZ Shimura varieties for parahoric hermitian lattices and study the singularity using Balloon--Ground stratification and local models. In \S \ref{section: global KR}, we introduce relevant special cycles on Shimura varieties. In \S \ref{section: mod over C and F} and \ref{section:Gmod}, we formulate prove global modularity by modification over $\BC$ and $\BF_q$. In \S \ref{section: global analytic and geometric side}, we introduce the global analytic side and geometric side. In \S \ref{section: the end}, we do a double induction and prove our arithmetic transfer conjectures by globalization.
	
	\subsection*{Acknowledgments} 
	I heartily thank my advisor Wei Zhang for introducing me into this beautiful area and his constant encouragement. I enjoy discussions with him during different stages of this project. I thank P. van Hoften, Y. Liu, A. Mihatsch, L. Xiao and Z. Yun for interesting and helpful discussions. I also thank C. Li and M. Rapoport for helpful comments on a draft.
	\subsection*{Notations and Conventions}

	\begin{itemize}
		\item 
		Unless otherwise stated, any hermitian space $V$ over a quadratic extension $F/F_0$ in this article is assumed to be non-degenerate. We always denote by $(\cdot,\cdot)_V$ the hermitian form on $V$, which is by definition $F$-linear on the first factor and conjugate symmetric. The $F_0$-valued bilinear form $(x, y) \to \tr_{F/F_0} (x, y)_V$ makes $V$ a quadratic space over $F_0$. 	
		\item Let $X$ be an affine variety with an action of a reductive group $G$ over a field $F_0$. Then an element $x \in X(F_0)$ is called regular semi-simple if and only if the stabilizer of $x$ inside $G$ is trivial, and the orbit of $x$ is closed in $X$ under the Zariski topology. If $F_0$ is a $p$-adic field, this is equivalent to that $G(F_0)x$ is closed in $X(F_0)$ under the analytic topology.
	\end{itemize}

	\subsubsection*{Local notations}
	
	\begin{itemize}
		\item
		In local set up, we denote by $F/F_0$ an unramified quadratic extension of $p$-adic fields with residue fields $\BF_{q^2}/\BF_q$, where $p \geq 2$. Denote by $\varpi=\varpi_{F_0}$ a uniformizer of $F_0$, and write the non-trivial Galois involution on $F$ by $x \mapsto \overline{x}$. The standard valuation map is denoted by $v=v_F: F^{\times} \surj \mathbb Z$, where $v_F(\varpi)=1$. Choose a unit $\delta \in O_{F_0}$ such that $O_F=O_{F_0}[\sqrt{\delta}]$ and $\bar {\delta}=-\delta$.
		\item
		We denote by $\eta: F_0^{\times} \surj F_0^{\times}/ NF^{\times} \cong \{\pm 1\} \subseteq \mathbb C^{\times}$  the quadratic character associated to $F/F_0$ by local class field theory. We extend it to $\eta: F^{\times} \rightarrow \mathbb C^{\times}$ by $\eta(x)=(-1)^{v_F(x)}$.
		\item
		For an $O_F$-lattice $\Lambda \subseteq V$ (of full rank), denote by $\Lambda^{\vee}= \{ x \in V | (x,\Lambda)_V  \subseteq  O_F\}$ its dual lattice. A lattice $\Lambda$ is called a \emph{vertex lattice} if $\Lambda \subset \Lambda^{\vee} \subset \varpi^{-1} \Lambda$. The type of a vertex lattice $\Lambda$ is $t(\Lambda) := \dim_{k_F}\Lambda^{\vee} /\Lambda \in [0, \dim_F V]$. 
		A \emph{self-dual} lattice is a vertex lattice of type $0$.  An \emph{almost self-dual} lattice is a vertex lattice of type $1$.  
		\item For an $O_{F_0}$-lattice $\Lambda \subseteq V_0$ (of full rank) in a $F_0$-vector space $V_0$, denote by $\Lambda^*=\{ y \in (V_0)^* | y(\Lambda) \subseteq O_{F_0} \}$ its linear dual lattice in $(V_0)^*$.
		Let $\breve{F}$ be the $p$-adic completion of the maximal unramified extension of $F_0$,  $W=O_{\breve{F}}$ be its ring of integers, and $\sigma \in \Gal(\breve{F}/F_0)$ be the Frobenius element $x \rightarrow x^q$.
		\item We fix a non-trivial unramified additive character $\psi: F_0 \rightarrow \mathbb C$, which induces a character $\psi_F: F \rightarrow \mathbb C$ given by $\psi_F:=\psi \circ \tr_{F/F_0}$.
		\item For a smooth affine algebraic variety $X$ over $F_0$, let $\CS(X(F_0))$ be the space of Schwartz functions on $X(F_0)$. Here all Schwartz functions on totally disconnected topological spaces are $\BQ$-valued locally constant functions. 	
		\item  For a $F/F_0$-hermitian space $V$, denote by $\omega$ the Weil representation of $\SL_2(F_0)$ on $\CS(V(F_0))$, which commutes with the natural right translation action of $\U(V)(F_0)$. 
		\item We use covariant version of the relative Dieudonne theory.
		\item For any $\xi \in \BR$ and $n \in \BZ$, let $W_{\xi}^{(n)} (h)$  be the standard weight $n$ Whittaker function on $h \in \SL_2(\BR)$. More concretely, for $h \in \SL_2(\BR)$ under the Iwasawa decomposition 
		$$h=\begin{pmatrix}
			1 & b \\ 0 & 1
		\end{pmatrix}\begin{pmatrix}
			a^{1/2} & 0 \\ 0 & a^{-1/2}
		\end{pmatrix} \kappa_\theta, \quad a \in \BR_{+}, \quad b \in \BR, $$ 
		and 
		$$
		\kappa_\theta=\begin{pmatrix}
			\cos \theta & \sin \theta \\ -\sin \theta & \cos \theta
		\end{pmatrix} \in \mathrm{SO}_2(\BR),
		$$
		we have 
		\begin{equation}\label{defn: Whittaker function on SL2}
			W_{\xi}^{(n)} (h):=|a|^{n/2} e^{2\pi \xi i (b+a i)} e^{in\theta}.
		\end{equation}
	\end{itemize}

	\subsubsection*{Global notations}

	\begin{itemize}
		\item In global set up, we always denote by $F/F_0$ a CM quadratic extension over a totally real number field. Let $F_{0,+}$ be the subset of totally positive elements in $F_0$.
		\item We denote by $\BA$, $\BA_0$, and $\BA_F$ for the adele rings of $\BQ$, $F_0$, and $F$, respectively. We use the subscript $(-)_f$ to denote the terms of finite adeles. We use the subscript $(-)_p$, $(-)_v$, and $(-)_w$ to denote the local term at a place $p$ of $\BQ$, a place $v$ of $F_0$, and a place $w$ of $F$, respectively.
		\item Let $\Delta$ be a finite collection of places of a number field $F_0$. We use the lower index $(-)_{\Delta}$ to denote the product of $(-)_v$ over all places $v \in \Delta$.
		\item We denote by $\Phi$ for a chosen CM type of $F$. A CM type $\Phi$ is called unramified at $p$, if $\Phi \otimes \BQ_p: F \otimes \BQ_p \to \ov \BQ_p$ is induced from a CM type of the maximal subalgebra $(F \otimes \BQ_p)^{u}$ of $F \otimes \BQ_p$ that is unramified over $\BQ_p$.
		\item For a smooth affine algebraic variety $X$ over $F_0$, let $\CS(X(\BA_0))$ be the space of Schwartz functions on $X(\BA_0)$.
		\item We use $\Hom^{\circ}(-, -)$ to mean $\Hom(-,-) \otimes \BQ$. All $K$-groups and Chow groups have $\BQ$-coefficients. 
		\item We use the standard additive character $\psi_{\BQ}: \BQ \backslash \BA  \rightarrow \mathbb C$, which induces an additive character $\psi_{F_0}: F_0 \backslash \BA_0 \rightarrow \mathbb C$ given by $\psi_{F_0}:=\psi_\BQ \circ \tr_{F_0/\BQ}$.
		\item 	
		Denote by $N^+ \leq \SL_2$ the subgroup of upper triangular unipotents. For any continuous function $f: \SL_2(\BA_0) \to \mathbb C$ that is left $N^+(F_0)$-invariant and $\xi \in F_0$, the $\xi$-th Fourier coefficient of $f$ is the following function on $h \in \SL_2(\BA_0)$:
		\begin{equation}\label{defn: Fourier coeff}
			W_{\xi, f}(h)=\int_{F_0 \backslash \BA_0}  f\left[ \begin{pmatrix} 1&b\\
				&1
			\end{pmatrix} h\right]\psi_{F_0}(-\xi b) db.
		\end{equation}		
		\item For a quadratic space $V$ over $F_0$, and $\xi \in F_0$, we denote by $V_{\xi}$ the $F_0$-subscheme of $V$ defined by $(x,x)_V=\xi$.
		\item For a $F/F_0$-hermitian space $V$, denote by $\omega$ the Weil representation of $\SL_2(\mathbb A_{0,f})$ on $\CS(V(\BA_{0,f}))$, which commutes with the natural right translation action of $\U(V)(\BA_{0,f})$. 
		\item We denote by $\CA_{\rm hol}(\SL_2(\BA_0), K_0, n)$ the space of automorphic forms on $\SL_2(\BA_0)$ (left invariant under $\SL_2(F_0)$) that is of parallel weight $n$, holomorphic and right invariant under a chosen compact open subgroup $K_0 \subseteq \SL_2(\BA_{0, f})$. 
	\end{itemize}

	\part{Jacquet--Rallis transfers}\label{Part 1}

	\section{Explicit transfer conjectures}\label{section: local set up transfer}
	
	Let $p$ be a prime. Let $F/F_0$ be an unramified quadratic extension of $p$-adic local fields, and $\eta$ be the associated quadratic character $\eta(x):=(-1)^{v_F(x)}: F^{\times} \to \{\pm 1\}$. In this section, we recall the set up of Jacquet--Rallis transfers and formulate explicit transfer conjectures for suitable hermitian lattices.
	
	\subsection{Semi-Lie version transfers} 
	Let $L$ be a hermitian lattice in a $n$-dimensional $F/F_0$ hermitian space $V$. We consider the natural right action of the unitary group $\U(V)$ on $\U(V) \times V$
	\[ 
	h.(g,u)=(h^{-1}gh, \, h^{-1}u).
	\]
	For $f \in \mathcal{S}((\U(V) \times V)(F_0))$, define its orbital integral at $(g,u) \in (\U(V) \times V) (F_0)_\rs$ : 
	\[ 
	\Orb((g, u), f) := \int_{h \in \U(V)} f(h.(g,u)) d h.
	\]
	Choose an orthogonal basis $\{e_i\}_{i=1}^n$ of $L$, which exists for $p \geq 2$ by \cite[Section 7]{classicalherm}. Let $\ov{(-)}$ be the induced Galois involution on $V$ with fixed subspace $V_0:=\oplus_{i=1}^n F_0 e_i$. Let $L_0=\oplus_{i=1}^n O_{F_0} e_i$. Consider the symmetric space over $F_0$:
	\begin{equation}
		S(V_0)=\{ \gamma \in \GL(V) | \gamma \ov \gamma =\id \}.
	\end{equation}
	The general linear group $\GL(V_0)$ acts on the space $S(V_0) \times V_0 \times V_0^*$ by
	$$
	h.(\gamma, \, u_1, \, u_2)=(h^{-1}\gamma h, \, h^{-1}u_1, \, u_2h).
	$$
	Define the \emph{transfer factor} on $(S(V_0) \times V_0 \times V_0^*)(F_0)_\rs$ with respect to $L$ by
	\begin{equation}\label{transfer factor: semi-Lie}
		\omega(\gamma, u_1, u_2) :=\eta( \det (\gamma^{i} u_1 )_{i=0}^{n-1}) \in \{\pm 1\}.
	\end{equation} 
	For $f' \in \mathcal{S}((S(V_0) \times V_0 \times V_0^{*})(F_0))$ and $(\gamma,u_1,u_2) \in  (S(V_0) \times V_0 \times V_0^*)(F_0)_\rs$, define twisted orbital integrals:
	\begin{equation}
		\Orb( (\gamma, u_1, u_2), f', s)= \omega(\gamma, u_1, u_2)  \int_{h \in \GL(V_0)} f'(h.(\gamma,u_1,u_2)) \eta(h) |h|^{s} d h, \quad s \in \BC.
	\end{equation}
	\begin{equation}
		\Orb((\gamma, u_1, u_2), f'):= \Orb((\gamma, u_1, u_2), f', 0).
	\end{equation}
We introduce the derived orbital integral:
	\begin{equation}\label{derived orbital int}
		\del\bigl((\gamma, u_1, u_2), f' \bigr):=\frac{d}{ds}\Big|_{s=0}  \Orb((\gamma, u_1, u_2), f', s).
	\end{equation} 
	Here we write $ \eta(h):=\eta(\det h)$ and $|h|^s:=|\det h|^s$ for short.  These integrals converge absolutely for regular semi-simple elements. By \cite{JR-GGP}, an orbit $(g, u) \in (\U(V) \times V )(F_0)$ is regular semi-simple if and only if $\{ g^iu \}_{i=0}^{n-1}$ spans the space $V$. 
	\begin{definition}\label{defn: semi-Lie match elements}
		We say $(g,u) \in (\U(V) \times V )(F_0)_\rs$ and $(\gamma,u_1,u_2) \in  (S(V_0) \times V_0 \times V_0^*)(F_0)_\rs$ match if they are conjugated by $\GL(V)$ in $\End(V) \times V \times V^*$, where we use the embedding $	\U(V) \times V   \to \End(V) \times V \times V^{*}$ given by $(g,u) \to (g, \, u, \, u^*:=(x \mapsto (x,u)_V)).$ 
	\end{definition} 
	The matching relation is equivalent to matching of the following invariants \cite{JR-GGP}:
	\[
	\det(T \id_{V}+g)=\det(T \id_{V}+\gamma) \in F[T], \quad (g^iu,u)=u_2 (\gamma^iu_1), \quad 0 \leq i \leq n-1.
	\]	
	The matching relation gives a natural bijection of orbit spaces of regular semisimple elements:
	\begin{align}
		\xymatrix{
			\bigl[ (\U(V) \times V)(F_0)\bigr]_\rs \coprod \bigl[(\U(\BV) \times \BV)(F_0) \bigr]_\rs  \ar[r]^-\sim& \bigr[(S(V_0) \times V_0 \times V_0^*)(F_0) \bigr]_\rs}
	\end{align}
	where $\BV$ is the nearby hermitian space of $V$ and the action of $\U(\BV)$ on $\U(\BV) \times \BV$ is defined similarly.
	
	Define $L_0^\vee:=L^\vee \cap V_0$. From now on, we normalize the Haar measure on $\U(V)(F_0)$ (resp. $\GL(V_0)(F_0)$) such that the stabilizer $\U(L)$ (resp. the stabilizer $\GL(L_0, L_0^{\vee})$ of $L_0 \subseteq L_0^\vee$) is of volume $1$. 
	
	\begin{definition}[Semi-Lie Transfers]\label{Semi-Lie Tranfers}
We say $f' \in \CS((S(V_0) \times V_0 \times V_0^{*})(F_0))$ and $f \in \CS((\U(V) \times V)(F_0))$ are \emph{transfers} to each other, if for any regular semisimple $(\gamma, u_1, u_2) \in (S(V_0) \times V_0 \times V_0^*)(F_0)_\rs$, 
	$$
	\Orb((\gamma,u_1,u_2),f')= \begin{cases}
		\Orb((g,u), \, f) & \text{if $(\gamma, u_1, u_2)$ matches some $(g, u) \in (\U(V) \times V)(F_0)_\rs$,} \\
		0 & \text{else.}
	\end{cases}
	$$ 
	\end{definition}

	\subsection{Group version transfers}  Consider an orthogonal decomposition
	$L=L^{\flat} \oplus O_Fe$ of hermitian lattices. Choose an orthogonal basis of $L^\flat$ which gives a basis of $L$ by adding $e$. Similarly, we consider the conjugacy action of $\U(V^\flat)$ (resp. $\GL(V_0^{\flat})$) on $\U(V)$ (resp. $S(V_0)$). 
	\begin{definition}\label{defn: gp match elements}
		We say $g \in \U(V)(F_0)_\rs$ and $\gamma \in  S(V_0)(F_0)_\rs$ match if they are conjugated by $\GL(V^\flat)$ in $\End(V)$.
	\end{definition} 
The matching relation is equivalent to matching of following invariants \cite{JR-GGP}:
\[
\det(T \id_{V}+g)=\det(T \id_{V}+\gamma) \in F[T], \quad (g^ie, e)=e^* (\gamma^ie), \quad 0 \leq i \leq n-1.
\]	
	
	Similarly, we formulate the matching relation and orbital integrals as in \cite{JR-GGP}. In particular, for $f' \in \mathcal{S}(S(V_0)(F_0))$ and $\gamma \in S(V_0)(F_0)_\rs$, we define ($s \in \mathbb C$)
	\[ 
	\Orb(\gamma, f', s):= \omega(\gamma)  \int_{h \in \GL(V_0^{\flat})} f'(h^{-1}\gamma h) \eta(h) |h|^{s} d h,
	\]
	where the \emph{transfer factor} on $S(V_0)(F_0)_\rs$ with respect to $L$ is defined by
	\begin{equation}\label{transfer factor: gp}
		\omega(\gamma) :=\eta( \det (\gamma^{i} e )_{i=0}^{n-1}) \in \{\pm 1\}.
	\end{equation}

	We normalize the Haar measure on $\U(V^{\flat})(F_0)$ (resp. $\GL(V^{\flat}_0)(F_0)$) such that the stabilizer $\U(L^{\flat})$ (resp. $\GL(L_0^{\flat},L_0^{\flat, \vee})$) is of volume $1$. 
	
	\begin{definition}[Group Transfers]
	We say $f' \in \CS((S(V_0)(F_0))$ and $f \in \CS(\U(V)(F_0))$ are \emph{transfers} to each other, if for any regular semisimple $\gamma \in S(V_0)(F_0)_\rs$, 
	$$
	\Orb(\gamma, f')= \begin{cases}
		\Orb(g, \, f) & \text{if $\gamma$ matches some $g \in \U(V)(F_0)_\rs$,} \\
		0 & \text{else.}
	\end{cases}
	$$ 
\end{definition}	
	
	\subsection{Conjectures}
	If $L$ is a hermitian lattice in $V$, define the \emph{valuation} of $L$ as 
	\[
	v(L):=\min_{x \in L, y \in L} v((x, y)_V).
	\]
	We say $L$ is \emph{maximal parahoric} if $v(L) + v(L^\vee) \geq -1$.  Recall $L$ is a \emph{vertex lattice} if $L \subseteq L^\vee \subseteq \varpi^{-1} L$, and the type of $L$ is $t(L):=\dim_{\BF_{q^2}} L^\vee/L$.  Using an orthogonal basis for $L$, we see that $L$ is maximal parahoric if and only if some scalar of $L$ or $L^\vee$ is a vertex lattice. 
	
	Define $L_0^{\vee}:=L^\vee \cap V_0$, and the linear dual of $L_0^{\vee}$ as $L_0^{\vee*}:=\{ f \in V^*_0| f(L_0^{\vee}) \subseteq O_{F_0} \}$. Let $S(L_0, L_0^\vee)$ be the stabilizer of $L_0 \subseteq L_0^\vee$ inside $S(V_0)$.  We formulate the following \emph{explicit Jacquet--Rallis transfer conjectures} (TC). 
	
	\begin{conjecture}{\label{conj: TC}}
		
		\begin{enumerate}
			\item  (Semi-Lie transfer for $L$) Let $L$ be a maximal parahoric hermitian lattice, then
			$f_L:=1_{\U(L)} \times 1_{L}$ and $f_{L}':=1_{S(L_0, L^{\vee}_0)} \times 1_{L_0} \times 1_{L_0^{\vee*}} $ 
			are transfers.
			\item  (Group transfer for the pair $(L,e)$)  Let $L=L^{\flat} \oplus O_Fe$ be an orthogonal decomposition of maximal parahoric hermitian lattices. Then $f_{(L,e)}:=1_{\U(L)}$	and  $f_{(L,e)}':=1_{S(L_0,L^{\vee}_0)}$ are transfers.
		\end{enumerate}
	\end{conjecture}
	By definition, the group TC for the pair $(L, e)$ is equivalent to the group TC for the pair $(L^\vee, e/(e,e)_V)$. 		
	
	\begin{remark}\label{rek: transfer factor and basis}
		Any $F$-basis of $V$ gives a trivialization $\wedge^{n} V \cong F$ hence a transfer factor $\omega(\gamma, u_1, u_2) :=\eta( \det (\gamma^{i} u_1 )_{i=0}^{n-1}) \in \{\pm 1\}$. For different $F$-basis of $V$, the associated transfer factors may differ up to a constant sign. In Conjecture \ref{conj: TC}, the transfer factors are defined using a basis of $L$ and depends only on $L$. We have $$\omega_{L^{\vee}}=\gamma_V \omega_L$$ where $\gamma_V=\eta(\det V) \in \{\pm1\}$ is the Weil constant of $V$.
	\end{remark}	
	
	We prove the explicit Jacquet--Rallis transfer conjectures. 
	
 	\begin{theorem}\label{proof of TC}
 	Conjecture \ref{conj: TC} holds for any maximal parahoric lattice $L$ i.e., $v(L)+v(L^\vee) \geq -1$ e.g. any vertex lattice $L$. Hence by Theorem \ref{thm: reduction orb}, the group version conjectures \ref{conj: TC}  also hold for any pair $(L, e)$.
 \end{theorem}
 
 For self-dual $L$, the method is similar to \cite{RBP19} which gives a pure local proof for the Lie algebra version Jacquet--Rallis fundamental lemma (FL). We handle the group version FL for $p=2$ which is not covered by \cite{RBP19}. In general, our strategy is to do reductions and inductions on two boundary cases via double uncertainty principle.

	\section{Reductions via relative Cayley maps}\label{section: red via relative Cayley}
	
	In this section, we introduce the relative Cayley map and do reductions of orbits and orbital integrals from group version to semi-Lie version. We have the reductions for $s$-variable orbital integrals, which is used for the reduction of arithmetic transfer conjectures in \S \ref{section: reduction Int}. 
	
	\subsection{The relative Cayley map}\label{relativeC}
	
	Let $V$ be a $F$-vector space with a decomposition $V=V^{\flat} \oplus Fe$. For any endomorphism $A' \in \End(V)$, we write 
	$$A'=
	\begin{pmatrix}
		a & b \\ c & d 
	\end{pmatrix} \in 	\begin{pmatrix}
		\End(V^\flat) & V^\flat \\ (V^{\flat})^* & F 
	\end{pmatrix} = \End(V).
	$$
	\begin{definition}
		The \emph{relative Cayley map} is the $\GL(V^{\flat})$-equivariant rational map
		\begin{equation}\label{relative Cayley map}
			\begin{gathered}
				\xymatrix@R=0ex{
					\mathfrak{c}: \End(V)  \ar[r]  &  \End(V^{\flat}) \times V^{\flat} \times (V^{\flat})^* \\
					A' \ar@{|->}[r]  &  (A=a+ \frac{bc}{1-d}, \, b_1= \frac{b}{1-d}, \, c_1=\frac{c}{1-d}). 
				}
			\end{gathered}
		\end{equation}
	\end{definition}

	\subsection{Reduction for unitary groups}
	
	Let $V$ be a $F/F_0$-hermitian space with an orthogonal decomposition $V=V^{\flat} \oplus Fe$. For any $g' =
	\begin{pmatrix}
		t & u \\ w & d 
	\end{pmatrix} \in \U(V), $ we have $d=\frac{(g'e,e)}{(e,e)}$ and $
	w(x)=\frac{(g'x,e)}{(e,e)}$ for $x \in V^{\flat}$.
	\begin{proposition}\label{Unitaryproj}
		Assume that $1-d \not =0$, consider the well-defined relative projection 
		$$ g = t +\frac{uw}{1-d}: V^{\flat} \to V^{\flat}.$$
		\begin{enumerate}
			\item We have $g \in \U(V^{\flat})$. If $x \in V^{\flat}$ such that $w(x)=0$, then $g(x)=g'(x)$.
			\item We have $\frac{w}{1-d}(x)=\frac{(gx,\frac{u}{1-d})}{(e,e)}$ for $x \in V^{\flat}$.
		\end{enumerate}
		
		Moreover, these properties uniquely characterize $g$ in terms of $g'$.  Given $g \in \U(V^\flat)$, $u \in V^\flat$ and $d \not =1$, there exists a unique $g' \in \U(V)$ with these invariants.
		
	\end{proposition}
	
	\begin{proof} 
		We have $d=\frac{(g'e,e)}{(e,e)}$ and $w(x)=\frac{(g'x,e)}{(e,e)}$. So for $x \in V^\flat$ we have $gx-g'x= - \lambda(x)(g'e-e)$
		where $\lambda(x):= \frac{(g'x, e)}{(g'e-e, e)}$. For $x, y \in V^\flat$, using $(g'e,  g'y)=(e,y)=0=(g'x, g'e)=(x,e)$ we find
		\[
		(gx,gy)-(x,y)= (gx,gy) - (g'x, g'y) = \lambda(x)(e, g'y) + \overline{\lambda(y)}(g'x, e) + \lambda(x) \overline{\lambda(y)} (g'e-e,g'e-e)
		\]
		which is $\lambda(x) \overline{\lambda(y)} ( \overline{(g'e-e,e)} + (g'e-e,e) + (g'e-e,g'e-e))=0.$  Hence $g \in \U(V^{\flat})$.
		
		If $w(x)=0$, then by definition $gx=tx=g'x$. For $x  \in V^{\flat}$, we compute $(gx,u)=$
		\[ 
		(g'x - \frac{(g'x,e)}{(g'e-e,e)}(g'e-e), g'e - d e)=(g'x,e) (0 - \frac{1-\overline{d}}{d-1}  - \overline{d} + \overline{d}) = \frac{1-\overline{d}}{1-d}  (g'x, e).
		\]
		As $w(x)=\frac{(g'x,e)}{(e,e)}$, we get the equality $
		\frac{w}{1-d}(x)=\frac{(gx,\frac{u}{1-d})}{(e,e)}.$ Finally, these properties determine $g^{-1}u$ and $g|_{\ker w}$ from $g'$, hence $g$ itself as $w(g^{-1}u)=\frac{(u,e)}{(e,e)}=1-d $ is non-zero. Conversely, we can reverse above computations and reconstruct $g'$.
	\end{proof}

	\begin{definition}
		The \emph{unitary relative Cayley map} is the rational map 
		\begin{equation}\label{UnitaryCayley}
			\begin{gathered}
				\xymatrix@R=0ex{
					\mathfrak{c}_\U:  \U(V) \ar[r]  &  \U(V^{\flat}) \times V^{\flat}  \\
					g' \ar@{|->}[r]  &  (g=t +\frac{uw}{1-d}, \, u_1=\frac{u}{1-d}). 
				}
			\end{gathered}
		\end{equation}
	\end{definition}
	
	The map is equivariant with respect to the natural conjugacy action of $\U(V^{\flat})$ on both sides.  By Proposition \ref{Unitaryproj}, we have
	\begin{corollary}\label{prop: red UnitaryInvariants}
		We have	$\mathfrak{c}(g')=(g, u_1, w_1)$ where $w_1(x)=\frac{(gx,u_1)}{(e,e)}$. The map $\mathfrak{c}_\U$ induces an isomorphism over the locus where $d \not =1$ is fixed. We have equalities of invariants ($i \in \BZ$):
		\[
		(g^iu_1,u_1)=(e,e)w_1(g^{i-1}u_1).
		\]
	\end{corollary}
	
	\begin{proposition}\label{regsemisimple:unitary}
		The element $g'$ is regular semisimple if and only if $\mathfrak{c}_\U(g')$ is regular semisimple.
	\end{proposition}
	\begin{proof}
		This follows from the $\U(V^\flat)$-equivariance of the relative Cayley map $\mathfrak{c}_\U$.

	\end{proof}
	
	Therefore, we have a natural reduction of $\U(V^{\flat})$-orbits. We now do reduction for $\U(V^{\flat})$-orbital integrals. Consider a hermitian lattice $L \subseteq V$ with an orthogonal decomposition $
	L=L^{\flat} \oplus O_Fe_L \subseteq V=V^\flat \oplus Fe$, where $e_L=\lambda_L e$ for some $\lambda_L \in F_0^{\times}$. 
	
	\begin{proposition}\label{Unitaryred}
		Let $g' \in \U(V)(F_0)$ with  $\mathfrak{c}_\U(g')=(g,u_1)$. Assume that  $1-d \in O_F^{\times}$, then for any $h \in \U(V^{\flat})$, $h^{-1}g'h \in \U(L)$ holds if and only if 
		$$
		h^{-1}gh \in \U(L^{\flat}), \quad h^{-1}u_1 \in \lambda_L^{-1} L^{\flat}, \quad h^{-1}u_1 \in (e,e)(\lambda_L^{-1} L^{\flat})^{\vee}.
		$$ 
	\end{proposition}
	\begin{proof}
		The map $\mathfrak{c}_\U$ is $\U(V^{\flat})$-equivariant and $d$ is $\U(V^{\flat})$-invariant, so we may assume $h=\id$. As $\det(g') \in F^{\Nm_{F/F_0}=1}$ is a unit, the condition $g' \in \U(L)$ is the same as $g' L \subseteq L$.  Under the decomposition, using $1-d \in O_F^{\times}$ we see that $g' L \subseteq L$ holds if and only if
		\[
		t(L^{\flat}) \subseteq L^{\flat}, \quad u_1 \in  \lambda_L^{-1} L^{\flat}, \quad w_1(L^{\flat}) \subseteq \lambda_L O_F. 
		\]
		The last two conditions imply $\frac{uw}{1-d} (L^{\flat}) \subseteq L^{\flat}$. Using this and $g=t+\frac{uw}{1-d} \in \U(V^\flat)$, we may replace the first condition $t(L^{\flat}) \subseteq L^{\flat}$ by $ g(L^{\flat}) = L^{\flat}$.
		By Proposition \ref{Unitaryproj} we have $w_1(x)=\frac{(gx, u_1)}{(e,e)}$, hence $w_1(L^{\flat}) \subseteq \lambda_L O_F$ is equivalent to $(L^{\flat}, u_1) \subseteq \lambda_L (e,e) O_F$ i.e., $u_1 \in (e,e) (\lambda_L^{-1} L^{\flat})^{\vee}$. 
	\end{proof}
	
	\begin{remark}
		We do not assume that $g'$ is regular semi-simple for this proposition. We have similar reduction result for intersection numbers on unitary Rapoport--Zink spaces in \S \ref{section: reduction Int}.
	\end{remark}

	\begin{corollary}\label{RedOrbUUU}
		For $g' \in \U(V)(F_0)_\rs$ with $1-d \in O_F^{\times}$ and any Haar measure on $\U(V^{\flat})(F_0)$, we have
		\[
		\Orb(g', 1_{\U(L)}) = \Orb(\mathfrak{c}_\U(g'), 1_{\U(L^{\flat})} \times 1_{(\lambda_L^{-1}L^{\flat} ) \cap (e,e)(\lambda_L^{-1}L^{\flat} )^{\vee}} ).
		\]
	\end{corollary}

	\begin{definition}\label{defn: red lat}
		We say the lattice $L=L^{\flat} \oplus O_F e_L$ is suitable for reduction if 
		$$ v_F(e_L,e_L)=\min_{x,y \in L} v_F((x,y)).$$  
		By definition, $L$ (resp. $L^\vee$) is suitable for reduction if and only if  $\lambda_L^{-1} L^{\flat} \subseteq (e,e)(\lambda_L^{-1}L^{\flat} )^{\vee}$ (resp. $(e,e)(\lambda_L^{-1}L^{\flat} )^{\vee} \subseteq \lambda_L^{-1} L^{\flat} $). 
	\end{definition}
	
	\begin{lemma}
		If $L$ is maximal parahoric, then either $L$ or $L^\vee$ is suitable for reduction.
	\end{lemma}	
	\begin{proof}
		Otherwise, we have $v_F(e_L, e_L) \geq v(L) + 1$ and $v_F(e_{L^\vee}, e_{L^\vee})=-v_F(e_L, e_L) \geq v(L^\vee)+1$, which sum together to give $0 \geq v(L)+v(L^\vee)  +2 \geq 1$, a contradiction.
	\end{proof}
	
	\begin{corollary}\label{SimpleU}
		If $L$ is suitable for reduction, then for any $g' \in \U(V)(F_0)_\rs$ with $1-d \in O_F^{\times}$ we have 
		\[
		\Orb(g', 1_{\U(L)}) = \Orb(\mathfrak{c}_\U(g'), 1_{\U(\lambda_L^{-1}L^{\flat})} \times 1_{(\lambda_L^{-1}L^{\flat})} ).
		\]
	\end{corollary}
	
	For $\xi \in F^{\Nm_{F/F_0}=1}$, we define the \emph{$\xi$-twisted unitary Cayley map} by 
	\begin{equation}
		\mathfrak{c}_{U, \xi}(g'):=\mathfrak{c}_\U (\xi g').
	\end{equation}
	
	For any element $g' \in \U(V)(F_0)_\rs$, we have $1-d\xi \in O_F^{\times}$ for infinitely many $\xi \in F^{N_{F/F_0}=1}$ because $\# \{ x \in \BF_{q^2}| \Nm_{\BF_{q^2}/\BF_q}(x)=1\}=q+1 > 1$. As $\xi$ is a unit, the elements $g'$ and $\xi g'$ have same orbital integrals for any $f \in \CS(\U(V)(F_0))$. Therefore, the reduction of orbital integrals holds for any regular semi-simple orbit in $\U(V)(F_0)$ under suitable $\xi$-twisted unitary Cayley maps.

	\subsection{Reduction for symmetric spaces}
	
	Consider a decomposition $V=V^{\flat} \oplus Fe$ of $F$-vector spaces. Choose a basis of $V^\flat$ to endow $V^\flat$ (resp. $V$) a $F_0$-rational structure $V_0^\flat$ (resp. $V_0$  by adding $e$), which induces a $F/F_0$ semi-linear involution $\overline{(-)}$ on $V$ with fixed subspace $V_0$. Denote also by $\overline{(-)}$ the induced involutions on $V^*$ and $\End(V)$.  We have the embedding of symmetric spaces:
	$$
	S(V_0^{\flat})= \{ \gamma \in \GL(V^{\flat}) | \gamma \overline{\gamma}=\id \} \hookrightarrow S(V_0)=\{ \gamma' \in \GL(V) | \gamma' \overline{\gamma'}=\id \}.
	$$
	Under the decomposition $V=V^{\flat} \oplus Fe$, we write any $\gamma' \in S(V_0)$ as $\gamma'=
	\begin{pmatrix}
		a & b \\ c & d 
	\end{pmatrix}.$
	\begin{proposition}\label{Symproj}
		Assume that $1-d \not =0$, consider the well-defined relative projection 
		$$\gamma= a +\frac{bc}{1-d} : V^{\flat} \to V^{\flat}. $$
		\begin{enumerate}
			\item We have $\gamma \in S(V^{\flat}_0)$. If $x \in V^{\flat}$ such that $c(x)=0$, then $\gamma(x)=\gamma'(x)$.
			\item We have 		$
			\gamma \overline{\frac{b}{1-d}}=\frac{b}{1-d}, \quad \overline{\frac{c}{1-d}} \gamma =\frac{c}{1-d}.$
		\end{enumerate}
		
		Moreover, these properties uniquely characterize $\gamma$ in terms of $\gamma'$.  Given $\gamma \in S(V_0^\flat), \, b \in V, \, c \in V^*$ and $d \not =1$ satisfying these properties, there exists a unique $\gamma' \in S(V_0)$ with these invariants.	
	\end{proposition}
	
	\begin{proof}
		By definition, if $x \in V^{\flat}$ such that $c(x)=0$, then we have $\gamma(x)=a(x)=\gamma'(x)$. From the equality $ \begin{pmatrix}
			a & b \\ c & d 
		\end{pmatrix}
		\begin{pmatrix}
			\overline{a} & \overline{b} \\ 
			\overline{c} & \overline{d} 
		\end{pmatrix} =\id$, we have
		$$ a \overline{a} = \id -b \overline{c}, \quad a\overline{b} = -\overline{d} b,  \quad c \overline{a}=-d\overline{c}, \quad c\overline{b}=1-d\overline{d}.
		$$
		We find $\gamma \overline{\gamma}= a \overline{a}+ a \frac{\overline{b}\overline{c} }{1-\overline{d}} + \frac{bc}{1-d} \overline{a} +  \frac{bc}{1-d} \frac{\overline{b}\overline{c} }{1-\overline{d}}$ is equal to $\id + b\overline{c} (-1 - \frac{\overline{d}}{1-\overline{d}} -\frac{d}{1-d} + \frac{1-d\overline{d}}{(1-d)(1-\overline{d})})=\id$.
		Therefore we have $\gamma \in S(V_0^\flat)$.  We have
		$$\gamma \overline{b}= a \overline{b}+ c\overline{b} \frac{b}{1-d}= -\overline{d} b + (1- d \overline{d}) \frac{b}{1-d}= \frac{1-\ov d}{1 -d}b, $$ 
		and similarly $\overline{c} {\gamma}= \overline{c} a + \overline{c} \frac{bc}{1-d}= - \overline{d}{c} + (1- d\overline{d}) \frac{c}{1-d}= \frac{1-\ov d}{1 -d}c$. The second proposition follows.  Conversely, the uniqueness and reconstruction of $g'$ follows by reversing the above computations.
	\end{proof}
	
	\begin{definition}
		The \emph{untwisted symmetric relative Cayley map} is the $\GL(V_0)$-equivariant rational map 
		\begin{equation}\label{SymCayley}
			\begin{gathered}
				\xymatrix@R=0ex{
					\mathfrak{c}_S:   S(V_0)   \ar[r]  &  S(V^{\flat}_0) \times V^{\flat} \times (V^{\flat})^* \\
					\gamma' \ar@{|->}[r]  &  (\gamma, \, b_1=\frac{b}{1-d}, \, c_1=\frac{c}{1-d}).
				}
			\end{gathered}
		\end{equation}
	\end{definition}
	Now we modify $\mathfrak{c}_S$ to make the image lying in $S(V^{\flat}_0) \times V^{\flat}_0 \times (V^{\flat}_0)^*$. 
	\begin{definition}
		A \emph{twisting element} for $\gamma \in S(V_0^{\flat})(F_0)$ is any $B_{\gamma} \in \GL(V^{\flat})(F)$ such that
		\[  \gamma=B_{\gamma} \overline{B_{\gamma}}^{-1} =  \overline{B_{\gamma}}^{-1} B_{\gamma}. \]
	\end{definition}
	
	By Hilbert $90$ for $F/F_0$, there is always some $B \in \GL(V^{\flat})$ such that $\gamma=B \overline{B}^{-1}$. Here the condition is slightly stronger. In particular, we have $\gamma B_{\gamma}=B_{\gamma} \gamma$ and $\gamma  \overline{B_{\gamma}}=\overline{B_{\gamma}} \gamma$.

	\begin{example}\label{IntHilbert90Exa}
		Choose any $\xi \in F^{\Nm_{F/F_0}=1}$ such that $\det(1+\xi \gamma) \not =0$. Write $\xi=a / \bar{a}$ for some $a \in F^{\times}$, then $a^{-1}(1+\xi \gamma)$ is a twisting element for $\gamma$.
	\end{example}
	Let $\gamma' \in S(V_0)(F_0)$ with   $1-d \not =0$ and $\mathfrak{c}_S(\gamma')=(\gamma, b_1, c_1)$. Choose a twisting element $B_{\gamma}$ for $\gamma$, set
	\[ 
	b_2:=B_{\gamma}^{-1}b_1, \quad c_2:= c_1 \overline{B_{\gamma}}.
	\]
	
	\begin{proposition}\label{prop: twisting}
		Choose any twisting element $B_{\gamma}$ for $\gamma$ as above.
		\begin{enumerate}
			\item We have $(b_2, c_2) \in V_0^{\flat} \times (V_0^{\flat})^*$, and we have following equalities of invariants ($i \in \mathbb Z$):
			$$
			c_2\gamma^{i+1} b_2 = c_1 \gamma^{i} b_1.
			$$
			\item The $\GL(V_0^{\flat})$-orbit of $(\gamma, b_2,c_2)$ doesn't depends on the choice of $B_\gamma$ and only depends on the $\GL(V_0^{\flat})$-orbit of $\gamma'$.
		\end{enumerate}
		
	\end{proposition}
	\begin{proof}
		By Proposition \ref{Symproj}, we have $\gamma \overline{b_1}=b_1$, $\overline{c_1} \gamma =c_1$. From
		$ \gamma=B_{\gamma} \overline{B_{\gamma}}^{-1}$, we see 
		$$\overline{B_{\gamma}^{-1}b_1}=B_{\gamma}^{-1}b_1, \quad \overline{c_1\overline{B_\gamma}}=c_1\overline{B_{\gamma}}$$
		which shows $(b_2, c_2) \in V_0^{\flat} \times (V_0^{\flat})^*$. Since $\gamma$ commutes with $B_{\gamma}$ and $\overline{B_{\gamma}}$, we get 
		$c_2\gamma^{i+1} b_2 = c_1 \gamma^{i+1} \overline{B_{\gamma}} B_{\gamma}^{-1} b_1= c_1 \gamma^{i}b_1.$ 
		
		Finally, if $B_{\gamma},  B_{\gamma}'$ are two twisting elements for $\gamma$, then $h_0:=B_{\gamma}^{-1}B_{\gamma}'$ is in $\GL(V_0^{\flat})(F_0)$ and commutes with $\gamma$. Hence 
		$h_0.(\gamma, b_2, c_2)=(\gamma, b_2', c_2').$ And for any $h \in \GL(V_0^{\flat})(F_0)$, the element $h^{-1} B_{\gamma} h$ is a twisting element for $h^{-1}\gamma h$. Property $(2)$ holds.
	\end{proof}

	\begin{definition}
		The \emph{relative symmetric Cayley map} on $\GL(V^\flat_0)$-orbits is the map  between $\GL(V_0^{\flat})(F_0)$-orbits
		\begin{equation}\label{SymTCayley}
			\begin{gathered}
				\xymatrix@R=0ex{
					\overline{\mathfrak{c}_S'}:   [S(V_0)(F_0)]    \ar[r]  &  [(S(V^{\flat}_0) \times V^{\flat}_0 \times (V^{\flat})^*_0)(F_0) ]  \\
					\gamma'  \ar@{|->}[r]  &  (\gamma, \, b_2:=B_{\gamma}^{-1}b_1, \, c_2:= c_1 \overline{B_{\gamma}}) 
				}
			\end{gathered}
		\end{equation}
		where $B_{\gamma}$ is any twisting element of $\gamma$. This is well-defined by Proposition \ref{prop: twisting}. We also write $\mathfrak{c}_S'(\gamma')=(\gamma, \, b_2, \, c_2)$  with respect to a chosen twisting element $B_{\gamma}$.
	\end{definition}
	
	\begin{proposition}
		The element $\gamma'$ is regular semisimple if and only if $\mathfrak{c}_S'(\gamma')=(\gamma, b_2 , c_2)$ is regular semisimple.
	\end{proposition}
	\begin{proof}
		By Proposition \ref{prop: twisting} we have $c_2\gamma^{i+1} b_2 = c_1 \gamma^{i} b_1$. So by the same argument in  Proposition \ref{regsemisimple:unitary} we see the equivalences:  $\mathfrak{c}_S'(\gamma')=(\gamma, b_2 , c_2)$ is regular semisimple $\Leftrightarrow$ $\mathfrak{c}_S(\gamma')=(\gamma, b_1 , c_1)$ is regular semisimple $\Leftrightarrow$ $\gamma'$ is regular semisimple.
	\end{proof}

	An element $\gamma \in S(V_0^{\flat})(F_0)$ is called \emph{integral} if its characteristic polynomial $\text{char}(T \id_{V^{\flat}} -\gamma)$ lies in $O_{F}[T]$.  Recall that $F/F_0$ is unramified and $p \geq 2$. 
	
	\begin{proposition} \label{IntHHH} (Integral Hilbert $90$)
		Let $\gamma \in S(V_0^{\flat})(F_0)$ be an integral element and $R=O_F[\gamma] \subseteq \End(V^\flat)$ be the commutative $O_F$-algebra generated by $\gamma$. Then there exists an element $B_{\gamma} \in (O_F[\gamma] )^{\times}$ such that $\gamma=B_{\gamma}\overline{B_{\gamma}}^{-1}$. 
	\end{proposition}
	\begin{proof}
		As $\gamma$ is integral and $\det(\gamma) \in F^{Nm_{F/F_0}=1}$ is a unit, we see $R$ is finite flat over $O_F$.  As $\overline{\gamma}=\gamma^{-1}$ lies in $R$ by Cayley--Hamilton theorem, we see $R$ is stable under the involution $\overline{(-)}$ on $\End(V^{\flat})$. Consider the finite flat commutative $O_{F_0}$-subalgebra 
		$
		R_0=\{ x \in R | x=\bar{x} \}
		$ of $R$.
		
		The ring map $R_0 \otimes_{O_{F_0}} O_F \to R$ is an isomorphism: we base change along the \'etale morphism $O_{F_0} \to O_F$ where things are trivial. If $p>2$,  we can see this by observing any $x \in R$ is in the image by the formula $x=\frac{x+\bar{x}}{2} + \frac{x+\bar{x}}{2\delta} \delta$ for chosen $\delta \in O_F^{\times}$ such that $\overline{\delta}=-\delta$. We are done by the commutative algebra lemma \ref{IntHHH: lemma} below. \end{proof}
	
	\begin{lemma}\label{IntHHH: lemma}
		For any finite flat commutative $O_{F_0}$-algebra $A_0$, consider the induced involution $\overline{(-)}$ on $A=A_0 \otimes_{O_{F_0}} O_F$ given by $a \otimes b \to a \otimes \overline{b}$. 	If $x \in A^{\times}$ satisfies $x\overline{x}=1$, then there exists $y \in A^{\times}$ such that $x=y\overline{y}^{-1}$.
	\end{lemma}
	
	\begin{proof}
		This is proved in \cite[Lem. 8.6]{Kottwitz-OrbGL3} and is a variant of Hilbert 90 for vector bundles on the semi-local ring $A_0$.
	\end{proof}
	Now we do reduction for orbital integrals. Consider any lattice $L$ in $V$ with a decomposition  $L=L^{\flat} \oplus O_Fe_L $
	that is stable under $\overline{(-)}$, where $e_L=\lambda_L e$ for some $\lambda_L \in F_0^{\times}$.  Then $L_0:=\{x \in L | \ov{x}=x\}$ is an $O_{F_0}$-lattice in $V_0$ such that $L=L_0 \otimes_{O_{F_0}} O_F$.	Consider the stabilizer for $L$:
	$$
	S(L):=\{\gamma \in S(V_0)(F_0)\,|\, \gamma L=L\}.
	$$
	\begin{proposition} \label{Symred}
		Let $\gamma' \in S(V_0)(F_0)$ with $\mathfrak{c}_S(\gamma')=(\gamma, b_1, c_1)$. Assume that $1-d \in O_F^{\times}$, then
		\begin{enumerate}
			\item  Assume that $\gamma$ is integral. Choose the twisting element $B_{\gamma} \in O_F[\gamma]^{\times}$ as in Proposition \ref{IntHHH}, and set 
			$$
			\mathfrak{c}_S'(\gamma')=(\gamma, b_2=B_{\gamma}^{-1}b_1, c_2=c_1 \overline{B_{\gamma}}).
			$$ 
			For any $h \in \GL(V^{\flat}_0)(F_0)$, we have $h^{-1}\gamma'h \in S(L)$ if and only if
			$$h^{-1}\gamma h \in S(L^{\flat}), \quad h^{-1}b_2 \in \lambda_L^{-1} L^{\flat}_0,  \quad c_2h \in (\lambda_L^{-1} L^{\flat}_0)^{*}.$$
			\item If $\gamma$ is not integral, then there is no $h \in \GL(V^{\flat}_0)(F_0)$ such that 
			$h^{-1}\gamma'h \in S(L)$ or $h^{-1}\gamma h \in S(L^{\flat})$.
		\end{enumerate}
	\end{proposition}
	\begin{proof}
		Because $\det(h^{-1}\gamma'h)$ is a unit, the condition $h^{-1}\gamma' h \in S(L)$ is equivalent to $h^{-1}\gamma' h L \subseteq L $. Under the decomposition $L=L^\flat \oplus O_F e$, we see (using $1-d \in O_F^{\times}$) that $h^{-1}\gamma' h L \subseteq L$ if and only if 
		$$
		h^{-1}ah(L^{\flat}) \subseteq L^{\flat}, \quad h^{-1}b_1 \in \lambda_L^{-1} L^{\flat}, \quad c_1h(L^{\flat}) \subseteq \lambda_L O_F. 
		$$
		The latter two conditions imply $\frac{bc}{1-d}(L^{\flat}) \subseteq L^{\flat}$. As $h^{-1}\gamma h=h^{-1} a h + \frac{bc}{1-d}$, above condition is equivalent to 	
		$$
		h^{-1}\gamma h (L^{\flat}) \subseteq L^{\flat}, \quad h^{-1}b_1 \in \lambda_L^{-1} L^{\flat}, \quad c_1 h (L^{\flat}) \subseteq \lambda_L O_F.
		$$
		 In particular $\gamma$ must be integral, the second part follows.
		
		Assume that $h^{-1}\gamma h (L^{\flat}) \subseteq L^{\flat}$ holds. By Proposition \ref{IntHHH},  $B_{\gamma}$ and $B_{\gamma}^{-1}$ lie in $O_F[\gamma]$ in particular $\det(B_{\gamma}) \in O_F^{\times}$. Hence we have $h^{-1} B_{\gamma} h (L^{\flat})=L^{\flat}$ and $h^{-1} \overline{B_{\gamma}} h (L^{\flat})=L^{\flat}$.  
		
		So the condition $h^{-1}b_1 \in \lambda_L^{-1} L^{\flat}$ (resp. $c_1 h (L^{\flat}) \subseteq \lambda_L O_F$) is equivalent to the condition $h^{-1}B_{\gamma}^{-1}h (h^{-1}b_1)=h^{-1}b_2 \in \lambda_L^{-1} L^{\flat}$ (resp. $c_2h (L^{\flat})=c_1 h ( h^{-1} \overline{B_{\gamma}} h) (L^{\flat})= c_1 h(L^{\flat}) \subseteq \lambda_L O_F$). As $(h^{-1}b_2, \, c_2h) \in V_0^{\flat} \times (V_0^{\flat})^*$, the result follows.
		
	\end{proof}
	
From now on, we choose and fix the twisting element $B_{\gamma} \in O_F[\gamma]^{\times}$ as in Proposition \ref{IntHHH}.
	
	\begin{corollary}\label{RedOrbSSS}
		For $\gamma' \in S(V_0)(F_0)_\rs$ with $1-d \in O_F^{\times}$, we have
		($s \in \BC$)
		\[
		\Orb(\gamma', 1_{S(L)}, s)= \Orb(\mathfrak{c}_S'(\gamma'),   1_{S(L^{\flat})} \times 1_{\lambda_L^{-1}L^{\flat}_0} \times 1_{(\lambda_L^{-1}L^{\flat}_0 )^{*}}, s).
		\]
		In general, for any collection of lattices $\{L_i\}$ of the form $L_i=L^{\flat}_i \oplus O_F\lambda_{L_i} e$, let $S(\{L_i\})$ be the stabilizer of all lattices $L_i$ inside $S(V_0)$. Then we have
		\[
		\Orb(\gamma', 1_{S(\{L_i\})}, s)=\Orb(\mathfrak{c}_S'(\gamma'),   1_{S(\{ L^{\flat} \} )} \times 1_{\cap_i \lambda_{L_i}^{-1}(L^{\flat}_i)_0} \times 1_{ ( \sum_i \lambda_{L_i}^{-1} (L^{\flat}_i)_0 )^{*}}, s).
		\]
	\end{corollary}
	
	\begin{proof}
		
		By Proposition \ref{Symred}, both integrals are zero unless $\gamma$ is integral.  Assume that $\gamma$ is integral and choose $B_{\gamma}$ as in Proposition \ref{Symred}, then two integrals have same supports. So we only need to show the equality of transfer factors $\omega(\gamma, b_2, c_2)=\omega(\gamma').$ From $\gamma^i b_2=B_{\gamma}^{-1} \gamma^i b_1$ and $\det(B_{\gamma}) \in O_F^{\times}$, we have $$\omega(\gamma, b_2, c_2)=\omega(\gamma, b_1, c_1)= \omega(\gamma, b, c).$$  By induction on $i \geq 0$, we see $\gamma^i b - a^{i} (b)$ is in the $F$-span of $\{ b, ab, \dots, a^{i-1}b \}$ so $\omega(\gamma, b, c)=\omega(a,b,c)$. As $\gamma'(x)=a(x)+c(x)e$ $(x \in V^{\flat})$, we have 
		$$ \omega(\gamma')=\det(e \wedge \gamma' e \dots \wedge \gamma'^{n-1} e) = \det(e \wedge b \wedge \gamma b \dots \wedge \gamma^{n-2} b) =\omega(\gamma, b, c).$$
		The result for $L$ follows. The reduction for any collection of lattices follows along the same line as in Proposition $\ref{Symred}$ and above arguments.
	\end{proof}

	For any $\xi \in F^{\Nm_{F/F_0}=1}$, define the $\xi$-twisted symmetric Cayley map by 
	$$\mathfrak{c}_{S, \xi}(\gamma'):=\mathfrak{c}_S (\xi \gamma'), \quad \mathfrak{c}_{S, \xi}'(\gamma'):=\mathfrak{c}_S' (\xi \gamma').$$
	For $\gamma' \in S(V_0)(F_0)_\rs$, we have $1-d\xi \in O_F^{\times}$ for infinitely many $\xi \in F^{\Nm_{F/F_0}=1}$ as $\# \{ x \in k_{F}| \Nm_{k_{F}/k_{F_0}}(x)=1\}=q+1 > 1$. The orbital integral is the same for $\gamma'$ and $\xi \gamma'$. Therefore, the reduction of orbital integrals holds for any regular semi-simple $\gamma' $, using suitable $\xi$-twisted symmetric Cayley maps.

	\subsection{Matching and the equivalence of two transfer conjectures}
	
	Return to our transfer Conjecture \ref{conj: TC}. Choose an orthogonal $F$-basis $\{e_i\}$ of $V^{\flat}$, and extend it to an orthogonal basis of $V$ by adding $e$. 
	
	\begin{proposition}\label{RedOrbits}
		If $g' \in \U(V)(F_0)_\rs$ matches $\gamma' \in S(V_0)(F_0)_\rs$, then $\mathfrak{c}_\U(g')=(g,u_1) \in (\U(V^{\flat}) \times V^{\flat})(F_0)$ matches 
		$$(e,e).\mathfrak{c}_S' (\gamma') :=(\gamma, b_2, (e,e)c_2) \in (S(V^{\flat}_0) \times V^{\flat}_0 \times (V^{\flat}_0)^*)(F_0).$$
	\end{proposition}
	\begin{proof}
		We only need to check invariants of $(g,u_1)$ and $(\gamma, b_2, (e,e)c_2)$ match. If $g'$ matches $\gamma'$, then there exist $h \in \GL(V^{\flat})$ such that $h^{-1}g'h=\gamma'$. This gives 
		$$h^{-1}u=b, \quad wh=c, \quad h^{-1}th=a, \quad d_{g'}=d_{\gamma'}=d.$$ As $\gamma=a+ \frac{bc}{1-d}$, $g=t+\frac{uw}{1-d}$, we see $h^{-1} gh =\gamma$. So we have 
		$$\text{char}(T \text{id}_{V^{\flat}}- g)=\text{char}(T \text{id}_{V^{\flat}} -\gamma)$$
		and $c_1 \gamma^i b_1= w_1 g^{i} u_1$ for any $i \in \mathbb Z$. By Proposition \ref{Unitaryproj}, we have $(g^iu_1,u_1)=(e,e) w_1(g^{i-1}u_1)$. By Proposition \ref{prop: twisting}, $c_2\gamma^ib_2= c_1 \gamma^{i-1} b_1$, so $(e,e)c_2\gamma^{i} b_2= (e,e) c_1 \gamma^{i-1} b_1$. 		Therefore we have $(g^iu_1, u_1)= (e,e)c_2\gamma^i b_2$ for any $i \in \mathbb Z$. So $(g, u_1)$ matches $(\gamma, b_2, (e,e)c_2) $.
	\end{proof}
	
	Consider any $\ov{(-)}$-stable hermitian lattice $L$ in $V$ with an orthogonal decomposition $L=L^{\flat} \oplus O_Fe_L$, $e_L=\lambda_L e$ for some $\lambda_L \in F_0^{\times}$. We have $L=L_0 \otimes_{O_{F_0}} O_F$ for the lattice $L_0=\{ x \in L | \ov{x}=x \}$.  Apply Corollary \ref{RedOrbSSS} to the collection of $L$ and $L^{\vee}$, we see
	
	\begin{proposition}
		For $\gamma' \in S(V_0)(F_0)_{rs}$ with $1-d \in O_F^{\times}$, we have
		$$
		\Orb(\gamma', 1_{S(L, L^{\vee})})=
		\Orb( (e,e).\mathfrak{c}_S' (\gamma'), f').
		$$
		where $(e,e).\mathfrak{c}_S' (\gamma') :=(\gamma, b_2, (e,e)c_2)$ and 
		$$
		f'=1_{S(L^{\flat}, (L^{\flat})^{\vee})} \times 1_{\lambda_L^{-1}L^{\flat}_0 \cap (e,e)  (\lambda^{-1}_L L^{\flat}_0)^{\vee} } \times 1_{(e,e)(\lambda_L^{-1}L^{\flat}_0 )^{*} \cap ( (\lambda_L^{-1}L^{\flat}_0)^{\vee} )^* }   .$$
	\end{proposition}
	
	\begin{corollary}\label{SimpleS}
		If $L$ is suitable for reduction (Definition \ref{defn: red lat}), then for $\gamma' \in  S(V_0)(F_0)_{rs}$ with $1-d \in O_F^{\times}$ we have
		$$
		\Orb(\gamma', 1_{S(L, L^{\vee})})=
		\Orb( (e,e).\mathfrak{c}_S' (\gamma'), f').
		$$  
		where $
		f':=1_{S(\lambda^{-1}_L L^{\flat}, (\lambda^{-1}_L L^{\flat} )^{\vee}) } \times 1_{\lambda^{-1}_L L^{\flat}_0} \times 1_{((\lambda^{-1}_L L^{\flat})^{\vee}_0)^* } .$
	\end{corollary}
	
	\begin{theorem}\label{thm: reduction orb}
		\begin{enumerate}
			\item 	If $L$ is suitable for reduction,  then the group transfer conjecture \ref{conj: TC} for $L=L^{\flat} \oplus O_F e_L$  is equivalent to the semi-Lie transfer conjecture \ref{conj: TC} for $L_{\new}:=\lambda^{-1}_L L^{\flat}$.
			\item 	If $L^\vee$ is suitable for reduction,  then the group transfer conjecture \ref{conj: TC} for $L=L^{\flat} \oplus O_F e_L$  is equivalent to the semi-Lie transfer conjecture \ref{conj: TC} for $L_{\new}:=\lambda^{-1}_{L^\vee} L^{\vee, \flat}$. 
		\end{enumerate}	
	\end{theorem}
	\begin{proof}
		By duality, we only need to show the first part. By Corollary \ref{SimpleS}, Corollary \ref{SimpleU} and Proposition \ref{RedOrbits}, semi-Lie version TC implies group version TC. Moreover, as relative Cayley maps are surjective on orbits (for given $d$) by Proposition \ref{Unitaryred} and \ref{Symproj}, we can reverse the process. 
	\end{proof}

	In particular, if $L$ is maximal parahoric then either $(1)$ or $(2)$ holds. The new lattice $L_\new $ is still maximal parahoric, so we can continue the process once we could relate (boundary) cases of semi-Lie version transfer conjectures to lower rank group version transfer conjectures.

	\section{Transfers for all vertex lattices}\label{section: TC proof}    
	In this section, we give the proof of explicit transfer Conjecture \ref{conj: TC} for any $p$-adic field $F_0$. 	Consider the space of regular semisimple $\GL(V_0)$-orbits 
	\begin{align}
		\xymatrix{
			B=  \bigr[S(V_0) \times V_0 \times V_0^* \bigr]_\rs(F_0)  \ar[r]^-\sim& [\U(V) \times V]_\rs(F_0) \coprod \bigl[\U(\BV) \times \BV\bigr]_\rs(F_0). }
	\end{align}
	where $\BV$ is the nearby hermitian space of $V$.
	
	Endow the $n$-dimensional hermitian space $V$ with the $F_0$-valued quadratic form $q(u):=(u,u)_V$, and the space $V_0 \times V_0^*$ with the $F_0$-valued quadratic form $q'(b,c):=c(b)$. They induce a map on orbit space $B$:
	$$
	q=q': B \to F_0, \quad [(\gamma,b,c)] \mapsto c(b).
	$$	
	
	Fix an unramified quadratic character $\psi$ of $F_0$. We denote by
	\begin{itemize}
		\item $\mathcal{F}_U$ the Fourier transform on $V$ with respect to the additive character $\psi \circ q$,
		\item $\mathcal{F}_S$ the Fourier transform on $V_0 \times V_0^*$ with respect to the additive character $\psi \circ q'$.
	\end{itemize}  
	
	We always use self-dual Haar measure on $V$ and $V_0 \times V_0^{*}$ for the Fourier transforms.  Then for a $\ov{(-)}$-stable hermitian lattice $L \subseteq V$ of type $t$, we have 
	$$\vol(L)=q^{-t/2}_{F}=1 \times q^{-t}=\vol(L_0 \times (L_0^{\vee})^*). $$

	\subsection{Double uncertainty principle on distributions}
	
	We have the Weil representation of $\SL_2(F_0)$ on $\mathcal{S}(S(V_0) \times V_0 \times V_0^{*})$ and $\mathcal{S}(\U(V) \times V)$ respectively, by acting on the linear factors $V_0 \times V_0^*$ and $V$ respectively. Concretely, the action of $\SL_2(F_0)$ on $f \in \mathcal{S}(\U(V) \times V)$ is given by 
	\begin{equation}
		\begin{pmatrix}
			1 & t \\ & 1
		\end{pmatrix}.f(g, u)=\psi(t(u,u)_V)f(g, u),
	\end{equation}
	\begin{equation}
		\begin{pmatrix} & -1 \\ 1 &
		\end{pmatrix}.f(g, u)=\gamma_V \mathcal{F}_Uf(g, u),
	\end{equation}
	where $\gamma_V:=\eta(\det(V)) \in \{\pm1\}$ is the Weil constant of $V$.
	The action of $\SL_2(F_0)$ on $f' \in \mathcal{S}(S(V_0) \times V_0 \times V_0^{*})$ is given by
	\begin{equation}
		\begin{pmatrix}
			1 & t \\ & 1
		\end{pmatrix}.f'(\gamma, b, c)=\psi(t(cb))f'(\gamma, b, c),
	\end{equation}
	\begin{equation}
		\begin{pmatrix} & -1 \\ 1 &
		\end{pmatrix}.f'(\gamma, b, c)=\mathcal{F}_S f'(\gamma, b,c ). 
	\end{equation}
	
For $f \in \mathcal{S}(\U(V) \times V)$, define the function 
\[
\Orb((\gamma, b, c), f)= \begin{cases}
\Orb(g,u, f) 	&   \text{   if $(\gamma, b, c)$ matches $(g,u) \in (\U(V) \times V)_\rs(F_0) $}  \\
	0, \, \, \text{else}. &  
\end{cases}
\]
	By pushforward, we define two \emph{space of orbital integral functions} $\Orb_U, \Orb_S \subseteq C^{\infty}(B)$: 

\[ 
	\Orb_{S}:=\{ \Orb(-, f'): (\gamma, b, c) \mapsto \Orb((\gamma, b, c), f')| f' \in \mathcal{S}(S(V_0) \times V_0 \times V_0^{*}) \},
\]
\[
\Orb_{U}=\{ \Orb(-, f): (\gamma, b, c) \mapsto \Orb((\gamma, b, c), f) | f \in \mathcal{S}(\U(V) \times V) \}.
\]

	By \cite[Theroem A.1]{AFL}, the Weil representations of $\SL_2(F_0)$ on $\mathcal{S}(S(V_0) \times V_0 \times V_0^{*})$ and  $\mathcal{S}(\U(V) \times V)$ descend to $\Orb_U$ and $\Orb_S$, and agree on the intersection $\Orb_U \cap \Orb_S$. In other words, 
	
	\begin{proposition}\label{OrbU+OrbS}
		There is an action of $\SL_2(F_0)$ on the linear subspace 
		$$\Orb_U+\Orb_S \subseteq C^{\infty}(B)$$
		compatible with the Weil representation on $V$ and $V_0 \times V_0^{*}$. And $ f \in \mathcal{S}(\U(V) \times V), f' \in \mathcal{S}(S(V_0) \times V_0 \times V_0^{*})$ are transfers, if and only if $\Orb(-, f)-\Orb(-, f') \in \Orb_U+\Orb_S  $ is zero.
	\end{proposition}
	
	\begin{remark}
		The proof of \cite[Theroem A.1]{AFL} is via local trace formula, which is eventually reduced to the fact that orbital integral function is locally integrable and Fourier transform on a quadratic space preserves the $L^2$-norm. Therefore, it also holds for $p=2$.
	\end{remark}

Using an orthogonal basis for $L$, we see the valuation
$v(L):=\min_{u \in L} v((u,u)_V)$ is equal to the valuation
$v(L_0, L_0^{\vee}):=\min_{b \in L_0, \, c \in (L_0^{\vee})^* } v(c(b)).$ We introduce the difference function
\begin{equation}\label{difference orbit integral function}
	\Phi_L=\Orb(-, 1_{\U(L)} \times 1_L)- \Orb(-, 1_{S(L_0,L_0^{\vee})} \times 1_{L_0} \times 1_{(L_0^{\vee})^*}) \in C^{\infty}(B).
\end{equation}

By Proposition \ref{OrbU+OrbS}, to prove Theorem \ref{proof of TC} for $L$, we only need to show $\Phi_L=0$. In other words,  both matching part and vanishing part of orbital integrals follow from $\Phi_L=0$.

The following double uncertainty principle says the support of a function and its ``Fourier transform'' can't be both too small, unless the function is zero.
	
	\begin{proposition}\label{Uncertain: double supports}
		Let $a_1, a_2$ be two integers with $a_1+a_2 \geq -1$. Assume that $\Phi \in \Orb_U+\Orb_S$ satisfies $\Phi(x)=0$ for all  $x \in B$ with $v(q(x)) \leq a_1$, and $\begin{pmatrix} & -1 \\ 1 &
		\end{pmatrix}.\Phi(x)=0$ for all  $x \in B$ with $v(q(x)) \leq a_2$.
		Then $\Phi=0$.
	\end{proposition}
	\begin{proof}
		Let $k=-a_1-1$ and $k'=-a_2-1$ so  $k+k' \leq -1$. By assumption, we have 
		$\begin{pmatrix} 1 & t \\  & 1
		\end{pmatrix}.\Phi =\Phi$ for all $t \in \varpi^{k}O_{F_0}$ and $\begin{pmatrix} 1 & t \\  & 1
		\end{pmatrix}.\begin{pmatrix} & -1 \\ 1 &
		\end{pmatrix}. \Phi =  \begin{pmatrix} & -1 \\ 1 &
		\end{pmatrix} \Phi$ for all $t \in \varpi^{k'}O_{F_0}$.
		The subgroups
		$\begin{pmatrix}
			1 & \\  \varpi^{k}O_{F_0} & 1
		\end{pmatrix}$ and  $\begin{pmatrix}
			1 & \varpi^{k'}O_{F_0} \\ & 1 \end{pmatrix}$ generates the whole group $\SL_2(F_0)$. Hence $\Phi$ is fixed by $\SL_2(F_0)$. As $\begin{pmatrix} 1 & t \\  & 1
		\end{pmatrix}.\Phi ([\gamma,b,c]) = \psi(t (cb))  \Phi([\gamma,b,c])$ for any $t \in F_0$, we see $\Phi=0$.
	\end{proof}
	
	\begin{proposition}\label{prop: support}
		We have
		\begin{enumerate}
			\item 	$\begin{pmatrix} & -1 \\ 1 &
			\end{pmatrix} . \Phi_L = \vol(L) \Phi_{L^{\vee}}.$
			\item 	$\Phi_L(x)=0$ for all  $x \in B$ with $v(q(x)) \leq v(L)-1$. And $\begin{pmatrix} & -1 \\ 1 &
			\end{pmatrix}.\Phi_L(x)=0$ for all  $x \in B$ with $v(q(x)) \leq v(L^\vee)-1$.
		\end{enumerate}	
	\end{proposition}
	\begin{proof}
		The first part follows from that $q(u)=q'(u_1, u_2)$ is an invariant for an orbit $x=(\gamma, u_1, u_2) \in B$, and that $v(L)=v(L_0, L_0^{\vee})$. By definition, $\begin{pmatrix} & -1 \\ 1 &
		\end{pmatrix}. \Phi_L$ is given by
		$$\vol(L) \Orb(1_{\U(L)} \times 1_{L^{\vee}},-)- \vol(L_0 \times (L_0^{\vee})^*) \gamma_V  \Orb(1_{S(L_0,L_0^{\vee})} \times 1_{L_0^{\vee}} \times 1_{L_0^*},-). $$ By Remark \ref{rek: transfer factor and basis}, we have 
		$	\omega_{L^{\vee}}=\gamma_{V} \omega_L.$
		And $\vol(L)=\vol((L_0 \times L_0^{\vee})^*)=|L^\vee/L|^{1/2}$, hence 
		$\begin{pmatrix} & -1 \\ 1 &
		\end{pmatrix} \Phi_L = \vol(L) \Phi_{L^{\vee}}$. Then the second part follows directly.
	\end{proof}

	\subsection{Proof of TC assuming inductions on boundary cases}

	We now prove Theorem \ref{proof of TC}.
	
	\begin{theorem}\label{proof of TC assuming boundary cases}
		Let $L$ be a rank $n$ maximal parahoric hermitian $O_F$-lattice i.e., $v(L)+v(L^{\vee})  \geq -1$. Then the semi-Lie transfer Conjecture \ref{conj: TC} for $L$ holds.
	\end{theorem}
	\begin{proof}
     	We do induction on the rank of $L$. The case $n=1$ is standard and works for any hermitian lattice $L$ (see also the computation in \S \ref{ATC n=1}). Now assume that Conjecture \ref{conj: TC} holds for all maximal parahoric lattices of rank $\leq n-1$.
		
		As $v(L)+v(L^{\vee}) \geq -1$, by Proposition \ref{prop: support} and Proposition \ref{Uncertain: double supports}, to show the conjecture for $L$ we just need to prove 
		$\Phi_L(x)=0$ if $v(q(x))=v(L)$,  and $\Phi_{L^{\vee}}(x)=0$ if $v(q(x))=v(L^{\vee})$.  By symmetry, we only need to deal with $L$. By Lemma \ref{lem: ind on boundary cases} below, we see that $\Phi_L(x)=0$ is implied by the group version transfer conjecture for $(L,e)$, which is equivalent to the semi-Lie version conjecture for the rank $n-1$ lattice $L^{\flat}$ by Proposition \ref{thm: reduction orb}, hence is true by induction. The result follows.
	\end{proof}
	
	Let $e \in L$ be a vector with maximal length i.e., $v((e,e)_V)=v(L)$. 
	
	\begin{lemma}\label{lem: ind on boundary cases}
		Semi-Lie transfer conjecture for $L$ in the boundary case $v(q(x))=v(L)$ is implied by the group version transfer conjecture for the pair $(L, e)$.
	\end{lemma}
	
	In the following three subsections, we prove Lemma \ref{lem: ind on boundary cases}.
	
	\subsection{Integral transitivity on maximal length vectors}
	We establish two integral transitivity lemmas for the induction on boundary cases. Recall that $F/F_0$ is unramified, the norm map is surjective on units i.e., $\Nm_{F/F_0}(O_F^{\times})=O_{F_0}^{\times}$.

	\begin{lemma}\label{lem: ortho basis}
		For any $O_F$-hermitian lattice $L$, $v(L)=\min_{x \in L} v_F((x,x))= \min_{x,y \in L} v_F((x,y)).$ 
	\end{lemma}
	\begin{proof}
		We have an orthogonal basis of $L$ ($p \geq 2$) by \cite[Section 7]{classicalherm}. The result follows directly. 
	\end{proof}
	
	Let $L$ be an $O_F$-hermitian lattice. Let $e \in L$ be a vector with maximal length i.e., $v((e,e))=v(L)$. By Lemma \ref{lem: ortho basis}, we have $\frac{(x,e)}{(e,e)} \in O_F$ for all $x \in L$.  Hence there is an orthogonal decomposition $L=L^{\flat} \oplus O_F e$ for $L^{\flat}:=\{ x \in L | (x,e)=0\}$.
	
	\begin{lemma}\label{lem: int tran U(L)} (Integral transitivity)
		Choose a vector $x_0 \in L$ with maximal length. Then the compact open subgroup $\U(L)$ of $\U(V)$ acts transitively on the sphere
		$$S_{x_0}=\{ x \in L | (x,x)=(x_0,x_0) \}.$$
	\end{lemma}
	\begin{proof}
		For any $x \in S_{x_0}$, we get a splitting $L=L_1 \oplus O_F x_0$. Choose a maximal vector of $L_1$ and keep doing splitting, finally we extend $x$ to an orthonormal basis $\{e_1, \hdots, e_{n-1}, x \}$ of $L$. For any other $x' \in S_{x_0}$, we get another orthonormal basis $\{e_1', \hdots, e_{n-1}', x'\}$ of $L$. By classifications of hermitian lattices, we can assume that $(e_i,e_i)= (e_i',e_i') \in \varpi^{\BZ}$, $1 \leq i \leq n-1$ after multiplying $e_i$ by units in $O_{F_0}^{\times}$ and rearrange the order. Then the endomorphism $h \in \GL(V)$ sending $e_i$ to $e_i'$ and $x$ to $x'$ is inside the compact open subgroup $\U(L)$. The result follows.
	\end{proof}
	
	\begin{remark}
		For integral transitivity on ``smaller spheres'' and integral Witt's theorem, see \cite{Morin-Strom}.
	\end{remark}
	
	Now let $L_1, L_2$ be two lattices inside a $F_0$-vector space $V_0$. We say a pair  $(u_1, u_2) \in L_1 \times L_2^*$ is of maximal length if we have $v(u_2u_1)=v(L_1, L_2):=\min_{x \in L_1, y \in L_2^*} v(y(x))$.
	
	\begin{lemma}\label{splitting}
		Choose a pair $(u_1,u_2) \in L_1 \times L_2^*$ of maximal length. Let $V_0^{\flat}=\mathrm{Ker}(u_2)$, and $(L_i)^{\flat}$ be the projection of $L_i$ to $V_0^{\flat}$ ($i=1,2$). Then  $V_0=V_0^{\flat}  \oplus F_0u_1$,  and 
		$$L_1=L_1^{\flat} \oplus O_{F_0}u_1 \, \, L_2=L_2^{\flat} \oplus O_{F_0} \varpi^{-v(L_1, L_2)} u_1.$$
		And $L_2^*=(L_2^{\flat})^* \oplus O_{F_0}u_2$. 
	\end{lemma}
	\begin{proof}
		For any $x \in V_0$, we have $x=(x-\frac{u_2(x)}{u_2(u_1)} u_1 ) + \frac{u_2(x)}{u_2(u_1)} u_1 \in V_0^{\flat} \oplus F_0 u_1$. 
		The first property follows. Now set $k=-v(L_1, L_2)$. For any $f_2 \in L_2^*$, we have $f_2(\varpi^{k}u_1) \in \varpi^{k+v(L_1, L_2)}O_{F_0}=O_{F_0}$, hence $\varpi^{k} u_1 \in L_2$. For $x \in L_2$, there exists $a \in O_{F_0}$ such that $u_2(x-a\varpi^{k}u_1)=0$. So we have the decomposition  $L_2=L_2^{\flat} \oplus O_{F_0} \varpi^{k} u_1$. Taking its dual, we see $L_2^*=L_2^{\flat *} \oplus O_{F_0}u_2$. 
	\end{proof}
	
	\begin{lemma}\label{lem: int tran S(L)} (Integral transitivity)
		Choose a pair $(u_1, u_2) \in L_1 \times L_2^*$ of maximal length. Then the compact open subgroup $\GL(L_1,  L_2)$ (stabilizer of $L_1$ and $L_2$) of $\GL(V_0)$ acts transitively on the sphere
		$$
		S_{u_1,u_2}=\{ (x_1, x_2) \in L_1 \times L_2^* | u_2u_1=x_2x_1 \}.
		$$
	\end{lemma}
	\begin{proof}
		Apply the Lemma \ref{splitting} several times to extend $u_1$ to a basis $\{e_i\}$ of $L_0$ (hence of $V_0$) such that $u_2$ is a multiple of the dual basis of $u_1$ and $L_2$ has a basis $\{ \varpi^{k_i} e_i^* \}$. Similarly, we extend $(x_1, x_2)$ to another basis $\{e_i'\}$ of $L_0$. Then the automorphism $h$ sending the basis $\{e_i' \}$ to $\{e_i\}$ lies in $S(L_1,  L_2)$ and we have $h^{-1} u_1 = x_1,  u_2 h = x_2 .$
	\end{proof}

	\subsection{Induction on boundary cases: unitary groups}
	
	Recall that $e$ is a fixed vector in $L$ of maximal length, and we have a decomposition $L=L^{\flat} \oplus O_F e$. On the unitary side, we now show orbital integrals for $L$ in the boundary case $v(q(x))=v(L)$ can be reduced to orbital integrals for $(L, e)$. 
	
	Consider $y=(g_0, u_0) \in [\U(V) \times V]_{rs}(F_0)$ such that $v((u_0,u_0)_V)=v(L)$. Conjugate $(g_0,u_0)$ by an element in $\U(V)$, we can assume that $u_0 \in L_0 \subseteq V_0$.  By Lemma \ref{lem: int tran U(L)}, we can even assume $u_0=e$.
	
	\begin{proposition}
		Assume that $u_0 \in L_0$ and $v((u_0,u_0)_V)=v(L)$, we have
		\[
		\Orb((g_0, u_{0}), 1_{\U(L)} \times 1_{L} )=\Orb(g_0, 1_{\U(L)})
		\]
		where the right hand side is group version orbital integral for the pair $(L, u_{0})$. 
	\end{proposition}
	
	\begin{proof}
		
		Let $h \in \U(V)$ be any element with $1_{L}(h^{-1}u_{0}) \not =0$. By Lemma \ref{lem: int tran U(L)}, there exists $h_1 \in \U(L)$ such that  $h^{-1}u_{0}=h_1u_{0}.$ Hence the support of such $h$ is in $\U(L)\U(V_0^{\flat})$. Recall that the volume of $\U(L)$ is normalized to be $1$. As $1_{\U(L)} \times 1_L$ is invariant under $\U(L)$,  $	\Orb((g_0, u_{0}), 1_{\U(L)} \times 1_{L} ) $ is equal to
		\[ 
		\int_{h \in \U(V_0^{\flat})} 1_{\U(L)} (h.g_0) 1_{L}(h^{-1}u_{0}) dh
		= \int_{h \in \U(V_0^{\flat})} 1_{\U(L)} (h.g_0) 1_{L}(u_{0}) dh=\Orb(g_0, 1_{\U(L)}).
		\]
	\end{proof}

	\subsection{Induction on boundary cases: symmetric spaces}
	
	Consider $x=(\gamma, u_1, u_2) \in [S(V_0) \times V_0 \times V_0^*]_{\rs}(F_0)$ such that $v(u_2u_1)=v(L_0, L_0^{\vee})=v(L)$. Conjugate $(\gamma, u_{1}, u_{2})$ by an element in $\GL(V_0)$, we can assume that $u_1 \in L_0, u_2 \in (L_0^{\vee})^*$.  By Lemma \ref{lem: int tran U(L)}, we can and do assume $u_1=e, \, u_2=(u_2u_1)e^*.$
	
	\begin{proposition}
		For all $s \in \BC$, we have 
		\[
		\Orb((\gamma, u_{1}, u_{2}), 1_{S(L_0, L^{\vee}_0)} \times 1_{L_0} \times 1_{L_0^{\vee*}}, s)=\Orb(\gamma, 1_{S(L_0, L^{\vee}_0)}, s), 
		\]
		where the right hand side is the group version orbital integral for the pair $(L, u_1)$. 
	\end{proposition}
	\begin{proof}

		Let $h \in \GL(V_0)$ be any element with $1_{L_0}(h^{-1}u_{1})1_{L_0^{\vee*}}(u_{2}h) \not =0$. By Lemma \ref{lem: int tran S(L)}, there exists $h_1 \in \GL(L_0, L_0^{\vee})$ such that $h^{-1}u_{1}=h_1u_{1}, u_{2}h=u_{2}h_1.$
		
		Hence the support of such $h$ is in $\GL(L_1,L_2)\GL(V_0^{\flat})$. Recall that the volume of $\GL(L_1,L_2)$ is normalized to be $1$. As $1_{S(L_0, L^{\vee}_0)} \times 1_{L_0} \times 1_{L_0^{\vee*}}$ is invariant under $\GL(L_1,L_2)$ and $\eta(\det(\GL(L_1,L_2)))=\{1\}$, $\Orb((\gamma, u_{1}, u_{2}), 1_{S(L_0, L^{\vee}_0)} \times 1_{L_0} \times 1_{L_0^{\vee*}}, s)$ is equal to
		\begin{align*} \omega(\gamma, u_{1}, u_{2})\int_{h \in \GL(V_0^{\flat})} 1_{S(L_0, L^{\vee}_0)} (h.\gamma) 1_{L_0}(h^{-1}u_{1})1_{L_0^{\vee*}}(u_{2}h) \eta(h) h|^s dh \\
			= \omega(\gamma)\int_{h \in \GL(V_0^{\flat})} 1_{S(L_0, L^{\vee}_0)} (h.\gamma) \eta(h) |h|^s dh=\Orb(\gamma, 1_{S(L_0, L^{\vee})}, s).
		\end{align*} 
		Here the transfer factor $\omega(\gamma, u_1, u_2)$ for $L$ and $\omega(\gamma)$ for $(L, u_1)$ agree by Definition (\ref{transfer factor: semi-Lie}) and (\ref{transfer factor: gp}).
	\end{proof}
	
	We finish the proof of Lemma \ref{lem: ind on boundary cases}. Assume that $y=(g_0, e)$ matches $x=(\gamma, e, (u_2u_1)e^*)$, where $e^*$ is the dual basis of $e$ under the chosen basis of $V$. Then there exists $h \in \GL(V)$ such that $h.(g_0, e)=(\gamma, e, (u_2u_1)e^*)$.  So $h \in \GL(V^{\flat})$ and $g$ matches $\gamma$. 
	
	By induction, the transfer conjecture for $(L,e)$ holds, sos
	\[\\\
	\Orb((g_0, u_{0}), 1_{\U(L)} \times 1_{L} )=\Orb(g_0, 1_{\U(L)})
	\]
	\[	
	=\Orb(\gamma, 1_{S(L_0, L^{\vee}_0)})=\Orb((\gamma, u_{1}, u_{2}), 1_{S(L_0, L^{\vee}_0)} \times 1_{L_0} \times 1_{L_0^{\vee*}}).
	\] 
	Lemma \ref{lem: ind on boundary cases} is proved. Therefore, on two boundary cases (for $L$ and $L^\vee$), the semi-Lie TC reduces to two lower rank group TCs. Hence $\Phi_L$ (c.f. \ref{difference orbit integral function}) satisfies the conditions in Corollary \ref{Uncertain: double supports}. Theorem \ref{proof of TC} is proved.

	\part{Arithmetic transfer conjectures}\label{local part}
	
	Let $p$ be an odd prime. Let $F/F_0$ be an unramified quadratic extension of $p$-adic local fields. Fix a uniformizer $\varpi$ of $F_0$ and an embedding $O_F \to O_{\breve{F}}$. Denote the residue field of $F_0$ (resp. $\breve{F}$) by $\BF_q$ (resp. $\BF$). Let $\sigma=\ov{(-)} \in \Gal(F/F_0)$ be the non-trivial Galois involution.
	
	Let $V$ be a $F/F_0$-hermitian space of dimension $n \geq 1$, and $L$ be  a vertex lattice in $V$ i.e., $L \subseteq L^\vee \subseteq \varpi^{-1}L$. Let $0 \leq t \leq n$ be the type of $L$. The space $V$ is split if and only if $t$ is even.
	
	\section{Special cycles and modified derived fixed points} \label{section:local RZ}
	
	In this section, we study the relative Rapoport--Zink space $\CN$ for the lattice $L$ and two kinds of Kudla--Rapoport cycles. We use the dual isomorphism $\lambda: \CN \to \CN^\vee$ as the tool of ``taking dual'' on the geometric side. We introduce the formal Balloon--Ground stratification on the special fiber and use it to resolve the singularity of $\CN \times \CN$, which is an arithmetic analog of the Atiyah flop \cite{Atiyah}.  We define a (derived) modification of the fixed point locus of an automorphism on $\CN$.
	
	For any $\Spf O_{\breve F}$-scheme $S$, \emph{a basic triple $(X, \iota, \lambda)$ of dimension $n$ and type $t$} (and signature $(1, n-1)$) over $S$ is the following data:
	\begin{itemize}
		\item
		$X$ is a strict formal $O_{F_0}$-module over $S$ of relative height $2n$ and dimension $n$. Strictness means the induced action of $O_{F_0}$ on $\Lie X$ is via the structure morphism $O_{F_0} \to \CO_S$. 
		\item $\iota: O_F \to \End(X)$ is an action of $O_F$ on $X$ that extends the action of $O_{F_0}$. We require that the \emph{Kottwitz signature condition} of signature $(1, n-1)$ holds for all $a \in O_{F} $:
		\begin{equation}\label{eq: Kottwitz}
			\charac (\iota(a)\mid \Lie X) = (T-a)(T-\ov a)^{n-1} \in \CO_S[T].
		\end{equation} 
		\item $\lambda: X \to X^\vee$ is a polarization on $X$ as strict $O_{F_0}$-module, which is $O_F/O_{F_0}$ semi-linear in the sense that the Rosati involution $\Ros_\lambda$ induces the non-trivial involution $\ov{(-)} \in \Gal({F}/{F_0})$ on $\iota: O_{F} \to \End(X)$.
		\item We require that the finite flat group scheme $\Ker\, \lambda$ over $S$ lies in $X[\varpi]$ and is of order $q^{2t}$.
	\end{itemize}
	
Note that here we endow the dual formal $O_{F_0}$-module $X^\vee$ with the $O_F$-action $\iota^\vee(a):=\iota(\bar{a})^\vee$ and the dual polarization $\lambda^\vee$ (such that $\lambda^\vee \circ \lambda=[\varpi]$ ).  The dual of a basic triple $(X, \iota, \lambda)$ (of dimension $n$ and type $t$) is the basic triple $(X^\vee, \iota^\vee, \lambda^\vee)$ (of dimension $n$ and type $n-t$).
	
 The Kottwitz signature condition gives a decomposition $\Lie X=(\Lie X)_0 \oplus (\Lie X)_1$, where $O_F$ acts on the rank $1$ piece $(\Lie X)_0$ by structure map, and on the rank $n-1$ piece $(\Lie X)_1$ by $\sigma$-conjugate of the structure map.	
	
	An isomorphism $(X_1, \iota_1, \lambda_1) \overset{\sim}{\longrightarrow} (X_2, \iota_2, \lambda_2)$ between two such triples is an $O_{F}$-linear isomorphism $\varphi\colon X_1 \overset{\sim}{\longrightarrow} X_2$ such that $\varphi^*(\lambda_2)=\lambda_1$.   Up to $O_F$-linear quasi-isogeny compatible with the polarization, there exists a unique such triple $(\BX, \iota_{\BX}, \lambda_{\BX})$ over $\BF$. Fix one choice of $(\BX, \iota_{\BX}, \lambda_{\BX})$ as the {\em framing object}.

	\begin{definition}\label{defn: local RZ space}
		The Rapoport--Zink space for $L$ is the functor 
		$$\CN_{\U(L)}=\CN_n^{[t]} \to \Spf O_{\breve F} $$
		sending $S$ to the set of isomorphism classes of tuples $(X, \iota, \lambda, \rho)$, where
		\begin{itemize}
			\item $(X, \iota, \lambda)$ is a basic  triple of dimension $n$ and type $t$ over $S$.
			\item $\rho: X \times_{S} {\ov S} \to \BX \times_\BF \ov S $ an $O_F$-linear quasi-isogeny of height $0$ over the reduction $\ov S:=S \times_{O_{\breve F}} \mathbb F$ such that $\rho^*(\lambda_{\BX, \ov S}) = \lambda_{\ov S}$.
		\end{itemize}
	\end{definition}
	
	From \cite[Thm. 2.16]{RZ96}, the functor $\CN_{\U(L)}=\CN_n^{[t]}$ is representable by a formal scheme locally formally of finite type over $\Spf O_{\breve F}$ of relative dimension $n-1$. 
	
	Let $\BE$ be a formal $O_{F_0}$-module over $\BF$ of relative height $2$ and dimension $1$. Complete it into a triple $(\BE, \iota_{\BE}, \lambda_{\BE})$ of \emph{signature $(0, 1)$}, dimension $1$ and type $0$ over $\BF$, and still denote it by $\BE$. Changing the $O_F$-action on $\BE$ by Galois involution $\ov{(-)}$, we obtain a framing object $(\BE, \iota_{\BE} \circ \sigma, \lambda_{\BE})$ for $\CN_1^{[0]}$, which we denote by $\ov{\BE}$.  The theory of canonical lifting \cite{Gross-canonical} gives an isomorphism $\CN_1^{[0]} \simeq \Spf O_{\breve F}$, where the universal triple $\ov{\CE}$ over $\CN_1^{[0]} \simeq \Spf O_{\breve F}$ is the canonical lifting of $\ov{\BE}$. Denote by $\CE$ the basic triple over $O_{\breve F}$ obtained from the canonical lifting of $\BE$.
	
	Let \emph{the space of special quasi-homomorphisms} be the $F$-vector space 
	\begin{equation}\label{eq: space of special quasi-homos}
		\BV=\Hom_{O_F}^{\circ}(\BE, \BX)
	\end{equation}
	equipped with the hermitian form ($x, \, y \in \BV$):
	\begin{equation}
		(x,y)_{\BV}:= \lambda_{\BE}^{-1} \circ y^{\vee} \circ \lambda_{\BX} \circ x \in \Hom_{O_F}^{\circ}(\BE, \BE) \cong F.
	\end{equation}
	
	The automorphism group of quasi-isogenies $\Aut^\circ(\BX, \iota_{\BX}, \lambda_{\BX})$ acts on $\BV$ naturally. By Dieudonne theory, we have $\Aut^\circ(\BX, \iota_{\BX}, \lambda_{\BX}) \cong \U(\BV)(F_0)$. The group $\U(\BV)(F_0)$ acts on $\CN_{\U(L)}$ by
	\begin{equation}
		g.(X, \iota, \lambda, \rho)=(X, \iota, \lambda, g \circ \rho).
	\end{equation}	
	
	From now on, we fix $n$ and $t$ and write $\CN=\CN_{\U(L)}=\CN_n^{[t]}$ for simplicity. 
	\begin{remark}\label{descent to O_F: moduli}
		The framing object may be defined over $\BF_{q^2}$, and we obtain a descent of $\CN$ to $\Spf O_{F}$ by the same moduli interpretation. We will still denote it by $\CN \to \Spf O_F$.
	\end{remark}

	\subsection{The dual isomorphism $\lambda_\CN: \CN \to \mathcal{N}^{\vee}$}\label{dual lambda}

	By taking dual, the framing object $(\BX, \iota_{\BX}, \lambda_{\BX})$ gives a framing object $(\BX^{\vee}, \iota_{\BX^{\vee}}, \lambda_{\BX^{\vee}})$ of dimension $n$ and type $n-t$, where $\lambda_{\BX^{\vee}}: \BX^{\vee} \to \BX$ is the dual polarization of $\lambda_{\BX}$ such that $\lambda_{\BX^{\vee}} \circ \lambda_{\BX}=[\varpi] : \BX \to \BX$.  And $\iota_{\BX^\vee}:=\iota_{\BX} \circ \ov {(-)}$.
	
	Let $\BV^{\vee}:=\Hom_{O_F}^{\circ}(\BE, \BX^{\vee})$ be the space of special quasi-homomorphisms for $(\BX^{\vee}, \iota_{\BX^{\vee}}, \lambda_{\BX^{\vee}})$. The polarization $\lambda_{\BX}$ induces a $F$-linear map $\lambda_\BV: \BV \to \BV^{\vee}$ sending $x$ to $\lambda_{\BX} \circ x$.

	\begin{proposition}\label{prop: dual iso herm form comp}
		The map $\lambda_\BV: \BV \to \BV^{\vee}$ is a bijection. For any $x, y \in \BV$, we have
		\begin{equation}
			(\lambda_{\BV} \circ x,  \lambda_{\BV} \circ y)_{\BV^{\vee}}=-\varpi (x,y)_{\BV}.
		\end{equation}
	\end{proposition}
	\begin{proof}
		Any polarization is an anti-symmetric isogeny. Unraveling the definitions we obtain the comparison between hermitian forms on $\BV$ and $\BV^\vee$.
	\end{proof}
	
	Let $V^{\vee}$ be the hermitian space with same underlying $F$-vector space as $V$, but equipped with the hermitian form $(x, y)_{V^{\vee}}:=-\varpi (x,y)_{V}$. Denote by $\lambda_V: V \cong V^{\vee}$ the identification of underlying vector spaces. Then $\lambda_V (L^{\vee})$ is a vertex lattice in $V^{\vee}$ of type $n-t$.
	
	\emph{The dual Rapoport--Zink space} of $\CN$ is the Rapoport--Zink space $\CN^{\vee}:=\CN_{\U( \lambda_V ( L^{\vee})) }=\CN_n^{[n-t]}$ for the framing object $(\BX^{\vee}, \iota_{\BX^{\vee}}, \lambda_{\BX^{\vee}})$. We have \emph{the dual isomorphism} 
	\[
	\lambda_\CN: \CN \to \CN^{\vee}
	\]
	\[
	(X, \iota_X, \lambda_X, \rho_X) \mapsto (X^{\vee}, \iota_{X^{\vee}}, \lambda_{X^{\vee}}, \rho_{X^{\vee}})
	\] 
	where $\lambda_{X^\vee}$ is the dual polarization of $\lambda_X$, and $\rho_{X^{\vee}}$ is the quasi-isogeny $((\rho_X)^{\vee})^{-1}: X^{\vee} \times_{S} {\ov S} \to \BX^{\vee} \times_\BF \ov S$ . By definition, we have 
	$$
	(\CN^{\vee})^{\vee}=\CN, \quad \lambda_{\CN^{\vee}} \circ \lambda_\CN=id_{\CN},  \quad\lambda_\CN \circ \lambda_{\CN^\vee}=id_{\CN^{\vee}}.
	$$

	\subsection{Kudla--Rapoport cycles}\label{section: local KR}
	
	The dual  $\ov{\BE}^\vee$ of $\ov{\BE}$ gives a framing object for $\CN_1^{[1]}$. For $i \in \{0,1\}$, denote by $\ov{\CE}^{[i]}$ the universal triple over $\CN_1^{[i]} \simeq \Spf O_{\breve F}$, and by $\CE^{[i]}$ the triple over $O_{\breve F}$ obtained from $\ov{\CE}^{[i]}$ by twisting the $O_F$-action using the Galois involution $\sigma$. So $\CE=\CE^{[0]}$.
	
	\begin{definition}
		For any non-zero $u \in \BV$, the Kudla--Rapoport cycle $\CZ(u) \to \CN$ is the closed formal subscheme of $\CN$ sending each $\Spf O_{\breve{F}}$-scheme $S$ to the subset $(X, \iota_X, \lambda_X, \rho_X) \in \CN(S)$ such that
		\[ (\rho_{X})^{-1}  \circ u: \BE \times_{\BF} \ov S \to X \times_{S} \ov S  \]  
		lifts to a homomorphism from $\CE$ to $X$ over $S$.
		
		Similarly, the Kudla--Rapoport cycle $\CY(u) \to \CN$ is the closed formal subscheme of $\CN$ where 
		\[ (\rho_{X^{\vee}})^{-1} \circ \lambda_{\mathbb{X}} \circ u: \BE \times_{\BF} \ov S \to X^{\vee} \times_{S} \ov S  \]  
		lifts to a homomorphism from $\CE$ to $X^{\vee}$ over $S$.
	\end{definition}
	
	\begin{proposition}\cite[Prop. 5.9]{Cho18}\label{Prop: Z and Y are Cartier divisors}
		The functors $\CZ(u)$ and $\CY(u)$ are representable by Cartier divisors of $\CN$.
	\end{proposition}

	\begin{remark}
		Unlike the case $t=0$ \cite[Prop. 3.5]{KR-local}, in general $\CZ(u)$ and $\CY(u)$ may not be flat over $\BZ_p$ nor relative Cartier divisors on $\CN$.
	\end{remark}
	
	\begin{proposition}\label{prop: relation of Z and Y}
		\begin{enumerate}
			\item We have natural inclusions:
			\begin{equation}
				\CZ(u) \hookrightarrow \CY(u) \hookrightarrow \CZ( \varpi u).
			\end{equation}	
			\item Under the dual isomorphism $\lambda_\CN: \CN \to \mathcal{N}^{\vee}$, we have 
			\begin{equation}
				\lambda_\CN (\CY(u)) = \CZ(\lambda_\BV \circ u), \quad \lambda_\CN (\CZ(u)) = \CY(\varpi^{-1} (\lambda_\BV \circ u)).
			\end{equation} 
			\item $\CZ(u)$ is empty unless $v((u,u)_\BV) \geq 0$; $\CY(u)$ is empty unless $v((u,u)_\BV) \geq -1$.
		\end{enumerate}
	\end{proposition}
	\begin{proof}
		We use the moduli interpretation. If $(X, \iota_X, \lambda_X, \rho_X) \in \CZ(u)(S)$, then $\lambda_X \circ \rho_{X}^{-1} \circ u : \mathbb E \rightarrow X^{\vee} \times_S \ov S$ lifts to a homomorphism from $\CE$ to $X^{\vee}$ over $S$. By compatibility of $\lambda_X$ and $\lambda_{\mathbb X}$ under the quasi-isogeny $\rho_X$, we have $\lambda_X \circ \rho_{X}^{-1} \circ u = \rho_{X}^{\vee} \circ \lambda_{\mathbb{X}} \circ u$ hence $(X, \iota_X, \lambda_X, \rho_X) \in \CY(u)(S)$. Similarly, from $\lambda_X^\vee \circ \lambda_X \circ \rho_{X}^{-1} \circ u=[\iota_X(\varpi)] \circ \rho_{X}^{-1} \circ u$ we see that $\CY(u) \subseteq \CZ(\varpi u) $. The second property holds by the definition of $\CY(u)$ and $\lambda_\CN$. 
		
		For the last property, if $(X, \iota_X, \lambda_X, \rho_X) \in \CZ(u)(S)$, then
		$$
		\lambda_\BE^{-1} \circ (\rho_{X}^{-1}  \circ u)^\vee \circ \lambda_X \circ (\rho_{X}^{-1}  \circ u) \in \Hom_{O_F}(\BE \times_\BF \ov S, \BE \times_\BF \ov S ) \simeq O_{F}
		$$
		which exactly gives $(u,u)_\BV \in O_{F}$ by the compatibility $\lambda_X \circ \rho_{X}^{-1} \circ u = \rho_{X}^{\vee} \circ \lambda_{\mathbb{X}} \circ u$. Using the dual isomorphism $\lambda: \CN \to \CN^\vee$ and Proposition \ref{prop: dual iso herm form comp}, the assertion for $\CY(u)$ follows by duality.
	\end{proof}

	\subsection{Formal Balloon-Ground stratification}\label{Formal Balloon-Ground}

	\begin{proposition}\label{prop:semi-stable}
		The formal scheme $\CN_{\U(L)} \to \Spf O_{\breve{F}_0} $ is formally smooth if $t=0$ or $n$, and is of strictly semi-stable reduction if $0<t<n$.  More precisely, the completed local ring of $\CN$ at any closed point  is isomorphic to 
		either $O_{\breve{F}}[[x_0,y_0,t_1, \hdots,t_{n-1}]]/(x_0y_0-\varpi)$
		or  
		$O_{\breve{F}} [[t_0,t_1,  \hdots, t_{n-1}]]$.
	\end{proposition}
	\begin{proof}
		This follows from the computation local models as in \cite[Prop 3.33]{Cho18}, which reduces to the $\GL_n$ case in \cite[\S 4.4.5]{Gortz01}.
	\end{proof}

	Let $S$ be any $\BF$-scheme. Choose any point $(X, \iota, \lambda, \rho) \in \CN(S)$. Consider the \emph{covariant} relative Dieudonne crystal $\BD(X)$ of $X$ over $S$, and the cotangent sheaf
	$$ \omega_{X}:=(\Lie X )^{\vee}. $$  We have short exact sequences of locally free $\CO_S$-modules
	\begin{equation}
		0 \to \omega_{X^\vee} \to \BD (X)(S) \to \Lie X \to 0,
	\end{equation}
    \begin{equation}
    	0 \to \omega_{X} \to \BD (X^\vee)(S) \to \Lie X^\vee \to 0.
    \end{equation}
	We have a natural identification $(\BD(X)(S))^\vee \simeq \BD(X^\vee)(S)$ via the perfect pairing 
	\begin{equation}
		(-,-): \BD(X)(S) \times \BD(X^\vee)(S) \to \CO_S.
	\end{equation}
	We obtain an antisymmetric bilinear form 
	$$
	(-,-)_\lambda: \BD(X)(S) \times \BD(X)(S) \to \CO_S, \, \, (x, y)_\lambda := (x, \BD(\lambda)(y)).$$
	The $O_F$-action on $X$ induces a decomposition
	$\BD(X)=\BD(X)_0 \oplus \BD(X)_1$ such that $O_F$ acts on $\BD(X)_0$ (resp. $\BD(X)_1$) by the structure map (resp. the structure map composed with the Galois involution on $O_F/O_{F_0}$). 
	And $\BD(X)_i$ ($i=0,1$) is isotropic under the form $(-,-)_\lambda$. From the Kottwitz signature condition $(1, n-1)$,  $(\omega_{X^\vee})_0$ (resp. $(\omega_{X^\vee})_1$) is locally free of rank $n-1$ (resp. $1$).  The polarizations  induce maps between line bundles
	$$
	\Lie \lambda: (\Lie X)_0 \to (\Lie X^\vee)_0, \quad  \Lie \lambda^\vee: (\Lie X^\vee)_0 \to (\Lie X)_0.
	$$
	We introduce \emph{the formal Balloon--Ground stratification} on the special fiber $\CN_\BF$: 
	\begin{itemize}\label{defn: formal Balloon-Ground}
		\item \emph{The formal Balloon stratum} $\CN^\circ$ is the vanishing locus of universal $\Lie \lambda$.
		\item  \emph{The formal Ground stratum} $\CN^\bullet$ is the vanishing locus of universal  $\Lie \lambda^\vee$.
		\item \emph{The formal Linking stratum} $\CN^\dagger$ is the intersection $\CN^\circ \cap \CN^\bullet$.
	\end{itemize}
	
	As $\Lie \lambda^\vee \circ \Lie \lambda =[\varpi]$, we have the stratification $\CN_\BF=\CN^\circ \cup \CN^\bullet$. 
	
	\begin{proposition}\label{prop: formal balloon-ground are smooth}
		The formal subschemes $\CN^\circ$ and $\CN^\bullet$ are Cartier divisors of $\CN$ and are formally smooth over $\BF$.  The structure map $\CN - \CN^{\dagger} \rightarrow O_{\breve{F}_0}$ is formally smooth. And for any closed point $x$ of $\CN^\dagger$, the completed local ring of $\CN$ at $x$ is isomorphic to $O_{\breve{F}}[[x_0,y_0,t_1, \hdots,t_{n-1}]]/(x_0y_0-\varpi)$, and the local equation of $\CN^\circ$ (resp. $\CN^\bullet$) may be taken to be $x_0=0$ (resp. $y_0=0$) .
	\end{proposition}
	\begin{proof}
		By Proposition \ref{prop:semi-stable}, $\CN$ is regular and $\varpi$ is locally non-zero. As vanishing locus of sections of line bundles, $\CN^\circ$ and $\CN^\bullet$ are Weil divisors of $\CN$. As $\Lie \lambda^\vee \circ \Lie \lambda =[\varpi]$, the local generators of $\CN^\circ$ and $\CN^\bullet$ are also locally non-zero, hence $\CN^\circ$ and $\CN^\bullet$ are in fact Cartier divisors of $\CN$. 
		
		We show that $\CN^\circ$ and $\CN^\bullet$ are formally smooth over $\BF$. For any closed point $x$, we pass to the completed local ring $\hat{O}_{\CN, x}$ of $\CN$ at $x$. As $\Lie \lambda \circ \Lie \lambda^\vee =[\varpi]$, the local generator $x^\circ$ (resp. $x^\bullet$) of $\CN^\circ$  (resp. $\CN^\bullet$) can be chosen such that $x^\circ x^\bullet=\varpi$ inside $\hat{O}_{\CN, x}$. By the description of completed local rings in Proposition \ref{prop:semi-stable}, there are two cases: firstly, if $\hat{O}_{\CN, x}$ is isomorphic to $O_{\breve{F}} [[t_0,t_1,  \hdots, t_{n-1}]]$ which is a UFD, then the equation $x^\circ x^\bullet=\varpi$ implies that either $x^\circ$ or $x^\bullet$ is a unit and the other generator agrees with $\varpi$ (up to a unit). In particular both of them are formally smooth over $\BF$ and we are done; second, if $\hat{O}_{\CN, x}$ is isomorphic to  $O_{\breve{F}}[[x_0,y_0,t_1, \hdots,t_{n-1}]]/(x_0y_0-\varpi)$, solving the equation $x^\circ x^\bullet=\varpi$ shows that $x^\circ$ (resp. $x^\bullet$) can be arranged as local generators $x_0$ (resp. $y_0$) in \ref{prop:semi-stable} after applying an automorphism on $\wh{O}_{\CN, x}$. We are done as well. 
		
		To see the structure map $\CN - \CN^{\dagger} \rightarrow O_{\breve{F}}$ is formally smooth, we apply Grothendieck--Messing theory. We only need to deform the Hodge filtration. As $(\omega_X)_1=(\Lie X)^\vee_1$ and $(\omega_{X^\vee})_0=(\Lie X^\vee)^\vee$ are orthogonal complements under the perfect pairing $(-,-)$ hence determine each other, we only need to deform the quotient lines $(\Lie X)_0$ and $(\Lie X^\vee)_0$ in the Hodge filtration. If $x \in \CN -\CN^\dagger$, then one of $(\Lie X)_0$ and $(\Lie X^\vee)_0$ determines the other using the non-zero section $\Lie \lambda$ or $\Lie \lambda^\vee$, so we only need to deform one quotient line, which forms a unobstructed deformation problem.
	\end{proof}
	
	\begin{definition}\label{eq: defn of local blow}
		Let $f: \wt{ \CN \times \CN} \to \CN \times \CN$ be the formal blow up of $\CN \times_{O_{\breve{F}}} \CN$ along the Weil divisor $\CN^{\circ} \times_{\BF} \CN^{\circ}$. 
	\end{definition}
	
	\begin{theorem}\label{resolution} 
		If  $0<t<n$, the formal scheme $\wt{ \CN \times \CN}$ is of semi-stable reduction over $\Spf O_{\breve{F}}$ and is in particular regular.  The map $f: \wt{ \CN \times \CN} \rightarrow \CN \times \CN$ is an isomorphism away from $\CN^{\dagger} \times_\BF \CN^{\dagger}$. The geometric fiber of $f$ is either a point or a projective line $\mathbb P^1_\BF$. In particular $f$ is a small resolution.
	\end{theorem}
	\begin{proof}
		By formal smoothness of $\CN - \CN^{\dagger} \rightarrow O_{\breve{F}_0}$, we know that $\CN \times \CN$ is regular away from $\CN \times \CN^\dagger$ or $\CN^\dagger \times \CN$. Hence the center $\CN^{\circ} \times_\BF \CN^{\circ}$ is a Cartier divisor of $\CN \times \CN$ away from $\CN^{\dagger} \times_\BF \CN^{\dagger}$, where blow up changes nothing.
		
		Blow up commutes with flat base change in particular open immersions, so we can pass to completed local rings. By Proposition \ref{prop: formal balloon-ground are smooth}, we can choose local generators of $\CN^\circ$ (resp. $\CN^\bullet$) as  $x^0$ (resp. $y^0$) in the isomorphism
		\[
		\wh{O}_{\CN, x} \simeq	O_{\breve{F}}[[x_0,y_0,t_1, \hdots,t_{n-1}]]/(x_0y_0-\varpi)
		\]
		for any closed point $x \in \CN^\dagger$. By direct computation, the blow up of 
		\[ 
		O_{\breve{F}}[[x_0,y_0,x_1,y_1,t_1, \hdots, t_{n-1},s_{1}, \hdots, s_{n-1}]]/(x_0y_0-\varpi,x_1y_1-\varpi)
		\]
		along $(x_0,x_1)$ is of strictly semi-stable reduction over $O_{\breve{F}}$. As the ideal $(x_0,x_1)$ is generated by two elements, the fiber of $f$ at $x \in \CN^{\dagger} \times_\BF \CN^{\dagger}$ can be embedded into $\mathbb P^1_\BF$.  By local computations, it is isomorphic to $\mathbb P^1_\BF$ . Hence the resolution $f$ has only $\dim \leq 1$ fibers and is an isomorphism outside the $\text{codim} \geq 3$ locus $\CN^{\dagger} \times_\BF \CN^{\dagger}$, hence a small resolution.
	\end{proof}
	
	The following proposition gives another definition of Balloon--Ground stratification in the spirit of \cite[Defn. 5.2.3]{LTXZZ}.
	
	\begin{proposition}\label{Balloon-Ground: comp with LTXZZ}
		The closed formal subscheme	$\CN^\circ \subseteq \CN_\BF$ is defined by the condition
		\[
		(\BD(X)(S)_0,  (\omega_{X^\vee})_1)_\lambda =0 \in \CO_S.
		\]
		Similarly, the space $\CN^\bullet \subseteq \CN_\BF$ is defined by the condition 
		\[
		(\BD(X)(S)_1)^\perp \subseteq (\omega_{X^\vee})_0.
		\] 
		Here $(-)^\perp$ means the (right) orthogonal complement under $(-,-)_\lambda$ inside $\BD(X)(S)_0$.
	\end{proposition}
	\begin{proof}
		We always have $( (\omega_{X^\vee})_0,  (\omega_{X^\vee})_1)_\lambda =0$, so $(\BD(X)(S)_0,  (\omega_{X^\vee})_1)_\lambda =0$ is equivalent to 
		$((\Lie X)_0, (\omega_{X^\vee})_1 )_\lambda=0$ i.e., $((\Lie X)_0, (\Lie X^\vee)_1 )_\lambda=0$ i.e., $\Lie \lambda=0$. 
		
		On the other hand, $(\BD(X)(S)_1)^\perp$ is the same as the orthogonal complement of $\Im(\BD(\lambda))(\BD(X)(S)_1) \subseteq \BD(X^\vee)(S)_0$ under the perfect pairing $(-,-)$ between $\BD(X)(S)_0$ and $\BD(X^\vee)(S)_0$. By perfectness, we see the condition $(\BD(X)(S)_1)^\perp \subseteq (\omega_{X^\vee})_0$ is equivalent to the following inclusion inside $\BD(X^\vee)(S)_0$:
		$$
		(\Lie X )_0^\vee = (\omega_X)_0 \subseteq \Im(\BD(\lambda))(\BD(X)(S)_1).
		$$
		As $S$ is of characteristic $p$, we have $\Im(\BD(\lambda))(\BD(X)(S)_1)=\Ker(\BD(\lambda^\vee))|_{\BD(X^\vee)(S)_0}$ by \cite[Lem. 3.4.12]{LTXZZ}. So the inclusion is the same as the condition that 
		$$\BD(\lambda^\vee)((\Lie X )_0^\vee )=0$$
		which is exactly $(\Lie \lambda^\vee)^\vee=0$ i.e., $\Lie \lambda^\vee=0$.
	\end{proof}

	\begin{remark}\label{descent to O_F: cycles strata}
		As in Remark \ref{descent to O_F: moduli}, the definition of Kudla--Rapoport cycles and formal Balloon--Ground stratification descends to $O_F$ by the same moduli interpretations.
	\end{remark}

	\subsection{Modifying derived fixed points}
	
	We recall some well-known facts about blow-ups.  
	
	\begin{proposition}
		Let $Z \hookrightarrow X$ be a closed immersion of formal schemes that are locally formally of finite type over $O_{\breve{F}}$.  Consider the blow up morphism $\Bl_{Z}X \to X$ along $Z$. For any closed formal subscheme $Y \rightarrow X$, denote by $\wt Y$ the strict transform of $Y$ in $X$. Then	
		\begin{enumerate}
			\item There is a natural isomorphism $\wt{Y} \cong \Bl_{Z \cap_{X} Y} Y$ over $Y$. Here $Z \cap _X Y$ is the fiber product of $Y$ and $Z$ in $X$.
			\item Either $\wt{Y}$ is empty, or $\dim_{\wt y} \wt{Y}=\dim_y Y$ for any closed point $\wt y$ of $\wt{Y}$ with image $y \in Y$.
			\item Blow up commutes with flat base change $X' \to X$ of the base.
		\end{enumerate}
	\end{proposition}
	
	For a closed formal scheme $Y \to \CN \times \CN$,  write $\wt {Y}$ its strict transform under the resolution (\ref{eq: defn of local blow}).
	
	\begin{remark}
		For the embedding $\delta_{(L, e)}: \CN^\flat \hookrightarrow \CN$ in \S \ref{section: embeddings}, we have $\CN^{\circ} \cap \CN^{\flat}=(\CN^{\flat})^{\circ}$. Hence for closed formal subscheme $Y \hookrightarrow \CN^\flat \times \CN^\flat$, the strict transform $\wt{Y}$ is independent of whether the strict transform is taken in $\CN^\flat \times \CN^\flat$ or  $\CN \times \CN$.
	\end{remark}

	\begin{remark}
		The resolution $f$ has a simple model. Consider the quadratic cone $X:=V(x_0y_0=x_1y_1) \subseteq \BA^4_{\BF}=\Spec{\BF[x_0, y_0, x_1, y_1]}$ which is singular at the origin. The resolution $f$ above is similar to the classical Atiyah flop \cite{Atiyah} $f_{x_0x_1}: X_{x_0x_1} \to X$ obtained by blowing up along $V(x_0=x_1=0) \subseteq X$. Consider similar resolutions $f_{y_0y_1}, \, f_{x_0y_1}, \, f_{x_1y_0}, \, f_{x_0y_0x_1y_1}$ of $X$ where $f_{x_0y_0x_1y_1}$ means blowing up of $X$ along the origin. Then we have identifications $f_{x_0x_1}=f_{y_0y_1}$, $f_{x_0y_1}=f_{x_1y_0}$ and a cartesian diagram
		
		\[
		\begin{tikzcd}
			& {X_{x_0y_0x_1y_1}} \\
			{X_{x_0x_1}} && {X_{x_0y_1} } \\
			& X
			\arrow[from=1-2, to=2-1]
			\arrow[from=2-1, to=3-2]
			\arrow[from=1-2, to=2-3]
			\arrow[from=2-3, to=3-2]
		\end{tikzcd}
		\]
		whose base change to the origin $\Spec \BF \to X$ is isomorphic to projections of $\BP^1 \times \BP^1 $ to two factors $\BP^1$.  The above resolution $f$ agrees with the blow up of $\CN \times \CN$ along products of Ground strata $\CN^\bullet \times_\BF \CN^\bullet$. We do not consider the blow up along $\CN^\bullet \times_\BF \CN^\circ$ as the diagonal $\Delta$ doesn't lift to it via strict transforms.
	\end{remark}

	\begin{proposition}\label{lem: local str=tot graph}
		For any automorphism $g \in \U(\BV)(F_0)$, the strict transform of its graph  $\Gamma_g \hookrightarrow  \CN \times \CN$ is an isomorphism $\wt{\Gamma_g} \simeq \Gamma_g$.
	\end{proposition}
	\begin{proof}
		The formal Balloon stratum $\CN^\circ$ is stable under $g$ by moduli interpretation. So the pullback of $\CN^\circ \times_\BF \CN^\circ$ to $\Gamma_g$  is isomorphic to $\CN^\circ  \hookrightarrow \CN$ via projection to the first factor $\CN$. As $\Gamma_g \cong \CN$ is regular, the Weil divisor $i_g: \CN^\circ \to \CN$ is Cartier hence the blow up changes nothing. 
	\end{proof}
	
	Therefore, we have a natural lifting $\Gamma_g=\wt{\Gamma_g} \to \wt {\CN \times \CN}$. Take $g=\id$, we get a natural lifting of the diagonal embedding $\Delta: \CN \to \CN \times \CN$:
	\[ 
	\wt{\Delta}: \CN=\wt{\CN} \to \wt{\CN \times \CN}.
	\] 
	
	\begin{definition}\label{def: Mod(g)} 
		For $g \in \U(\BV)(F_0)_{rs}$, the \emph{fixed point locus} of $g$ inside $\CN$ is the formal scheme $\Fix(g):=\Gamma_g \cap_{\CN \times \CN} \Delta$. The \emph{modified fixed locus} of $g$ is the fibered product 
		$$
		\wt {\Fix}(g):= \wt{\Gamma_g} \times _{\wt{ \CN \times \CN} } \wt{\Delta}
		$$ 
		in the category of formal schemes over $O_{\breve{F}}$. The \emph{modified derived fixed locus} of $g$ is the derived tensor product
		\[
		\xymatrix{ \wt {\Fix}^\BL  (g) = \CO_{\wt{\Gamma_g} } \otimes^{\BL}_{\wt{ \CN \times \CN} }}   \CO_{\wt{\CN} }
		\]
		viewed as an element in the Grothendieck $K$-group  $K_0'(\CN)_{\wt{\Fix}(g)}$ ($\BQ$-coefficients) of coherent sheaves on $\CN$ supported on $\wt {\Fix}(g)$. 	
	\end{definition} 
	As $\wt{\Gamma_g} \to \wt{ \CN \times \CN}$ is a closed immersion, $\wt {\Fix}(g) \to \wt{\CN} = \CN$ is a closed immersion. We have a map $\wt{\Fix}(g) \to \Fix(g)$, which is an isomorphism if $\Fix(g) \cap  \CN^{\dagger} =\emptyset$.

	\subsection{Embeddings}\label{section: embeddings}
	
	Let $L=L^{\flat} \oplus O_Fe$ be an orthogonal decomposition of vertex lattices where $v_F((e,e)_V)=i \in \{0,1\}$. We have $t(L^{\flat})=t(L)-i$. Choose a decomposition of framing objects:
	\begin{equation}
		(\BX, \iota_{\BX}, \lambda_{\BX})=  (\BX^\flat, \iota_{\BX^\flat}, \lambda_{\BX^\flat}) \times  (\BE^{[i]}, \iota_{\BE^{[i]}}, \lambda_{\BE^{[i]}}).
	\end{equation}
	Set $\CN^\flat:=\mathcal{N}_{\U(L^\flat)}$ and $ \CN:=\mathcal{N}_{\U(L)}$ with above framing objects. There is a natural embedding:
	\begin{align}\label{gp embedding of RZ}
		\delta=\delta_{(L, e)}: \CN^\flat \to \CN
	\end{align} 
	\begin{align*}
		(X^\flat, \iota^\flat, \lambda^\flat, \rho^\flat) \longmapsto  (X, \iota, \lambda, \rho):=  (X^\flat, \iota^\flat, \lambda^\flat, \rho^\flat) \times  (\CE^{[i]}, \iota_{\CE^{[i]}}, \lambda_{\CE^{[i]}}, \rho_{\CE^{[i]}}).
	\end{align*}
	
	Taking dual, we have decompositions of vertex lattices and framing objects:
	\[
	\lambda_V  (L^{\vee})=\lambda_{V^\flat}((L^{\flat})^\vee) \oplus O_F \lambda_V \frac{e}{(e,e)_V}, 
	\]
	\[
	(\BX^\vee, \iota_{\BX^\vee}, \lambda_{\BX^\vee})=  (\BX^{\vee, \flat}, \iota_{\BX^{\vee, \flat}}, \lambda_{\BX^{\vee,\flat}}) \times  (\BE^{[i^\vee]}, \iota_{\BE^{[i^\vee]}}, \lambda_{\BE^{[i^\vee]}}).
	\]
	Here $i^\vee:=v((\lambda_V \frac{e}{(e,e)_V}, \lambda_V \frac{e}{(e,e)_V})_{V^\vee})=1-i$.
	We have an embedding $$\delta_{(L,e)}^\vee:=\delta_{(\lambda_V (L^{\vee}), \lambda_V \frac{e}{(e,e)_V})}: \CN^{\vee, \flat} \to \CN^\vee.
	$$
	From the moduli interpretation, the embeddings are compatible under dual isomorphisms:
	\begin{equation}\label{compatibility of embedding under duality}
		\lambda_\CN \circ \delta_{(L, e)}=\delta_{(L,e)}^\vee \circ  \lambda_{\CN^\flat}.
	\end{equation}
	Define the vector $e^{[i]} \in \BV=\Hom^\circ_{O_E}(\BE^{[0]}, \BX)$ by
	\begin{equation}
		e^{[i]}=\begin{cases}
			\BE^{[0]} \to \BX^\flat \times \BE^{[0]} =\BX, & i=0. \\
			\BE^{[0]} \overset{\lambda_{\BE^{[0]}}}{\cong} \BE^{[1]} \to \BX^\flat \times \BE^{[1]} = \BX, & i=1.
		\end{cases}
	\end{equation}
	
	\begin{proposition}\label{base case of Z and Y}
		\begin{enumerate}
			\item If $i=0$, then $\CN^{\flat}=\CZ(e^{[0]})$ under the embedding $\delta_{(L, e)}$.
			\item If $i=1$, then $\CN^{\flat}=\CY(\varpi^{-1}e^{[1]})$ under the embedding $\delta_{(L, e)}$.
		\end{enumerate}
	\end{proposition}
	\begin{proof}
		If $i=0$, then the valuation of $e^{[0]}$ is a unit. Hence for any point $(X, \iota_X, \lambda_X, \rho_X)$ of $\CZ(e^{[0]})(S)$, the lifting of $\rho^{-1}_X \circ e^{[0]}$ to $X$ gives a splitting of $X$ to identify it with a point in $\CN^\flat(S)$. The first property follows. The second property follows by duality using the dual isomorphism (\refeq{compatibility of embedding under duality}).
	\end{proof}

	\section{Arithmetic transfer conjectures and reductions}\label{section: conj ATC}
	
	\subsection{Semi-Lie version}
	
	For any bounded complex $\CF$ in the $K$-group of coherent sheaves on the regular formal scheme $\CN$ with proper support, recall  the Euler characteristic of $\CF$ is 
	$$
	\chi(\CN , \CF):=\sum_{i \in \BZ} (-1)^i\length_{O_{\breve{F}}} H^i(\CN, \CF).
	$$
	
	\begin{definition}
		For any regular semisimple pair  $(g, u) \in (\U(\BV) \times \BV)(F_0)_{rs} $, we consider the following \emph{derived intersection numbers}
		\begin{equation}
			\wt {\Int}^{\CZ}(g, u):=\chi( \CN, \CO_{\CZ(u)} \otimes^\BL \wt {\Fix}^\BL(g)).
		\end{equation}
		\begin{equation}
			\wt {\Int}^{\CY}(g, u):=\chi( \CN, \CO_{\CY(u)} \otimes^\BL \wt {\Fix}^\BL(g)).
		\end{equation}
	\end{definition}

Note that $\wt {\Fix}^\BL(g)$ has bounded degrees by the regularity of $\wt {\CN \times \CN}$. 	By Proposition \ref{Int = proper scheme} below, $\wt {\Int}^\CZ(g,u)$ and $\wt {\Int}^\CY(g,u)$ are well-defined and finite.
	
	\begin{proposition}\label{Int = proper scheme}
\begin{enumerate}
	\item 	The intersections $\CZ(u) \cap \Gamma_g$ and $\CY(u) \cap \Gamma_g$ are \emph{proper scheme} over $O_{\breve{F}}$. 
	\item $\wt{ \CZ(u) \cap \Gamma_g} \subseteq \wt{ \CZ(u)}  \cap \wt {\Gamma_g}$ are both proper schemes over 
	$O_{\breve F}$. 
	\item $\wt{ \CY(u) \cap \Gamma_g} \subseteq \wt{ \CY(u)}  \cap \wt {\Gamma_g}$ are both proper schemes over 
	$O_{\breve F}$.
\end{enumerate}
	\end{proposition}
	\begin{proof}
We prove the first property firstly. Using the dual isomorphism, it is sufficient to prove that $\CZ(u) \cap \Gamma_g$ is a proper scheme. Note $\CZ(u) \cap \Gamma_g$ is a closed subscheme of the Kudla--Rapoport cycle $\CZ(L_{g,u})$ where $L_{g,u}$ is the lattice spanned by $\{u, gu, \hdots, g^{n-1}u\}$. We only need to show $\CZ(L_{g,u})$ is a proper scheme. By regular semi-simpleness, $L_{g,u}$ is a full lattice in $\BV$. We follow the argument of Li--Zhang \cite[Lemma 2.10.1]{LZ19} via vertical and horizontal parts. The horizontal part of $\CZ(L_{g,u})$ is empty: otherwise we may find a $O_K$-point $\CX$ of $\CZ(L_{g,u})$ for some extension $K$ of $\breve{F}$, and get an isogeny $\mathcal{E}^n \to \CX$ induced by $\{u, gu, \hdots, g^{n-1}u\}$, contradict to the assumption that $\CX$ has Kottwitz signature $(1,n-1)$. The vertical part of $\CZ(L_{g,u})$ is a scheme by the argument in \cite[Lemma 5.1.1]{LZ19}, and is contained in finitely many (proper) irreducible components of $\CN^\red$ using the Bruhat--Tits stratification in our set up (see \S \ref{section: local BT strata}) hence proper. 

Now the first property holds. As the blow up morphism is proper, the natural map $f: \wt{ \CZ(u)}  \cap \wt {\Gamma_g} \to \CZ(u) \cap \Gamma_g$ is proper with image being a proper scheme. Hence $\wt{ \CZ(u)}  \cap \wt {\Gamma_g}$ is also proper scheme, so is its closed formal subscheme $\wt{ \CZ(u) \cap \Gamma_g}$. The second property holds, and the third property follows by duality.
	\end{proof}

	\begin{remark}
Finiteness of intersection numbers for regular semi-simple orbits appear in other contexts, see e.g. \cite[\S 9.2]{AFL2021}, \cite[Lemma 2.8]{AFL-Invent} via basic uniformization.  
	\end{remark}
	
	Recall that in \S \ref{section: local set up transfer} we have set up orbital integrals on the analytic side. Choose an orthogonal basis of $L$ to endow $L$ (resp. $V$) with an $O_{F_0}$ (resp. $F_0$)-structure $L_0$ (resp. $V_0$). Consider the symmetric space over $F_0$
	$$S(V_0)=\{  \gamma \in \GL(V) | \gamma \ov \gamma =id \,\}.$$ 
	The general linear group $\GL(V_0)$ acts on $S(V_0) \times V_0 \times V_0^*$ by
	$$ h.(\gamma,u_1,u_2)=(h^{-1}\gamma h,h^{-1}u_1, u_2h). $$
	
	For $f' \in \CS(S(V_0) \times V_0 \times V_0^*)(F_0))$, recall the derived orbital integral (\ref{derived orbital int}). 
	\[ 
	\del\bigl((\gamma,u_1,u_2), f'):= \omega_{L}(\gamma, u_1, u_2) \frac{d}{ds}\Big|_{s=0}  \int_{h \in \GL(V_0)} f'(h.(\gamma,u_1,u_2)) \eta(h) |h|^{s} d h
	\]
	
	We consider the following test functions on $(S(V_0) \times V_0 \times V_0^* )(F_0)$:
	$$
	f_{\text{std}}:=1_{S(L, L^{\vee})} \times 1_{L_0} \times 1_{(L^{\vee}_0)^*}, \quad f_{\text{std}}'=1_{S(L, L^{\vee})} \times 1_{L_0^\vee} \times 1_{(L_0)^*}.
	$$
	
	Now we introduce the explicit arithmetic transfer conjectures at maximal parahoric levels.
	
	\begin{conjecture}[Semi-Lie ATC for $L$] \label{conj: semi-Lie version ATC}
		For any regular semisimple pair $(g, u) \in (\U(\BV) \times \BV)(F_0)_\rs$ matching $(\gamma, u_1, u_2) \in ( S(V_0) \times V_0 \times V_0^* )(F_0)_\rs$, we have equalities in $\BQ \log q$:
		\begin{enumerate}
			\item $
			\del\bigl ((\gamma,u_1,u_2),  f_{\text{std}})
			=-\wt {\Int}^{\CZ}(g, u) \log q. $
			\item 	$
			\del\bigl ((\gamma,u_1,u_2),  f_{\text{std}}')
			=-(-1)^{t} \wt {\Int}^{\CY}(g, u) \log q.
			$
		\end{enumerate}
		
	\end{conjecture}
	
	\begin{proposition}
		The part $(2)$ of Conjecture \ref{conj: semi-Lie version ATC} for $L \subseteq V$ is equivalent to part $(1)$ of Conjecture \ref{conj: semi-Lie version ATC} for the dual vertex lattice $\lambda_V( L^\vee)  \subseteq V^\vee$.
	\end{proposition}
	\begin{proof}
		By Proposition \ref{prop: relation of Z and Y}, we have $\wt {\Int}^\CY_{\CN}(g, u)=\wt {\Int}^\CZ_{\CN^\vee}(g, \lambda_\BV \circ u)$. On the analytic side, we have 
		\[
		f_{\text{std}, L}' (\gamma, u_1, u_2)= f_{\text{std}, \lambda_V ( L^\vee)}(\gamma, u_1, u_2).
		\]
		We are done by noting the transfer factor $\omega_L$ and $\omega_{\lambda_V( L^\vee)}$ for $L$ and $\lambda_V( L^\vee)$ differ by the multiplication of the Weil constant $\gamma_V=(-1)^t$.
	\end{proof}

	\subsection{Group version}
	
	We have group version arithmetic transform conjectures for any orthogonal decomposition of vertex lattices $L=L^{\flat} \oplus O_Fe$ where $v((e,e)_V)=i \in \{0,1\}$.
	
	\begin{definition}
		With respect to the embedding $\delta_{(L, e)}: \CN^\flat \to \CN$ (\ref{gp embedding of RZ}), we consider the \emph{derived intersection number} of $g \in \U(\BV)(F_0)_{\rs}$ as the following Euler characteristic number
		\begin{equation}
			\wt {\Int}(g):=\chi( \CN,  \CO_{\CN^\flat} \otimes^\BL \wt {\Fix}^\BL(g)).
		\end{equation}
	\end{definition}
	
	By Proposition \ref{base case of Z and Y}, we have
	\begin{equation}
		\wt {\Int}(g)=\begin{cases}
			\wt{\Int}^\CZ(g, e^{[0]}),  & \text{if $i=0$}. \\
			\wt{\Int}^\CY(g, \varpi^{-1}e^{[1]}), & \text{if $i=1$}. \\
		\end{cases}
	\end{equation}
	
	The GGP type embedding $\CN^\flat \to \CN^\flat \times \CN$ lifts to an embedding of formal schemes 
	$$\CN^\flat \to \wt{\CN^\flat \times \CN}$$
	via strict transforms .
	
	\begin{proposition}\label{lem: comp of blow ups}
		The natural commutative diagram 
		\[
		\xymatrix{
			\wt { \CN^\flat \times \CN^\flat} \ar[r] \ar[d] & \wt { \CN \times \CN}  \ar[d]\\		 \CN^\flat \times \CN^\flat	 \ar[r] & \CN \times \CN }
		\]
		is cartesian. Moreover, the commutative diagram for the embedding $\wt { \CN^\flat \times \CN} \to \wt { \CN \times \CN}$ over $ { \CN^\flat \times \CN} \to { \CN \times \CN}$ is also cartesian.
		
	\end{proposition}
	\begin{proof}
		Denote the fiber product by $X$, then we have a closed immersion $i: \wt { \CN^\flat \times \CN^\flat} \to  X$ over $\CN^\flat \times \CN^\flat$. It is sufficient to show $i$ is an isomorphism when localizing at every closed point $P \in \CN^\flat \times \CN^\flat$. The blow up is an isomorphism outside $\CN^\dagger \times_\BF \CN^\dagger$, we can assume $P \in \CN^\dagger \times_\BF \CN^\dagger$. The formulation of Balloon--Ground strata is compatible with the embedding $\delta_{(L, e)}$.  By Proposition \ref{prop: formal balloon-ground are smooth} and Theorem \ref{resolution}, we can choose local generators such that the embedding
		$
		\CN^\flat \times \CN^\flat	 \rightarrow \CN \times \CN
		$ at $P$ is isomorphic to the standard embedding from
		$$
		O_{\breve{F}}[[x_0,y_0,x_1,y_1, t_1, \hdots, t_{n-1}, s_{1}, \hdots, s_{n-1}]]/(x_0y_0-\varpi,x_1y_1-\varpi)
		$$
		$$ 
		\surj O_{\breve{F}}[[x_0,y_0,x_1,y_1, t_1, \hdots, t_{n-2}, s_{1}, \hdots, s_{n-2}]]/(x_0y_0-\varpi,x_1y_1-\varpi)
		$$
		where $x_0$ (resp. $y_0$) is a local generator for $\CN^\circ$ (resp. $\CN^\bullet$). Then the compatibility follows from direct computations. The claim for the embedding $\wt { \CN^\flat \times \CN} \to \wt { \CN \times \CN}$ follows similarly.
	\end{proof}
	
	The formal scheme $\wt{\CN^\flat \times \CN}$ is regular by compatibility of blow ups in Proposition \ref{lem: comp of blow ups}.
	
	\begin{proposition}\label{prop: gp TC for GGP embedding}
		For $g \in \U(\BV)(F_0)_\rs$, we have 
		\[ 
		\wt{\Int} (g)=\chi( \wt{\CN^\flat \times \CN}, \CO_{\CN^\flat} \otimes^\BL \CO_{(1 \times g) \CN^\flat}).
		\]
	\end{proposition}
	\begin{proof}
		By definition, $\wt{\Int}(g)$ is the derived intersection number between $\CN^\flat$ and $(1 \times g) \CN$ inside the resolution $\wt{\CN \times \CN}$. By projection formula, we only need to show the (derived) pullback of $(1 \times g) \CN$ along the embedding $ \wt{\CN^\flat \times \CN } \to \wt{\CN \times \CN}$ agrees with $(1 \times g) \CN^\flat$. This is true before taking blow ups, and we are done by the compatibility of blow ups in Proposition \ref{lem: comp of blow ups}.
	\end{proof}
	
	Consider the conjugacy action of $\GL(V_0^{\flat})$ on $S(V_0)$. For $f' \in \mathcal{S}(S(V_0)(F_0))$ and $\gamma' \in S(V_0)(F_0)_\rs$, recall the derived orbital integral
	\[ 
	\del\bigl (\gamma', f')= \omega_L(\gamma')  \frac{d}{ds}\Big|_{s=0} \int_{h \in \GL(V_0^{\flat})} f'(h^{-1}\gamma' h) \eta(\det(h)) |\det h|^{s} d h
	\]
	where $\omega_L(\gamma') :=\eta( \det (\gamma'^{i} e )_{i=0}^{n-1}) \in \{\pm 1\}$ is the transfer factor for the basis of $L$ (\ref{transfer factor: gp}).
	
	\begin{conjecture} (Group ATC for $(L, e)$) \label{conj: gp version ATC}
		For regular semisimple element $g' \in \U(\BV)(F_0)_{rs}$ matching $\gamma' \in S(V_0)(F_0)_{rs}$, we have an equality in $\BQ \log q$
		\begin{equation}
			\del\bigl (\gamma', 1_{S(L,L^{\vee}) } )= -\wt {\Int}(g') \log q.
		\end{equation}	
	\end{conjecture}
	
	The group version arithmetic transfer conjecture includes the case \cite[Section 10]{RSZ-regular} as a special case i.e., $L=L^\flat \oplus O_F e$ where $L^\flat$ is self dual and $(e, e)_V=\varpi$.

	\subsection{Reduction of intersection numbers}\label{section: reduction Int}
	
	Consider any orthogonal decomposition of vertex lattices $L=L^{\flat} \oplus O_Fe$ where $v_F((e,e)_V)=i \in \{0,1\}$. In this subsection, we show the reduction from Conjecture \ref{conj: gp version ATC} for $(L,e)$ to part $(i+1)$ of Conjecture \ref{conj: semi-Lie version ATC} for $L^\flat$. 
	
	Assume that $i=0$ firstly. Consider the map $\ast: \Hom_{O_F}(\BE, \BX^\flat) \to \Hom_{O_F}( \BX^\flat,\BE)$ defined by $f \mapsto \lambda_{\BE}^{-1}\circ f^{\vee} \circ \lambda_{\BX^\flat}$. It induces an isomorphism of $F$-vector spaces
	\begin{equation}
		\ast:  \BV^\flat=\Hom^\circ_{O_F}(\BE, \BX^\flat) \simeq  (\BV^\flat)^*=\Hom^\circ_{O_F}( \BX^\flat,\BE).
	\end{equation}
	
	Write any element $g'\in \U(\BV)(F_0)$ in the matrix form
	\begin{align}\label{diag g'}
		\xymatrix{
			\BX^\flat \times \BE \ar[rr]^-{g'=\left(\begin{matrix} a & u\\
					w & d
				\end{matrix}\right)} &		   &  \	\BX^\flat \times \BE  } 
	\end{align}
	where 
	$
	a \in \End^\circ_{O_F}(\BX^\flat)=\End(\BV^\flat) $,  \, $ u\in \BV^{\flat}$, \, $w \in (\BV^{\flat})^*$  and $d \in \End^\circ_{O_F}(\BE)=F$.
	
	For the decomposition $\BV=\BV^\flat \oplus F e^{[i]}$, consider its unitary relative Cayley map  (\ref{UnitaryCayley}):
	\begin{equation}
		\mathfrak{c}_\U: \U(\BV) \to \U(\BV^{\flat}) \times \BV^{\flat},
	\end{equation}
	\begin{equation}\label{Reduction Int by Cayley}
		g' =
		\begin{pmatrix}
			a & u \\ w & d 
		\end{pmatrix} \mapsto (g= a +\frac{uw}{1-d}, \, u_1=\frac{u}{1-d}). 
	\end{equation}
	
	For $g' \in \U(\BV)(F_0)_{rs}$ with $\mathfrak{c}_\U(g')=(g, u_1)$, consider the graph $\Gamma_{g'}$ (resp. $\Gamma_g$) of $g'$ (resp. $g$) inside $\CN \times \CN$ (resp. $\CN^\flat \times \CN^\flat$). 
	
	\begin{lemma}\label{lem: induction}
		If $i=0$ and $1-d \in O_F^{\times}$, then we have
		$$
		(\CN^\flat \times \CZ(u_1) ) \cap_{\CN^\flat \times \CN^\flat } \Gamma_g= (\CN^\flat \times \CN^\flat) \cap_{\CN \times \CN} \Gamma_{g'}.
		$$
	\end{lemma}
	\begin{proof}
As $1-d \in O_F^\times$, we have $\CZ(u)=\CZ(u_1)$.	Set $\epsilon_d:=\frac{1-\ov d}{(1-d)(e,e)} \in F^\times$ which is an unit by our assumption. By Proposition \ref{Unitaryproj}, for the element $w^\vee:=g^{-1}\epsilon_d u$ we have
		$$
		(w^\vee)^*=w \in (\BV^\flat)^*.
		$$
		For any test scheme $S$, consider a $S$-point $(X_1, X_2)$ of  $\CN^\flat \times \CN^\flat$.  On the one hand, if $(X_1,X_2)$ is on the graph $\Gamma_{g'}$ via the embedding $\delta: \CN^\flat \to \CN$, then we have a homomorphism  $\varphi':X_1 \times \CE \to X_2 \times \CE$ lifting $g'$. Write $\varphi'$ in the matrix form
		\[
		\xymatrix{
			X_1 \times \CE \ar[rr]^-{\varphi'=\left(\begin{matrix} \varphi & \psi_u \\
					\psi_w & d
				\end{matrix}\right)} &		 & X_2\times \CE} 
		\]
		lifting the diagram \eqref{diag g'} (after applying framings of $X_1$ and $X_2$). Here $\varphi: X_1 \rightarrow X_2$, $\psi_u: \CE \rightarrow X_2$, $\psi_w: X_1 \rightarrow \CE$ are all homomorphisms, in particular $X_2 \in \CZ(u)$. As $1-d \in O_F^{\times}$, the quasi-isogeny $$\wt{\varphi}=\varphi+(1-d)^{-1}\psi_u \circ \psi_w$$ is a homomorphism lifting $g$ (\ref{Reduction Int by Cayley}). Hence $(X_1, X_2)$ is on the intersection $(\CN^\flat \times \CZ(u) ) \cap_{\CN^\flat \times \CN^\flat} \Gamma_{g}$.
		
		On the other hand, we start with homomorphisms $\psi_u: \CE \to X_2$ (resp. $\varphi_g : X_1 \to X_2$) lifting $u$ (resp. $g$). By the formula $w^\vee=g^{-1}\epsilon_d u$, the homomorphism ($\epsilon_d $ is a unit) $$\psi_{w^\vee}:=[\epsilon_d] \circ  \varphi^{-1}_g \circ \psi_u: \CE \to X_2$$  lifting the vector $w^\vee \in \BV^\flat$.  As $i=0$, the polarization $\lambda_\CE: \CE \to \CE^\vee$ is an isomorphism. Then the homomorphism	$ \psi_w: = \lambda_{\CE}^{-1} \circ (\psi_{w^\vee})^{\vee} \circ \lambda_{X} $
		gives the desired lifting of $w$. The homomorphism $\varphi:= \varphi_g - (1-d)^{-1} \psi_u \circ \psi_w$ lifts $a$, hence we get the desired lifting of $g'$. Therefore, $(X_1, X_2)$ is on the intersection $(\CN^\flat \times \CN^\flat) \cap_{\CN \times \CN} \Gamma_{g'}$.
	\end{proof}
	
	\begin{remark}
		This generalizes the reduction lemma \cite[Lem. 4.4]{AFL} to maximal parahoric levels.
	\end{remark}

	\begin{proposition}\label{update1}
		We have an identification of two fiber products as  closed formal subschemes of $ \wt {\CN^\flat \times \CN^\flat} $:
		$$
		(\CN^\flat \times \CN^\flat ) \cap_{\CN \times \CN} \Gamma_{g'} = ( \wt {\CN^\flat \times \CN^\flat} ) \cap_{\wt {\CN \times \CN } } \wt{\Gamma_{g'}} . 
		$$
	\end{proposition}
	\begin{proof}	
		The strict transform of $\Gamma_{g'}$ equals to itself by Proposition \ref{lem: local str=tot graph}. Then $T= ( \wt {\CN^\flat \times \CN^\flat} ) \cap_{\wt {\CN \times \CN } } \wt{\Gamma_{g'}}$ is the fiber product of the diagram:
		\[
		\xymatrix{
			T \ar[r] \ar[d] \ar@{}[rd]|*{\square}  & \ \wt {\Gamma_{g'}}=\Gamma_{g'} \ar[d]  \\
			\wt { \CN^\flat \times \CN^\flat} \ar[r] \ar[d]  & \wt { \CN \times \CN}  \ar[d]\\		 \CN^\flat \times \CN^\flat	 \ar[r] & \CN \times \CN }
		\]
		The bottom square is cartesian by Proposition \ref{lem: comp of blow ups}. So $T=	(\CN^\flat \times \CN^\flat ) \cap_{\CN \times \CN} \Gamma_{g'} $.
	\end{proof}

	Consider the small resolution $f^\flat: \wt { \CN^\flat \times \CN^\flat} \rightarrow \CN^\flat \times \CN^\flat $ introduced before for $\CN^\flat$.
	
	\begin{proposition}\label{update2}
		We have an identification of two fiber products as closed formal subschemes of  $\wt {\CN^\flat \times \CN^\flat} $:		
		$$
		(\CN^\flat \times \CZ(u_1) ) \cap_{\CN^\flat \times \CN^\flat } \Gamma_g = (f^{\flat})^{-1} (\CN^\flat \times \CZ(u_1) ) \cap_{\wt { \CN^\flat \times \CN^\flat}} \wt {\Gamma_g},
		$$
	\end{proposition}
	\begin{proof}
		By Proposition \ref{lem: local str=tot graph}, the graph $\Gamma_g=\wt {\Gamma_g} \rightarrow\CN^\flat \times \CN^\flat$ factors through $f^\flat$. The equality follows by base change.
	\end{proof}

	\begin{proposition}\label{update3}
		Assume that $u \in \BV^\flat$ is non-zero, then we have equalities in the $K$-group of $\wt {\CN^\flat \times \CN^\flat}$
		\begin{equation}
			( \wt {\CN^\flat \times \CN^\flat} ) \cap_{\wt {\CN \times \CN } } \wt{\Gamma_{g'}} = ( \wt {\CN^\flat \times \CN^\flat} ) \cap^\BL_{\wt {\CN \times \CN } } \wt{\Gamma_{g'}},
		\end{equation}
		\begin{equation}
			(f^{\flat})^{-1} (\CN^\flat \times \CZ(u) ) \cap_{\wt { \CN^\flat \times \CN^\flat}} \wt {\Gamma_g} =(f^{\flat})^{-1} (\CN^\flat \times \CZ(u) ) \cap^\BL_{\wt { \CN^\flat \times \CN^\flat}} \wt {\Gamma_g}.
		\end{equation}
	\end{proposition}
	\begin{proof}
		By Lemma \ref{Prop: Z and Y are Cartier divisors}, $\CZ(u)$ is a Cartier divisor of $\CN^\flat$. 
		By Lemma \ref{lem: induction}, Proposition \ref{update1} and Proposition \ref{update2}, the intersection $( \wt {\CN^\flat \times \CN^\flat} ) \cap_{\wt {\CN \times \CN } } \wt{\Gamma_{g'}}  = (\CN^\flat \times \CN^\flat ) \cap_{\CN \times \CN} \Gamma_{g'}$  has expected dimension $\dim \CZ(u)=n-2$, and $(f^{\flat})^{-1} (\CN^\flat \times \CZ(u) ) \cap_{\wt { \CN^\flat \times \CN^\flat}} \wt {\Gamma_g}$ has expected dimension $\dim \CZ(u)=n-2$.
		
		As $\wt { \CN^\flat \times \CN^\flat}$ and $\wt { \CN \times \CN}$ are regular, we conclude by applying \cite[Lem. B.2]{AFL}. The assumptions in \emph{loc. cit.} are satisfied: $\CN^\flat \times \CZ(u) $ is a Cartier divisor of $\CN^\flat \times \CN^\flat$, so the pullback $(f^{\flat})^{-1} (\CN^\flat \times \CZ(u))$ is a Cartier divisor in the regular formal scheme $\wt {\CN^\flat \times \CN^\flat}$. Hence $(f^{\flat})^{-1} (\CN^\flat \times \CZ(u))$ is Cohen-Macaulay and of pure dimension $2n-4$. 
	\end{proof}

	Now we finish the reduction of group version ATCs to semi-Lie version ATCs.	
	
	\begin{theorem}\label{thm: semi-Lie and gp}
		Consider any regular semi-simple element $g' \in \U(\BV)(F_0)_{rs}$. Let $\mathfrak{c}_\U(g')=(g, u_1) \in (\U(\BV^\flat) \times \BV^\flat)(F_0)_\rs$. 
		\begin{enumerate}
			\item If $i=0$, then we have 
			\[
			\wt{\Int}(g')=\wt{\Int}^\CZ(g, u_1).
			\]
			\item If $i=1$, then we have
			\[
			\wt{\Int}(g')=\wt{\Int}^\CY(g, u_1).
			\]	
			\item Conjecture \ref{conj: gp version ATC} for $(L,e)$ is implied by part $(i+1)$ of Conjecture \ref{conj: semi-Lie version ATC} for $L^\flat$.
		\end{enumerate}
		
	\end{theorem}
	\begin{proof}
		Assume that $i=0$ firstly. Twisting $g'$ by elements in $F^{\Nm_{F/F_0}=1}$, we can assume $1-d \in O_F^\times$. Then $\CZ(u_1)=\CZ(u)$. To simplify the notations, we write the space to mean its coherent structure sheaves in the K-groups. By Proposition \ref{update3},  we have 
		\[
		( \wt {\CN^\flat \times \CN^\flat} ) \otimes^\BL_{\wt {\CN \times \CN } } \wt{\Gamma_{g'}} =(f^{\flat})^{-1} (\CN^\flat \times \CZ(u) ) \otimes^\BL_{\wt { \CN^\flat \times \CN^\flat}} \wt {\Gamma_g}.
		\]
		Apply $-\otimes^{\BL}_{\wt {\CN^\flat \times \CN^\flat} } \wt {\CN^\flat} $ to both sides, as elements in $K_0'(\wt{\Fix}(g) \cap \CZ(u) )$ we obtain
		\begin{align*} 
			(\wt {\Fix}^\BL (g'), \CN^\flat ) _{\CN} &=\wt {\CN^\flat} \otimes^{\mathbb{L}}_{\wt {\CN \times \CN}} \wt {\Gamma_{g'}} \\ 
			&= \wt {\CN^\flat} \otimes^{\BL}_{\wt {\CN^\flat \times \CN^\flat}}  ( \wt {\CN^\flat \times \CN^\flat} \otimes^{\BL}_{\wt {\CN \times \CN}} \wt {\Gamma_{g'}}) \\
			&= \wt {\CN^\flat} \otimes^{\BL}_{\wt {\CN^\flat \times \CN^\flat}} ((f^{\flat})^{-1} (\CN^\flat \times \CZ(u) ) \otimes^\BL_{\wt { \CN^\flat \times \CN^\flat}} \wt {\Gamma_g} ) \\ 
			&=  ( \wt {\CN^\flat} \otimes^{\BL}_{\wt {\CN^\flat \times \CN^\flat}}  (f^{\flat})^{-1}(\CN^\flat \times \CZ(u) ) ) \otimes^{\BL}_{\wt { \CN^\flat }} \wt {\Fix}^\BL (g).
		\end{align*}
		We have 
		$$
		\wt {\CN^\flat} \otimes_{\wt {\CN^\flat \times \CN^\flat}}  (f^{\flat})^{-1} (\CN^\flat \times \CZ(u) )=\CN^\flat \otimes_{\CN^\flat \times \CN^\flat}  (\CN^\flat \times \CZ(u) )=\CZ(u)
		$$ which is of the expected dimension $n-1+2n-4 - (2n-3)=n-2$. By \cite[Lem. B.2]{AFL}, we have
		$$
		\wt {\CN^\flat} \otimes^\BL_{\wt {\CN^\flat \times \CN^\flat}}  (f^{\flat})^{-1} (\CN^\flat \times \CZ(u) ) =\CZ(u).
		$$
		Therefore, $\wt{\Int}(g')=(\wt {\Fix}^\BL (g'), \CN^\flat ) _{\CN}=( \wt {\Fix}^\BL  (g), \CZ(u) )_{\CN^\flat}=\wt{\Int}^\CZ(g, u)$.
		
		If $i=1$, we apply the dual isomorphism $\lambda_\CN: \CN \to \CN^\vee$ to reduce to the case $i=0$ using the compatibility (\ref{compatibility of embedding under duality}).
		
		For the last part, consider matching elements $g' \in \U(\BV)(F_0)_{\rs}$ and $\gamma' \in S(V_0)(F_0)_\rs$. Apply relative Cayley maps to $g'$ and $\gamma'$. By local constancy of intersection numbers \cite{AFLlocalconstant}, we can always assume $u$ is non-zero by approximation. We have the reduction on the analytic side by Corollary \ref{RedOrbSSS}, which concludes the proof.
	\end{proof}
	
	\begin{remark}
		Although \cite{AFLlocalconstant} only lists as examples the unitary Rapoport--Zink spaces in the self-dual case, the argument in \emph{loc. cit.} on local constancy of intersection numbers is quite general and applies to our current situation. 
	\end{remark}

	\section{Bruhat--Tits strata}\label{section: local BT strata}
	
	In this section, we study Bruhat--Tits stratification on the reduced locus $\CN^{\red}$ based on \cite{Cho18, Vollaard05, VW11}, which is useful for explicit computations. We redefine Bruhat--Tits strata via the Kudla--Rapoport cycles, which is used later to prove AT conjectures in the unramified maximal order case and some local modularity results.
	
Our group is not quasi--split in general, and the geometry of the reduced locus $\CN^\red$ is more involved in particular  $\CN^\red$ may be not equidimensional. 	See \cite{ADLV-hyperspecial} for some general results on top dimensional irreducible components and their stabilizers (which are special parahorics) of the reduced locus of Rapoport-Zink spaces  for quasi--split groups. 
	
For a vertex lattice $L \subseteq V$ of rank $n$ and type $t$, consider the associated \emph{Rapoport--Zink space} $\CN=\CN_{\U(L)} \to \Spf O_{\breve{F}}$. Recall we introduce the nearby hermitian space $\BV=\Hom_{O_F}^{\circ}(\BE, \BX)$. 
	
	\begin{definition}
		For a vertex lattices $L' \subseteq \BV$, the closed \emph{Bruhat-Tits stratum} $\mathrm{BT}(L')$ is the closed subscheme of $\CN^\red$ over $\BF$ defined by
		\begin{enumerate}
			\item 
			if $L'=L^{\bullet}$ has type $\geq t+1$, then $\mathrm{BT}(L'):=\CZ(L^{\bullet})^\red$.	
			\item 
			If $L'=L^{\circ}$ has type $\leq t-1$, then $\mathrm{BT}(L'):=\CY((L^{\circ})^{\vee})^\red$.
		\end{enumerate}	
	\end{definition}
	
	\begin{proposition}
		Consider an orthogonal decomposition $L=L^\flat \oplus O_F e$ with $v((e,e)_V)=i \in \{0,1\}$, with induced decomposition $\BV=\BV^\flat \oplus Fe^{[i]}$ and the embedding $\delta_{(L, e)}: \CN^\flat \to \CN$. If $L'$ is a vertex lattice in $\BV$ containing $e^{[i]}$, then
		\begin{equation}
			\mathrm {BT}(L') \cap \CN^\flat= \mathrm {BT}(L'^\flat). 
		\end{equation}	
	\end{proposition}	
	\begin{proof}
		We may assume $i=0$ using the dual isomorphism. Identify $\CN^\flat$ with $\CZ(e^{[0]})$. The result follows from moduli definition of $\CZ(L')$ or $\CY(L'^\vee)$.
	\end{proof}

	\subsection{Relative Dieudonne theory}
	
	Recall $\sigma$ is the $F_0$-linear Frobenius automorphism on $\breve{F}_0$. We have $O_F \otimes_{O_{F_0}} O_{\breve{F}_0} \simeq O_{\breve{F}_0} \times O_{\breve{F}_0}$ where $O_F$ acts on the first (resp. second) factor $O_{\breve{F}_0}$ by the fixed embedding $i_0: O_F \hookrightarrow O_{\breve{F}_0}$ (resp. $\sigma \circ i_0$).
	
	Let $ \BM(\BE)$ (resp. $\BM(\BX)$) be the (covariant) \emph{relative Dieudonne module} of the framing object $\BE$ (resp. $\BX$), which is a free $O_{\breve{F}_0}$-module of rank $2$ (resp. rank $2n$) with the relative Frobenius morphism $\CF_\BE$ (resp. $\CF_\BM$) and the relative Verschiebung $\CV_\BE$ (resp. $\CV_\BM$). 
	
	\begin{remark}
		See \cite[Prop 3.56, Defn. 3.57]{RZ96} for backgrounds on relative Dieudonne theory. Let $f$ be the inertial degree of $F_0$ over $\BQ_p$. The absolute Dieudonne module of $\BX$ is a free $W(\BF)$ module of rank $2n[F_0:\BQ_p]$, which decomposes into direct summands under the action of $O_F \otimes_{\BZ_p} W(\BF) \cong \prod_{i=1}^{f} O_{\breve{F}}$. The $0$-th summand has $O_{\breve{F}}$-rank  $\frac{2n[F_0:\BQ_p]}{f} \times \frac{1}{ \text{rk}_{W(\BF)} O_{\breve{F}}}=2n$ and is exactly the relative Dieudonne module $\BM(\BX)$ of $\BX$.
	\end{remark}
	
	Denote by $\BN(\BX)=\BM(\BX) \otimes \BQ$ the associated isocrystal. The $O_F$-action on $\BX$ induces $\BZ/2$-gradings:
	\begin{equation}
		\BM(\BX)=  \BM(\BX)_0 \oplus \BM(\BX)_1, \quad \BN(\BX)=  \BN(\BX)_0 \oplus \BN(\BX)_1.
	\end{equation}
	
	The polarization $\lambda_\BX$ induces a nondegenerate $\breve{F}$-bilinear alternating form $\left<-,-\right>_\BX$ on $\BN(\BX)$. For $x, y \in \BN(\BX)$ and $a \in O_F$, we have
	\begin{equation}
		\left<\CF_{\BX} x , y \right>_\BX= \sigma(\left<x , \CV_{\BX} y \right>_\BX),  \quad \left< \iota(a) x , y \right>_\BX= \left<x , \iota(\ov a) y \right>_\BX. 
	\end{equation}
	
	\begin{definition}
		Let $\{-,-\}$ be the nondegenerate form on $\BN(\BX)_0$ given by $$\{x,y\} := \left<x,\CF_{\BX} y \right>_\BX$$ which is linear in the first variable and $\sigma$-linear in the second variable. The restriction of the form $\{-,-\}|_{C(\BX)}$ to $C(\BX)$ is skew-hermitian.
	\end{definition}

For any lattice $\Lambda_1 \subseteq \BN(\BX)_1$, write $(\Lambda_1)^\perp \subseteq \BN(\BX)_0$ as its dual inside $\BN(X)_0$ under the pairing $\left<-,-\right>$. We use $(-)^\perp$ to distinguish it with the dual under the natural form $\{-,-\}$ on $\BN(\BX)_0$.

Consider the $\sigma^2$-linear operator $\tau= \CV^{-1}_{\BX}\CF_{\BX}$ on $\BN(\BX)_0$. Let $C(\BX)=\BN(\BX)_0^{\tau=1}$ be the $F$-vector space of invariant elements in $\BN(\BX)_0$. 

By the (relative) supersingular assumption on $\BX$, we have 
\begin{equation}
	\BN(\BX)_0 \simeq C(\BX) \otimes_{F} \breve{F}.
\end{equation}

	We have similar decompositions and forms on the isocrystal $\BN(\BE)=\BM(\BE) \otimes \BQ$. Choose a unit $\delta \in O_{F}^{\times}$ such that $\bar{\delta}=-\delta$. By \cite[Remark. 2.5]{KR-local} (the case of signature $(0,1)$), we can find $\tau$-invariant  generators $1_i$ of $\BM(\BE)_i \, (i=0, 1)$ such that
	\begin{equation}\label{eq: Dieduonne for BE}
		\CF_{\BE}1_1=1_0, \quad \CF_{\BE}1_0 = \varpi 1_1, \quad \left\{ 1_0, 1_0 \right\}_{\BE} = \varpi \delta.
	\end{equation}	
	
	A vector $x \in \BV=\Hom_{O_F}^{\circ}(\BE, \BX)$  induces a $F \otimes_{F_0} \breve{F}$-linear map from $\BN(\BE)$ to $\BN(\BX)$. As $\CF_{\BE}(1_{0})=\CV_{\BE}(1_0)$, we have $x(1_0) \in C(\BX)$. From \cite[Lem. 3.9]{KR-local} we obtain:
	
	\begin{proposition}\label{prop: compare herm}
		The association $x \mapsto x(1_0)$ gives a bijection
		\[
		\BV \cong C(\BX)
		\]
		under which the hermitian form $(-,-)_{\BV}$ on $\BV$ is identified with $\frac{1}{\varpi \delta }\{-,-\}$ on $C(\BX)$.
	\end{proposition}
	
	Consider the following non-degenerate $\sigma$-sesquilinear form on $\BN(\BX)_0$ and $C(\BX)$:
	\begin{equation}\label{defn: pairing on BN_0}
		(x,y)_\BN:=\frac{1}{\varpi \delta }\{x, y\}.
	\end{equation}
	For  $x, y \in \BN(\BX)_0$,  we have
	$\sigma( (x,y)_\BN )= (\tau(y), x)_\BN, \quad (\tau(x), \tau(y))_\BN=\sigma^2((x, y)_\BN).$

	\subsection{$\BF$-points of $\CN$}	
	
	For any $O_{\breve{F}}$ lattice $A$ in $\BN(\BX)_0$, denote by $A^\vee$ its dual under $(-,-)_\BN$:
	\begin{equation}
		A^\vee=\{ x \in \BN(\BX)_0 | (x, A)_\BN \subseteq O_{\breve{F}} \}.
	\end{equation}
	And denote by $(-)^\perp$ the dual with respect to $\left<-,-\right>_\BX$. For two $O_{\breve{F}}$ lattices $A_1, A_2$ in $\BN(\BX)_0$ and $m \geq 0$, write $A_1 \overset{m} \subseteq A_2$ if $A_1 \subseteq A_2$ and $A_2/A_1$ is a $\BF$-vector space of dimension $m$. Note we have
	\begin{equation}
		(A^{\vee})^{\vee}=\tau(A).
	\end{equation}
	
	We regard $O_{F}$ lattices in $\BV \simeq C(\BX)$ the same as $\tau$-stable $O_{\breve{F}}$ lattices in $\BN(\BX)_0$. For any $O_{F}$ lattice $L'$ in $\BV \simeq C(\BX)$, we denote by $L'^\vee$ its dual under the form $(-,-)_{\BV} \simeq (-,-)_\BN$. We have  $(L'^{\vee})^{\vee}=L'$.

	Consider a point $(X, \iota, \lambda, \rho) \in \CN(\BF)$. The framing gives an isomorphism $\BN(X) \simeq \BN(\BX)$. Hence the relative Dieudonne module of $(X, \iota, \lambda, \rho)$ gives an $O_{\breve{F}}$-lattice $M=M_0 \oplus M_1$ inside $\BN(\BX)$. Consider two $O_{\breve{F}}$-lattices inside $\BN(\BX)_0$ given by
	\begin{equation}
		A=M_0, \quad B=M_1^\perp.
	\end{equation}
	By definition, we have $\CF_{\BX} M_1= ((M_1)^\perp)^\vee$. A pair of $O_{\breve{F}}$-lattices $(A,B)$ in $\BN(\BX)_0$ is call \emph{special} if we have 
	\[
	B^{\vee} \overset{1}  \subseteq A, \quad A^{\vee} \overset{1}  \subseteq B, \quad A \overset{t}  \subseteq B.
	\]
	\begin{proposition}\label{prop: special pair classifying x in N}
		We have the following description of $\BF$-points of $\CN$:
		\begin{equation}
			\CN(\BF)= \{ \text{special pairs of $O_{\breve{F}}$-lattices \,} (A,B) \subseteq \BN(\BX)_0  \}.
		\end{equation}
	\end{proposition}
	\begin{proof}
		This follows from \cite[Lem. 1.5]{Vollaard05} \cite[Prop. 2.4]{Cho18}.  The condition on the polarization is equivalent to $A \overset{t}  \subseteq B$, and the Kottwitz signature condition is equivalent to $B^{\vee} \overset{1}  \subseteq A, \, A^{\vee} \overset{1} \subseteq B$. Hence we conclude by relative Dieudonne theory. 
	\end{proof}

	\begin{remark}
		By Proposition \ref{prop: dual iso herm form comp}, we have $(-,-)_{\BV^\vee}=-\varpi (-,-)_{\BV}$ under the dual isomorphism $\lambda_\BV : \BV \cong \BV^\vee$. On $\BF$-points, the dual isomorphism $\lambda: \CN \to \CN^\vee$ (\ref{dual lambda}) is given by
		\[
		\lambda_\CN: \CN(\BF) \to \CN^\vee(\BF),  \,\,	(A,B) \mapsto (A',B'):=(\lambda_\BV(B), \lambda_\BV(\varpi^{-1} A)).
		\]
	\end{remark}
	
	Given a special pair $(A,B)$, denote by $L_A$ (resp. $L_B$) the smallest $\tau$-stable lattice containing $A$ (resp. $B$). By definition, 
	\begin{equation}\label{eq: generate tau-stable lat}
		A \subseteq L_A, \quad L_A^{\vee} \subseteq A^{\vee},
	\end{equation}
	\begin{equation}
		L_B^{\vee} \subseteq B^{\vee}, \quad B \subseteq L_B.
	\end{equation}
	
	\begin{proposition}\label{special pair: choose 1 in 2}\cite[Lem. 2.7]{Cho18} 
		For a special pair $(A,B)$, at least one of the two following cases holds:
		\begin{itemize}
			\item $\varpi L_A^{\vee}  \subseteq L_A \subseteq L_A^{\vee}$, in particular $A \overset{t-1} \subseteq A^{\vee}$. Hence $L^{\circ}=L_A$ is a vertex lattice of type $ \leq  \dim_\BF (A^{\vee}/A)=t-1$ .
			\item $\varpi L_B  \subseteq  L_B^{\vee} \subseteq L_B$, in particular $B^{\vee} \overset{t+1} \subseteq B$. Hence $L^{\bullet}=L_B^{\vee}$ is a vertex lattice of type $\geq \dim_\BF (B/B^{\vee})=t+1$.
		\end{itemize}
	\end{proposition}
	
	\begin{remark}
		For the Balloon--Ground stratification,  $\CN^{\red, \circ}(\BF) \subseteq \CN^\red(\BF)$ is the locus $A \subseteq A^\vee$ ($A^\vee \subseteq \varpi^{-1} A$ is automatic). And $\CN^{\red, \bullet}(\BF)$ is the locus $B \subseteq \varpi^{-1}B^\vee$ ($B^{\vee} \subseteq B$ is automatic). Any automorphism $g \in \U(\BV)(F_0)$ preserves $\CN^{\red, \bullet}(\BF)$ and $\CN^{\red, \circ}(\BF)$. By Proposition \ref{special pair: choose 1 in 2}, we have a stratification 
		\[
		 \CN^\red(\BF)=\CN^{\red, \bullet}(\BF) \cup \CN^{\red, \circ}(\BF).
		\]
	\end{remark}

	\subsection{KR cycles and BT strata}
	
	Consider a non-zero vector $u \in \BV$ and a point $(X, \iota, \lambda, \rho) \in \CN(\BF)$. By Grothendieck--Messing theory,  $\rho^{-1} \circ u$ lifts to $\Hom_{O_F}(\BE, X)$ if and only if
	\begin{equation}\label{eq: special cycle Dieudonne}
		u(1_0) \in M_0, \quad u(1_1) \in M_1.
	\end{equation}
	We have $\CF_\BE1_1=1_0, \CF_\BE1_0=\varpi1_1$ by (\ref{eq: Dieduonne for BE}). Hence the condition (\ref{eq: special cycle Dieudonne}) is equivalent to that $u(1_0) \in \CF_\BX M_1= B^{\vee}$, and we have the following description of $\CZ(u)(\BF)$ (\cite[ Prop. 5.5.]{Cho18}):
	\begin{equation}
		\CZ(u)(\BF)= \{  (A,B) \in \CN(\BF)| u \in B^{\vee} \}.
	\end{equation}  
	
	The relative Dieudonne module of $X^\vee$ can be identified with $M^{\perp}=B \oplus A^{\perp}$. We obtain similar description for $\CY(u)$:
	\begin{equation}
		\CY(u)(\BF)= \{  (A,B) \in \CN(\BF)| u \in A^{\vee} \}.
	\end{equation}
	
	The inclusion $
	\CZ(u) \subseteq \CY(u) \subseteq \CZ(\varpi^{-1}u)$ on $\BF$-points corresponds to the inclusion $B^\vee \subseteq A^\vee \subseteq \varpi^{-1}B^\vee $.
	
	For a vertex lattice $L^\bullet \subseteq \BV$ of type $\geq t+1$, consider the $\BF_{q^2}$-vector space 
	$\BV(L^\bullet)_0:=(L^\bullet)^\vee / L^\bullet
	$ of dimension $t(L^\bullet)$. Here we regard $L^\bullet$ as a finite rank $O_F$-lattice. Equip $\BV(L^\bullet)_0$ with the $\mathbb F_{q^2}$-valued perfect $\sigma$-hermitian form $(x,y):= \ov {\varpi(x,y)_\BV}$. From above we have  $\mathrm{BT}(L^\bullet)(\BF)= \CZ(L^\bullet)(\BF)=$
	\[
	\{ (A,B) \in \CN(\BF) | L^\bullet \subseteq B^{\vee} \subseteq A \subseteq B \subseteq (L^\bullet)^\vee,
	L^\bullet \subseteq B^{\vee} \subseteq A^\vee \subseteq B \subseteq (L^\bullet)^\vee \}.
	\]
	Set $U_1$ (resp. $U_2$) as the reduction of $B^{\vee}$ (resp. $A$) inside $\BV(L^\bullet)_0 \otimes \BF$.  Then $\CZ(L^\bullet)(\BF)$ can be identified with 
	\begin{equation} \label{eq: Dieudone defn of BT(L bullet)}
		\{ (U_1, U_2) \subseteq \BV(L^\bullet)_0 \otimes \BF | U_1 \subseteq U_2, U_2 \subseteq U_1^\perp, U_1 \subseteq U_2^\perp, \dim U_1 = \frac{t(L^\bullet)-t-1}{2}, \dim U_2= \dim U_1+1 \}.
	\end{equation}
	Here by definition $U^\perp$ is the subspace $\{ x \in  \BV(L^\bullet)_0 \otimes \BF | (x,y)=0, \, \forall y \in U \}.$ 
	
	For a vertex lattice $L^\circ \subseteq \BV$ of type $\leq t-1$,  consider the $\BF_{q^2}$-vector space 
	$\BV(L^\circ)_0:=L^\circ/\varpi (L^\circ)^\vee$ of dimension $n-t(L^\circ)$. Here we regard $L^\circ$ as a finite rank $O_F$-lattice. Equip $\BV(L^\circ)_0$ with the perfect $\sigma$-hermitian form $(x,y):= \ov {(x,y)_\BV}$. Similarly we have $\mathrm{BT}(L^\circ)=\CY((L^\circ)^\vee)(\BF)= $
	\[
	\{ (A,B) \in \CN(\BF) |  \varpi (L^\circ)^\vee  \subseteq \varpi  A^\vee \subseteq \varpi B \subseteq A \subseteq L^\circ , \varpi (L^\circ)^\vee  \subseteq \varpi  A^\vee \subseteq B^\vee \subseteq A \subseteq L^\circ \}
	\]
	Set $U_3$ (resp. $U_4$) as the reduction of $\varpi A^{\vee}$ (resp. $\varpi B$) inside $\BV(L^\circ)_0 \otimes \BF$. Then $\CY((L^\circ)^\vee)(\BF)$ can be identified with the set 
	\begin{equation}\label{eq: Dieudone defn of BT(L circ)}
		\{ (U_3, U_4) \subseteq  \BV(L^\circ)_0 \otimes \BF | U_3 \subseteq U_4, U_3 \subseteq U_4^\perp, U_4 \subseteq U_3^\perp, \dim U_3= \frac{t-1-t(L^\circ)}{2}, \dim U_4= \dim U_3+1 \}.
	\end{equation}
	
	Here $U^\perp$ is by definition the subspace $\{ x \in  \BV(L^\circ)_0 \otimes \BF | (x,y)=0, \, \forall y \in U \}.$ 
	
	We now work with the descent of $\CN$ to $\Spf O_F$ by Remark \ref{descent to O_F: moduli}, the definition of Bruhat--Tits strata also descends to $\Spf O_F$ by Remark \ref{descent to O_F: cycles strata}.
	For any perfect field $k$ over $\mathbb F_{q^2}$, we have natural bijections of $k$ points of BT strata as above by applying relative Dieudonne theory over $k$.  In conclusion, we have (see also \cite[Defn. 2.9, Prop. 2.18]{Cho18}): 
	
	\begin{proposition}\label{prop: BT int combinatorics}
		\begin{enumerate}
			\item  For two vertex lattices $L^{\circ}_1, L^{\circ}_2$ of type $\leq t-1$, we have 
			$$
			\mathrm{BT}(L_1^\circ) \cap \mathrm{BT}(L_2^\circ) = \begin{cases}
				\mathrm{BT}(L^{\circ}_1 \cap L^{\circ}_2) & \text{if $L^{\circ}_1 \cap L^{\circ}_2$ is a vertex lattice} \\
				\emptyset & \text{else} \\
			\end{cases}
			$$ 
			\item For two vertex lattices $L^{\bullet}_1, L^{\bullet}_2$ of type $\geq t+1$, we have 
			$$
			\mathrm{BT}(L_1^\bullet) \cap \mathrm{BT}({L_2^\bullet} )= \begin{cases}
				\mathrm{BT}(L^{\bullet}_1 +L^{\bullet}_2)& \text{if $L^{\bullet}_1 + L^{\bullet}_2$ is a vertex lattice} \\
				\emptyset & \text{else} \\
			\end{cases}
			$$ 
		\end{enumerate}
	\end{proposition}

	\begin{theorem}\label{thm: moduli int of BT strata}
		The collection $\{\mathrm{BT}(L') \}_{L'}$ indexed by vertex lattices $L'$ in $\BV$ gives a natural stratification on $\CN^{\red}$. Each stratum $\mathrm{BT}(L')$ is projective, smooth, and irreducible over $\BF$, and isomorphic to a (closed) Deligne--Lusztig variety for $\U(\BV(L')_0)$ which is not of Coxeter type in general. 
		\begin{enumerate}
			\item We have 
			$$
			\dim \mathrm{BT}(L')= \begin{cases}
				\frac{t-1-t(L')}{2} + n-t& \text{if $t(L') \leq t-1$} \\
				\frac{t(L')-t-1}{2}+t & \text{if $t(L') \geq t+1$} \\
			\end{cases}
			$$ 	
			\item  $\mathrm{BT}(L^\bullet)$ is a union of two classical Deligne--Lusztig varieties, depending on whether $U_2 \subseteq U_2^{\perp}$. And $\mathrm{BT}(L^\circ)$ is a union of two classical Deligne--Lusztig varieties, depending on whether $U_4 \subseteq U_4^{\perp}$.		
			\item $\mathrm BT(L^\circ) \cap \mathrm BT(L^\bullet)$ is non-empty if and only if $L^\bullet \subseteq L^\circ$. In this case, 
			\begin{equation}
				\mathrm BT(L^\circ) \cap \mathrm BT(L^\bullet)(\BF)=\{ (A, B) \in \CN(\BF) | L^\bullet \subseteq B^\vee \subseteq A \subseteq L^\circ \}. 
			\end{equation}
			is equal to $\BF$-points of the flag variety $Fl_{(L^\bullet, L^\circ)}$ parameterizing flags $(\ov{B^\vee}, \ov{A})$ of dimension $(\frac{t(L^\bullet)-t-1}{2}, \frac{t(L^\bullet)-t+1}{2})$ in the $\frac{t(L^\bullet)-t(L^\circ)}{2}$ dimensional $\BF$-vector space $L^\circ/L^\bullet$. 
			\item In particular if $t(L^\bullet)=t+1$ and $t(L^\circ)=t-1$, then $\mathrm BT(L^\circ)(\BF) \cong \BP^{n-t}(\BF)$ and $\mathrm BT(L^\bullet)(\BF) \cong \BP^{t}(\BF)$. And $\mathrm BT(L^\circ) \cap \mathrm BT(L^\bullet)(\BF)$ is exactly one point if $L^\bullet \subseteq L^\circ$, and empty otherwise.
		\end{enumerate}
	\end{theorem}
	\begin{proof}
		This is \cite[Thm. 1.1]{Cho18} after identifying the strata in \emph{loc. cit.} with our definitions via special cycles. Proposition \ref{special pair: choose 1 in 2} shows the collection $\{\mathrm{BT}(L')\}_{L'}$ covers $\CN^\red$ which forms a stratification. In \emph{loc.cit.}, we identify $\mathrm{BT}(L')$ with (closed) Deligne--Lusztig varieties (\ref{eq: Dieudone defn of BT(L bullet)}, \ref{eq: Dieudone defn of BT(L circ)}). By the identification, we obtain dimensions, projectivity, smoothness and geometrical irreducibility of $\mathrm{BT}(L')$. 
	\end{proof}

	\begin{remark}\label{rek: Ballloon-Ground Bruat-Tits}
		We can describe the induced Balloon-Ground stratification on the reduced locus $\CN^\red$ using Bruhat--Tits strata. For example, if $t=1$, then $(\CN^{\circ})^\red$ is the union of all type $0$ Bruhat--Tits strata $\CY((L^{\circ})^{\vee} )^\red \simeq \BP^{n-1}_{L^{\circ}} $ in $\CN^\red$, and $\CN^{\dagger} \subseteq (\CN)^\red$ is the disjoint union of Fermat hypersurfaces in $\BP^{n-1}_{L^{\circ}}$ of degree $q+1$. 
	\end{remark}

	\section{Local modularity}\label{section: local modularity}

	In this section, we prove some local modularity results, which is used in \S \ref{section:Lmod} to show (double) modularity of arithmetic theta series over the basic locus, see also Proposition \ref{key: int modularity}.
	
	Let $L$ be a vertex lattice of rank $n$ and type $t$. Consider the intersection pairing between divisors $\CZ$ on $\CN=\CN_{\U(L)}$ and $1$-cycles $\CC$ in $\CN^\red$:
	$$
	(\CZ, \CC) := \chi(\CN, \CO_\CZ \otimes^\BL \CO_\CC).
	$$
	Consider the following  functions on  $u \in \BV-0$: 
	$$f_{L}(\CC)(u):=(\CZ(u), \CC), \quad f_{L^\vee}(\CC)(u):=(\CY(u), \CC).$$ 
	
	Consider the Weil representation of $\SL_2(F_0)$ on $\CS(\BV)$ with respect to a fixed unramified non-trivial additive character $\psi: F_0 \to \BC$. Set $\psi_{F}=\psi \circ \tr_{F/F_0}$.  Then for $f \in \CS(\BV)$, we have 
	\begin{equation}
		\begin{pmatrix} & -1 \\ 1 &
		\end{pmatrix}. f (u)= \gamma_{\BV} \mathcal{F}_\BV(f)(u),
	\end{equation}
	where $\gamma_{\BV}=\eta(\det(\BV)) \in \{\pm1\}$ is the Weil constant of $\BV$, and $\mathcal{F}_\BV$ is the Fourier transform on $\BV$ with respect to $\psi \circ q_\BV$:
	\begin{equation}
		\mathcal{F}_\BV(f)(y) := \int_{x \in \BV} f(x) \psi_{F}((x, y)_\BV) dx, \quad y \in \BV.
	\end{equation}
	Here the Haar measure is the self-dual one with respect to $\psi_F$.
	
	A $1$-cycle $\CC$ in $\CN^\red$ is called \emph{very special}, if it is a finite linear combination of $1$-cycles from standard embeddings of the form 
	\[ \delta: \CN_2^{[1]} \to  \CN_n^{[t]}= \CN \]
	\[
	X^\flat \mapsto X^\flat \times (\CE^{[1]})^{t-1} \times \CE^{n-t-1}.
	\] 
	induced by a decomposition of framing objects $\BX=\BX^\flat \times (\BE^{[1]})^{t-1} \times (\BE^{[0]})^{n-t-1}$.
	
	\begin{theorem}\label{thm: partial local modularity}
		Choose any $1$-cycle $\CC \subseteq \CN^{\red}$, and assume it is very special if $0<t<n$. Then $f_L(\CC)$ and $f_{L^\vee}(\CC)$ extend to Schwartz functions on $\BV$.  Moreover, we have the dual relation
		\begin{equation}
			\begin{pmatrix} & -1 \\ 1 &
			\end{pmatrix}. f_L(\CC)= (-q)^{-t} f_{L^\vee}(\CC).
		\end{equation}
	\end{theorem}
	
	\begin{conjecture}[ local modularity conjecture ] \label{local mod conj}
	Theorem	\ref{thm: partial local modularity} holds for any $1$-cycle $\CC \subseteq \CN^\red$.
\end{conjecture}

	In the rest of this section, we prove the theorem. The case $t=0$ follows from  \cite[Lem. 6.2.1, Thm. 6.4.9]{LZ19}. Consider the dual isomorphism $\lambda_\CN: \CN \to \CN^\vee$. By Proposition \ref{prop: relation of Z and Y}, we have
	\[
	\lambda_\CN (\CY(u)) = \CZ(\lambda_\BV \circ u), \quad \lambda_\CN (\CZ(u)) = \CY(\varpi^{-1} (\lambda_\BV \circ u)).
	\]
	We have $(\lambda_{\BV} -, \lambda_{\BV} -)_{\BV^\vee}=-\varpi (-,-)_{\BV}$ by Proposition \ref{prop: dual iso herm form comp}. Hence the self-dual Haar measure on $\BV^\vee$ is $|\varpi|^{n/2}_{F}=q^{-n}$ times the Haar measure on $\BV$. And for $f \in \CS(\BV)$, we have
	\begin{equation}
		\CF_{\BV^\vee}(f \circ \lambda^{-1}_\BV) = q^{-n} \CF_{\BV} (f) \circ (-\varpi \lambda^{-1}_\BV) \in \CS(\BV^\vee).
	\end{equation}
	From this we may compute $\CF_{\BV^\vee}(\wt{f}_{\wt{L}^\vee}(\lambda_\CN(\CC) )) (x)=q^{n} \CF_{\BV} (f_L(\CC)) (\lambda_{\BV}^{-1} \circ x)$ for $x \in \BV^\vee -0$, where $\wt{f}_{\wt{L}^\vee}$ means $f_{L^\vee}$ on the dual Rapoport--Zink space $\CN^\vee$.
	Hence the case $t=t_0$ follows from the case $t=n-t_0$.

	So we may assume that $0<t<n$ from now on.  We could assume the very special $1$-cycle $\CC$ lives on a chosen embedding $\delta: \CN_2^{[1]} \to  \CN_n^{[t]}= \CN$. Write  $u \in \BV$ as
	\begin{equation}
		u=u_1+ u_2 + u_3 \in \BV=\BV^\flat \oplus  (\oplus_{i=1}^{t-1} Fe^{[1]}_i ) \oplus (\oplus_{j=1}^{n-t-1} Fe^{[0]}_j ). 
	\end{equation}
	\begin{proposition}
		After restriction to the embedding $\delta$, for any $u \in \BV$ we have 
		\begin{enumerate}
			\item $
			\CZ(u)|_{\delta}= \CZ(u_1) 1_{O^{t-1}_{F_0}}(u_2) 1_{O^{n-t-1}_{F_0}} (u_3).
			$
			\item $
			\CY(u)|_{\delta}= \CY(u_1) 1_{\varpi^{-1}O^{t-1}_{F_0}}(u_2) 1_{O^{n-t-1}_{F_0}} (u_3).
			$
		\end{enumerate}
	\end{proposition}
	\begin{proof}
		We can decompose the embedding into $\delta: \CN_2^{[1]} \overset{\delta_2}\to \CN_{t+1}^{[t]} \overset{\delta_1} \to  \CN_n^{[t]}= \CN$, then the result follows from moduli definitions of Kudla--Rapoport cycles.
	\end{proof}
	
	So by the projection formula for $\delta$, Theorem \ref{thm: partial local modularity} for the very special $1$-cycle $\CC$ is reduced to the base case $n=2$ and $t=1$.

	\subsection{The base case $n=2,\, t=1$.}

	We now show Theorem \ref{thm: partial local modularity} is true for  $n=2$ and $t=1$. As the starting vertex lattice in $V$ is fixed, we use the notation $L$ instead of $L'$ for a vertex lattice in $\BV$, and write $\BV=\BV_2$. By the main theorem in \cite{KR-Drinfeld plane}, The space $\CN_2^{[1]}$ is isomorphic to the Drinfeld half plane $\wh{\Omega}^1_{F_0} \to \Spf O_{\breve F}$.  Any Bruhat--Tits stratum $\mathrm{BT}(L)$ in $\CN_2^{[1]}$ is a projective line  (see also \ref{thm: moduli int of BT strata}):
	$$
	\mathbb{P}_{L}=\mathbb{P}(L^\vee/ \varpi L^\vee).
	$$
	
	We work with non-isotropic vector $u \in \BV$ i.e., $(u, u) _\BV \not =0$. So $m=v((u,u)_\BV)$ is finite, and we set $u_1=\varpi^{-\lfloor \frac{m+1}{2} \rfloor} u $.  Note 
	\begin{equation}
		v((u_1,u_1)_\BV) = \begin{cases}
			0 & \text{$m$ is even},  \\
			-1 & \text{$m$ is odd}. \\ 
		\end{cases}
	\end{equation}
	
	By \cite[Lem. 3.8]{San13} (the notion of vertex lattices in \emph{loc.cit.} is dual to us) which works for any $F_0$, there exists a unique vertex lattice $L_u$ such that $u_1 \in  L^{\vee}_u - \varpi L^{\vee}_u$. Moreover, $L_u$ is of type $0$ (resp. type $2$) if $m$ is even (resp. odd). We call $L^{\vee}_u$ the \emph{central lattice} of $u$.

	Let $B$ denote the Bruhat--Tits tree for $\SU(\BV_2)(F_0)$, which is a graph with the following description \cite[Defn. 3.3]{San13}. The vertices are vertex lattices in $\BV_2$, and edges can only occur between vertex lattices of differing type. Two vertex lattices $L_1, L_2$ are joined by an edge if and only if $L_1 \subseteq L_2 \subseteq \varpi^{-1}L_1$ or $L_2 \subseteq L_1 \subseteq \varpi^{-1}L_2$ .This graph is a $q+1$-regular tree.
	
	Let $d(L_1, L_2)$ be the distance of two vertex lattices $L_1, L_2 \subseteq \BV_2$. For a vector $u \in \BV_2$ and a vertex lattice $L \subseteq \BV_2$, define the multiplicity 
$$
m(u, L)= \max( \max\{ r \in \mathbb Z | \varpi^{-r} u \in L\}, 0).
$$
	
	\begin{theorem}\label{thm: n=2 Z-cycle}
		Let $u \in \BV_2$ with $m=v((u,u)_{\BV_2}) \geq 0$, then we have a decomposition
		\begin{equation}
			\CZ(u)=\CZ(u)^{h}+\sum_{u \in L} m(u, L) \mathbb{P}_{L}
		\end{equation}
		in the K-group of $\CN_2^{[1]}$ supported on $\CZ(u)$ where 
		\begin{enumerate}
			\item $\CZ(u)^{h} \cong \Spf O_{\breve{F}}$ is the horizontal part of $\CZ(u)$. It meets $\CN^\red$ at an ordinary point $x_u$ on $\BP_{L_u}$, corresponding to the non-isotropic line given by $u_1$ in $L_u^\vee/ \varpi L_u^\vee$. 
			\item For a vertex lattice $L$, $u \in L$ if and only if $d(L, L_u) \leq m$, in which case we have
			\begin{equation}
				m(u,L)= \begin{cases}
					\frac{1}{2}(m-d(L, L_u)), & m=d(L, L_u) \mod 2, \\
					\frac{1}{2}(m+1-d(L, L_u)), & m+1=d(L, L_u) \mod 2.\\ 
				\end{cases} 
			\end{equation}
		\end{enumerate}
		
	\end{theorem}
	\begin{proof}
		The decomposition of $\CZ(u)$ is \cite[Thm. 3.14]{San13} by computing explicit equations of $\CZ(u)$ on each Bruhat--Tits stratum. The multiplicity formula is \cite[Lem. 3.12]{San13}. The proof works over any $F_0$.
	\end{proof}
	
	\begin{remark}\label{rek: base shape Z Y n=1 t=2}
		The case $m=0$ and $m=1$ form the base shapes of $\CZ$-cycles. Note $m(u, L)=0$ if $u \in L- \varpi L$. If $v((u,u)_\BV) \geq 0$, we have $\CZ(\varpi u) - \CZ(u) = \sum_{u \in L } \BP_L $.
	\end{remark}
	
	\begin{theorem}\label{thm: n=2 Y-cycle}
		Let $u \in \BV_2$ with $v((u,u)_{\BV_2})=m \geq -1$. Then we have a decomposition in the K-group of $\CN_2^{[1]}$ with supported on $\CY(u)$:
		\begin{equation}
			\CY(u)=\CY(u)^{h}+\sum_{u \in L} m^{\vee}(u, L) \BP_{L}
		\end{equation}
		\begin{enumerate}
			\item Here $\CY(u)^{h} \cong \Spf O_{\breve{F}}$ is the horizontal part of $\CY(u)$. It meets $\CN^\red$ at an ordinary point $x_u$ on $\BP_{L_u}$, corresponding to the non-isotropic line given by $u_1$ in $L_u^\vee/ \varpi L_u^\vee$. 
			\item If $u \in L$, the multiplicity $m^\vee(u, L)$ is equal to
			\begin{equation}
				m^{\vee}(u,L)=\begin{cases}
					\frac{1}{2}(m+2-d(L, L_u)), & m=d(L, L_u) \mod 2, \\
					\frac{1}{2}(m+1-d(L, L_u)), & m+1=d(L, L_u) \mod 2.\\ 
				\end{cases} 
			\end{equation}
		\end{enumerate}
	\end{theorem}
	\begin{proof}
		We are done by applying the dual isomorphism $\lambda: \CN \to \CN^\vee$ (\ref{dual lambda}) to Theorem \ref{thm: n=2 Z-cycle}. Note if $d(L, L_u) = m+1$, then $m^{\vee}(u, L)=0$. Hence we can write the summation $\sum m^{\vee}(u, L) \BP_L $ over vertex lattices $L$ such that $d(L, L_u) \leq m$ i.e., $u \in L$ instead of $d(L, L_u) \leq m+1$. 
	\end{proof}

	\begin{theorem}\label{thm: n=2 local modularity}
		Consider a vertex lattice $L \subseteq \BV_2$.
		\begin{enumerate}
			\item If $L$ has type $0$, then for any $u \not =0$ we have 
			\begin{equation}
				(\CZ(u), \mathbb P^1_{L} )=1_L(u),  \, \, (\CY (u), \mathbb P^1_{L} )=  -q1_{L}(u) .
			\end{equation} 
			\item  If $L$ has type $2$, then for any $u \not =0$ we have 
			\begin{equation}
				(\CZ(u), \mathbb P^1_{L} )=-q 1_L(u), \, \, (\CY (u), \mathbb P^1_{L} )= 1_{L^\vee}(u) .
			\end{equation}		
		\end{enumerate}
		
	\end{theorem}
	\begin{proof}
		The number $(\mathcal{Z}(u), \mathbb P_{L} )$ for non-isotropic vector $u$ is computed in \cite[Lem 2.10]{San17} explicitly. The result depends on the parity of $d(L, L_u)-v((u,u)_\BV)$. But it is equivalent to our formulation, as if $L$ is of type $0$ (resp. $2$), then $d(L, L_u)-v((u,u)_\BV)$ is always even (resp. odd). We could also (re)prove the theorem using the Remark \ref{rek: base shape Z Y n=1 t=2} and the case $m=0$ and $m=1$. For isotropic vectors $u$, we reduce to non-isotropic case by the local constancy of intersection numbers \cite{AFLlocalconstant}.
		
		Applying the dual isomorphism $\CN \to \CN^\vee$, we have $(\CY(u), \mathbb P^1_{L})=(\mathcal Z(\lambda_\BV \circ u), \mathbb P^1_{\lambda (L)})$. The result for $\CY$-cycles follows. 
	\end{proof}
	
	The following corollary finishes the proof of Theorem \ref{thm: partial local modularity}.
	
	\begin{corollary}
		Assume that $n=2, \, t=1$, in particular $\BV=\BV_2$ is split and $\gamma_{\BV}=1$. Then for any $1$-cycle $C \subseteq \CN^{\red}$, we have
		\begin{equation}
			\begin{pmatrix} & -1 \\ 1 &
			\end{pmatrix}. (\CZ(u), \CC)= - q^{-1} (\CY(u), \CC).
		\end{equation}
	\end{corollary}
	\begin{proof}
		Any vertical $1$-cycle in $(\CN^1_2)_{red}$ is a finite linear combination of  $\mathrm {BT}(L)$ for vertex lattices $L \subseteq \BV_2$. By self-duality of the Haar measure on $\BV$, we have $\mathcal{F}_\BV 1_L = \vol(L) 1_{L^\vee}$ where $\vol(L)= (\# (L^\vee /L) )^{-1/2}=q^{-t(L)}$.  So the result follows from Theorem \ref{thm: n=2 local modularity}. 
	\end{proof}

	\section{ATCs for unramified maximal orders}\label{section: ATC unram max order}
	
	In this section, we show Conjecture $\ref{conj: semi-Lie version ATC}$ for $(g, u) \in (\U(\BV) \times \BV)(F_0)_\rs$ holds in the case $O_F[g]$ is a maximal order that is \'etale over $O_F$. By duality, we only need to work with Part $(1)$ i.e., for $\CZ$-cycles. Then we generalize the thesis of Mihatsch \cite[Section 9]{M-Thesis} on AFL in the maximal order case to our set up assuming $F_0/\BQ_p$ is unramified.

	\subsection{Semi-Lie ATC for $n=1$}\label{ATC n=1}
	
	Assume that $n=1$, and the type of $L \subseteq V$  is $t_0 \in \{0,1\}$. Recall the standard function (\ref{std function})
	$$
	f_{\mathrm{std}}:=1_{S(L, L^{\vee})} \times 1_{L_0} \times 1_{(L^{\vee}_0)^*}.
	$$
	 and a regular semisimple pair $(\gamma, b,c) \in (S(V_0) \times V_0 \times (V_0)^*)(F_0)_\rs$.
	
	\begin{proposition}
		
		\begin{enumerate}	
			\item If $v(cb)=t_0 \mod 2$, then 
			\begin{equation}
				\Orb((\gamma, b, c), f_{\mathrm{std}})=\begin{cases}
					1 & v(cb) \geq t_0, \\
					0 & \text{else.} \\
				\end{cases}
			\end{equation}
			
			\item If $v(cb)=t_0+1 \mod 2$, then $\Orb((\gamma, b, c), f_{\mathrm{std}})=0$ and
			\begin{equation}
				\del\bigl((\gamma, b, c), f_{\mathrm{std}})= 
				\begin{cases}
					- (\log q) \frac{v(cb)-t_0+1}{2}  & v(cb) \geq t_0+1, \\
					0 & \text{else.} \\
				\end{cases}
			\end{equation}
		\end{enumerate}	
	\end{proposition}
	
	\begin{proof}
		Choose a generator of $L$, then we are computing orbital integrals for $f_{\mathrm{std}}=1 \times 1_{O_{F_0}} \times 1_{\varpi^{t_0}O_{F_0}^*}$. And the transfer factor is $\omega_{L}(\gamma, b, c)=(-1)^{v(b)}$. So $\Orb((\gamma, b, c), f_{\mathrm{std}}, s)$ is
		\[
		(-1)^{v(b)} \int_{h \in \GL_{1}(F_0)} 1_{O_{F_0}}(h^{-1}b) 1_{\varpi^{t_0}O_{F_0}^*}(ch)  (-1)^{v_F(h)} |h|^sdh
		= (-1)^{v(b)}\sum^{v(b)}_{k=-v(c) + t_0 } (-1)^k q^{-ks}.
		\] 
		Take $s=0$, then the first part follows. Now assume that $v(cb)=t_0+1 \mod 2$.  Then we obtain
		\[ \del\bigl(\gamma, b, c, 0)= (-1)^{v(b)} \sum_{k=-v(c)+t_0}^{v(b)} (-1)^k (\log q) (-k) = - \log q \frac{v(cb)-t_0+1}{2}.\]

	\end{proof}
	
	\begin{proposition}\label{prop: Fix and Z in the case n=1}
		For any regular semisimple pair $(g, u) \in (\U(\BV) \times \BV)(F_0)_\rs$,  we have identifications
		$
		\Fix(g)=\CN_1^{[t_0]}$ and $
		\CZ(u) \simeq \Spf O_{\breve{F}} / \varpi^{k}$
		where $k:=\max\{0, \frac{v((u,u))-t_0+1}{2}\}$.
	\end{proposition}
	\begin{proof}
		By the theory of canonical lifting \cite{Gross-canonical}, we have $\CN_1^{[t_0]} \cong \Spf O_{\breve{F}}$ (by duality we may assume $t_0=0$). As $\CN_1^{[t_0]} \times_{ O_{\breve{F}} } \CN_1^{[t_0]} =\CN_1^{[t_0]}$,  $\Fix(g)$ is the whole space $\CN_1^{[t_0]}$.  The lifting condition for $\CZ(u)$ doesn't depend on the polarization. Only the hermitian form on $\BV$ changes when we change $t_0=0$ to $t_0=1$ by scaling. So we may assume $t_0=0$. Again by the theory of canonical lifting, we have $
		\CZ(u) \simeq \Spf O_{\breve{F}} / \varpi^{k}$.
	\end{proof}

	\subsection{Reduction in the maximal order case}\label{MaxorderATC}
	
	Return to the semi-Lie Conjecture \ref{conj: semi-Lie version ATC} for a vertex lattice $L \subseteq V$ of type $t$ and rank $n$. Recall the hermitian space $\BV=\Hom_{O_F}^{\circ}(\BE, \BX)$. 
	
	Consider a maximal order $O_F[g]=\prod_{i=1}^{m} O_{F_i} \hookrightarrow \End (\BV)$ which is \'etale over $O_F$ i.e., all $F_i$ are \emph{unramified} field extensions of $F$. Then each $F_i$ is an unramified quadratic extension of its fixed subfield $F_{i,0}$. Let $f_i$ be the residual degree of $F_{i,0}$ over $F_0$. As $F_i=F_{i,0} \otimes_{F_0} F$ is a field, the degree $f_i$ is odd for any $i$. We have a decomposition of $F/F_0$-hermitian spaces
	\begin{equation}
		\BV=\prod_{i=1}^{m} \BV_i, \quad \BV_i \cong F_i.
	\end{equation}
	Lift the $F/F_0$-hermitian form on $\BV_i$ to a $F_i/F_{i,0}$-hermitian form on $\BV_i$ along the trace morphism $\tr: F_i \to F$. As $F_i$ is unramified over $Fs$, the hermitian dual of any $O_{F_i}/O_{F_{i,0}}$ hermitian lattice agrees with the hermitian dual of its underlying $O_F/O_{F_0}$ hermitian lattice.

	\subsection{Fixed points via Bruhat--Tits strata}
	
	We can compute $\BF$-points of $\Fix(g)$ for $g \in \U(\BV)(F_0)_{\rs}$ using the Bruhat--Tits strata in \S \ref{section: local BT strata}.
	
	\begin{proposition}\label{prop: O_F[g] stable lat}
		Assume that $O_F[g]$ is an \'etale maximal order. Then there exists a unique vertex lattice $L \subseteq \BV$ that is $O_F[g]$-stable. Moreover, its type $t(L)=\sum_{i} f_i$ where the index runs over $i$ such that $\BV_i$ is non-split as a $F/F_{0}$-hermitian space. 
	\end{proposition}
	\begin{proof}
		Consider any $O_F[g]$-stable vertex lattice $L \subseteq \BV$. Then $L=\prod_i L_i$ as $O_F/O_{F_0}$-hermitian lattices. As $L \subseteq L^{\vee} \subseteq \varpi^{-1}L$, we have inclusions of $O_{F_i}$-stable lattices
		$
		L_i \subseteq L^{\vee}_i \subseteq \varpi^{-1}L_i.
		$
		
		We know that there is a unique $O_{F_i}$-stable lattices in $\BV_i$ up to $O_{F_i}$-scaling.  As $F_i/F_0$ is unramified, $\varpi$ is a uniformizer for $F_i$.  We have two cases: either $L_i^{\vee}=L_i$ thus $t(L_i)=0$,  which happens if and only if $\BV_i$ is a split $F/F_{0}$-hermitian space; or $L_i^{\vee}=\varpi^{-1}L_i$ thus $t(L_i)=f_i$, which happens if and only if $\BV_i$ is a non-split $F/F_{i,0}$-hermitian space. 
		
		Hence $L_i$ are all uniquely determined. Conversely, we can construct an $O_F[g]$-stable lattice by reversing the above process.
	\end{proof}
	
	If $x \in \CN(\BF)$ is a fixed point of $g$, then the associated special pair $(A, B)$ by Proposition \ref{prop: special pair classifying x in N} is also fixed by $g$. Recall $L_A$ (resp. $L_B$) the smallest $\tau$-stable lattice containing $A$ (resp. $B$) (\ref{eq: generate tau-stable lat}) is generated by $\tau^i(A)$ (resp. $\tau^i(B)$) ($i=0, 1,2,\dots$). So $L_A, L_B$ are $O_F[g]$-stable. Therefore, from Proposition \ref{special pair: choose 1 in 2} and Proposition \ref{prop: O_F[g] stable lat} we see that  
	
	\begin{corollary}
		Assume that $O_F[g]$ is an \'etale maximal order. Let $L_g$ be the unique $O_F[g]$-stable vertex lattice in $\BV$. Then we have an inclusion
		$$
		\Fix(g)(\BF) \subseteq \mathrm{BT}(L_g).
		$$
		Hence $\Fix(g)(\BF)$ agrees with fixed points of $\ov g \in \U(\BV(L_g)_0)(\BF_q)$ on the smooth projective variety $\mathrm{BT}(L_g)$. 
	\end{corollary}
	
	\begin{example}\label{Exa: Fix(g) on Pn}
		Consider the case $\sum_{\text{$\BV_i$ is not quasi-split}} f_i = t-1$. Then $\Fix(g)(\BF)$ is equal to the fixed points of $g \in \U(L_g/ \varpi L_g^\vee)$ on $\CY(L^{\vee}_g) = \BP(L_g/ \varpi (L^{\vee}_g))$. They corresponds to $O_F[g] = \prod_{i=1}^{m} O_{F_i}$-stable lines in $L_g/ \varpi L^{\vee}_g \otimes_{\BF_{q^2}} \BF \cong \prod_{\text{$\BV_i$ is quasi-split}} k_{F_i} \otimes_{\BF_{q^2}} \BF$. Hence the number of fixed points of $g$ is the number of index $i$ such that $f_i=1$ and $\BV_i$ is split. 
		
		If $n=2$ and $t=1$, then $F[g]=F \times  F$ and $\BV=\BV_1 \times \BV_2$. We see any unramified maximal order $g \in \U(\BV_2)(F_0)$ has exactly two fixed points on $(\CN^{[1]}_2)^\red$, which lies in a type $0$ (resp. $2$) Bruhat--Tits stratum $\BP^1_L$ if $\BV_1$ and $\BV_2$ are both split (resp. non-split).
	\end{example}

	\subsection{Fixed points via moduli interpretation}
	
	By definition, the subscheme $\Fix(g) \to  \CN$ parameterizes tuples $(X, \iota, \lambda, \rho) \in \CN$ with an $O_F$-linear isomorphism $ \varphi: X \to X $
	such that $\varphi^* \lambda=\lambda$ and that we have a commutative diagram
	\[
	\xymatrix{ X \times_S \ov S \ar[d]^{\rho} \ar[r]^{\varphi}  & X \times_S \ov S 
		\ar[d]^{\rho} 
		\\  \BX \times_\BF \ov S  \ar[r]^{g} & \BX \times_\BF \ov S
	}
	\]
	
	For any point $(X, \iota, \lambda, \rho) \in \Fix(g)(S)$, using the idempotents in $O_F[g]=\prod_{i=1}^m O_{F_i}$ we have a decomposition
	\begin{equation}
		(X, \iota, \lambda, \rho)=\prod_{i=1}^m (X_i, \iota_i, \lambda_i, \rho_i)
	\end{equation}
	with induced decomposition  $\ker \lambda = \prod_{i} \ker \lambda_i $ Here $X_i$ is a supersingular hermitian $O_{F_i}$-$O_{F_0}$  module of dimension $n_i=\dim_{F} \BV_i$ and $O_{F_0}$-type $t_i$ for some integer $t_i$. See \cite[Defn. 3.2]{M-Thesis}, \cite[Defn. 3.21]{Cho18}) for the notion. In particular, the $O_{F_i}$-$O_{F_0}$ rank of $X_i$ is $1$ and $\ker \lambda_i \subseteq X_i[\varpi]$ has order $q^{2t_i}$. By our assumption, $n_i=\dim_{F} \BV_i= f_i$. 
	
	As $X$ satisfies Kottwitz condition of signature $(1, n-1)$, there is a unique index $i_0 $ such that $X_{i_0}$ has signature $(1, n_{i_0}-1)$ and other $X_i$ ($i \not=i_0$) has signature $(0, n_i)$. Fix $i_0$ and $t_i$ $(i=1, \hdots, m)$. The $O_F[g]$ action on $\BX$ induces a decomposition (depending on $i_0$ and $t_i$)
	$ \BX=\prod_{i} \BX_i$
	where $\BX_i$ is a supersingular hermitian $O_{F_i}-O_{F_0}$  module of $O_{F_i}/O_{F_0}$ rank $1$ and $O_{F_0}$-type $t_i$.

	Consider the universal object $X$ on $\Fix(g)$. As $F_i/F_0$ is unramified, we have the comparison theorem between hermitian $O_{F_i}$-$O_{F_0}$ modules and hermitian $O_{F_i}$ modules, see \cite[Prop. 3.29]{Cho18}. We obtain the following generalization of \cite[Prop. 4.14]{M-Thesis}.

	\begin{theorem}
		There is a natural identification of formal schemes over $O_{\breve{F}}$:
		\begin{equation}\label{what is Fix(g)}
			\Fix(g)=  \coprod_{(i_0, \{t_i\}_{i=1}^{m})} ( \mathcal{N}_{F_{i_0}, (1, 0)}^{[t_{i_0}/f_{i_0}]} \times \prod_{i \not =i_0} \mathcal{N}^{[t_i/f_i]}_{F_i, (0, 1)} ) 
		\end{equation}	
		where $(i_0, \{t_i\}_{i=1}^{m})$ runs over integers $1 \leq i_0 \leq m, \, 0 \leq t_i \leq m$ such that
		\begin{enumerate}
			\item $f_i | t_i$ and $\sum_{i=1}^{m} t_i =t$.
			\item $t_{i_0}$ is odd if $\BV_{i_0}$ is split, and is even if $\BV_{i_0}$ is non split.
			\item For $i \not=i_0$, $t_i$ is even if $\BV_i$ is split, and is odd if $\BV_i$ is non split.
		\end{enumerate}
	\end{theorem}

	Note $\sum_{i=1}^{m} f_i = n \geq t=\sum_{i=1} t_i =t$, so condition $(1)$ implies $t_i=0$ or $f_i$. From this, we see 
	$$
	t_i=\begin{cases}
		f_i & \text{if $i= i_0$, $\BV_{i_0}$ is split,}   \\
		0 & \text{if $i= i_0$, $\BV_{i_0}$ is not split,} \\
		0 & \text{if $i\not = i_0$, $\BV_{i}$ is split,} \\
		f_i & \text{if $i\not = i_0$, $\BV_{i}$ is not split.} \\
	\end{cases}
	$$
	
	\begin{corollary}\label{prop: max order Int is not derived}
		Assume that $O_F[g]$ is an \'etale maximal order. Then
		\begin{enumerate}
			\item $\Fix(g)$ is formally smooth over $\Spf O_{\breve{F}}$.
			\item We have $
			\Fix(g)= {}^\BL \wt{\Fix}(g)$.
		\end{enumerate} 
	\end{corollary}
	\begin{proof}
		As $\CN_1^{[1]}$ and $\CN_1^{[1]}$ are isomorphic to $\Spf O_{\breve{F_i}}=\Spf O_{\breve{F}}$, $\Fix(g)$ is formally smooth over $\Spf O_{\breve{F}}$.
		
		For any fixed point $(A,B) \in \CN(\BF)$ of $g$, both $L_A$ and $L_B$ in Proposition \ref{special pair: choose 1 in 2} are $O_F[g]$-stable, so can't both be vertex lattices by the uniqueness in Proposition \ref{prop: O_F[g] stable lat}. So $\Fix(g)$ doesn't intersect with the link stratum $\CN^{\dagger}$, hence $\Fix(g)=\wt{\Fix}(g)$.	As $\wt{\Fix}(g)$ is purely $1$-dimensional of the expect dimension, we apply  \cite[Lem. B.2.]{AFL} (note that $\CN$ and $\wt{\CN \times \CN}$ are regular) to see that
		${}^\BL \wt{\Fix}(g) = \wt{\Fix}(g)$.
	\end{proof}
	
	\begin{example}
		Assume that $t_0=1$, $f_i=1$ and $\BV_i$ is split for all $i$, then $\Fix(g)= \coprod_{i=1}^{n} \Spf O_{\breve{F}}$ which matches Example \ref{Exa: Fix(g) on Pn}.
	\end{example}
	
	\subsection{Geometric side}
	
	Consider a non-zero vector $u=(u_i) \in \BV=\prod_{i=1}^m \BV_i$. 
	
	\begin{proposition}\label{prop: Z(u) in pieces of max order Fix(g) }
		On the $(i_0, \{t_i\}_{i=1}^{m})$-th copy of $\Spf O_{\breve{F}}$ of $\Fix(g)$,	we have 
		$$
		\CZ(u)|_{\Spf O_{\breve{F}}}  \cong \CZ(u_{i_0}) \hookrightarrow \mathcal{N}_{F_{i_0}, (1, 0)}^{[t_{i_0}/f_{i_0}]} .
		$$
		The $O_{\breve{F}}$-length of $\CZ(u)|_{\Spf O_{\breve{F}}}$ is
		$ \begin{cases}
			\max\{0, \frac{v_F((u_i,u_i)_\BV)}{2}\} & \text{if $\BV_{i_0}$ is split,}   \\
			\max\{0, \frac{v_F((u_i,u_i)_\BV+1}{2}\}  & \text{if $\BV_{i_0}$ is not split.} \\
		\end{cases}
		$.
	\end{proposition}	
	\begin{proof}
		We observe that
		\[ 
		\Hom_{O_F}(\CE, X)= \prod_{i} \Hom_{O_{F_i}}(\CE \otimes_{O_F} O_{F_i}, X_i).
		\]
		Therefore, the lifting of $u$ corresponds to liftings of each $u_i$ on the right hand side of (\ref{what is Fix(g)}).
		
		If $i \not = i_0$ then $X_i$ has signature $(0, 1)$, hence $u_i$ lifts to the whole space $\mathcal{N}^{[t_i/f_i]}_{F_i, (0, 1)}$ by Lubin--Tate theory. So we have $\CZ(u)|_{\Spf O_{\breve{F}}}  \cong \CZ(u_{i_0})$. The computation of $\length(\CZ(u)|_{\Spf O_{\breve{F}}})$ follows from the case $n=1, t \in \{0,1\}$, see Proposition \ref{prop: Fix and Z in the case n=1}.
	\end{proof}
	
	\begin{theorem}\label{thm: max order Int}
		For any regular semisimple pair $(g,u) \in (\U(\BV) \times \BV)(F_0)_\rs$ such that $O_F[g]$ is an \'etale maximal order, we have 
		\[
		\wt {\Int}^{\CZ}(g, u)= \sum_{i=1}^m \max\{0, \frac{v_F((u_i,u_i)_\BV)+a_i}{2}\}   .
		\]
		where $a_i=0$ if $\BV_i$ is split and $a_i=1$ if $\BV_i$ is not split.
	\end{theorem}
	\begin{proof}
		As $(u,u)_{\BV}=\sum_i (u_i, u_i)_{\BV}$, the result follows from Corollary \ref{prop: max order Int is not derived} and Proposition \ref{prop: Fix and Z in the case n=1}.		
	\end{proof}

	\subsection{Analytic side via counting lattices}
	
	Recall that $L \subseteq V$ is a vertex lattice of rank $n$ and type $t=t_0$. Choose an orthogonal basis of $L$ to endow $L$ (resp. $V$) with an $O_{F_0}$ (resp. $F_0$)-structure $L_0$ (resp. $V_0$), which induces a $F/F_0$ semi-linear involution $\tau=\overline{(-)}$ on $V$ with fixed subspace $V_0$. Recall the Haar measure on $\GL(V_0)(F_0)$ is normalized so that the volume of $\GL(L_0, L_0^\vee)$ is $1$. 
	
	Under the basis, we identify $L_0$ (resp. $L_0^\vee$) with $O_{F_0}^n$ (resp. $\varpi^{-1} O_{F_0}^{t_0} \oplus O_{F_0}^{n-t_0}$). The $\GL(V_0)$-orbit of $(L, L^{\vee})$ in the set of pairs of $O_F$-lattices $(L_1, L_2) \subseteq V$ can be identified with the set
	\[
	\mathrm{Lat}_{n, F}^{[t]}:=\{ (L_1, L_2) \subseteq V| L_1 \subseteq L_2  \subseteq \varpi^{-1}L_1, \quad [L_2: L_1]=t, \quad \tau(L_1)=L_1, \quad \tau(L_2)=L_2\}.  
	\]
	Set 
	\[
	M_{n, F}^{[t]}:=\{ (L_1, L_2) \in \mathrm{Lat}_{n, F}^{[t]}| u_1 \in L_1, \quad u_2 \in L_2^*, \quad \gamma L_1=L_1, \quad \gamma L_2=L_2 \}, 
	\]
	\[
	M_{n, F, i}^{[t]}:=\{  (L_1, L_2) \in  M_{n, F}^{[t]} | [L_1:L]=i \}.
	\]
	For the function $f_{\mathrm{std}}=1_{S(L, L^{\vee})} \times 1_{L_0} \times 1_{(L^{\vee}_0)^*}$ and a regular semisimple pair $(\gamma, u_1, u_2) \in ( S(V_0) \times V_0 \times V_0^* )(F_0)_\rs$, write $(s \in \BC)$
	$$
	\Orb_F((\gamma, u_1, u_2), s, t):=\Orb((\gamma, u_1, u_2), f_{\mathrm{std}}).
	$$ 
	Consider the lattice  $L_{\gamma, u_1}=\sum_{i} O_F \gamma^i u_1 .$ Write $\ell=[L_{\gamma, u_1}: L]$.

	\begin{proposition}\label{prop: comput orb by couting lat}
		We have 
		\[
		\Orb_F((\gamma, u_1, u_2), s, t)=  \sum_{i \in \BZ} (-1)^i (\# M_{n, F, i}^{[t]}) q^{-(i+\ell)s}.
		\]
	\end{proposition}
	\begin{proof}
		By our choice of Haar measure, we have
		$$
		\Orb_F((\gamma, u_1, u_2), s, t)=\omega_L(\gamma, u_1, u_2)  \int_{h \in \GL(V_0)/ \GL(L_0, L_0^\vee)} f_{\mathrm{std}}(h.(\gamma, u_1, u_2))\eta(h) |h|^{s} d h. 
		$$
		By definition, $f_{\mathrm{std}}(h.(\gamma, u_1, u_2))$ has value $1$ if $(hL, hL^\vee) \in  M_{n, F}^{[t]}$ and value $0$ else.  For $(L_1, L_2)=(hL, hL^\vee)  \in M_{n, F}^{[t]}$, we have 
		$v(\det(h))=[hL:L]=[L_1:L]=[L_1: L_{\gamma, u_1}] + \ell$. We conclude by noting that the parity of $[L_1: L_{\gamma, u_1}]$ is equal to the transfer factor $\omega_L(\gamma, u_1, u_2)$. 
	\end{proof}

	Now we assume $(\gamma, u_1, u_2)$ matches $(g, u)$ such that $O_F[g]=\prod_{i=1}^{m} O_{F_i}$ is an \'etale maximal order. As the characteristic polynomial of $\gamma$ agrees with $g$, the embedding $O_F[\gamma] \subseteq \End(V)$ induces a compatible decomposition $V=\prod_{i=1}^{m} V_i $.
	Here $V_i$ is a $F_i/F_0$ hermitian space such that $V_i \not \simeq \BV_i$ and  $\dim_{F} V_i= f_i$. Consider the induced decomposition
	$$
	(\gamma, u_1, u_2)=\prod_{i=1}^{m} (\gamma_i, u_{1i}, u_{2i}).
	$$
	\begin{proposition}
		We have
		\begin{equation}
			\Orb_F((\gamma, u_1, u_2), s, t)=\sum_{\{t_i\}_{i=1}^m}  \prod_{} \Orb_{F_i}((\gamma_i, u_{1i}, u_{2i}), s/f_i, t_i/f_i)
		\end{equation}
		where $\{t_i\}_{i=1}^{m}$ runs over integers $0 \leq t_i \leq n$ such that $t_i \in \{0, f_i \}$ and $\sum_{i=1}^{m} t_i =t$. 
	\end{proposition}
	\begin{proof}
		From definition, we have 
		\[
		M_{n, F}^{[t]} = \coprod_{\{t_i\}_{i=1}^m} \prod_{i=1}^{m} M_{1, F_i}^{[t_i / f_i]} .
		\]
		where $\{t_i\}_{i=1}^{m}$ runs over integers $0 \leq t_i \leq n$ such that $f_i | t_i$ and $\sum_{i=1}^{m} t_i =t$. Note $\sum_{i=1}^m f_i =n \geq t$, we see $t_i \leq f_i$ hence $t_i \in \{0, f_i \}$. By definition, $\omega_L(\gamma, u_1, u_2)=\ell=\prod_{i} \ell_i=\prod_{i=1}^m \omega_{L_i}(\gamma_i, u_{1i}, u_{2i})$. So the result follows from Proposition \ref{prop: comput orb by couting lat}.
	\end{proof}

	\begin{corollary}\label{thm: max order Del} We have an equality in $\BQ \log q$:
		\[
		\del_{F}\bigl ((\gamma, u_1, u_2), t)=\sum_{i_0, \{t_i\}_{i=1}^m}  \frac{1}{f_{i_0}} \del_{F_{i_0}}((\gamma_{i_0}, u_{1i_0}, u_{2i_0})),  t_{i_0}/f_{i_0})  
		\prod_{i \not = i_0} \Orb_{F_i}((\gamma_i, u_{1i}, u_{2i}), t_i/f_i).
		\] 
	\end{corollary}
	where $i_0$ runs over integers $1 \leq i_0 \leq n$  with
	$$
	t_i=\begin{cases}
		f_i & \text{if $i= i_0$, $\BV_{i_0}$ is split,}   \\
		0 & \text{if $i= i_0$, $\BV_{i_0}$ is not split,} \\
		0 & \text{if $i\not = i_0$, $\BV_{i}$ is split,} \\
		f_i & \text{if $i\not = i_0$, $\BV_{i}$ is not split.} \\
	\end{cases}
	$$
	\begin{proof}
		The equality follows from above proposition by taking derivative. We only need to explain the conditions on $(i_0, \{t_i\}_{i=1}^{m})$. The pair $(\gamma_i, u_{1i}, u_{2i})$ matches the pair $(g_i, u_{i}) \in (\U(\BV_i) \times \BV_i)(F_0)_\rs$. 
		
		Consider any $i \not =i_0$. From \S \ref{ATC n=1}, we have the vanishing of  $\Orb_{F_i}((\gamma_i, u_{1i}, u_{2i}), 0, t_i/f_i)$ unless $t_i=0$ if $\BV_{i}$ is split and $t_i=f_i$ if $\BV_{i}$ is non-split.
		
		Compare $\sum_{i=1}^m t_i =t$ and $\BV = \oplus_{i=1} \BV_i$ (using $f_i$ is odd) , we see $t_{i_0}=f_{i_0}$ (resp. $t_{i_0}=0$ ) if $\BV_{i}$ is split (resp. not split).
	\end{proof}
	
	\begin{theorem}
		Let $L$ be a vertex lattice of dimension $n$ and type $t$. Consider a regular semisimple pair $(g,u) \in (\U(\BV) \times \BV)(F_0)_\rs$ matching $(\gamma, u_1, u_2) \in ( S(V_0) \times V_0 \times V_0^* )(F_0)_\rs$. Assume that $O_F[g]$ is an \'etale maximal order, then Conjecture \ref{conj: semi-Lie version ATC} holds for $(g, u)$ and $(\gamma, u_1, u_2)$.
	\end{theorem}
	\begin{proof}
		This follows from Theorem \ref{thm: max order Int} and Corollary \ref{thm: max order Del} and the computation in the case $n=1$ in \S \ref{ATC n=1}.
	\end{proof}

	\part{Modularity of arithmetic theta series}\label{global part}

	\section{Integral models and Balloon--Ground stratification}\label{section: global Shimura var}

	In this section, we introduce the RSZ-variant of unitary Shimura varieties \cite{RSZ-Shimura, RSZ-Diagonalcycle} for parahoric hermitian lattices as a globalization of Rapoport--Zink spaces studied in Part \ref{local part}, which admits a natural integral model $\CM$ with PEL type moduli interpretation.  Then we globalize the formal Balloon--Ground stratification  in \S \ref{Formal Balloon-Ground} to the Balloon--Ground stratification on mod $p$ fiber of $\CM$.

	\subsection{The Shimura data}\label{section: global set up of Shimura data}
	
	Let $F / F_0$ be a totally imaginary quadratic extension of a totally real number field. We always assume that $F_0 \not = \BQ$. 
	
	Write $\Sigma_{F}:=\Hom (F, \ov{\BQ})$ and choose a CM type $\Phi \subseteq \Sigma_F$ of $F$ with a distinguished element $\varphi_0 \in \Phi$. Let $V$ be a $F/F_0$-hermitian space of dimension $n \geq 1$ with signature $\{(r_{\varphi}, s_{\varphi})_{\varphi \in \Phi} \}= \{ (n-1,1)_{\varphi_0}, (n,0)_{\varphi \in \Phi -  \{ \varphi_0 \} }  \} $ under a fixed embedding $\overline{\mathbb Q} \hookrightarrow \mathbb C$.  
	
	Consider the reductive group $G:=\U(V)$ over $F_0$, and the following reductive groups over $\mathbb Q$:
	\[
	Z^{\mathbb Q}:=\{ x \in  \Res_{F/\mathbb Q} \mathbb G_m | \, \Nm_{F/F_0} x \in \mathbb G_m \}, 
	\]
	\[
	G^{\mathbb Q} := \{g \in \Res_{F_0/\mathbb Q} \GU(V) | \, c(g) \in  \mathbb G_m \},
	\]
	where $c: \Res_{F_0/\mathbb Q} \GU(V)  \to  \Res_{F_0/\mathbb Q} \mathbb G_m$ denotes the similitude character. 
	
	For any $\varphi \in \Phi$, choose a basis of $V_{\varphi}$ over $\BC$ under which the hermitian form is given by $\diag\{  1_{r_{\varphi}}, -1_{s_{\varphi}} \}$. It induces a homomorphism 
	\[
	h_{G^{\BQ}, \varphi}: \Res_{\mathbb C/\mathbb R} \mathbb G_m \to \GU(V_{\varphi}),  
	\quad  z \mapsto \diag\{z 1_{r_{\varphi}}, \overline{z} 1_{s_{\varphi} } \}.
	\]
	
	The homomorphism
	\begin{equation}\label{h for Shimura datum GQ}
		h_{G^{\BQ}}: \Res_{\BC/\BR} \mathbb G_m \to G^{\BQ}_{\BR} \hookrightarrow \prod_{\varphi \in \Phi}  \GU(V_{\varphi}),  \quad z \mapsto (h_{G^{\BQ}, \varphi}(z))_{\varphi \in \Phi}
	\end{equation}
	gives a Shimura datum $(G^{\BQ}, \{ h_{G^{\BQ}} \} )$.  The homomorphism
	\begin{equation}\label{h for Shimura datum G}
		h_{\Res_{F_0/\BQ}G}: \Res_{\BC/\BR} \mathbb G_m \to (\Res_{F_0/\BQ}G)_\BR  \hookrightarrow \prod_{\varphi \in \Phi} \U(V_{\varphi}),  \quad z \mapsto ( \diag\{ 1_{r_{\varphi}}, \overline{z}/z 1_{s_{\varphi} } \} )_{\varphi \in \Phi}
	\end{equation}
	gives a Shimura datum $(\Res_{F_0/\BQ}G, \{ h_{\Res_{F_0/\BQ}G} \} )$ (see \cite[Section 2.3]{RSZ-Shimura}).

	The homomorphism
	\begin{equation}
		h_{\Phi}: \Res_{\mathbb C/\mathbb R} \mathbb G_m \to Z^{\mathbb Q}_{\mathbb R} \hookrightarrow \prod_{\varphi \in \Phi} \Res_{\mathbb C/\mathbb R} \mathbb G_m, \quad z \mapsto (\varphi(z))_{\varphi \in \Phi}
	\end{equation}
	gives a Shimura datum $(Z^{\mathbb Q}, h_{\Phi} )$ whose reflex field $F^\Phi \subseteq \ov{\BQ}$ is the subfield fixed by $\{\sigma \in \Gal(\overline{\mathbb Q} / \mathbb Q) | \sigma \circ \Phi =\Phi \}$. Consider the fibre product
	\[
	\wt{G}=Z^{\mathbb Q} \times_{\mathbb G_m} G^{\mathbb Q}
	\]	
	along projection maps $\Nm_{F/F_0}: Z^{\mathbb Q} \to \BG_m$ and $c: G^{\mathbb Q}  \to \BG_m$. The homomorphism 
	\[
	h_{\wt G}=(h_{\Phi}, h_{G^{\mathbb Q}}) : \Res_{\BC/\BR} \mathbb G_m \to \wt{G}_{\BR} 
	\] 
	gives a Shimura datum $(\wt{G}, \{ h_{\wt G} \} )$ whose reflex field $E$ is the composition $\varphi_0(F)F^\Phi \subseteq \overline{\mathbb Q}$. Note there is a natural isomorphism
	\begin{equation}
		Z^{\mathbb Q} \times_{\mathbb G_m} G^{\mathbb Q} \cong Z^{\mathbb Q} \times \Res_{F_0/\mathbb Q}G, \quad (z, g) \mapsto (z, z^{-1}g).
	\end{equation} 
	For neat compact open subgroups $K_{Z^{\mathbb Q}} \subseteq Z^{\mathbb Q}(\mathbb A_f)$ and $K_G \subseteq G(\mathbb A_{0,f})$, the \emph{RSZ unitary Shimura variety} \cite{RSZ-Shimura} with level $K_{\wt{G}}=K_{Z^{\mathbb Q}} \times K_G$ is the Shimura variety associated to $(\wt{G}, \{ h_{\wt G} \} )$  (over $\BC$):
	\begin{equation}\label{eq: Shimura=Shimura x Shimura}
		\Sh_{K_{\wt{G}}}( \wt{G}, \{ h_{\wt{G}} \} )(\BC) \cong  \Sh_{K_{Z^{\BQ}}}( Z^{\BQ}, h_{\Phi})(\BC) \times \Sh_{K_G}( \Res_{F_0/\BQ}G, \{ h_{\Res_{F_0/\BQ}G} \})(\BC).
	\end{equation}
	
	\begin{remark}\label{rek: reflex field}
		By \cite[Rem. 3.1]{RSZ-Diagonalcycle}, We have $E=\varphi_0(F)$, if $F$ is Galois over $\mathbb Q$ or $F=F_0K$ for an imaginary quadratic field $K$ over $\mathbb Q$.
	\end{remark}
	
   Denote by 
	$$M=M_{\wt{G}, \wt{K}} \to \Spec E$$
	the canonical model of $\Sh_{\wt{K}}( \wt{G}, h_{\wt{G}})$, which is a $n-1$ dimensional smooth \emph{projective} (as $F_0 \not = \BQ$) variety over $E$.

	\subsection{The level for $L$ and $\Delta$}\label{section: distinguished v0}

	Choose a finite collection $\Delta$ of finite places for $F_0$ such that $F/F_0$ is unramified outside $\Delta$ and all $2$-adic places are in $\Delta$.
	
	Let $L$ be a hermitian lattice in $V$. If $v \not \in \Delta$ is a finite place of $F_0$ such that the localization $L_v$ is not self-dual, then we assume $v$ is inert in $F$ and moreover
	\begin{itemize}
		\item $L_v \subseteq L_v^\vee \subseteq \varpi_{v}^{-1}L_v$ and $L_v^\vee/L_v$ has size $q_v^{2t_v}$ for some $0 \leq t_v \leq n$ i.e., $L_v$ is a vertex lattice in $V_v$ of type $t_v$.
		\item If $0<t_v<n$, then $E \otimes_\BQ  \BQ_{p_v}$ decomposes into unramified extensions of $\BQ_{p_v}$. Here $p_v>2$ is the underlying prime of $v$. In particular, $F_v$ is unramified over $\BQ_{p_v}$.
	\end{itemize}
	Let $K_{Z^\BQ, \Delta}$ (resp. $K_{G, \Delta}$) be a compact open subgroup of $Z^\BQ(F_{0,\Delta})$ (resp. $\U(V)(F_{0, \Delta})$). Let $K_G(L)^{\Delta}$ be the (completed) stabilizer of $L$ in $ \U(V)(\mathbb A_{0, f}^{\Delta})$,  and $K_{Z^{\mathbb Q}}^{\Delta} $ be the unique maximal compact open subgroup of 
	$Z^{\BQ}(\BA_{0, f}^{\Delta})$. We call
	\begin{equation}\label{level for L and Delta}
		{
			K_{\wt{G}}= K_{Z^\BQ, \Delta} \times K_{Z^{\mathbb Q}}^{\Delta} \times K_{G, \Delta} \times K_G(L)^{\Delta} \subseteq \wt{G}(\mathbb A_f).
       }
   \end{equation}
	a level structure for $L$ and $\Delta$.

	\subsection{The PEL type moduli}\label{moduli int of RSZ}
	
	Now we define the integral model for $M$ with level for $L$ and $\Delta$ generalizing the construction in \cite[Section 14.1]{LZ19}, \cite[Section 6.1]{RSZ-Shimura}. Consider the functor 
	$$
	\CM_0 \to \Spec O_E[\Delta^{-1}]
	$$
	sending a locally noetherian $O_E [\Delta^{-1}]$-scheme $S$ to the groupoid $\mathcal{M}_0(S)$ of tuples $(A_0, \iota_0, \lambda_0, \ov {\eta}_0)$ where
	\begin{itemize}
		\item $A_0$ is an abelian scheme over $S$ of dimension $[F_0 : \mathbb Q]$.
		\item $\iota_0: O_F[\Delta^{-1}] \to \End(A_0)[\Delta^{-1}]$ is an $O_F[\Delta^{-1}]$-action satisfying the \emph{Kottwitz condition} of signature $\{(1,  0)_{\varphi \in \Phi} \}$:
		\begin{equation}
			\charac (\iota_0(a) \mid \Lie A_0) =\prod_{\varphi \in \Phi} (T- \varphi(a)) \in \CO_S[T]. 
		\end{equation}
		\item $\lambda_0: A_0 \to A_0^\vee$ is an away-from-$\Delta$ principal polarization such that for all $a \in O_F[\Delta^{-1}]$ we have $ \lambda^{-1}_0 \circ \iota_0(a)^{\vee} \circ \lambda_0 = \iota_0(\overline{a})$.
		\item $\ov \eta_0$ is a $K_{Z^{\mathbb Q}, \Delta}$-level structure as in \cite[Section C.3]{Liu-Fourier-Jacobi}. 
	\end{itemize}	
	An isomorphism between two tuples is a quasi-isogeny preserving the polarization and $K_{Z^{\mathbb Q}, \Delta}$-level structure. The functor $\mathcal{M}_0 \to \Spec O_E[\Delta^{-1}]$ is representable, finite and \'etale \cite[Prop. 3.1.2]{Ho-CM}.
	
	From now on, we assume that $\CM_0$ is non-empty which holds if $F/F_0$ is ramified at some place, see \cite[Rem. 3.5 (ii)]{RSZ-Diagonalcycle}. The generic fiber $M_0$ of $\mathcal{M}_0$ is a disjoint union of copies of $\Sh_{K_{Z^{\mathbb Q}}} (Z^{\mathbb Q}, h_{\Phi} )$, see \cite[Lem 3.4]{RSZ-Diagonalcycle}. We work with one copy and still denote its \'etale integral model by $\CM_0$.
	\begin{remark}
		We allow $L$ to be non-self dual at some $v \not \in \Delta$, so $\mathcal{M}_0$ can be non-empty even if $F/F_0$ is unramified everywhere. 		
	\end{remark}
	\begin{definition}\label{defn: integral RSZ model}
		The integral RSZ Shimura variety with level $K$ for $L$ and $\Delta$ is the functor 
		$$
		\CM_{\wt{G}, \wt{K}} \to \Spec O_E[\Delta^{-1}]
		$$
		sending a locally noetherian $O_E [\Delta^{-1}]$-scheme $S$ to the groupoid	of tuples 
		$(A_0, \iota_0, \lambda_0, \ov \eta_0,  A, \iota, \lambda, \overline{\eta}_{\Delta})$ where
		\begin{itemize}
			\item $(A_0, \iota_0, \lambda_0, \ov \eta_0)$ is an object of $\CM_0(S)$.
			\item $A$ is an abelian scheme over $S$ of dimension $n[F_0:\BQ]$.
			\item $\iota: O_F[\Delta^{-1}] \to \End(A)[\Delta^{-1}]$ is an $O_F[\Delta^{-1}]$-action satisfying the \emph{Kottwitz condition} of signature $\{ (n-1, 1)_{\varphi_0}, (n,0)_{\varphi \in \Phi - \{\varphi_0 \}} \} $.
			\item $\lambda : A \to A^{\vee}$ is a polarization such that for all $a \in O_F[\Delta^{-1}]$ we have $ \lambda^{-1} \circ \iota(a)^{\vee} \circ \lambda= \iota(\overline{a}). $
			\item For any finite place $v \not \in \Delta$ of $F_0$, the induced map $\lambda_v:  A[\varpi_v^{\infty}] \to A^\vee[\varpi_v^{\infty}]$ has kernel $\Ker \lambda_v \subseteq A[\varpi_v]$ which has order $\# L^{\vee}_v/L_v = q_v^{2t_v}$.
			\item $\overline{\eta}_{\Delta}$ is a $K_{G, \Delta}$-orbit of isometries of hermitian modules 
			\[
			\eta_{\Delta}: V_{\Delta}(A_0, A)= \prod_{v \in \Delta} V_{v}(A_0, A) \cong V(F_{0, \Delta})
			\]
			as smooth $F_{0, \Delta}$-sheaves on $S$. Here 
			$V_{v}(A_0, A)=\Hom_{F \otimes_{F_0} F_{0, v}} (V_vA_0, V_vA)$ is the Hom space between rational Tate modules of $A_0$ and $A$, with the hermitian form
			\[
			(x, y):= \lambda_0^{-1} \circ y^\vee \circ \lambda \circ x \in  F_v.
			\]
			\item For any finite place $v \not \in \Delta$ of $F_0$, we put the sign condition and Eisenstein condition at $v$ as in \cite[Section 4.1]{RSZ-Diagonalcycle}. By \cite[Section 4.1]{RSZ-Diagonalcycle} and \cite[Rem. 6.5 (i)]{RSZ-Shimura}, the Eisenstein condition is automatic when the places of $F$ above $v$ are unramified over $p$, and the sign condition is automatic when $v$ is split in $F$. 
		\end{itemize}
		An isomorphism between two tuples is a pair of quasi-isogenies $(\phi_0, \phi): (A_0, A) \to (A_0', A')$ preserving polarizations and the $K_{Z^\BQ, \Delta}\times K_{G, \Delta}$-orbit of level structures.
	\end{definition}
	
	\begin{theorem}\cite[Thm. 6.2]{AFL}
		The functor $\mathcal{M}_{\wt{G}, \wt{K}}$ is representable by a separated scheme flat and of finite type over $\Spec O_E[\Delta^{-1}]$. 
		Moreover, $\mathcal{M}_{\wt{G}, \wt{K}} $ is smooth over $O_E \otimes_{O_F} O_{F,v}$ for any finite place $v \not \in \Delta$ such that $L_v$ is self-dual.
	\end{theorem}
	
	From now on, for simplicity we write $\CM=\CM_{\wt{G}, \wt{K}}$ and write $(A_0, A, \ov\eta)$ for a $S$-point of $\CM$. By the assumption $F_0 \not =\BQ$, $\CM$ is in fact projective over $O_E[\Delta^{-1}]$.

	\subsection{The singularity and Balloon--Ground stratification}\label{section: global KR strata}
	
	Let $v \not \in \Delta$ be a finite place of $F_0$ such that $L_v$ is not self-dual. Choose a place $\nu | v$ of the reflex field $E$.  Denote by $k_{v}=\BF_{q_v}$ (resp. $k_{\nu}$) the residue field of $F_{0,v}$ (resp. $E_{\nu}$). Denote by $\CM_{k_{\nu}}$ the special fiber of $\CM$ over $k_{\nu}$.  
	
	\begin{theorem}\cite[Prop.  4.2.]{Cho18}
		If $t_v=n$, then $\CM \to \Spec O_{E_\nu}$ is smooth. If $0<t_v<n$, then $\CM \to \Spec O_{E_\nu}$ is of strictly semi-stable reduction.
	\end{theorem}
	The proof is based on computations of local models. For our need, we give a different proof using Grothendieck--Messing theory, which contains finer information and relates the singularity of $\CM \to  \Spec O_{E, \nu}$ with the Balloon-Ground stratification on $\CM$. 
	
	Let $S$ be a locally noetherian scheme over $k_{\nu}$. For any point $(A_0, A, \ov \eta) \in \CM(S)$, consider the bundle $\omega_{A^\vee}:= (\Lie A^{\vee} )^{\vee}$. The Hodge filtration induces a short exact sequence of locally free $\CO_S$-modules:
	\begin{equation}
		0 \to \omega_{A^\vee}  \to H^{\mathrm{dR}}_1 (A) \to \Lie A \to 0.
	\end{equation}
	which decomposes under the $O_F[\Delta^{-1}]$-action into $\varphi$-parts for $\varphi \in \Phi \cup \Phi^c = \Hom(F, \ov{\BQ})$:
	\[
	0 \to (\omega_{A^\vee})_\varphi  \to H^{\mathrm{dR}}_{1} (A)_\varphi \to (\Lie A)_\varphi \to 0.
	\]
	By the Kottwitz signature condition, $(\omega_{A^\vee})_{\varphi_0}$ (resp. $(\omega_{A^\vee})_{\varphi_0^c}$) is a vector bundle on $S$ of rank $1$ (resp. $n-1$). 
	
	\begin{definition}
		\begin{enumerate}
			\item 
			The \emph{Balloon stratum} $\CM_{k_{\nu}}^\circ$ is the vanishing locus of the universal section $\Lie \lambda: (\Lie A)_{\varphi_0^c} \to (\Lie A^\vee)_{\varphi_0}$ inside $\CM_{O_{E_\nu}}$. 
			\item 
			The \emph{Ground stratum} $\CM_{k_{\nu}}^\bullet$ is the vanishing locus of the universal section $\Lie \lambda^\vee: (\Lie A^\vee)_{\varphi_0} \to (\Lie A)_{\varphi_0^c} $  inside $\CM_{O_{E_\nu}}$. Here $\lambda^\vee$ is the dual polarization of $\lambda$ at $v$ such that $\lambda^\vee \circ \lambda= [\varpi_v]$.
			\item The \emph{Linking stratum} $\CM_{k_{\nu}}^\dagger:= \CM_{k_{\nu}}^\circ \cap \CM_{k_{\nu_0}}^\bullet \hookrightarrow \CM_{k_{\nu_0}}$.
		\end{enumerate}
	\end{definition}
	
	As $\Lie \lambda^\vee \circ \Lie \lambda =[\varpi_v]$ , we have the \emph{Balloon--Ground stratification} $\CM_{k_{\nu}}=\CM_{k_{\nu}}^\circ \cup \CM_{k_{\nu}}^\bullet$, which globalizes the
	\emph{formal Balloon--Ground stratification} in \S \ref{Formal Balloon-Ground}. Consider the bilinear form 
	$$
	(-,-)_{\varphi}: H^{\mathrm{dR}}_1 (A)_{\varphi}  \times H^{\mathrm{dR}}_1 (A)_{\varphi^c} \to O_S
	$$
	induced by the polarization $\lambda: A \to A^{\vee}$.  As the kernel $\Ker \, \lambda[\varpi_{v}]$ has order $q^{2t_v}$, the orthogonal complements $(H^{\mathrm{dR}}_1 (A)_{\varphi_0^c} )^{\perp} \subseteq H^{\mathrm{dR}}_1 (A)_{\varphi_0}$ and $(H^{\mathrm{dR}}_1 (A)_{\varphi_0})^{\perp} \subseteq H^{\mathrm{dR}}_1 (A)_{\varphi_0^c}$ are vector bundles of rank $t_v$ over $S$.
	
	\begin{proposition}\label{Balloon-Ground: global comp with LTXZZ}
		The Balloon stratum $\CM_{k_{\nu}}^\circ$ (resp. Ground stratum $\CM_{k_{\nu}}^\bullet$) is the closed locus of $\CM_{k_{\nu}}$ where $(\omega_{A^\vee})_{\varphi_0} \subseteq (H^{\mathrm{dR}}_1 (A)_{\varphi_0^c} )^{\perp} $ (resp. $(H^{\mathrm{dR}}_1 (A)_{\varphi_0})^{\perp} \subseteq  (\omega_{A^\vee})_{\varphi_0^c} $).
	\end{proposition}
	\begin{proof}
		This is proved in the same way as Proposition \ref{Balloon-Ground: comp with LTXZZ}.
	\end{proof}
	
	\begin{theorem}\label{proof: global semi-stable}
		If $0<t_v<n$, then $\CM \to \Spec O_{E_\nu}$ is of strictly semi-stable reduction. Moreover, 
		\begin{enumerate}
			\item For a geometric point $x$ of $\CM_{k_{\nu}}^\dagger$, the completed local ring of $\CM$ at $x$ admits a smooth map to $W(k_\nu)[[x_1, y_1]]/(x_1y_1 - p)$ such that $\CM_{k_{\nu}}^\circ$ (resp.  $\CM_{k_{\nu}}^\bullet$) is the pullback of the Cartier divisor $x_1=0$ (resp. $y_1=0$). 
			\item The Balloon stratum $\CM_{k_{\nu}}^\circ$ and the Ground stratum $\CM_{k_{\nu}}^\bullet$ are smooth projective varieties of dimension $n-1$ over $k_\nu$. 
		\end{enumerate}
		
	\end{theorem}
	\begin{proof}
		The proof is similar to Proposition \ref{prop: formal balloon-ground are smooth} and \cite[Thm. 5.2.5]{LTXZZ}. We compute the completed local rings via Grothendieck--Messing theory. By our assumption on $L$, the local ring $O_{E_{\nu}} \simeq W(k_\nu)$ is the ring of Witt vectors for $k_\nu$. Let $k$ be any perfect field over $k_{\nu}$ and $A$ be a $k$-point of $\CM$. 
		
		Denote by $\BD(A)$ the covariant Dieudonne crystal of $A$. Consider the Dieudonne module $M=\BD(A)(W(k))$ which is finite free over $W(k)$ of rank $2n[F_0: \BQ]$. By the signature condition, we only need to consider deformations $\varphi_0$ and $\varphi_{0}^c$ part of the Hodge Filtration of $A$. Now  $M_{\varphi_0}$ and $M_{\varphi^c_0}$ are finite free over $W(k)$ of rank $n$.  The short exact sequence   $ 0 \rightarrow \Lie(A^{\vee})^{\vee} \rightarrow H_1^{\mathrm{dR}}(A) \rightarrow \Lie(A) \rightarrow 0 $
		can be identified by Dieudonne theory with  ($V_M$ is the Verschiebung operator on $M$)
		\[ 0 \rightarrow V_M M/ \varpi_v M \rightarrow M/\varpi_v M \rightarrow M/V_MM \rightarrow 0. 
		\]
		Let $V=K^n$ be a $n$-dimensional quadratic space over $K:=W(k)[1/p]$ with the bilinear form
		\[ (x, y)= \sum_{i=1}^{t_v} \varpi_{v} x_iy_i + \sum_{i=t_v+1}^{n} x_iy_i . 
		\]	
		Let $\Ker$ be the radical of this bilinear form. We can identify $M_{\varphi_0} \otimes K$ and $M_{\varphi^c_0} \otimes K$ with $V$ such that the pairing between them induced by $\lambda$ can be identified with the quadratic form above. By Grothendieck--Messing theory we see that the local model for $\CM \to \Spec O_{E_\nu}$ is the following moduli problem $(\ell, H)$ for $W(k_\nu)$-algebras $R$ where $p \in R$ is locally nilpotent:
		\[	
		\text{ $\ell$ is a line in $V \otimes R$ and $H$ is a hyperplane in $V \otimes R$ such that $(\ell, H)=0$.}
		\]
		The Hodge filtration of $A$ defines such a pair $(\ell_0, H_0)$ for $R=k$. And we can compute the completed local ring of $\CM$ at $x$ by deforming $(\ell_0, H_0)$ in the local model. If $\Ker \not \subseteq H$, then $\Ker + H= V \otimes R$, hence $(\ell, V \otimes R))=0$ i.e., $\ell \subseteq \Ker$. So we have the following three cases:
		\begin{itemize}
			\item $\ell \subseteq \Ker$, and $\Ker \subseteq H$, which corresponds to $x \in \CM^\dagger_{k_{\nu_0}}$.
			\item $\ell \not \subseteq \Ker$, and $\Ker \subseteq H$,  which corresponds to $x \in \CM^\bullet_{k_{\nu_0}} -\CM^\dagger_{k_{\nu_0}}$.
			\item $\ell \subseteq \Ker$, and $\Ker \not \subseteq H$, which corresponds to $x \in \CM^\circ_{k_{\nu_0}} -  \CM^\dagger_{k_{\nu_0}}$.
		\end{itemize}
		Standard computations show the semi-stable reduction results $(1)$ and $(2)$.
	\end{proof}

	\begin{remark}
		If $t_v=1$ i.e, $L_v$ is almost-self dual, the Balloon--Ground stratification is introduced in \cite[Section 5]{LTXZZ} under the assumption that $p$ is a special inert prime in particular $F_{0,v_0}=\BQ_p$. In this case, the Balloon stratum $\CM_{k_\nu}^\circ$ lies in the basic locus of $\CM_{k_{\nu}}$. Hence the basic locus of $\CM_{k_\nu}$ contains some irreducible components of $\CM_{k_\nu}$. Such phenomenon also appears in Stamm's example \cite{Stamm} for Hilbert modular surfaces at Iwahori levels.
	\end{remark}

	\section{Kudla--Rapoport cycles and modified Hecke CM cycles}\label{section: global KR}
	
	In this section, we globalize the construction in \S \ref{section:local RZ}. Fix a finite place $v_0 \not \in \Delta$ of $F_0$ such that $L_{v_0}$ is not self-dual. We introduce two kinds of Kudla--Rapoport divisors on the integral model with testing functions $1_{L_{v_0}}$ and $1_{L_{v_0}^{\vee}}$ at $v_0$ respectively. Then we desingularize the self-product $\CM \times \CM$ and define (derived) modified Hecke CM cycles on $\CM$. Finally, we prove basic uniformization results to connect with \S \ref{section:local RZ}.

	\subsection{The Kudla--Rapoport divisors over the generic fiber}
	
	Consider any neat compact open subgroup $K_{\wt G}=K_{Z^\BQ} \times K_G$ of $\wt G(\BA_f)$ (not necessarily from a lattice $L$). For any $(A_0, A, \ov\eta) \in M_{\wt G, \wt K}(S)$, endow the $F$-vector space $\Hom^\circ_F(A_0,A)$ with the hermitian form
	\begin{equation}
		(u_1,u_2)
		:= \lambda_0^{-1}\circ u_2^\vee\circ \lambda \circ u_1 \in \End ^\circ_F(A_0)\simeq F.
	\end{equation}
	
	\begin{definition}\label{def:KR generic fiber}
		Let $\xi\in F_{0}$ and $\mu\in  V(\BA_{0,f})  /K_{G}$.	The \emph{Kudla--Rapoport cycle} $Z(\xi, \mu)$ is the functor sending any locally noetherian $E$-scheme $S$ to the groupoid of tuples
		$(A_0, A, \ov\eta, u)$ where
		\begin{altitemize}
			\item  $(A_0, A, \ov \eta) \in M_{\wt G, \wt K}(S)$.
			\item  $u\in \Hom^\circ_F(A_0,A)$ such that $(u,u)=\xi$.
			\item $\ov\eta (u)$ is in the $K_{G}$-orbit $\mu$. 
		\end{altitemize}
		An isomorphism between two tuples $(A_0, A, \ov\eta, u)$ and $(A_0', A', \ov\eta', u')$ is an isomorphism of tuples $(\phi_0, \phi): (A_0, A, \ov\eta) \simeq (A_0', A', \ov\eta') $ such that $\phi \circ u= u' \circ \phi_0$.
	\end{definition}
	
	By positivity of the Rosati involution, $Z(\xi, \mu) $ is empty unless $\xi \in F_{0,+}$ is totally positive or $\xi=0$. The natural forgetful morphism $i: Z(\xi, \mu) \to M$ is finite and unramified \cite[Prop. 4.22]{Liu-Fourier-Jacobi}. It is \'etale locally a Cartier divisor if $\xi \not = 0$. 
	
	By taking the image of $Z(\xi, \mu)$ (with multiplicities) inside $M_{\wt G, \wt K}$, we view $Z(\xi, \mu)$ as elements in the Chow group (of $\BQ$-coefficients) $\Ch^1(M_{\wt G, \wt K})$. For a Schwartz function $\phi \in \CS(V(\BA_{0,f}))^{K_{G}}$ and $\xi \in F_{0,+}$,  the $\phi$-averaged Kudla--Rapoport divisor is the finite summation 
	\begin{equation}\label{KR generic average}
		Z(\xi, \phi)\colon=\sum_{\mu\in V_\xi(\BA_{0,f})/K_{G}} \phi(\mu)\,Z(\xi, \mu) \in \Ch^1(M_{\wt G, \wt K}).
	\end{equation}Here $V_{\xi}=\{x \in V| (x,x)= \xi \}$ is the hyperboloid in $V$ of length $\xi$.  We put 
	\[ 
	Z(0, \phi) := -\phi(0) c_1(\omega)  \in \Ch^1(M_{\wt G, \wt K})
	\]
	where $\omega$ is the automorphic line bundle (Hodge bundle) as in \cite{Kudla97Duke}, which is a descent of the tautological line bundle on the hermitian symmetric domain $X$ of $M_{\wt G, \wt K}$ to $E$.

	\subsection{Generating series}
	
	For $\xi \in F_0$, consider the weight $n$ Whittaker function on $\SL_2(F_{0, \infty})$:
	$$W^{(n)}_\xi (h_{\infty})=\prod_{v | \infty}  W^{(n)}_\xi (h_{v}).
	$$ Consider the Weil representation $\omega=\omega_V$ of $\SL_2(\mathbb A_{0,f})$ on $\CS(V(\BA_{0,f}))$ which commutes with the natural action of $\U(V)(\BA_{0,f})$.  For a Schwartz function $\phi \in \CS(V(\BA_{0,f}))^{K_{G}}$, we form \emph{the generating series of Kudla--Rapoport divisors} on $M_{\wt G, \wt K}$ by
	\begin{equation}\label{eq: generating series generic fiber}
		Z(h, \phi) :=W^{(n)}_0(h_{\infty}) Z(0, \omega(h_f) \phi ) + \sum_{\xi \in F_{0, +}} W^{(n)}_\xi (h_{\infty})  Z(\xi, \omega(h_f) \phi )\in \Ch^1(M_{\wt G, \wt K}),
	\end{equation}
	for any $h=(h_\infty, h_f) \in \SL_2(\BA_0)=\SL_2(F_{0, \infty}) \times \SL_2(\BA_{0, f})$.

	For any $h_f \in \SL_2(\mathbb A_{0,f})$, by definition we have 
	\begin{equation}\label{eq: dual relation on generic fiber}
		Z(h, \omega(h_f)\phi)=Z(hh_f, \phi). 
	\end{equation}	
	
	The known modularity over $E$ \cite{Liu-thesis, YZZ} is one main input of our modularity result.
	
	\begin{proposition}\cite[Thm. 8.1]{AFL}\label{modularity over generic fiber} 
		The function $Z(h, \phi)$ is a weight $n$ holomorphic automorphic form valued in $\Ch^1(M_{\wt G, \wt K})$ i.e., we have 
		$$Z(h, \phi) \in \CA_{\rm hol}(\SL_2(\BA_0), K_0, n)_{\overline{\BQ}} \otimes_{\overline{\BQ}} \Ch^1(M_{\wt G, \wt K})_{\overline{\BQ}}$$
		where $K_0 \subseteq \SL_2(\BA_{0,f})$ is any compact open subgroup that fixes $\phi$ under the Weil representation.
	\end{proposition}

	\subsection{The Kudla--Rapoport divisors on the integral model $\CM$}
	
	Consider the level structure for $L$ and $\Delta$ in \S \ref{section: distinguished v0}:
	\[ K_{\wt{G}}= K_{Z^{\BQ}}^{\Delta} \times K_{Z^{\BQ}, \Delta}  \times K_{G, \Delta} \times K_G(L)^{\Delta} \leq  \wt{G}(\mathbb A_f) \]
	So $K_{G,v_0}=\U(L_{v_0})$ and $L_{v_0}$ is a vertex lattice of type $0 \leq t_0 \leq n$.
	
	For any point $(A_0, A, \ov\eta) \in \CM_{\wt G, \wt K}(S)$, consider the hermitian form $(\cdot,\cdot)$ on $\Hom_{O_F}(A_0,A)[1/\Delta]$ given by
	\[	
	(u_1, u_2)= \lambda_0^{-1}\circ u_2^\vee\circ \lambda \circ u_1 \in \End_{O_F}(A_0)[1/\Delta] \simeq O_F[1/\Delta].
	\] 
	
	The polarization $\lambda: A \to A^{\vee}$ gives an injection 
	$$\lambda \circ-: \Hom_{O_F}(A_0,A)[1/\Delta]  \to \Hom_{O_F}(A_0,A^\vee)[1/\Delta]
	$$
	
	\begin{definition}\label{defn: global int KR}
		Let $\xi\in F_{0,+}$ and $\mu_\Delta \in V(F_{0,\Delta})/K_{G,\Delta}$.	\emph{The Kudla--Rapoport cycle}  $\CZ(\xi, \mu_{\Delta})$ (resp. $\CY(\xi, \mu_{\Delta})$) is the functor sending any locally noetherian scheme $S$ over $O_E[1/\Delta]$ to the groupoid of tuples $(A_0, A, \ov\eta, u)$ where
		\begin{altitemize}
			\item  $(A_0, A, \ov\eta) \in\CM_{\wt G, \wt K}(S)$;
			\item $u \in \Hom^\circ_F(A_0, A)$ with $(u,u)=\xi$.  And $u \in \Hom_{O_F}(A_0,A)[1/\Delta]$ (resp.  $\lambda \circ u \in \Hom_{O_F}(A_0,A^\vee)[1/\Delta]$). 
			\item $\ov\eta (u)$ is in the $K_{G,\Delta}$-orbit $\mu_{\Delta}$. 
		\end{altitemize}
	\end{definition} 
	
	\begin{proposition}
		The morphisms $\CZ(\xi, \mu_\Delta) \rightarrow \CM_{\wt G, \wt K}$ and $\CY(\xi, \mu_\Delta)\rightarrow \CM_{\wt G, \wt K}$ are representable, finite and unramified, and \'etale locally Cartier divisors on $\CM_{\wt G, \wt K}$.
	\end{proposition}
	\begin{proof}
		This follows from the proof of \cite[Prop. 4.22]{Liu-Fourier-Jacobi}.
	\end{proof}
	
	\begin{remark}
		We have a natural inclusion $\CZ(\xi, \mu_\Delta) \to \CY(\xi, \mu_\Delta)$. In general, the Cartier divisors $\CZ(\xi, \mu_\Delta)$  and $\CY(\xi, \mu_\Delta)$ may not be flat over $O_E$, see the Shimura curve example \cite{KR-height}.
	\end{remark}

	Choose $\phi^{\Delta, v_0} \in \CS(V(\BA^{\Delta, v_0}_f))$ as the indicator function of the completion of $L$ inside $V(\BA^{\Delta, v_0}_f)$.  For $i \in  \{1, 2\}$, consider the Schwartz function $\phi_i=\phi_{\Delta} \otimes \phi_{i, v_0} \otimes \phi^{\Delta, v_0} \in \CS(V(\BA_{0,f}))^{K_G}$ where $\phi_{\Delta} \in \CS(V(F_{\Delta}))^{K_{G, \Delta}}$ and
	\begin{equation}
		\phi_{1, v_0}=1_{L_{v_0}}, \quad \phi_{2, v_0}=1_{L_{v_0}^{\vee}}.
	\end{equation} 
	For $\xi \in F_{0,+}$,  the $\phi_i$-averaged Kudla--Rapoport divisor is the finite summation
	\begin{align}\label{def:KR int}
		\CZ(\xi, \phi_1)\colon=\sum_{\mu_{\Delta} \in V_\xi(F_{0,\Delta})/K_{G,\Delta}} \phi_\Delta(\mu_{\Delta})\,\CZ(\xi, \mu_{\Delta}),
	\end{align}
	\begin{align}
		\CZ(\xi, \phi_2)\colon=\sum_{\mu_{\Delta} \in V_\xi(F_{0,\Delta})/K_{G,\Delta}} \phi_\Delta(\mu_{\Delta})\,\CY(\xi, \mu_{\Delta})
	\end{align}
	viewed as elements in the Chow group $\Ch^1(\CM_{\wt G, \wt K})$.  
	
	\begin{proposition}
		For $\phi_i \in \CS(V(\BA_{0,f}))^{K_G}$ above ($i=1,2$),  the generic fiber of $\CZ(\xi, \phi_i)$ over $E$ agrees with $Z(\xi, \phi_i)$ defined by \eqref{KR generic average}.
	\end{proposition}
	\begin{proof}
		This follows from PEL type moduli interpretations of the canonical model $M_{\wt G, \wt K}$ similar to Definition \ref{defn: integral RSZ model}.
	\end{proof}

	\subsection{The Hecke CM cycles on the integral model $\CM$}
	
	Fix a degree $n$ monic polynomial $\alpha \in F [t]$ that is conjugate self-reciprocal i.e., 
\begin{equation}\label{conjuga self-reciprocal}
t^n \, \alpha(t^{-1})=\alpha(0) \ov{\alpha}(t).
\end{equation}

	\begin{definition}
		For any $\mu_{G, \Delta} \in K_{G, \Delta}\backslash G(F_{0, \Delta})/K_{G, \Delta}$, the \emph{naive Hecke CM cycle}
		$$
		\mathcal{CM}(\alpha, \mu_{G, \Delta}) \to \CM
		$$
		is the functor sending any locally noetherian scheme $S$ over $O_E[1/\Delta]$ to the groupoid of tuples $(A_0, A, \ov \eta, \varphi)$ where
		\begin{enumerate}
			\item $(A_0, A, \ov \eta) \in \CM(S)$.
			\item $\varphi \in \End_{O_F}(A)[1/\Delta]$ is a prime-to-$\Delta$ endomorphism such that $\text {char}(\varphi)=\alpha \in F[t]$.
			\item $\varphi^*\lambda=\lambda$, and $\eta_1 \varphi \eta_2^{-1} \in \mu_{G, \Delta}$ for some $\eta_1, \eta_2 \in \ov \eta$. 
		\end{enumerate}
	\end{definition}
	
	\begin{definition}\label{CM cycle: generic fiber}
		For such $\alpha \in F[t]$, let the Hecke CM cycle $CM(\alpha, \mu_{G, \Delta})$ over $E$ be the generic fiber of $\mathcal{CM}(\alpha, \mu_{G, \Delta})$.
	\end{definition}
	
	\begin{definition}
		The \emph{Hecke correspondence} $$
		\Hk_{\mu_{G, \Delta}} \to \CM \times_{O_E[1/\Delta]} \CM
		$$ is the functor sending any locally noetherian scheme $S$ over $O_E[1/\Delta]$
		to the groupoid of tuples $(A_0, A_1, \ov \eta_{A_1}, A_2, \ov \eta_{A_2}, \varphi)$ where
		
		\begin{enumerate}
			\item $(A_0, A_1, \ov \eta_{A_1}), \, (A_0, A_2, \ov \eta_{A_2}) \in \CM(S)$.
			\item $\varphi \in \Hom_{O_F}(A_1, A_2)[1/\Delta]$ is a prime-to-$\Delta$ homomorphism such that $\varphi^*\lambda_2=\lambda_1$.
			\item  We have $\eta_{A_2} \varphi \eta_{A_1}^{-1} \in \mu_{G, \Delta}$ for some $\eta_{A_1} \in \ov \eta_{A_1}$ and $\eta_{A_2} \in \ov \eta_{A_2}$. 
		\end{enumerate}
	\end{definition}

	\begin{proposition}
		The pullback of $\Hk_{\mu_{G, \Delta}}$ along the diagonal morphism $\Delta: \CM \to \CM \times_{O_E[1/\Delta]} \CM$ has the following decomposition as schemes over $O_E[\Delta^{-1}]$:
		\begin{equation}\label{eq: Hk deomp}
			\Hk_{\mu_{G, \Delta}}|_{\mathcal{M}} = \coprod_{\alpha } \mathcal{CM}(\alpha, \mu_{G, \Delta}).
		\end{equation}
		where the index $\alpha$ runs over all degree $n$ conjugate self-reciprocal monic polynomials $\alpha \in O_F[\Delta^{-1}][t]$ (c.f. \ref{conjuga self-reciprocal}).
	\end{proposition}
	\begin{proof}
		This follows from the moduli interpretation directly similar to \cite[Section 7.5]{AFL}.
	\end{proof}
	
	Assume that $\alpha \in F[t]$ is irreducible. Then the map $\mathcal{CM}(\alpha, \mu_{G, \Delta}) \to \CM$ is finite and unramified as in \cite[Section 7.5]{AFL}. Hence $\mathcal{CM}(\alpha, \mu_{G, \Delta})$ is proper over $O_E[1/\Delta]$ by the assumption $F_0 \not =\BQ$.

	\subsection{Derived Hecke CM cycles and blow up}\label{section: derived Hk}
	
	Assume that $\alpha \in O_F[\Delta^{-1}][t]$ is irreducible. Then $F_\alpha:=F[t] / \alpha(t)$ is a CM number field. The involution on $F$ extends to $F_\alpha$ by sending $t$ to $t^{-1}$. The subalgebra $F_{\alpha,0}$ fixed by the involution on $F_\alpha$ is a totally real field over $F_0$. We have an isomorphism $F_\alpha \cong F_{\alpha,0} \otimes_{F_0} F $. 
	
	\begin{definition}
		\emph{The modified self product of $\CM$} is the blow up of $\CM \times_{O_E[1/\Delta]} \CM$ along union of Weil divisors 
		$\CM^{\circ}_{k_{\nu}} \times_{k_{\nu}} \CM^{\circ}_{k_{\nu}}$ for all finite places $\nu$ of $E$ over $v$ such that $L_v$ is not self-dual:
		$$
		\wt{\CM \times_{O_E[1/\Delta]} \CM} \to \CM \times_{O_E[1/\Delta]} \CM. 
		$$
	\end{definition}
	
	\begin{proposition}
		The scheme $\wt{\CM \times_{O_E[1/\Delta]} \CM} \to \CM \times_{O_E[1/\Delta]} \CM$ is regular and of strictly semi-stable reduction over $O_E[1/\Delta]$ .
	\end{proposition}
	\begin{proof}
		We have assumed that $E_\nu$ is unramified over $\BQ_p$. This follows from Theorem \ref{proof: global semi-stable} and computations on completed local rings.
	\end{proof}

	\begin{proposition}\label{lem: global str=tot graph}
		The Hecke correspondence $\Hk_{\mu_{G, \Delta}}$ lifts naturally along $\wt{ \CM \times_{O_E[1/\Delta]} \CM} \to \CM \times_{O_E[1/\Delta]} \CM$ via strict transforms.
	\end{proposition}
	\begin{proof}
		This is a global analog of Proposition \ref{lem: local str=tot graph}. We only put the level structure $\mu_{G, \Delta}$ on $\Delta$, so the first projection $pr_1: \Hk_{\mu_{G, \Delta}} \to \CM$ is smooth and $\Hk_{\mu_{G, \Delta}}$ is regular. 
		
		From moduli interpretations, the pullback of $\CM^{\circ}_{k_{\nu}} \times_{k_{\nu}} \CM^{\circ}_{k_{\nu}}$ to $\Hk_{\mu_{G, \Delta}}$ is the preimage of $\CM^{\circ}_{k_\nu}$ under $pr_1$, hence a divisor in $\Hk_{\mu_{G, \Delta}}$. Due to regularity of $\Hk_{\mu_{G, \Delta}}$, $pr_1^{-1} (\CM^{\circ}_{k_\nu})$ is a Cartier divisor hence the strict transform of $\Hk_{\mu_{G, \Delta}}$ is isomorphic to $\Hk_{\mu_{G, \Delta}}$.
	\end{proof}
	
	In particular, we get a natural lifting of the diagonal morphism
	\begin{equation}
		\wt \Delta: \CM \to \wt{\CM \times_{O_E[1/\Delta]} \CM}.
	\end{equation} 

We now generalize the construction of derived CM cycles \cite[\S 5.1]{AFL2021} to our parahoric levels.
	
	\begin{definition}
		The \emph{derived fixed point locus of the Hecke correspondence} $\Hk_{\mu_{G, \Delta}}$ is the derived tensor product  
		\[ {}^{\BL} \Fix(\mu_{G, \Delta}):= \CO_{\Hk_{\mu_{G, \Delta}} } \otimes^{\BL}_{\wt{\CM \times_{O_E[1/\Delta]} \CM}} \CO_\CM,   \]
		viewed as an element in the $K$-group of coherent sheaves on $\CM$ with support on the image of $\Hk_{\mu_{G, \Delta}}$.
		 The \emph{modified derived Hecke CM cycle} $${}^{\BL}\mathcal{CM}(\alpha, \mu_{G, \Delta}) \in K_0'(\mathcal{CM}(\alpha, \mu_{G, \Delta}))
		$$ is the restriction of ${}^{\BL} \Fix(\mu_{G, \Delta})$ to  $K_0'(\mathcal{CM}(\alpha, \mu_{ \Delta}))$ under the decomposition (\ref{eq: Hk deomp}).
	\end{definition}
	
	By \cite[Apend. B.3]{AFL} and regularity of the modified self product of $\CM$, ${}^{\BL}\mathcal{CM}(\alpha, \mu_{G, \Delta})$ is a virtual $1$-cycle in the sense that it lies in the subgroup $F_1K_0'(\mathcal{CM}(\alpha, \mu_{G, \Delta}))$ generated by coherent sheaves supported on $1$-dimensional subschemes of $\mathcal{CM}(\alpha, \mu_{G, \Delta})$.
	
	\begin{remark}
		By the theory of complex multiplication, the generic fiber $CM(\alpha, \mu_{G, \Delta})$ of $\mathcal{CM}(\alpha, \mu_{G, \Delta})$ is $0$-dimensional over $E$. Moreover, $\mathcal{CM}(\alpha, \mu_{G, \Delta})$ itself is $1$-dimensional away from finitely many primes. Hence ${}^{\BL}\mathcal{CM}(\alpha, \mu_{G, \Delta})_E=CM(\alpha, \mu_{G, \Delta})$ by \cite[Apend. B.2]{AFL}. 
	\end{remark}

	 Let $\mathcal C_1(\CM)$ be the group of $1$-cycles on $\CM$ (with $\BQ$-coefficient), and $\mathcal C_1(\CM)_{\sim}$ be the quotient of $\mathcal C_1(\CM)$ by the subgroup generated by $1$-cycles on $\CM$ that are contained in a closed fiber and rationally trivial within that fiber. For a coherent sheaf $\CF$ on $\mathcal{CM}(\alpha, \mu_{G, \Delta})$ with at most $1$-dimensional support $C$. Let $\eta_1, \hdots, \eta_k$ be the generic points of the $1$-dimensional irreducible components of $C$. Let $C_i$ denote the closure of $\eta_i$ with reduced scheme structure. Then  we have a natural map 
\begin{equation}\label{from K to 1-cycle}
	 F_1K_0'(\mathcal{CM}(\alpha, \mu_{G, \Delta})) \to \CC_1(\CM)_{\sim} :  \, \, \CF \to \sum_{i=1}^k \length_{\CO_{C, \eta_i}} \CF_{\eta_i}  [C_i].
\end{equation}
	 We still denote by ${}^{\BL}\mathcal{CM}(\alpha, \mu_{G, \Delta}) \in \mathcal C_1(\CM)_{\sim}$ the image of ${}^{\BL}\mathcal{CM}(\alpha, \mu_{G, \Delta})$ under the above natural map (\ref{from K to 1-cycle}). 
	
		Consider the space $\CS(K_{G,\Delta} \backslash G(F_{0,\Delta})/ K_{G,\Delta})$  of bi-$K_{G,\Delta}$-invariant Schwartz functions on $G(F_{0,\Delta})$.     	
	
	\begin{definition}\label{defn: Derived CM}
		For $\phi_{CM, \Delta} \in \CS(K_{G,\Delta} \backslash G(F_{0,\Delta})/ K_{G,\Delta})$, the $\phi_{CM, \Delta}$-averaged version of modified derived Hecke CM cycle is the finite summation
		\begin{equation}
			{}^{\BL}\mathcal{CM}(\alpha, \phi_{CM})= \sum_{\mu_{G, \Delta} \in K_{G,\Delta} \backslash G(F_{0,\Delta}) / K_{G,\Delta} } \phi_{CM,\Delta}(\mu_{G, \Delta})
			{}^{\BL}\mathcal{CM}(\alpha, \mu_{G, \Delta}) \in \CC_1(\CM)_{\sim}.
		\end{equation}
		
		Here $\phi_{CM}:=\phi^{\Delta}_{CM} \otimes \phi_{CM, \Delta}$ with $\phi^{\Delta}_{CM}:=1_{K_G^\Delta}$.
	\end{definition}

	\subsection{Basic uniformization and local-global compatibility}\label{section: basic uni local-global} 
	
	In the rest of this section, we establish basic uniformizations of special cycles at inert places of $F_0$, see also \cite[Thm. 8.15]{RSZ-Diagonalcycle}. We show local-global compatibility under blow ups and relate (modified) local and global intersection numbers. 
	
	Let $\nu$ be a finite place of $E$ with $\nu|_{F}=w_0$, $\nu|_{F_0}=v_0$. Assume that $v_0 \not \in \Delta$ is inert in $F$. Recall that if $0<t_{v_0}<n$  we assume that $v_0$ is unramified over $p$. 
	
	 Let $V^{(v_0)}$ be the nearby $F/F_0$-hermitian space of $V$ at $v_0$, which is positively definite at all $v| \infty$, and locally isomorphic to $V$ at all finite places $v \not = v_0$ of $F_0$. 
	
	Consider reductive groups $G^{(v_0)}:=\U(V^{(v_0)})$ over $F_0$ and $\wt{G}^{(v_0)}:=Z^{\BQ} \times \Res_{F_0/\BQ} \U(V^{(v_0)})$ over $\BQ$.  Consider the RSZ Shimura variety 
	$$
	\CM \to \Spec O_E[1/\Delta]
	$$ 
	associated to $L$ in \S \ref{section: global Shimura var}. Denote by $\CM^{\wedge,  \basic}_{O_{\breve{E}_{\nu}}}$ the formal completion of $\CM_{O_{\breve{E}_{\nu}}}$ along the basic locus of the geometric fiber $\CM \otimes \BF_{\nu}$.  The formal scheme $\CM^{\wedge,  \basic}_{O_{\breve{E}_{\nu}}}$ is regular and \emph{formally finite type} over $O_{\breve{E}_{\nu}}$.
	
	Consider the (relative) unitary Rapoport--Zink space 
	$$\CN \to \Spf O_{\breve{F}}$$ associated to $L_{v_0}$ in \S \ref{section:local RZ}. Note we use Kottwitz signature condition $(1, n-1)$ in the Part $2$ and $(n-1,1)$ in Part $3$. In this section we identify our Rapoport--Zink spaces of signature $(1, n-1)$ to $(n-1, 1)$ by chang the $O_{F_{v_0}}$ action with the Galois conjugation in $\Gal(F_{v_0}/F_{0v_0})$. We have the following basic uniformization theorem.

	\begin{proposition}\label{prop: basic uniformization RSZ}
		There is a natural isomorphism of formal schemes over $O_{\breve{E_\nu}}$:
		\begin{equation}\label{eq: basic basic unif}
			\Theta: \, \, \CM^{\wedge,  \basic}_{O_{\breve{E}_{\nu}}} \stackrel{\sim}{\rightarrow}  \wt{G}^{(v_0)}(\mathbb{Q}) \backslash ( \mathcal{N}^{\prime} \times \wt{G} (\mathbb{A}_{f}^{p}) / K_{\wt{G}}^{p} ), 
		\end{equation}  
		where 
		\[ \mathcal{N}^{\prime} \simeq (Z^{\BQ} \left(\mathbb{Q}_{p}\right) / K_{Z^{\BQ}, p} ) \times \mathcal{N}_{O_{\breve{E}_{\nu}}} \times \prod_{v|p, v \not = v_0} \U(V)(F_{0, v}) / K_{G, v}. \]
		And there is a natural projection map
		\begin{equation}\label{projection map from RSZ to unitary}
			\mathcal{M}^{\wedge,  \basic}_{O_{\breve{E}_{\nu}}}  
			\to Z^{\BQ}(\BQ)  \backslash (Z^{\BQ}(\BA_{0, f}) / K_{Z^\BQ} ) 
		\end{equation}  
		with fibers isomorphic to
		\[ \Theta_0: \mathcal{M}^{\wedge,  \basic}_{O_{\breve{E}_{\nu}}, 0} \stackrel{\sim}{\rightarrow}  G^{(v_0)}(F_0) \backslash [ \mathcal{N}_{O_{\breve{E}_{\nu}}}  \times G (\mathbb{A}_{0, f}^{v_0}) / K_{G}^{v_0} ]. \] 
	\end{proposition}
	Moreover, the isomorphism descends to $O_{E_\nu}$ when we take the descent of $\CN$ to $\Spf O_{F_{v_0}}$ as in Remark \ref{descent to O_F: moduli}.
	\begin{proof}
		As $E_{\nu}$ is unramified over $\BQ_p$, we have the comparison between absolute and relative Rapoport--Zink spaces \cite[Prop. 3.30]{Cho18}. The result follows from \cite[Thm. 4.3]{Cho18}.
	\end{proof}

	Hence we denote by $\mathcal{M}^{\wedge,  \basic}_{O_{\breve{E}_{\nu}}, 0}$ any fiber of $\mathcal{M}^{\wedge,  \basic}_{O_{\breve{E}_{\nu}}} $ under the projection map (\ref{projection map from RSZ to unitary}).
	
	\subsection{Basic uniformization for Kudla--Rapoport cycles}
	
	In Definition \ref{defn: global int KR}, we defined Kudla--Rapoport cycles on $\CM$ for suitable $\phi_i \in \CS(V(\BA_{0,f}))^{K_G}$ ($i=1, 2$). Denote by $\mathcal{Z}(\xi, \phi_i)_0^\wedge$ any fiber of the formal completion $\mathcal{Z}(\xi, \phi_i)^\wedge$ under the projection map (\ref{projection map from RSZ to unitary}). 
	
	\begin{proposition}\label{prop:basic unif KR}
		For any $\xi \in F_{0, +}$, we have
		\[ 
		\mathcal{Z}(\xi, \phi_1)_0^\wedge= \sum_{(u, g) \in G^{(v_0)}(F_0) \backslash  V^{(v_0)}_{\xi}(F_0) \times G (\mathbb{A}_{0, f}^{v_0}) / K_{G}^{v_0}  }  \phi_1^{v_0}(g^{-1}u) [\CZ(u, g)]_{K_G^{v_0}} , 
		\]
		\[ 
		\mathcal{Z}(\xi, \phi_2)_0^\wedge= \sum_{(u, g) \in G^{(v_0)}(F_0) \backslash  V^{(v_0)}_{\xi}(F_0) \times G (\mathbb{A}_{0, f}^{v_0}) / K_{G}^{v_0}  }  \phi_2^{v_0}(g^{-1}u) [\CY(u, g)]_{K_G^{v_0}} .
		\]
		Here $[\CZ(u, g)]_{K_G^{v_0}}$ is the descent of 
		$$ \sum_{(u',g') \in G^{(v_0)}(F_0).(u, g) \subseteq V^{(v_0)}(F_0) \times G (\mathbb{A}_{0, f}^{v_0}) / K_{G}^{v_0}} \CZ(u') \times 1_{g'K^{v_0}_G}$$ 
		to $\mathcal{M}^{\wedge,  \basic}_{O_{\breve{E}_{\nu}}, 0}$. And $[\CY(u, g)]_{K_G^{v_0}}$ is the descent of $\sum  \CY(u') \times 1_{g'K^{v_0}_G}$ over the same index set.
		
	\end{proposition}
	\begin{proof}
		This follows from checking moduli definitions on both sides using the uniformization map $\Theta$ constructed in \cite[Thm. 4.3]{Cho18}.
	\end{proof}

	\subsection{Comparison of Balloon--Ground stratification}
	
	Return to the local set up in \S \ref{Formal Balloon-Ground} and consider
	$$  
	\wt{\CN \times_{O_{\breve{F}_{v_0}}} \CN} \to \CN \times_{O_{\breve{F}_{v_0}}} \CN
	$$ the formal blow up morphism along the Weil divisor  
	$\CN^{\circ} \times_{k_{v_0}} \CN^{\circ}$.

	Let $	(\wt{\CM_{O_{\breve{E}_{\nu}}} \times \CM_{ O_{\breve{E}_{\nu}}} })^{\wedge}$ be the completion of $\wt{\CM_{O_{\breve{E}_{\nu}}} \times \CM_{ O_{\breve{E}_{\nu}}} }$ along the strict transform of the basic locus in the special fiber of $\CM_{O_{\breve{E}_{\nu}}} \times \CM_{ O_{\breve{E}_{\nu}}}$. Let $	(\wt{\CM_{O_{\breve{E}_{\nu}},0} \times \CM_{ O_{\breve{E}_{\nu}},0} })^{\wedge}$ be any fiber of $(\wt{\CM_{O_{\breve{E}_{\nu}}} \times \CM_{ O_{\breve{E}_{\nu}}} })^{\wedge}$ under the $2$-fold product of the projection map (\ref{projection map from RSZ to unitary}). 
		
	\begin{proposition}\label{prop: comp of balloon-ground}
		The uniformization map 
		$$ \Theta_0: \CM^{\wedge,  \basic}_{O_{\breve{E}_{\nu}}, 0} \stackrel{\sim}{\rightarrow}  G^{(v_0)}(F_0) \backslash [ \mathcal{N}_{O_{\breve{E}_{\nu}}}  \times G (\mathbb{A}_{0, f}^{v_0}) / K_{G}^{v_0} ] 
		$$
		induces an isomorphism
		\[
		(\wt{\CM_{O_{\breve{E}_{\nu}},0} \times \CM_{ O_{\breve{E}_{\nu}},0} })^{\wedge}
		\stackrel{\sim}{\rightarrow}  (G^{(v_0)}(F_0) \times G^{(v_0)}(F_0) )\backslash [ \wt{ \CN_{O_{\breve{E}_{\nu}}} \times \CN_{O_{\breve{E}_{\nu}}}}  \times (G (\mathbb{A}_{0, f}^{v_0}) / K_{G}^{v_0} \times G (\mathbb{A}_{0, f}^{v_0}) / K_{G}^{v_0} )  ].
		\]
	\end{proposition}

	\begin{proof}
		Blow up commutes with flat base change $O_{\breve{E}_\nu} \to O_{E_\nu}$ so we can work over $O_{\breve{E}_\nu}$. From Proposition \ref{Balloon-Ground: comp with LTXZZ} and the definition of $\Theta_0$, 
		the base change of the Balloon stratum $\CM_{k_{E_\nu}, 0}^\circ$ along the flat morphism
		\begin{equation}
			G^{(v_0)}(F_0) \backslash \mathcal{N}_{O_{\breve{E}_{\nu}}}  \times G (\mathbb{A}_{0, f}^{v_0}) / K_{G}^{v_0} \overset{\Theta^{-1}}{\cong} \CM^{\wedge,  \basic}_{O_{\breve{E}_{\nu}}, 0} \to \CM_{O_{\breve{E}_{\nu}}, 0} \to \CM_{O_{E_\nu}, 0}
		\end{equation} 
		is equal to the formal subscheme $G^{(v_0)}(F_0)\backslash \CN_{O_{\breve{E}_{\nu}}}^\circ \times G (\mathbb{A}_{0, f}^{v_0}) / K_{G}^{v_0}$. The result follows. 
	\end{proof}

	\subsection{Basic uniformization for Hecke CM cycles}
	
	Let $\Hk_{\mu_{G, \Delta}}^{\wedge}$ be the locally noetherian formal scheme over $O_{\breve{E}_{\nu}}$ obtained as pullback of $\Hk_{\mu_{G, \Delta}}$ along the map 
	\[ 
	\CM^{\wedge}_{O_{\breve{E}_{\nu}}} \times_{O_E[\Delta^{-1}]} \CM^{\wedge}_{O_{\breve{E}_{\nu}}} \to \CM \times_{O_E[\Delta^{-1}]} \CM .
	\]
	
	There is a natural projection 
	\begin{equation}
		\Hk_{\mu_{G, \Delta}} \to Z^{\BQ}(\BQ)  \backslash Z^{\BQ}(\BA_{0, f}) / K_{Z^\BQ} .
	\end{equation}
	Denote by $\Hk_{\mu_{G, \Delta},0}^{\wedge}$ any fiber of $\Hk_{\mu_{G, \Delta}}^{\wedge}$ of the projection. Consider the discrete set 
	\[
	\Hk_{\mu_{G, \Delta},0}^{(v_0)}= \{ (g_1,g_2) \in G(\BA_f^{v_0})/K_G^{v_0} \times G(\BA_f^{v_0})/K_G^{v_0} | g_1^{-1}g_2 \in K_G\mu_{G, \Delta}K_G \}. 
	\]
	
	\begin{proposition}
		We have an isomorphism:
		\[ 
		\Hk_{\mu_{G, \Delta},0}^{\wedge} \cong G^{(v_0)}(F_0) \backslash [ \CN_{O_{\breve{E}_{\nu}}} \times \Hk_{\mu_{G, \Delta},0}^{(v_0)}]
		\]
		compatible with the projection under the isomorphism $\Theta$ (\ref{eq: basic basic unif}). 
		
	\end{proposition}
	\begin{proof}
		This follows from the proof of \cite[Prop. 7.16]{AFL}.
	\end{proof}

	For $(\delta, h) \in G^{(v_0)}(F_0) \times G(\BA_{0,f}^{v_0}) / K_G^{v_0} $, consider the Hecke CM cycle
	\begin{equation}
		\mathcal{CM} (\delta, h)_{K_G^{v_0}} := [\CN^{\delta=id}_{O_{\breve{E}_{\nu}}} \times 1_{hK_G^{v_0}}] \rightarrow [ \mathcal{N}_{O_{\breve{E}_{\nu}}}  \times G (\mathbb{A}_{0, f}^{v_0}) / K_{G}^{v_0}].
	\end{equation}
	
	The summation 
	$$\sum_{(\delta', h')} \mathcal{CM} (\delta, h)_{K_G^{v_0}} $$ over the $G^{(v_0)}(F_0)$-orbit of $(\delta, h)$ in $G^{(v_0)}(F_0) \times G(\BA_{0,f}^{v_0}) / K_G^{v_0} $ descends to a cycle $[\mathcal{CM} (\delta, h) ]_{K_G^{v_0}}$ on the quotient $\mathcal{M}^{\wedge,  \basic}_{O_{\breve{E}_{\nu}}, 0}$. Consider $\varphi_{CM}=\varphi_{CM}^{\Delta} \otimes \phi_{CM, \Delta} \in \CS(\U(V)(\BA_{0, f}), K)$ where $\varphi_{CM}^{\Delta}=1_{K_G^{\Delta}}$ (including $v_0$).  As in the proof of \cite[Prop. 7.17]{AFL}, we have
	
	\begin{corollary}
		The restriction of the formal completion $\mathcal{CM}(\alpha, \phi_{CM})^{\wedge}$ to any fiber of the projection $\Theta$ (\ref{eq: basic basic unif}) is the sum
		\begin{equation}\label{basic naive Hk}
			\sum_{(\delta, h) \in G^{(v_0)}(F_0) \backslash G^{(v_0)}(\alpha)(F_0)  \times G(\BA_{0,f}^{v_0})/K_G^{v_0} } \phi_{CM}^{v_0}(h^{-1}\delta h) [\mathcal{CM} (\delta, h) ]_{K_G^{v_0}}.
		\end{equation}
		Here $G^{(v_0)}(\alpha)(F_0) \subseteq G^{(v_0)}(F_0)$ consist of elements with characteristic polynomial $\alpha$.
	\end{corollary}
	
	Note $\CN^{\delta=id}=\Fix(\delta)$ is introduced in Definition \ref{def: Mod(g)}. Replace $\Fix(\delta)$ by $\wt {\Fix}^\BL  (\delta)$ in the definition, we get a virtual $1$-cycle $[{}^\BL\mathcal{CM} (\delta, h) ]_{K_G^{v_0}}$. By Proposition \ref{prop: comp of balloon-ground} we get 
	
	\begin{corollary}
		The restriction of the basic uniformization ${}^\BL \mathcal{CM}(\alpha, \phi_{CM})^{\wedge}$ to any fiber of the projection $\Theta$ (\ref{eq: basic basic unif}) is the sum
		\begin{equation}\label{basic derived Hk}
			\sum_{(\delta, h) \in  G^{(v_0)}(F_0) \backslash  G^{(v_0)}(\alpha)(F_0)  \times G(\BA_{0,f}^{v_0})/K_G^{v_0} } \phi_{CM}^{v_0}(h^{-1}\delta h) [{}^\BL\mathcal{CM} (\delta, h) ]_{K_G^{v_0}}.
		\end{equation}
	\end{corollary}

	\section{Modification over $\BC$ and $\BF_q$}\label{section: mod over C and F}	
	
	In this section, we use basic geometric results on the generic fiber and special fibers of our Shimura varieties to understand $1$-cycles on our integral model by modification. Firstly, we show that the degree function of any Hecke CM cycle $CM(\alpha, \mu) \to E$ is equidistributed on geometric connected components of $M=M_{\wt{G}, \wt{K}}$ under mild assumptions. This shows we can use any $1$-cycle on $M$ with non-empty complex points to modify general $1$-cycle to a cycle of degree zero over $\BC$.	 Secondly, we use Balloon--Ground stratification and basic uniformization to approximate $1$-cycles on non-smooth special fibers by very special $1$-cycles in the basic locus.
	
	\subsection{Modification over $\BC$}\label{section: equidistribution} 
	
	We firstly recall the Hecke and Galois actions on connected compoenents of complex Shimura varieties. Let $(G,X)$ be a Shimura datum,  $G^{\ab}=G/G^{\text{der}}$ be the maximal abelian quotient  of $G$ which is a torus over $\BQ$, and $Z$ be the center of $G$. 
	
	For any neat compact open subgroup $K \leq G(\mathbb A_f)$, consider the complex uniformization: 
	$$ 
	\Sh_K(G,X)(\mathbb C)=G(\BQ) \backslash X \times G(\mathbb A_f) / K.
	$$	
	
	Assume that $G^{\text{der}}$ is simply connected. 	
	
	\begin{proposition}\label{prop: geom conn comp of Sh}
		The natural projection $\det_G: G \to G^{\ab}$ induces an isomorphism of finite sets
		\begin{equation}
			\pi_0(\Sh_K(G,X)(\BC)) \cong  G^{\ab}(\mathbb Q)^{\dagger} \backslash G^{\ab}(\BA_f) / \mathrm{det}_G (K), 
		\end{equation}
		where $G^{\ab}(\mathbb Q)^{\dagger}=G^{\ab}(\BQ) \cap \Im(\det_G |_{Z}: Z(\BR) \to G^{\ab}(\BR))$.
	\end{proposition}
	
	\begin{proof}
		This follows from the strong approximation theorem for non-compact type, simply connected semisimple algebraic groups over number fields, see \cite[Thm. 5.17]{Milne}. 
	\end{proof}
	
	\begin{remark}
		The Hecke action of $G(\BA_f)$ on the tower $\{\Sh_K(G,X)\}_K$ is compatible with the projection $\det_G: G(\BA) \to G^{\ab}(\BA)$, hence is transitive on $\lim_K \pi_0(\Sh_K(G,X)(\BC))$.
	\end{remark}	
	
	Let $E$ be the reflex field of $(G, X)$. The conjugacy class of the Hodge cocharacter $\mu_h: \BG_{m,\BC} \to G_{\BC}$ is defined on $E$. The composition 
	$
	\det_G \circ \mu_h : \BG_{m, \BC} \to G^\ab_{\BC}
	$
	 is independent of the choice of $\mu_h$ in the conjugacy class $\{\mu_h\}$ and is defined over $E$. We denote by 
\begin{equation} \label{reflex hom over E}
	r_{G, E} : (\mathbb G_m)_E \to G^{\ab}_E
\end{equation}
the descent of $\det_G \circ \mu_h$ to $E$. 

We have the norm map $\Nm_{E/\BQ}: \Res_{E/\BQ} (G^\ab_E) \to G^\ab$ (from $E$ to $\BQ$).  \textit{The reflex norm map} is defined as the composite map
\begin{equation}\label{reflex norm map}
	r_G=\Nm_{E/\BQ} \circ \Res_{E / \BQ}r_{G, E}: \Res_{E /\mathbb Q} \mathbb G_m \to \Res_{E/\BQ}G^\ab \to G^{\ab}.
\end{equation}

	We describe the Galois action on $\pi_0(\Sh_K(G,X)(\mathbb C))$ via the theory of Shimura reciprocity law as in \cite[Section 12-13]{Milne}. By global class field theory, there is a continuous surjective homomorphism $\Art_E: \Res_{E/\BQ} \BG_m (\BA) \to \Gal_E$. By \cite[p. 119]{Milne} and Proposition \ref{prop: geom conn comp of Sh}, we have 
	
	\begin{proposition}\label{Gal on pi}	
		For any Galois automorphism $\sigma \in \Gal_E$, choose some $s=(s_\infty, s_f) \in \mathbb A_E^{\times} = \Res_{E/ \mathbb Q} \mathbb G_m (\mathbb A)$ such that $\Art_E(s)=\sigma|_{E^{ab}}$ such that $s_\infty$  lies in the identity connected component of $(E \otimes_\BQ \BR)^\times$. Then $\sigma$ acts on the finite abelian group $\pi_0(\Sh_K(G,X)(\BC)) \cong  G^{\ab}(\mathbb Q)^{\dagger} \backslash G^{\ab}(\BA_f) / \mathrm{det}_G(K)$
		via multiplication with $r_G(s_f)$:
		\begin{equation}
			\sigma[t]= [r_G(s_{f}).t] .
		\end{equation}
	\end{proposition}
	
	\begin{remark}
		We can compute connected components of PEL type Shimura varieties using moduli interpretation. For the Siegel modular variety $\CA_{g,N}$ with level structure $\Gamma(N)$, we have $\pi_0(\CA_{g, N}(\BC)) \cong (\BZ/N\BZ)^\times$: given $(A, \lambda) \in \CA_{g, N}(\BC)$, its image in $(\BZ/N\BZ)^\times$ is the ratio of a fixed perfect pairing on $(\BZ/N\BZ)^{2g}$ and the Weil pairing on $A[N]$ under $\eta: A[N] \cong (\BZ/N\BZ)^{2g}$. 
	\end{remark}
	
	Return to our set up. We have $\wt{G}=Z^{\BQ} \times \Res_{F_0/\BQ}G$ where $G=\U(V)$ is a reductive group over $F_0$. Denote by $M_{\wt{G}, \wt{K}}$ (resp.  $M_{G,K}$) the canonical model of $\Sh_{\wt{K}}( \wt{G}, \{ h_{\wt{G}} \})$ (resp. $\Sh_{K}( \Res_{F_0/\BQ}G, \{ h_{\Res_{F_0/\BQ}G} \})$) over the reflex field $E$ (resp. $\varphi_0(F) \subseteq \ov \BQ$). We consider the abelianization 
	\[
	T^1:=(\Res_{F_0/ \BQ}G)^\ab=\Res_{F_0/ \BQ} \mathbb G_{m}^{\Nm_{F/F_0}=1}.
	\]
	As $T^1(\BR)$ is connected, we have $T^1(\BQ)^{\dagger}=T^1(\BQ)$.	Consider the reflex norm map (\ref{reflex norm map}) for the Shimura datum $(\Res_{F_0/\BQ}G, \{ h_{\Res_{F_0/\BQ}G} \})$:
	\[
	r_{\Res_{F_0/\BQ}G}: \Res_{F/\BQ} \BG_m \to T^1.
	\]
	
	\begin{lemma}
The reflex norm map for the Shimura datum $(\Res_{F_0/\BQ}G, \{ h_{\Res_{F_0/\BQ}G} \})$ could be identified with the map
		\begin{equation}\label{eq: reflex norm}
			r_{\Res_{F_0/\BQ}G}: \Res_{ F/\BQ} \BG_m \to \Res_{F_0/ \mathbb Q} \mathbb G_{m}^{\Nm_{F/F_0}=1}, \quad z \mapsto \bar{z}/z.
		\end{equation}
	\end{lemma}
	\begin{proof}
This is computed in \cite[Section 2.4-2.5]{Boumasmoud} via the equivalence between tori over a number field and $\BZ$-lattices with actions of the Galois groups of the number field, given by sending a torus $T$ to the cocharacter lattice $X_*(T_{\BC} ) =\Hom( {\BG_m}_{\BC}, T_{\BC}) $. 
\end{proof}
	
	\begin{corollary}\label{Gal_F stable on M_{G,K}}
		The Galois action of $\Gal_F$ on $\pi_0(M_{G,K}(\BC)) $ is transitive. In other words, $M_{G,K}$ is connected over $F$.
	\end{corollary}
	\begin{proof}
		By above lemma and Hilbert 90, the reflex norm map $r_{\Res_{F_0/\BQ}G}$ induces a surjection $\Res_{F/\BQ} \BG_m(\BA_f) \surj T^1(\BA_f)$. We conclude by Proposition \ref{Gal on pi}.
	\end{proof}

	\subsection{Complex uniformization of Hecke CM cycles}
	Denote by $X$ the Grassmannian of negative definite $\mathbb C$-lines in $V \otimes_{F,\varphi_0} \BC$. Consider any neat compact open subgroup $\wt{K}=K_{Z^{\BQ}} \times K$.  
	We have the following complex uniformization \cite[Prop. 3.7]{RSZ-Diagonalcycle}  of $M_{\wt{G}, \wt {K}}$ along  the complex  embedding $E \to \BC$ induced by $\varphi_0$.
	
	\begin{proposition}\label{RSZ complex uniformization} 
		There is an isomorphism between  complex algebraic varieties over $\BC$:
		\begin{equation}
			M_{\wt{G}, \wt {K}} \otimes_{E, \varphi_0} \mathbb C \cong \Sh_{K_{Z^{\mathbb Q}}}(Z^{\mathbb Q}, \{h_{ Z^{\mathbb Q}}\} )_{\BC} \times M_{G,K} \otimes_{F, \varphi_0}{\BC}. 
		\end{equation}
		Each fiber of the projection $ M_{\wt G, \wt K}(\BC) \to Z^{\BQ}(\BQ)  \backslash  Z^{\BQ}(\BA_{f}) / K_{Z^\BQ}$ is isomorphic to 
		$$M_{G, K}(\BC)= \U(V)(F_0) \backslash X \times \U(V)(\BA_{0, f}) / K_G.$$
	\end{proposition}

	Let $\mu_{G, \Delta} \in K_{G, \Delta}\backslash G(F_{0, \Delta})/K_{G, \Delta}$ and $\alpha \in F [t]$ be a degree $n$ conjugate self-reciprocal monic polynomial (c.f. \ref{conjuga self-reciprocal}). Assume $\alpha$ is irreducible.    For any $(\delta, h) \in \U(V)(\alpha)(F_0) \times \U(V)(\mathbb A_{0,f})/K_G$, the cycle $[CM(\alpha, h)]_{K_G} \to M_{G, K}(\BC)$ is the descent of the summation
	\[
	\sum_{(\delta', h')} X^{\delta'=id} \times 1_{g'K_G}
	\]
	where $(\delta', h')$ runs over the $G(F_0)$-orbit of $(\delta, h)$. Similar to \cite[Prop. 7.17]{AFL}, the complex points of Hecke CM cycles have the following description.
	
	\begin{proposition}\label{prop: complex unif for CM cycles}
		The restriction of $CM(\alpha, \mu_{G, \Delta})(\BC)$ on each fiber of the projection $M_{\wt G, \wt K}(\BC) \to Z^{\BQ}(\BQ)  \backslash Z^{\BQ}(\BA_{f}) / K_{Z^\BQ} $ is the disjoint union
		\[ \coprod [CM(\alpha, h)]_{K_G} \rightarrow M_{G, K}(\BC) \]
		where the index runs over the set 
		\[
		A_{\mu_{G,\Delta}}:=\{(\delta, h) \in \U(V)(F_0) \backslash ( \U(V)(\alpha)(F_0) \times \U(V)(\mathbb A_{0,f})/K_G | h^{-1}\delta h \in K_G \mu_{G, \Delta} K_G \}.
		\]
	\end{proposition}

	\subsection{Hecke CM cycles are Galois stable}
	
	Denote by $\wt{\deg}_{C'}$ the degree function of $C'=CM(\alpha, \mu_{G, \Delta})$ on 
	$$ \pi_0(M_{\wt G, \wt K}(\BC)) \cong  Z^\BQ(\BQ) \backslash Z^\BQ(\BA_{f}) / K_{Z^\BQ} \times \pi_0(M_{G, K}(\BC)).$$
	
	The Hecke action of $Z^{\BQ}(\mathbb A_f)$ on $M_{\wt G, \wt K}(\BC)$  preserves $CM(\alpha, \mu)$ by moduli interpretation, and induces a transitive action on $Z^\BQ(\BQ) \backslash Z^\BQ(\BA_{f}) / K_{Z^\BQ}$. Hence
	$$ 
	\wt{\deg}_{C'}=  c \times \deg_{C'}
	$$
	where $c$ is a constant and $\deg$ is a function on $\pi_0(M_{G, K}(\BC))$.

	Note the degree function 
	$\wt{\deg}_{C'}$ is $\Gal_E$-invariant because $CM(\alpha, \mu)$ is defined over $\Spec E$.

	\begin{corollary}\label{prop: equistribution when E=F}
		Assume that $E=F$. Then $\wt{\deg}_{C'}$ is a constant function on $\pi_0(M_{\wt G, \wt K}(\BC))$.
	\end{corollary}
	\begin{proof}
		As $E=F$ the degree function $\wt{\deg}_{C'}= c \times \deg_{C'}$ is $\Gal_F$-invariant. By Corollary \ref{Gal_F stable on M_{G,K}}, $\Gal_F$ acts transitively on $\pi_0(M_{G,K})$, this shows $\wt{\deg}_{C'}$ is a constant function.
	\end{proof}
	
	\begin{conjecture}\label{Important conj: eq over C}
		The degree $\deg_{C'}$ is constant on $\pi_0(M_{G, K}(\BC))$. 
	\end{conjecture}
	
	\begin{remark}
		By Proposition \ref{prop: complex unif for CM cycles}, the conjecture is equivalent to the image of the natural projection
		\[ \coprod_{(\delta, h)} X^{\delta=id} \times 1_{gK_G} \rightarrow \U(V)(F_0) \backslash X \times \U(V)(\BA_{0, f}) / K_G \]
		has constant degrees, where $(\delta, h)$ runs over $ \U(V)(\alpha)(F_0) \times \U(V)(\mathbb A_{0,f})/K_G$ with $g^{-1}\delta g \in K_G \mu_{G, \Delta} K_G $.
		Using the explicit description of Galois actions on special points, it might be possible to prove the conjecture for $E \not = F$ by firstly showing the image of the projection could be descent into a cycle over the canonical model $M_{G,K}$ over the reflex field $F$ and then apply the Galois action.
	\end{remark}

	\begin{proposition}\label{prop: mod over BC by max order Hecke CM}
		Assume that $E=F$. For any $0$-cycle $C \to M_{\wt G, \wt K}$ over $E$ stable under the Hecke action of $Z^{\BQ}(\mathbb A_f)$, there exists some $C'=\mathcal{CM}(\alpha, \mu_{G, \Delta})$ such that
		\[
		\wt{\deg}_C=a (\wt{\deg}_{C'})
		\]
		for some $a \in \BQ$. Moreover, we can assume $\alpha$ is maximal order at any finite place $v$ of $F_0$ outside $\Delta$, and is unramified at $v_0$.
	\end{proposition}     
	\begin{proof}
		
		The degree function $\wt{\deg}_C$ is $\Gal_F$-invariant and stable under Hekce action of $Z^{\BQ}(\mathbb A_f)$, hence also a constant. It suffices to show $C'=\mathcal{CM}(\alpha, \mu_{G, \Delta})$ can be chosen to be non-empty. This follows from the complex uniformization (\ref{prop: complex unif for CM cycles}) above.
		
	\end{proof}
	
	\begin{remark}
		As Kudla--Rapoport divisors are invariant under Hecke action of $Z^{\BQ}(\mathbb A_f)$, the assumption that $C$ is invariant under $Z^{\BQ}(\mathbb A_f)$ is harmless for the application to our modularity results.
	\end{remark}

	\subsection{Modification over $\BF_q$}\label{section: mod over BF}

	Let $v \not \in \Delta$ be a finite place of $F_0$ such that $L_v$ is not self-dual. So $L_{v}$ is a vertex lattice of type $0 < t_v \leq n$. The case $t_v=n$ is same to the case $t_v=0$ by duality, so assume $0<t_v<n$.  By assumption, $v$ is inert in $F$ and $E_\nu$ is unramified over $\BQ_p$. Hence $\CM \to \Spec O_{E_\nu}$ is of strictly semi-stable reduction by Theorem \ref{proof: global semi-stable}. 
	
	For a $1$-cycle $\CC$ on the special fiber $\CM_{k_{\nu}}$, consider its degree function on irreducible components $\CX_i$ of $\CM_{k_{\nu}}$ 
	which is obtained via the intersection pairing  $(-,-)_\CM$ between $\CX_i$ and $\CC$ inside the regular scheme $\CM$. Recall the Balloon--Ground stratification in \S \ref{section: global KR strata}:
	\[
	\CM_{k_\nu}=\CM_{k_\nu}^\circ \cup \CM_{k_\nu}^\bullet.
	\]	
	Note $\CM_{k_\nu}^\circ$ and $\CM_{k_\nu}^\bullet$ are closed smooth subvarieties of $\CM_{k_\nu}$ over $k_\nu$ of dimension $n-1$.  So irreducible components of $\CM_{k_\nu}$ are exactly connected components of the Balloon and Ground stratum:
	\[
	\CM_{k_\nu}= \coprod_{i=1}^{a} \CX_i^{\circ} \cup \coprod_{j=1}^{b} \CX_j^{\bullet}.
	\]
	Consider the basic uniformization in Proposition \ref{prop: basic uniformization RSZ}:
	\[
	\Theta_0: \mathcal{M}^{\wedge,  \basic}_{O_{\breve{E}_{\nu}}, 0} \stackrel{\sim}{\rightarrow}  G^{(v_0)}(F_0) \backslash [ \mathcal{N}_{O_{\breve{E}_{\nu}}}  \times G (\mathbb{A}_{0, f}^{v_0}) / K_{G}^{v_0} ]. 
	\]
	Here $\CM_{O_{\breve{E}_{\nu},0}}^{\wedge,\basic}$ is any fiber of the projection $\mathcal{M}^{\wedge,  \basic}_{O_{\breve{E}_{\nu}}}  
	\to Z^{\BQ}(\BQ)  \backslash (Z^{\BQ}(\BA_{0, f}) / K_{Z^\BQ} ) $.
	
	\begin{definition}\label{very special}
		Call a $1$-cycle $C'$ in $\CM_{k_{\nu}}^{\text{basic}}$ \emph{very special}, if via basic uniformization it can be generated from the projection of a cycle in an embedding $\CN_{L^\flat_v} \to \CN_{L_v}$ for a decomposition of vertex lattice $L_v=L^\flat_v \oplus M_v$ where $L^\flat_v$ has rank $2$ and type $1$.
	\end{definition}

	\subsection{The case $t_v=1$ or $n-1$}	
	
	\begin{proposition}
		If $t_v=1$, then $\CM_{k_\nu}^\circ$ is contained in the basic locus of $\CM_{k_{\nu}}$.
	\end{proposition}
	\begin{proof}
		This is \cite[Thm. 5.3.4]{LTXZZ}, whose proof works over any prime that is unramfied in $F$. It is done by the computation of slopes of associated Dieudonne modules and by the construction of basic isogenies. Alternatively, as the formal Balloon strata
		already has top dimension $n-1$ if $t_v=1$ via our Bruhat--Tits stratification (Theorem \ref{thm: moduli int of BT strata}), using the basic uniformization in (Proposition \ref{prop: basic uniformization RSZ}) and that any connected component of $\CM_{k_\nu}$ has a basic point (Proposition \ref{connect comp has basic point}), we see that $\CM_{k_\nu}^\circ$ must lie in the basic locus.
	\end{proof}
	
	\begin{proposition}
		Any irreducible component $\CX_i^{\circ}$ in $\CM_{k_\nu}^\circ$ is isomorphic to the $n-1$ dimensional projective space $\BP^{n-1}$. Moreover, $\CX_i^{\circ} \cap \CM^\bullet_{k_\nu}$ is isomorphic to the degree $q+1$ Fermat hypersurface in $\BP^{n-1}$.
	\end{proposition}	
	\begin{proof}
		As $\CM_{k_\nu}^\circ$ is in the basic locus, we can use basic uniformization (Proposition \ref{prop: basic uniformization RSZ}, which descends to $k_\nu$) to prove similar statements in the basic locus. Then we use explicit descriptions of Bruhat--Tits strata in Theorem \ref{thm: moduli int of BT strata} and Remark \ref{rek: Ballloon-Ground Bruat-Tits}. The irreducible components of $(\CN^\red)^\circ$ are exactly $\BT(L^\circ) \cong \BP^{n-1}$ where  $L^\circ$ is a vertex lattice of type $t_v-1=0$. The intersection of $\BT(L^\circ)$ with $(\CN^\red)^\bullet$ is given by the condition $B^\vee \subseteq \varpi^{-1} B$ which is the equation of Fermat hypersurface in $\BP^{n-1}$.
	\end{proof}			
	
	\begin{proposition}\label{Ground is irr t=1}
	Assume $n \geq 3$. In any given connected component of $\CM_{k_\nu}$, the Ground stratum is geometrically irreducible.
	\end{proposition}		
	\begin{proof}
		Otherwise, we can write $\CM^\bullet$ in this component as a disjoint union of two non-empty smooth varieties $X^\bullet$ and $Y^\bullet$. For any $\CX_i^{\circ}$ in this component, we know $\CX_i^{\circ} \cap \CM^\bullet_{k_\nu}$ is a Fermat hypersurface and in particular connected for $n-1 \geq 2$. This shows either $\CX_i^{\circ} \cap X^\bullet=\emptyset$ or $\CX_i^{\circ} \cap Y^\bullet=\emptyset$. This divides this connected component of $\CM_{k_\nu}$ into two non-empty disjoint unions, a contradiction! 
	\end{proof}
	
	Let $\Ch^1_{\nu}(\CM)$ be the subgroup of the Chow group $\Ch^1(\CM)$ generated by irreducible components $\CX_i$ of  $\CM_{k_{\nu}}$.
	
	\begin{proposition}\label{prop: mod over F for t=1}
		If $t_v=1$ or $n-1$, then for any linear functional $\ell: \Ch^1_{\nu}(\CM) \to \BR_\Delta$, there exists a very special $1$-cycle $C'$ in $\CM_{k_{\nu}}^{\text{basic}}$ such that 
		$$
		(\CX_i, C')_\CM=\ell(\CX_i)
		$$
		for any irreducible component $\CX_i$ of  $\CM_{k_{\nu}}$. In particular, the statement applies to $\ell=(-,\CC)_\CM$ for any $1$-cycle $\CC$ in $\CM$.
	\end{proposition}
	\begin{proof}
		We can work in any connected component of $\CM_{k_\nu}$, hence in some fiber $\CM_{O_{\breve{E}_{\nu}, 0}}$. We only need to deal with the case $t_v=1$, and the case $t_v=n-1$ follows by applying the dual isomorphism. 
		
		If $n=2$, then the Shimura curve $\CM_{k_\nu}$ has Drinfeld uniformization the union of two families of projective lines and any $1$-cycle is in $\CM_{k_{\nu}}^{\text{basic}}$ very special. Any linear functional $\ell$ can be regarded as a function on irreducible components of $\CM_{k_\nu}$. We are done by non-degeneracy of the intersection pairing (intersection matrix) between irreducible components of the curve $\CM_{k_\nu}$ (inside the semi-stable proper scheme $\CM_{O_{E_\nu}}$).
		
	 Now assume that $n \geq 3$. By Proposition \ref{Ground is irr t=1},  there is only one irreducible component $\CX^\bullet_{j_0}$ in $\CM_{k_\nu, 0}^\bullet$, and possibly many irreducible components $\CX^\circ_{i}$ ($i=1, \hdots, r$) in $\CM_{k_\nu, 0}^\circ$. Then $\ell$ is given by a collection of numbers $\{\ell_{j_0}=\ell(\CX^\bullet_{j_0}), \ell_1=\ell(\CX^\circ_{1}), \hdots, \ell_r= \ell(\CX^\circ_{r}) \}$ such that $\ell_{j_0}=-\sum_{i=1}^r \ell_i.$ It suffices to find for each $1 \leq i \leq r$ a very special cycle $C_i'$ such that 
		$$
		C_i' \subseteq \CX^\circ_{i}, \quad (\CX^\bullet_{j_0}, C_i')_\CM \in \BQ^\times.
		$$
		Then $C'=\sum_{i=1}^r \frac{\ell_i}{(\CX_i^\circ, C_i')} C'_i$ is what we want.  To find such $C_i'$, by applying basic uniformization we only need to find special $1$-cycle $C_i'$ in a given Bruhat-Tits stratum $\mathrm{BT}(L^\circ) \cong \BP^{n-1}$. Here $L^\circ$ is a vertex lattice in $V^{(v_0)}$ of type $t_v-1=0$. Choose an orthogonal decomposition of vertex lattices
		\[
		L^\circ_v=(L^\circ_v)^\flat \oplus M_v
		\]
		corresponding to an orthogonal decomposition $L_v=L_v^\flat \oplus M_v$ such that $(L^\circ_v)^\flat$ is a self-dual lattice of rank $2$.  Then we can choose
		$$
		C'_i= \CZ(M_v) = \mathrm{BT}(L^\circ)^\flat \cong \BP^1 \subseteq \CN_{L^\flat_v}.
		$$
	\end{proof}

	\subsection{The case $1<t_v<n-1$}	
	
	If $1<t_v<n-1$, the irreducible components $(\CN^\circ)^\red$ and $(\CN^\bullet)^\red$ are of dimension less than $n-1$ by Theorem \ref{thm: moduli int of BT strata}. So the Balloon stratum $\CM_{k_\nu}^\circ$  and Ground stratum $\CM_{k_\nu}^\bullet$ don't lie in the basic locus of $\CM_{k_{\nu}}$.
	
	\begin{proposition}\label{connect comp has basic point}
 For $0 \leq t_v \leq n$, any geometric connected component of $\CM_{k_{\nu}}$ contains a basic point $x \in \CM_{k_\nu}$. 
	\end{proposition}
	\begin{proof}
		This follows from \cite[Prop. 3.5.4]{IrredKR}, using the compatibility of the basic uniformization (\ref{prop: basic uniformization RSZ}) and the complex uniformization of geometric connected components (\ref{prop: geom conn comp of Sh}).  Firstly, we can work on any fiber $\CM_{O_{\breve{E_\nu}}, 0}$ of the projection $\CM_{O_{\breve{E_v}}} \to Z^{\BQ}(\BQ)  \backslash (Z^{\BQ}(\BA_{0, f}) / K_{Z^\BQ} )$. 
		
		As the scheme $\CM_{O_{\breve{E_v}},0}$ is proper and flat over $O_{E_\nu}$ with reduced special fiber, by Lemma \cite[Lemma 0E0N]{StacksProject} we have an identification of geometric connected components $\pi_0(\CM_{\ov{k_\nu},0}) \cong \pi_0(\CM_{\ov{E_\nu},0})$ by specialization.
		Recall by Proposition \ref{prop: geom conn comp of Sh} we have
		$$
		\pi_0(\CM_{\ov{E_{\nu}},0})= G^{\ab}(F_0) \backslash G^{ab}(\BA_{0,f}) / \det(K_G).
		$$	 
		The basic uniformization
		\[
		\Theta_0: \mathcal{M}^{\wedge,  \basic}_{O_{\breve{E}_{\nu}}, 0} \stackrel{\sim}{\rightarrow}  G^{(v_0)}(F_0) \backslash [ \mathcal{N}_{O_{\breve{E}_{\nu}}}  \times G (\mathbb{A}_{0, f}^{v_0}) / K_{G}^{v_0} ]
		\]
		is compatible with the projection
		\[
		\pi_0(\CM^{\basic}_{\ov{k_\nu},0}) \to \pi_0(\CM_{\ov{k_\nu},0}) \cong  \pi_0(\CM_{\ov{E_{\nu}},0})= G^{\ab}(F_0) \backslash G^{ab}(\BA_f) / \det(K_G).
		\]
		We are reduced to show the projection
		\[
		G^{(v_0)}(F_0) \backslash [ \mathcal{N}_{O_{\breve{E}_{\nu}}}  \times G (\mathbb{A}_{0, f}^{v_0}) / K_{G}^{v_0}] \to G^{\ab}(F_0) \backslash G^{ab}(\BA_{0,f}) / \det(K_G)
		\]
		is surjective, which follows from $[G^{(v_0)}]^\ab= G^\ab=T^1$ in our cases.
	\end{proof}
	
\begin{conjecture} \label{irreducible comp contains basic point}
 For $1 < t_v < n-1$, any geometric connected component of the Balloon (resp. Ground) stratum contains a basic point.	
\end{conjecture}

We make the following topological definitions.

\begin{definition}[Good stratification]\label{defn:Good strata}
Let $X$ be a proper scheme over a field $k$. A stratification 
$X= \coprod_{i \in I} X_i= \bigcup_{i \in I} \overline{X_i}$
 indexed by a finite partial order set $I$ is good, if
 \begin{enumerate}
 	\item  The closure relation $\overline{X_i}= \bigcup_{j \leq i} X_j$ holds.
 	\item  The stratum $X_i$ is locally closed and equi-dimensional.
 	\item  If $i <j$, the dimension of $X_i$ is strictly less than $X_j$.
 	\item $X_i$ is quasi-affine.
 \end{enumerate}
\end{definition}

 For a good stratification on $X$, any irreducible component of $X$ is an irreducible component of some $\overline{X_i}$ ($i$ maximal). By dimensional assumptions on $X_i$, this gives a bijection between the set of irreducible components of $X$ and disjoint union of the set of irreducible components of $\ov{X_i} $ for all maximal $i$. We have following topological lemmas.

\begin{lemma}\label{lemma: closed EKOR int with basic locus}
Let $\{X_i\}_{i \in I}$ be a good stratification of $X$. If $Z$ is a (non-empty)
connected component of $\ov{X_i}$ for some $i$, then there exists a minimal element $j$ in $I$ such that $Z \cap X_j \not = \emptyset $ and $\dim X_j=0$. 
\end{lemma}
\begin{proof}
There is a minimal $i' \leq i$ such that $Z \cap X_{i'} \not = \emptyset$. The minimality implies $Z \cap X_{i'}= Z \cap \ov{X_{i'}}$, which is proper and quasi-affine hence finite. As $Z$ is a connected component of $\ov{X_i}$, $Z \cap \ov{X_{i'}}$ is a
union of connected components of $X_{i'}$. As $X_{i'}$ is equi-dimensional, we see that $\dim X_{i'}=0$ hence $i'$ is minimal in $I$. The result follows for $j=i'$. 
\end{proof}

\begin{lemma}\label{lemma: good stratification and irreduciblity by Y}
	Let $\{X_i \}_{i \in I}$ be a good stratification of $X$. Let $Y$ be a closed subset of
	$X$. Fix $i \in I$. Set $\ov{Y_i} := \ov{X_i} \cap Y$. Assume that for any connected component $Z$ of $\ov{X_i}$, $Z \cap \ov{Y_i} \not = \emptyset$. If $\ov{Y_i}$
	is connected, then $\ov{X_i}$ is connected. In particular, if $\ov{X_i}$ is normal then
	$\ov{X_i}$ is irreducible.
\end{lemma}	
\begin{proof}
Firstly, if $\ov{X_i}$  is not connected hence a disjoint union of opens, then $\ov{Y_i}$ is  disjoint union of opens. The non-emptyness of theses opens comes from the assumption $Z \cap \ov{Y_i} \not = \emptyset$. A contradiction to that $\ov{Y_i}$ is connected. Hence $\ov{X_i}$ is connected. 
\end{proof}

By  \cite[Thm. C]{EKOR}, the EKOR stratification on special fiber of abelian type Shimura varieties at parahoric level is good, if the level is Iwahori at $p$ or the axiom $(4c)$ of Rapoport--He \cite{HR17-Axiom} holds. The following proposition is an adaption of \cite[Thm. 3.7.1]{IrredKR} to our cases.
	
	\begin{proposition}  \label{Important conj: irred}
	In any geometric connected component of $\CM_{k_\nu}$, the Balloon and Ground stratum are geometrically irreducible.
	\end{proposition}
	\begin{proof}
		Using the definition of $\CM^\circ$ (resp. $\CM^\bullet$) as vanishing locus of $\Lie \lambda$ (resp. $\Lie \lambda^\vee$) and the computation of local models, they correspond to different orbits of the parahoric group scheme on the local model. Therefore, our Balloon (resp. Ground) stratum is the closure of a maximal dimensional Kottwitz--Rapoport stratum hence the closure of a non-basic EKOR stratum \cite{EKOR}. We are reduced to prove the claim that closure of non-basic EKOR strata in $\CM_{k_\nu}$ are irreducible in each geometric connected component.
		
		We consider the canonical integral model of unitary Shimura varieties $\CM^{Iw}_{O_{E_\nu}} \to \CM_{O_{E_\nu}}$ by replacing the maximal parahoric level with the Iwahori level at $v$ with PEL type modulis. Note we have the axiom $(4c)$ of Rapoport--He \cite{HR17-Axiom} for our Shimura variety $\CM^{Iw}_{O_{E_\nu}} \to \CM_{O_{E_\nu}}$ by basic uniformization (Proposition \ref{prop: basic uniformization RSZ}) and \cite[Proposition A.4.5]{EKOR} for PEL type Shimura varieties. The axiom 4 (c) implies that for every EKOR stratum at our parahoric level, there is
		a EKOR stratum (i.e. Kottwitz--Rapoport stratum) at Iwahori level surjecting onto it. We are reduced to prove the claim that closure of each EKOR stratum is irreducible on any connected component of $\CM_{k_\nu}^{Iw}$.

		From now on, we consider the EKOR stratification on a given connected component of the proper scheme $\CM_{k_\nu}^{Iw}$. We firstly recall that this is a good stratification in the sense of Definition \ref{defn:Good strata}. Its strata are quasi-affine, equi-dimensional and have the closure relation in Definition \ref{defn:Good strata} by  \cite[Thm. C]{EKOR}. Now choose a  non--basic EKOR strata $X_i$. We need to show that $\ov{X_i}$ is irreducible.  We apply Lemma \ref{lemma: good stratification and irreduciblity by Y} with $Y$ being the basic locus of (the given connected component of) $\CM_{k_\nu}^{Iw}$, where the normality of $\ov{X_i}$ is known by  \cite[Thm. C]{EKOR}. We need to check the non-empty condition in Lemma \ref{lemma: good stratification and irreduciblity by Y}:  let $Z$ be any connected component of $\ov{X_i}$. By Lemma \ref{lemma: closed EKOR int with basic locus}, $Z$ must have  non-empty intersection with a dimension $\leq 1$  Kottwitz--Rapoport stratum $W$. By \cite[Prop 5.6]{HodgeNewtonBasic}, we see that $W$ is in the basic locus hence $Z \cap \ov{Y_i} \not= \emptyset$. Finally, we need to check that $\ov{Y_i}$ is connected. Via basic uniformization, we have connectedness of the basic locus and $\ov{Y_i}$ by \cite[Thm. 3.7.1, Lem. 3.7.4]{IrredKR} (where our notation $\ov{Y_i}$ agrees with $Y_w$ in \emph{loc. cit.} ). By  Lemma \ref{lemma: good stratification and irreduciblity by Y} ($Y$=the basic locus), we obtain the claim that $\ov{X_i}$ is irreducible.  
	\end{proof}

	\begin{proposition}\label{prop: mod over F for 1<t<n-1}
		If $1<t_v<n-1$, then for any linear functional $\ell: \Ch^1_{\nu}(\CM) \to \mathbb R_\Delta$, there exists a very special $1$-cycle $C'$ in $\CM_{k_{\nu}}^{\text{basic}}$ such that 
		$$
		(\CX_i, C')_\CM=\ell(\CX_i)
		$$
		for any irreducible component $\CX_i$ of  $\CM_{k_{\nu}}$. In particular, the statement applies to $\ell=(-, \CC)_\CM$ for any $1$-cycle $\CC$ in $\CM$.
	\end{proposition}
	\begin{proof}
		We may work in a given connected component $X$ of $\CM_{k_{\nu}}$. By Proposition \ref{Important conj: irred}, there are only $2$ irreducible components in this component i.e., the Balloon stratum $X^\bullet$ and the Ground stratum $X^\circ$. So $X^\bullet + X^\circ$ is principal. We only need to show there exists a very special $1$-cycle $C'$ such that 
		$(X^\bullet, C') = - (X^\circ, C') \not =0$. Then $\frac{\ell(X^\circ)}{(X^\circ, C')}C'$ is the $1$-cycle we want. By basic uniformization, any minimal dimensional Bruhat--Tits strata $\BP^{n-t_v}$ and $\BP^{t_v}$ with non-empty intersection will intersect at a point. Choose any $\BP^1 \subseteq \BP^{n-t_v}$ passing exactly  one such point will do the job.
	\end{proof}

In fact, Conjecture \ref{irreducible comp contains basic point} and Proposition \ref{Important conj: irred} make sense for mod $p$ fibers of general Shimura varieties (with suitable integral models). We make the following conjectures. 

\begin{conjecture}
	\begin{enumerate}
		\item Any (geometric) irreducible component of the mod $p$ fiber contains a basic point. 
		\item Closure of any (geometric) non-basic Kottwitz--Rapoport strata of the mod $p$ fiber is irreducible in given connected component of the mod $p$ fiber.
	\end{enumerate}
\end{conjecture}

The first part $(1)$ is closely related to the Hecke orbit conjecture for PEL type Shimura varieties \cite{Basic-point}. Conditional on \cite[Prop. 4.2 (2)]{Basic-point} which is to be filled, a prime-to-$p$ Hecke orbit is finite if and only if it contains basic points \cite[Prop. 4.8-4.9]{Basic-point}).  The second part $(2)$ is proved \cite[Thm. 3.7.1, Prop. 3.7.3]{IrredKR} for Hodge type Shimura varieties under certain quasi-split assumptions.

	\section{Global modularity via modification}\label{section:Gmod}
	
	In this section, we introduce the modification method towards arithmetic modularity. We establish the modularity of arithmetic theta series at maximal parahoric levels when intersecting with (derived) $1$-cycles on $\CM$, in particular with modified derived Hecke CM cycles.
	
	Let $\CX \to \Spec O_E[\Delta^{-1}]$ be a pure dimensional, regular, flat and proper scheme with smooth generic fiber $X$. Set the quotient $\BQ$-vector space 
	$$\BR_{\Delta}=   \BR / \text{span}_{\BQ} \{ \log \ell | \exists \nu  \in \Delta, \, \nu  |\ell \}. $$
	
	\subsection{Truncated arithmetic intersection pairing}\label{section: adm arithmetic int}
	We recall basics of the Gillet--Soul\'e	arithmetic intersection theory following \cite[\S 4.1]{AFL2021} and \cite{GS90}. 
	
	We consider arithmetic divisors to relate local and global intersection numbers. An arithmetic divisor on $\CX$ is a finite $\BQ$-linear combinations of tuples $(Z,g_Z)$, where $Z$ is a divisor on $\CX$, and $g_Z = (g_{Z,\nu})_\nu$ ($\nu$ runs over all infinite places of $E$) is a tuple of Green functions on $X_\nu (\mathbb C) \backslash  Z_\nu (\BC)$  with
	\[
	dd^c g_{Z, \nu} + \delta_{Z_\nu (\BC)}=[\omega_\nu],  \quad \text{$\omega_\nu$ a smooth $(1,1)$-form on $Z_\nu(\BC)$. }
	\]
	Fix a K\"{a}hler metric on $X_\nu (\BC)$.	The Green function $g_{Z,\nu}$ is called \emph{admissible} if $\omega_v$ is harmonic with respect to the metric. Then the arithmetic Chow group $\wh{\Ch}^1(\CX)$ (with $\BQ$-coefficients)
	is generated by arithmetic divisors, modulo the relation given by the $\BQ$-span of principal arithmetic divisors  i.e., tuples $\wh{\div}(f):=(\div(f), (-\log|f|_\nu^2)_{\nu| \infty})$ associated to rational functions $f \in E(\CX)^{\times}$:
	\[
	\widehat{\Ch}^1(\mathcal{X}):= \{ \text{Arithmetic divisors} \} / \mathrm{Span}_{\mathbb Q}\{ (\text{div}(f), (- \text{log} |f|^2_\nu)_{\nu|\infty}) \}_{f \in E(\CX)^\times}.
	\]
	An admissible Green function for a given divisor $Z$ exists, and is unique up to adding locally constant functions on $X_\nu(\BC)$. Consider the admissible Chow group $\wh{\Ch}^{1,\adm}(\CX)$ as the subgroup of the arithmetic Chow group generated by tuples $(Z, (g_{Z,\nu})_{\nu})$ with admissible Green functions $g_{Z, \nu}$.    There is a natural map
	\begin{equation}
		\wh{\Ch}^{1,\adm}(\CX) \to \Ch^1(X), \quad
		(Z, g_Z) \mapsto Z_E.
	\end{equation}      
	
	For any place $\nu$ of $E$, define $\wh{\Ch}^{1}_\nu(\CX) \subseteq \wh{\Ch}^{1,\adm}(\CX)$ as the subgroup generated by
	\begin{itemize}
		\item $(0, c_\nu)$ where $c_\nu$ is a locally constant function on $X_\nu(\BC)$, if $\nu|\infty$.
		\item $(\CX_{k_\nu, i}, 0)$ where and $\CX_{k_\nu, i}$ is an irreducible component of the special fiber $\CX_{k_\nu}$ at $\nu$, if $\nu$ is a finite place with residue field $k_\nu$ .
	\end{itemize}
	
	Let $\wh{\Ch}^{1}_{\mathrm{Vert}}(\CX) \subseteq \wh{\Ch}^{1,\adm}(\CX)$ be the subgroup generated by $\wh{\Ch}^{1}_\nu(\CX)$ for all places $\nu$ of $E$, and $\wh{\Ch}^{1}_{\mathrm{\infty}}(\CX) :=\sum_{\nu_\infty | \infty} \wh{\Ch}^{1}_{\nu_\infty}(\CX)$.
	
	\begin{lemma}\label{factorization of Chow}
		\begin{enumerate}
			\item The projection $\wh{\Ch}^{1,\adm}(\CX) \to \Ch^1(X)$, $(Z, g_Z) \mapsto Z_E$ induces an isomorphism $\wh{\Ch}^{1,\adm}(\CX)/\wh{\Ch}^{1}_{\mathrm{Vert}}(\CX) \cong \Ch^1(X)$. Hence we have a short exact sequence:
			$$
			0 \to \wh{\Ch}^{1}_{\mathrm{Vert}}(\CX) \to \wh{\Ch}^{1,\adm}(\CX) \to \Ch^1(X) \to 0.
			$$
			\item For any finite place $\nu$ such that $\CX_{\nu}:=\CX \otimes O_{E, \nu}$ is smooth over $O_{E, \nu}$, we have $\wh{\Ch}^{1}_\nu(\CX) \subseteq \wh{\Ch}^{1}_{\mathrm{\infty}}(\CX)$.
		\end{enumerate}
	\end{lemma}      
	\begin{proof}
		The first part: by taking Zariski closure and using the existence of admissible Green function, we see the natural map $ \wh{\Ch}^{1,\adm}(\CX) \to \Ch^1(X)$ is surjective. Then we show the natural map has kernel $\wh{\Ch}^{1}_{\mathrm{Vert}}(\CX)$. If $[Z]_E=\emptyset$, then $Z$ is a finite linear combination of  arithmetic divisors supported on special fibers of $\CX$. By adding elements in $\wh{\Ch}^{1}_\nu(\CX)$ ($\mu$ finite place of $E$), we may assume $Z=\emptyset$, then by admissibility  of $g$, we see $g$ must be a constant. This finishes the proof.
		
		The second part: assume that $\CX_{\nu}$ is smooth over $O_{E, \nu}$. Then any $\CX_{k_\nu, i}$ is smooth and is the reduction of a connect component $\CX_{\nu, i}$ of $\CX_\nu$. The argument of \cite[Lemma 4.1]{AFL2021} applies. We use the Stein factorization $\CX \to \Spec R \to \Spec O_E[\Delta^{-1}]$ where $R=H^0(\CX, O_{\CX})$. The connected components of $\CX_{\nu}$ are in bijection with $|\Spec R/ \frak{p}_\nu|$. By finiteness of class number of $R$, we may find an element $a$ in $R$ whose divisor is supported on the prime ideal of $R$ corresponding to $\CX_{k_\nu, i}$. As the arithmetic divisor $\wh{\div}(a|_{\CX_{\nu, i}})$ is principal, we see that $(\CX_{k_\nu, i}, 0) \in \sum_{\nu_\infty | \infty} \wh{\Ch}^{1}_{\nu_\infty}(\CX)$.
	\end{proof}

	Let $\CC_{1}(\CX)$ be the group of $1$-cycles on $\CX$ (with $\BQ$-coefficient). Consider the $\Delta$-truncated version standard arithmetic intersection pairing \cite{GS90} \cite[Section 4.2]{AFL2021} between $\BQ$-vector spaces:
	\begin{equation}\label{eq: Truncated Arithmetic int pairing}
		(\cdot,\cdot): \wh\Ch^1(\CX)\times \CC_{1}(\CX)  \to \BR_\Delta.
	\end{equation}

	Consider the subgroup $\CC_{1}(\CX)_{\deg=0}$ consisting of $1$-cycles with zero degree on each connected component of $X$ over $E$, which is the orthogonal complement of  $\wh{\Ch}^{1}_{\infty}(\CX)$.  Let $\CC_{1}(\CX)^{\perp} \subseteq \CC_{1}(\CX)$ be the orthogonal complement of $\wh{\Ch}^{1}_{\text{Vert}}(\CX)$. 
	
	From Lemma \ref{factorization of Chow}, the $\Delta$-truncated arithmetic intersection pairing (\ref{eq: Truncated Arithmetic int pairing}) induces a pairing
	\begin{equation}
		(\cdot,\cdot)^\text{adm}: \, \Ch^1(X)\times \CC_{1}(\CX)^{\perp}  \to \BR_\Delta.
	\end{equation}

	\subsection{Green functions and arithmetic special divisors}\label{section:arch}
	
	Let $K$ be a neat compact open subgroup of $G(\BA_{0,f})$,  and fix an embedding $\nu: E \hookrightarrow \BC$.  
	
	Choose a Schwartz function  $\phi \in \CS(V(\BA_{0,f}))^K$. Choose any compact open subgroup $K_0 \subseteq \SL_2(\BA_{0,f})$ such that $\phi $ is invariant under $K_0$ by the Weil representation.
	
	For $\xi \in F_{0}$, we have two kinds of Green functions \cite[Section 3]{AFL2021} for the Kudla--Rapoport divisor $Z(\xi,\phi)_\BC:=Z(\xi,\phi) \otimes_{E, \nu} \BC$ on $M_{G,K} \otimes_{E, \nu}\BC$:
	\begin{itemize}
		\item The \emph{Kudla Green function} \cite{Kudla97Ann} $\CG^{\bf K}(\xi,h_{\infty},\phi)$ with a variable $h_\infty \in \SL_2(F_{0, \infty})$, which occurs in the analytic side naturally via computations of archimedean orbital integrals. 
		\item The \emph{automorphic Green function} \cite{JBGreen} $\CG^{\bf B}(\xi, h_\infty, \phi)=\CG^{\bf B}(\xi,\phi)$ for $\xi \in F_{0, +}$ which is admissible by the work of Bruinier \cite[Cor. 5.16]{JBGreen}. 
	\end{itemize} 
	
	Consider the difference  $\mathcal{G}^{\bf K-\bf B}(\xi, h_{\infty}, \phi)= \begin{cases}
		\CG^{\bf K}(\xi,h_{\infty},\phi)- \CG^{\bf B}(\xi,\phi), & \text{for } \xi \in F_{0,+}, \\
		\CG^{\bf K}(\xi,h_{\infty},\phi), & \text{else.}
	\end{cases}$ 
	
	We form the generating series of Green functions as the archimedean part of the generating series of arithmetic special divisors ($h \in \SL_2(\BA_0)$):
	\begin{equation}\label{eq: generating series Green}
		\mathcal{G}^{?}(h, \phi):= \sum_{\xi \in F_0} \mathcal{G}^{?}(\xi, h_{\infty}, \omega(h_f)\phi)W^{(n)}_\xi (h_{\infty}), \quad ?=\bf K, \bf B, \bf K- \bf B. 
	\end{equation}
	
	\begin{theorem}\label{thm: Green}
		The generating function $\mathcal{G}^{\bf K-\bf B}(h, \phi)$ when evaluating at a degree zero cycle on $\CM_{\wt G, \wt K}(\BC)$,  is a smooth function on $\SL_2(\BA_0)$, invariant under left $\SL_2(F_0)$-action and right $K_0$-action, and of parallel weight $n$.
	\end{theorem}
	\begin{proof}
		See \cite[Thm. 3.13]{AFL2021}.
	\end{proof}
	
	\begin{definition}\label{Hodge bundle}[Hodge bundle]
		We take $\omega_\CM=(\Lie A_0) \otimes (\Lie A)_1^\vee$ as in \cite[Section 7.2]{BHKRY} as an extension of the automorphic line bundle $\omega$ to $\CM_{\wt{G}, \wt{K}}$. Define the arithmetic line bundle $\wh \omega=(\omega_\CM, \lVert  -  \rVert_{\mathrm{Pet}})$ endowed with the natural Petersson metric. 
	\end{definition} 
	Define constant terms
	\begin{equation}
		\wh \CZ^{B} (0, \phi):=-\phi(0) c_1(\wh \omega) \in \wh{\Ch}^1(\CM_{\wt G, \wt K}), 
	\end{equation}
	\begin{equation}
		\wh \CZ^{K} (0, h_\infty, \phi)=-\phi(0) c_1(\wh \omega) + (0, (\CG^{\bf K}_{\nu}(\xi,h_{\nu},\phi))_{\nu} ) \in \wh{\Ch}^{1}(\CM_{\wt G, \wt K}).
	\end{equation}

	Set $\CG^{\bf K}(\xi,\phi):=\CG^{\bf K}(\xi,1, \phi)$.  For $\xi \in F_{0}$, define \emph{arithmetic Kudla--Rapoport divisors} ($?=\bf K, \bf B, \bf K- \bf B$)
	\begin{equation}\label{eq: Fourier coeff of arithmetic generating series}
		\wh \CZ^{?} (\xi, h_\infty, \phi)= (\CZ (\xi, \phi), ( \mathcal{G}^{?}_{\nu}(\xi, h_\nu, \phi))_{\nu} ) \in \wh{\Ch}^1(\CM_{\wt G, \wt K}),  
	\end{equation}
	\begin{equation}
		\wh \CZ^{?} (\xi, \phi)= (\CZ (\xi, \phi), ( \mathcal{G}^{?}_{\nu}(\xi, \phi))_{\nu} ) \in \wh{\Ch}^1(\CM_{\wt G, \wt K}).
	\end{equation}
	We have that $\wh \CZ^{\bf B} (0, \phi), \wh \CZ^{\bf B} (\xi, \phi) \in \wh{\Ch}^{1, \text{adm}}(\CM_{\wt G, \wt K})$.

	\subsection{Modularity of arithmetic theta series at maximal parahoric levels}

	Return to the set up in \S \ref{section: global Shimura var}, we have the integral model (recall $F_0 \not =\mathbb Q$)
	$$
	\CM=\CM_{\wt G, \wt K} \to \Spec O_E[\Delta^{-1}].
	$$ The level $\wt{K}$ is related to a chosen lattice $L$ such that $L_{v_0}$ is a vertex lattice of type $0 \leq t_0 \leq n$.

	\begin{theorem}\label{thm: almost global modularity}
		Let $\phi \in \CS(V(\BA_{0,f}))^K$ be a Schwartz function  such that $\phi$ is invariant under a compact open subgroup $K_0 \subseteq \SL_2(\BA_{0,f})$  by the Weil representation. Then for any $C^\perp \in \CC_{1}(\CM)^{\perp}$, the generating function $h \in \SL_2(\BA_0) \mapsto (Z(h, \phi), C^\perp)^{\text{adm}}$ lies in $\CA_{\rm hol}(\SL_2(\BA_0), K_0, n)_{\ov \BQ} \otimes_{\ov \BQ} \BR_{\Delta, \ov \BQ}$. Moreover, the same is true if we replace $(-,C^\perp)$ by any linear functional $\ell: \Ch^1(\CM) \to \mathbb R_\Delta$ that is trivial on $\wh{\Ch}^{1}_{\text{Vert}}(\CM)$.
	\end{theorem}
	\begin{proof} 
		This follows from by applying Lemma \ref{factorization of Chow} to the regular scheme $\CM$ and Proposition \ref{modularity over generic fiber}.
	\end{proof}

	Let $$
	\phi=\phi_i \in \CS(V(\BA_{0,f}))^{K}$$ be of the form $\phi_i$ $(i=1,2)$ introduced in \S \ref{def:KR int}. Then $\phi$ is invariant under a compact open $K_0 \subseteq \SL_2(\BA_{0, f})$ such that for $v \not \in \Delta$ we have $K_{0, v}=\SL_2(O_v)$ if $L_v$ is self dual; $K_{0, v}=\Gamma_0(v)$ is a upper congruence subgroup if $v \not =v_0$ and $L_v$ is not self dual; $K_{0, v_0}=\Gamma_0(v_0)$ if $i=1$ and $K_{0, v_0}=\Gamma_0^-(v_0)$ is a lower triangular congruence subgroup if $i=2$. Consider the generating series 
	$$
	\wh \CZ^{\bf B}(h_\infty, \phi)= \sum_{\xi \in F_0} \wh \CZ^{\bf B} (\xi, \phi) W_\xi^{(n)}(h_\infty).
	$$

	\begin{theorem}\label{thm: global modularity}
		Assume $E=F$ and $F_0 \not= \mathbb Q$.  
		\begin{enumerate}
			\item Assume that $\Delta$ is large enough so that there exists a non-empty Hecke CM cycle $C=CM(\alpha_0, \mu_0)$ (c.f. (\ref{CM cycle: generic fiber}) ), where $\alpha_0 \in F[t]$ is maximal order at any finite place $v$ outside $\Delta$ and is unramified at $v_0$.  Then for any $1$-cycle $\CC \to \CM$ over $O_E[\Delta^{-1}]$, the function 
			$$(\wh \CZ^{\bf B}(h_\infty, \phi), \CC) $$ is a holomorphic modular form on $\SL_2(F_{0,\infty})$ of parallel weight $n$ of level $K_0$, up to adding a constant in $\BR$.
			\item let $\ell: \wh{\Ch}^{1}(\CM) \to \mathbb R_\Delta$ be a linear functional that has degree $0$ over $\mathbb C$, i.e. $\ell|_{\wh{\Ch}^{1}_{\mathrm{\infty}}(\CM)}=0$. Assume $\ell$ is also trivial on irreducible components of special fibers of $\CM$ over $v  \not \in \Delta$, i.e. $\ell|_{\wh{\Ch}^{1}_{v}(\CM)}=0$. Then $\ell(\wh \CZ^{\bf B}(h_\infty, \phi))$ is a holomorphic modular form on $\SL_2(F_{0,\infty})$ of parallel weight $n$ of level $K_0$.
		\end{enumerate}
	\end{theorem}
	
	\begin{proof}
	Firstly, we know the arithmetic modularity for specific $1$-cycles $\CC$ of the following two forms:
		\begin{enumerate}
			\item $\CC$ is a very special $1$-cycle in the basic locus of $\CM_{k_{\nu}}$ where $\nu$ is a place of $E$ over an inert place $v_0$ of $F_0$. This is done by Theorem \ref{thm: int in basic locus = theta series}.
			
			\item (Only needed for part $(1)$) $\CC$ is the normal integral model of a Hecke CM cycle $C=CM(\alpha, \mu_{G, \Delta})$. We use the set up in \S \ref{section: global analytic and geometric side}. We choose a standard partial transfer $\Phi'$ (Definition \ref{standard partial transfers}) in $\CS((S(V_0) \times V')(\BA_0))$ of the standard function 
			$$
			\Phi=\Phi_\Delta \otimes \prod_{v \not \in \Delta}(1_{U(L_v)} \otimes 1_{L_v}), \, \, \Phi_\Delta:=\phi_\Delta \otimes 1_{\mu_0}.
			$$
			For $\xi \in F_0^\times$, by Proposition \ref{prop: global-semiglobal-local} and the proof in Theorem \ref{prop: local ATC is from global id of Fourier coeff derived at v0} we have the matching
			\[
			2 \partial_{\alpha}(\xi, \Phi') + \Int^{\bf K- \bf B}(\xi, \Phi) = - (\wh \CZ^{\bf B}(\xi, \phi), \CC) .
			\]
			Note by known ATCs for unramified maximal orders in \S \ref{section: ATC unram max order}, the AFL for maximal orders in \cite[Cor. 9.1]{M-Thesis} the assumption in Proposition \ref{prop: global-semiglobal-local} holds.
			Then the modularity follows from the modularity of the automorphic generating function $\partial \BJ_\alpha(h_\infty,\Phi')$ known by the Poisson summation formula and modularity of $\Int^{\bf K- \bf B}(h_\infty, \Phi_f)$ known by \cite{JBGreen}.			
		\end{enumerate}
	Now we consider a general $1$-cycle $\CC$. From the modifications over $\BC$ (Proposition \ref{prop: mod over BC by max order Hecke CM} which uses $E=F$) and $\BF_q$ (Proposition \ref{prop: mod over F for t=1} and Proposition \ref{prop: mod over F for 1<t<n-1}), we know there exists a $1$-cycle $\CC'$ that is a finite linear combination of specific $1$-cycle above such that the functional $(-, \CC-\CC')$ factors through $\Ch^1(\CM_E)$ via the surjection (\ref{factorization of Chow}). Hence the modularity of $(-, \CC-\CC')$ holds by the modularity over the generic fiber (Theorem \ref{modularity over generic fiber}). And the modularity of $\CC$ follows from $\CC'$ and $\CC-\CC'$. The second part is proved similarly.
	\end{proof}

	\begin{remark}
	We may use complex uniformization (\ref{prop: complex unif for CM cycles}) to pin down the condition that a Hecke CM cycle $CM(\alpha, \mu_{G, \Delta})$ is non-empty. Starting with an unramified maximal order $\alpha_{v_0} \in F_{v_0}[t]$,  We may enlarge $\Delta$ to make sure the condition holds.
	\end{remark}

	\subsection{Modularity for very special $1$-cycles in basic locus}\label{section:Lmod}
	
	Assume $\CC$ is a very special $1$-cycle in the basic locus of $\CM_{k_{\nu}}$ where $\nu$ is a place of $E$ over an inert place $v_0$ of $F_0$.
	
	To show $(\wh \CZ^{\bf B}(h_\infty, \phi), \CC) $ is modular, we use basic uniformization and local modularity results (Section \ref{section: local modularity}). In the end, we will find $(\wh \CZ^{\bf B}(h_\infty, \phi), \CC) $ can be adelized into a theta series for suitable Schwartz function $\phi' \in \CS(V^{(v_0)}(\BA_{0,f}))$ in the nearby hermitian space of $V$ at $v_0$, hence is modular.  Without loss of generality, we can assume $\CC$ is of the form
	\[
	\CC= \CC_0 \times 1_{g_0K^{v_0}} \to \CM_{k_{\nu},0}
	\]
	for some $1$-cycle $\CC_0 \subseteq \CN^{\red}$ and $g_0 \in G(\BA_{0, f}^{(v_0)})/ K^{v_0}$. Note the functions $(\CZ(u), \CC_0)$ and$(\CZ(u), \CC_0)$ ($u \in V^{(v_0)}(F_{0,v_0})-0$) extend to smooth functions on $V^{(v_0)}(F_{0, v_0})$ by Theorem \ref{thm: partial local modularity} under the assumption $\CC$ is a very special $1$-cycle.  So we have a Schwartz function $\phi'_\CC \in \CS(V^{(v_0)}(\BA_{0,f}))$ defined by
	\[
	\phi'_{\CC, v}(u):= \begin{cases}
		\phi_v(g_0^{-1}u) & \text{if $v \not= v_0$}, \\
		(\CZ(u), \CC_0)& \text{if $v = v_0, \, i=1$}, \\
		(\CY(u), \CC_0)& \text{if $v = v_0, \, i=2$}. \\
	\end{cases}
	\]

\subsection{Constant terms}
Recall that $\wh \CZ^{B} (0, \phi)=-\phi(0) c_1(\wh \omega)$. By moduli interpretation,  the pullback of the Hodge bundle $\wh \omega$ (\ref{Hodge bundle}) to the basic locus is the Hodge bundle $\omega=(\Lie X)_0^\vee$. We need to show
\[
-\phi(0) \deg(\omega|_\CC)=\phi'_\CC(0).
\]
Identifying the prime-to-$v_0$ part, we reduce this to the following proposition.
\begin{proposition}\label{constant terms}
	For a very special $1$-cycle $\CC \subseteq \CN_{L_{v_0}}$, we have 
	\[
	- \deg \omega |_\CC = \lim_{u \to 0} (\CZ(u), \CC_0), \quad q \deg \omega |_\CC  =  \lim_{u \to 0} (\CY(u), \CC_0).
	\] 
\end{proposition} 
\begin{proof}
	As $\CC$ is very special, via pullback we only need to do the case $n=2$ and $t_v=1$ i.e., the unitary Shimura curve case with Drinfeld uniformization at $v_0$. We use computations in Theorem \ref{thm: n=2 Z-cycle}. Consider the case $\CC=\BP^1_{L^\circ}=\BP(L^\circ / \varpi L^\circ)$ for a type $0$ lattice $L^\circ$. Then by Theorem \ref{thm: n=2 local modularity}, the right hand side is $1$ for $\CZ$-cycles and $-q$ for $\CY$-cycles. So we are reduced to show
	\[
	\deg \omega |_{\BP^1_{L^\circ}} = -1. 
	\]
	This is true by relative Dieudonne theory:  by equation \ref{eq: Dieudone defn of BT(L circ)}, a special pair $(A, B)$ (recall $A=M_0, B=M_1^\perp, B^\vee=\CF_\BX M_1$) determines a point $x=\varpi B / \varpi A^\vee$ in $\BP(L^\circ / \varpi (L^\circ)^\vee )$ if $A=L^\circ$. The line bundle $(\Lie X)_0$ at $x$ gives the line $M_0/\CV_{\BX}M_1 = A/ \sigma^{-1}(B) \overset{\sigma} \cong A^\vee /B $ in $\BP(L^\circ / \varpi (L^\circ)^\vee )$. So $\omega$ is the tautological line bundle $O(-1)$ on $\BP^1_{L^\circ}$, hence has degree $-1$. The case $\CC=\BP^1_{L^\bullet}$ for a type $2$ lattice $L^\bullet$ is similar. 
\end{proof}

\begin{remark}
	See \cite[Section 11.1]{KR-2004} for similar computations on orthogonal Shimura curves.
\end{remark}

	\begin{theorem}\label{thm: int in basic locus = theta series}
	Let $\CC$ be a very special $1$-cycle in the basic locus of $\CM_{k_{\nu}}$ where $\nu$ is a place of $E$ over an inert place $v_0$ of $F_0$. Then we have an equality
		\[
		(\wh \CZ^{\bf B}(h_\infty, \phi), \CC) =  \phi_{\CC}'(0) + \sum_{\xi \in F_{0, +}} \sum_{u \in V^{(v_0)}_\xi(F_0)} \phi_{\CC}'(u) W_\xi^{(n)}(h_\infty).
		\]	
		Hence  $(\wh \CZ^{\bf B} (\xi, \phi), \CC)$  is the $\xi$-th Fourier coefficient of the theta series $\Theta(h, \phi_{\CC}' )$  
		for $\phi_{\CC}'$ on $V^{(v_0)}$. In particular, $(\wh \CZ^{\bf B}(h_\infty, \phi), \CC)$ is modular. 
	\end{theorem}
	\begin{proof}
		The non-constant term follows from basic uniformization in Proposition \ref{prop:basic unif KR} directly. The constant term is done by Proposition \ref{constant terms}.
	\end{proof}
	
	\begin{remark}
		For $n=2, t=1$, basic uniformization is used in \cite[Section 4.3]{KR-height} to show the modularity of their arithmetic theta series on orthogonal Shimura curves when intersecting with vertical projective lines $\BP^1$. We will use the local duality between $\CZ$ and $\CY$ cycles in Theorem \ref{thm: partial local modularity} later to relate $(\wh \CZ^{\bf B}(h_\infty, \phi_1), \CC)$ and $(\wh \CZ^{\bf B}(h_\infty, \phi_2), \CC)$. 
	\end{remark}

	\part{The proof of ATC for $p>2$}
	
Let $p$ be an odd prime. Assuming $F_0/\BQ_p$ is unramified if $0<t<n$, now we prove arithmetic transfer Conjectures \ref{conj: gp version ATC} and \ref{conj: semi-Lie version ATC}. 

By equivalences in Theorem \ref{thm: semi-Lie and gp}, we only need to prove the semi-Lie version ATC in part $(1)$ of Conjecture \ref{conj: semi-Lie version ATC} for any vertex lattices.
	
The strategy is local-global via a double induction: 
\begin{itemize}
	\item We globalize the data and produce a pair of automophic forms (using modularity) on both geometric side and analytic side.
	\item Non-singular Fourier coefficients of geometric and analytic side (away from $\Delta$) with ``boundary valuations'' at the distinguished finite place $v_0$ can be matched by Theorem \ref{proof of TC assuming boundary cases} and previous local reductions and inductions. 
	\item By double induction Lemma \ref{the double induction lemma}, we obtain global identities for all Fourier coefficients. 
	\item We deduce local identities at $v_0$ and finally the desired local equalities by subtracting local terms in the modification.
\end{itemize}

	\section{The global analytic and geometric side}\label{section: global analytic and geometric side}
	
	In this section, we construct the global analytic side and geometric side in the semi-Lie setting at maximal parahoric levels, following \cite[\S 11-15]{AFL}. We use the set up in \S \ref{section: global Shimura var}:
	
	\begin{itemize}
		\item $F/F_0$ is a totally imaginary quadratic extension of a totally real number field. Choose a CM type $\Phi$ of $F$ and 
		uished element $\varphi_0 \in \Phi$.
		\item $V$ is a $F/F_0$-hermitian space of dimension $n \geq 1$ of signature $\{ (n-1,1)_{\varphi_0}, (n,0)_{\varphi \in \Phi -  \{ \varphi_0 \} }  \} $.
		\item Fix a finite place $v_0$ of $F_0$ that is inert in $F$, with residue characteristic $p>2$.  Fix an integer $0 \leq t_0 \leq n$. 
		 \item $\Delta$ is a finite collection of places of $F_0$. Assume that $p \nmid \Delta$, i.e. all places $v$ of $F_0$ above $p$ are not in $\Delta$, and that all places $v \not \in \Delta$ is unramified in $F$.  
		 \item $L$ is a hermitian lattice satisfying assumptions in \S \ref{section: distinguished v0}. Assume that $L$ is self--dual away form $\Delta$ and $v_0$, and $L_{v_0}$ is a vertex lattice of type $t_0$.
	\end{itemize}	
	
Choose a level structure $\wt{K}$ as in (\ref{level for L and Delta}) for $L$ and $\Delta$. We have the integral model (\ref{defn: integral RSZ model})
	$$
	\CM=\CM_{\wt G, \wt K} \to \Spec O_E[\Delta^{-1}].
	$$
	of the RSZ unitary Shimura variety for the hermitian space $V$ with level $\wt{K}$.
	
Fix a degree $n$ conjugate self-reciprocal (\ref{conjuga self-reciprocal}) monic polynomial $\alpha \in F [t]$. Assume that $\alpha$ is irreducible. Then $F':=F[t]/(\alpha(t))$ is a field with a Galois involution on $F'$ extending the non-trivial Galois involution of $F/F_0$ and sending $t$ to $t^{-1}$. Denote by $F_0'$ the fixed subfield of $F'$ under this involution.

	\subsection{Analytic generating functions and derivatives at $s=0$}
	
	We firstly globalize the analytic side in \S \ref{section: local set up transfer} based on \cite[\S 12]{AFL-Invent}. Fix an orthogonal $F$-basis of $V$ to endow it with a $F_0$-rational structure $V_0$ and a $F/F_0$ semi-linear involution $\ov{(-)}$ with fixed subspace $V_0$. Consider the symmetric space
	\begin{equation}
		S(V_0)=\{ \gamma \in \GL(V) | \gamma \ov \gamma =id \}
	\end{equation}
	and the natural $(F_0 \times F_0) / F_0$-hermitian space 
	$$
	V' =V_0 \times (V_0)^*.
	$$
	
	Let $\eta: \BA_{0}^\times \to \BC^\times$ be the quadratic character associated to $F/F_0$ by global class field theory. Consider the subscheme 
	$$ 
	S(V_0)(\alpha) \subseteq S(V_0)
	$$ of elements with characteristic polynomial $\alpha$.  Then $S(V_0)(\alpha)(F_0)$ consists of exactly one $\GL(V_0)(F_0)$-orbit.
	
\begin{definition}\label{Gaussian test function}
A decomposable function $\Phi'=\otimes_{v} \Phi'_v \in \CS((S(V_0) \times V')(\BA_0))$ is called Gaussian, if for every place $v | \infty$ of $F_0$,  $\Phi'_v$ is the chosen partial Gaussian test function (relative to $\alpha$) in \cite[Section 12.4]{AFL}, which partially transfers to the Gaussian function $e^{-\pi (u,u)}$ on $(g, u) \in \U(n)(\BR) \times \BC^n$.
\end{definition}	

Choose a Gaussian decomposable function $\Phi' \in \CS((S(V_0) \times V')(\BA_0))$. Consider the action of $h \in \GL(V_0)$ on $(\gamma,u_1,u_2) \in S(V_0) \times  V'$  by
\begin{equation}
	h.(\gamma,u_1,u_2)=(h^{-1}\gamma h,h^{-1}u_1, u_2h).
\end{equation}

Consider a pair $(\gamma, u') \in (S(V_0)(\alpha)\times V')(F_0)$ that is regular semisimple, which is equivalent to $u' \not =0$ as $\alpha$ is irreducible. Define the global orbital integral ($v$ runs over all places of $F_0$, $s \in \mathbb C$):
	\begin{equation}
		\Orb((\gamma,u'), \Phi',s)= \prod_{v} \Orb((\gamma,u'), \Phi'_v,s)
	\end{equation}
	
	Denote by $[(S(V_0)(\alpha) \times V')(F_0)]$ (resp. $[(S(V_0)(\alpha) \times V')(F_0)]_\rs$) the set of (resp. regular semisimple) $\GL(V_0)(F_0)$-orbits in $(S(V_0)(\alpha) \times V')(F_0)$.  Set the adelic quotient $[\GL(V_0)]:= \GL(V_0)(F_0) \backslash \GL(V_0)(\BA_0)$.  
	
Let $\omega$ be the Weil representation of $\SL_2(\BA_0) \times \U(V)$ on $V(\BA_0)$ as in \cite[Section 11.1]{AFL}. For $h \in \SL_2(\BA_0)$, consider the regularized integral ($s \in \BC$): 
	\begin{equation}\label{defn of BJ}
		\BJ_{\alpha}(h,\Phi',s)=\int_{g \in [\GL(V_0)] } \left(\sum_{(\gamma, u')\in (S(V_0)(\alpha)\times V')(F_0)}\omega(h)\Phi'(g^{-1}. (\gamma, u'))\right)   |g|^s\eta(g)  dg.
	\end{equation}
	
	By \cite[Thm. 12.14]{AFL}, we see that $\BJ_{\alpha}(h,\Phi',s)$ is a smooth function of $(h, s)$, entire in $s \in \BC$ and is left invariant under $h \in \SL_2(F_0)$ by Poisson summation formula. 
	
	We have a decomposition:
	\begin{equation}
		\BJ_{\alpha}(h,\Phi',s)= \BJ_{\alpha}(h,\Phi',s)_{0}+\sum_{(\gamma,u')\in [(S(V_0)(\alpha) \times V')(F_0)]_\rs} \Orb((\gamma,u'),\omega(h)\Phi',s),
	\end{equation}
	where $\BJ_{\alpha}(h,\Phi',s)_{0}$ is the term over the two regular nilpotent orbits as in \cite[Section 12.6]{AFL}.
	
For $\xi \in F_0$, let $\BJ_{\alpha}(\xi, h, \Phi',s)$ be the $\xi$-th Fourier coefficient  (\ref{defn: Fourier coeff}) of $\BJ_{\alpha}(\cdot,\Phi',s)$. If $\xi \not =0$, we have
	\begin{align}
	\BJ_{\alpha}(\xi, h, \Phi',s)=	\sum_{(\gamma,u')\in [(S(V_0)(\alpha) \times V'_\xi)(F_0)]_\rs} \Orb((\gamma,u'),\omega(h)\Phi',s).
	\end{align}
 Here $V'_\xi$ is the subscheme of $(u_1, u_2) \in V'$ defined by the condition $u_2(u_1)=\xi$. Consider the derived functions and orbital integrals:
	\begin{equation}
		\partial \BJ_\alpha(h,\Phi') :=\frac{d}{ds}\Big|_{s=0}   \BJ_\alpha(h,\Phi',s).
	\end{equation}
	\begin{equation}
		\del((\gamma,u'),\Phi'_v) := \frac{d}{ds}\Big|_{s=0}  \Orb((\gamma,u'),\Phi'_v,s).
	\end{equation}
By similar properties of $\BJ_{\alpha}(\xi,\Phi',s)$, $\partial \BJ_\alpha(h,\Phi')$ is a smooth function on $h \in \SL_2(\BA_0)$, left invariant under $\SL_2(F_0)$, and of parallel weight $n$. If $\Phi'$ is $K_0$-invariant under the Weil representation, then  $\partial \BJ_\alpha(h,\Phi')$ is right $K_0$-invariant.

	By Leibniz's rule, we have the standard decomposition ($v$ runs over all places of $F_0$):
	\begin{equation}
		\partial \BJ_\alpha(h,\Phi') = \partial \BJ_{\alpha}(h,\Phi')_{0} + \sum_{v}    \partial \BJ_{\alpha ,v}(h, \Phi'),
	\end{equation}
	where
	\begin{equation}\label{def partial BJ a,v}
		\partial \BJ_{\alpha ,v}(h, \Phi') :=\sum_{(\gamma,u')\in [(S(V_0)(\alpha) \times V')(F_0)]_\rs}  \del((\gamma,u'), \omega(h)\Phi'_v)\cdot  \Orb((\gamma,u'), \omega(h)\Phi'^{v}).
	\end{equation}
	and the nilpotent term $\partial \BJ_{\alpha}(h,\Phi')_{0}$ is part of the $0$-th Fourier coefficient of $\partial \BJ_{\alpha}(h,\Phi')$. We have the Fourier expansion
	\begin{equation}\label{eq: Del BJ}
		\partial \BJ_{\alpha, v}(h, \Phi') = \sum_{\xi \in F_0} \partial \BJ_{\alpha ,v}(\xi, h, \Phi').
	\end{equation}
For $\xi \in F_0^\times$, we have 
	\begin{equation}\label{eq: Del BJ (xi, h, Phi')}
	\partial \BJ_{\alpha ,v}(\xi, h, \Phi'):= \sum_{(\gamma,u')\in [(S(V_0)(\alpha) \times V'_\xi)(F_0)]_\rs}  \del((\gamma,u'), \omega(h)\Phi'_v)\cdot  \Orb((\gamma,u'), \omega(h)\Phi'^{v}).
	\end{equation}

As $\Phi' \in \CS((S(V_0) \times V')(\BA_0))$, for all but finitely many finite place $w$ of $F_0$, we know that $\Phi'_w$ is the standard function (\ref{std function}) for the self-dual lattice $L_{0,w} \subseteq V_{0,w}$ (spanned by the chosen orthogonal basis of $V$):
$$
\Phi'_w=1_{S(L_{0,w})} \times L_{0,w} \times (L_{0, w})^*.
$$

	\subsection{Comparison to intersection numbers}\label{section: compare Int and DJ}
	
	Let $V$ be the $n$-dimensional hermitian space we start with. Consider a decomposable function $\Phi=\otimes_{v<\infty}\Phi_v \in\CS((\U(V)\times V)(\BA_{0,f}))$.  
	
\begin{definition} \label{Partial transfers}
We say a decomposition function $\Phi' \in \CS((S(V_0) \times V')(\BA_0))$ is a partial transfer (relative to $\alpha$) of $\Phi$, if $\Phi'$ is Gaussian (\ref{Gaussian test function}) and for any finite place $v$ of $F_0$, equalities for $\Phi_v'$ and $\Phi_v$ in Definition \ref{Semi-Lie Tranfers} are required to hold only for $(\gamma, u_1, u_2) \in (S(V_0)(\alpha) \times V_0 \times V_0^*)(F_0)_\rs$.
\end{definition}	

Choose a partial transfer $\Phi'$ of $\Phi$, which exists by the existence of smooth Jacquet--Rallis transfers \cite{Zhang-GGP}. Let $v$ be a place of $F_0$. We hope to compare $\partial \BJ_{\alpha ,v}(h, \Phi')$ (\ref{def partial BJ a,v}) and $\partial \BJ_{\alpha ,v}(\xi, h, \Phi')$ (\ref{eq: Del BJ (xi, h, Phi')}) on the analytic side with relevant geometric terms for $\Phi$.

\begin{proposition}
If $v$ splits in $F$, then $\partial \BJ_{\alpha ,v}(h, \Phi')=0$.
\end{proposition}	
\begin{proof}
This follows the argument in \cite[Lem. 14.2]{AFL} (see also \cite[Prop. 3.6]{AFL-Invent} for the group case).	
\end{proof}

	If $v$ is a non-split place (including the case $v | \infty $), consider the $v$-nearby hermitian space $V^{(v)}$ of $V$. As $\Phi'$ is a partial transfer of $\Phi$, the $(\gamma,u')$-term in the summation (\ref{def partial BJ a,v}) of $\partial \BJ_{\alpha ,v}(h, \Phi')$ is $0$, unless $(\gamma,u')$ matches an orbit $(g,u)\in [(\U(V^{(v)})\times V^{(v)})(F_0)]_\rs$.
	
	Let $v$ be a finite non-split place. For $h=h_\infty h_f \in \SL_(\BF_{0, \infty})\SL_(\BA_{0, f})$, as $\Phi'$ is Gaussian the decomposition (\ref{eq: Del BJ}) becomes 
	\begin{equation}\label{eq: important decomp}
		\partial \BJ_{\alpha ,v}(h, \Phi') = \sum_{\xi \in F_{0}, \xi \geq 0} \partial \BJ_{\alpha ,v}(\xi, 1, \omega(h_f)\Phi')  \frac{W_{\xi}^{(n)}(h_\infty)}{W_{\xi}^{(n)}(1)}.
	\end{equation}
	Here $W_{\xi}^{(n)}(-)$ is the product of standard weight $n$ Whittaker functions (\ref{defn: Whittaker function on SL2}) over infinite places $v | \infty$ of $F_0$.
	We normalize the value at $h=1$ by 
	$$
	\partial \BJ_{\alpha ,v}(\xi, \Phi'):= \frac{1}{W^{(n)}_\xi(1)}\partial \BJ_{\alpha ,v}(\xi,1,\Phi').
	$$
	By the computation of archimedean orbital integrals \cite[\S 12]{AFL}, we have the following decomposition
\begin{proposition}
\begin{equation}\label{key: DJ local-global}
	\partial \BJ_{\alpha ,v}(\xi,\Phi')=\sum_{(\gamma,u')\in [(S(V_0)(\alpha) \times V'_\xi)(F_0)]_\rs}  \del((\gamma,u'),\Phi'_v)\cdot \Orb((\gamma,u'),\Phi'^{v,\infty}).
\end{equation}	
\end{proposition}

The form (\ref{key: DJ local-global}) on the analytic side could be compared with the geometric side (see Proposition \ref{key: Int local-global} below).  Recall $K$ is the level associated to $L$ and $\Delta$. Choose $\Phi$ to be of the standard form $\Phi_i=\phi_{CM} \otimes \phi_i \in\CS((\U(V)\times V)(\BA_{0,f}))^{K}$ for $i=1, 2$, where
	\begin{enumerate}\label{Local-global: Assumption on Phi}
		\item For $v \notin \Delta$ and $v \not = v_0$, we have  $\Phi_{i, v}=1_{\U(L_v)}\otimes 1_{L_v}$. 
		\item At $v_0$, we have $\Phi_{1v_0}=1_{\U(L_{v_0})}\otimes 1_{L_{v_0}}$ and $\Phi_{2v_0}=1_{\U(L_{v_0})}\otimes 1_{L_{v_0}^\vee}$ .
	\end{enumerate}
	
 We have introduced Kudla--Rapoport divisors (\ref{def:KR int}) and derived Hecke CM cycles (\ref{defn: Derived CM}). Moreover, we have introduced relevant Green functions in \S \ref{section:arch} and the \emph{arithmetic Kudla--Rapoport divisors} (\ref{eq: Fourier coeff of arithmetic generating series}) for $\xi \in F_{0}$ and $?=\bf K, \bf B, \bf K- \bf B$:
 \begin{equation}
 	\wh \CZ^{?} (\xi, h_\infty, \phi_i)= (\CZ (\xi, \phi), ( \mathcal{G}^{?}_{\nu}(\xi, h_\nu, \phi_i))_{\nu} ) \in \wh{\Ch}^1(\CM).  
 \end{equation}

	Recall the volume factor $\tau(Z^\BQ)=\#Z^\BQ(\BQ) \backslash Z^\BQ(\BA_f)$. For $\xi \in F_0$, we define \emph{global arithmetic intersection numbers}
	\begin{equation}
		\Int^{?}(\xi, \Phi_i) := \frac
		{1}{\tau(Z^{\BQ}) [E:F]}
		(\wh \CZ^{?} (\xi, \phi_i), {}^{\BL}\mathcal{CM}(\alpha, \phi_{CM})) \in \BR_\Delta. 
	\end{equation}
We set $\Int(\xi, \Phi_i)=\Int^{\bf B}(\xi, \Phi_i)$.

 If $\xi \not =0$, we may lift $\Int(\xi, \Phi_i)$ to $\BR$ using the physical divisor $\wh \CZ^{\bf B} (\xi, \phi_i)$.

	\begin{theorem}\label{int support}
		Assume that $\xi \neq 0$, consider a place $\nu$ of $E$ above a finite place $v$ of $F_0$. Consider the support of $\CZ(\xi,\phi_i) \cap \mathcal{CM}(\alpha,\phi_{CM})$.
		\begin{altenumerate}
			\item  The support does not meet the generic fiber $\CM \otimes_{O_E} E $. 
			\item If $v$ is split in $F$, then the support does not meet the special fiber
			$\CM\otimes_{O_E} k_{\nu}$. 
			\item If $v$ is inert in $F$, then the support meets the special fiber 
			$\CM \otimes_{O_E} k_{\nu}$ only in its basic locus.
		\end{altenumerate}
	\end{theorem}
	\begin{proof}
		The proof of \cite[Theroem 9.2]{AFL} can be applied directly to our levels. 
	\end{proof}
	
	Hence $\Int(\xi, \Phi)$ $(\Phi=\Phi_i)$ are well-defined, and can be decomposed into a sum of archimedean contributions and non-archimedean contributions away from $\Delta$ in $\BR_{\Delta}$ ($v$ means a place of $F_0$.):
\begin{equation}\label{eq: decom of Int xi}
	\Int(\xi, \Phi_f)=\sum_{v| \infty} \Int_v^{\bf B}(\xi, \Phi_f)   + \sum_{v  \in \mathrm{Inert}(F/F_0) - \Delta } \Int_v(\xi, \Phi_f) \in \BR_{\Delta}.
\end{equation}

	\begin{proposition}\label{key: Int local-global}
		Assume $\xi\neq 0$ and $v \not \in \Delta$.
		\begin{enumerate}
			\item If $v$ is inert in $F$ and $v \not = v_0$, then 
\begin{equation}\label{Int v xi  Phi, good v}
\Int_{v}(\xi,\Phi_i) =  
2\log q_{v} \sum_{(g,u)\in [(\U(V^{(v)})(\alpha)\times V^{(v)}_\xi)(F_0)]_\rs} \Int_{v}(  g,u ) \cdot \Orb\left((g,u), \Phi^{v}_i \right).
\end{equation}
			Here $\Int_{v}(g, u)$ is the local quantity defined in the AFL (i.e., Conjecture \ref{conj: semi-Lie version ATC}) for rank $n$ self-dual lattices) with respect to the quadratic extension $F_{v}/ F_{0v}$. 
			\item   If $v=v_0$, then	
\begin{equation}\label{Int v xi  Phi, distinguished v0}
			\Int_{v_0}(\xi,\Phi_i) =  
			2\log q_{v_0} \sum_{(g,u)\in [(\U(V^{(v_0)})(\alpha)\times V^{(v_0)}_\xi)(F_0)]} \Int_{v_0}^i(  g,u ) \cdot \Orb\left((g,u), \Phi^{v_0}_i \right).
\end{equation}
			Here $\Int_{v_0}^1(g, u)$ (resp. $\Int_{v_0}^2(g, u)$) is the local quantity defined in Part $(1)$ (resp. $(2)$) of Conjecture \ref{conj: semi-Lie version ATC} for rank $n$ and type $t_0$ vertex lattice with respect to the quadratic extension $F_{v_0}/ F_{0v_0}$. 
			\item 	If $v|\infty$, then
\begin{equation}\label{Int v, archimedean}
	\Int_{v}^{\bf K}(\xi,\Phi_i) =  
\sum_{(g,u)\in [(\U(V^{(v)})(\alpha)\times V^{(v)}_\xi)(F_0)]} \Int_{v}(g,u ) \cdot \Orb\left((g,u), \Phi_i \right).
\end{equation}
			Here $\Int_{v}(g,u)$ is defined as the value $\Int_{v}(g,u )= \,{\bf\CG}^{\bf K}(u, h_\infty)(z_g)$ where $z_g$ is the unique fixed point of $g$ on the Grassmanian $X_v$.
		\end{enumerate}
		
	\end{proposition}
	\begin{proof}
		This follows the proof of \cite[Thm. 9.4, Thm. 10.1]{AFL}, replacing basic uniformization results in \emph{loc.cit.} by basic uniformization results for Kudla--Rapoport cycles and Hecke CM cycles proved in \S \ref{section: basic uni local-global}.
	\end{proof}
	
	We form the generating series ($h_\infty  \in \SL_2(F_{0, \infty})$)
	\begin{equation}
		\Int(h_\infty, \Phi) = \Int(0, \Phi) + \sum_{\xi \in F_{0,+} } \Int (\xi, \Phi) W_{\xi}^{(n)}(h_\infty).
	\end{equation}

\begin{definition}[standard partial transfers]\label{standard partial transfers}
Let $\Phi=\Phi_i$ be of the form (\ref{Local-global: Assumption on Phi}). A partial transfer $\Phi'=\Phi'_i$ of $\Phi$ is said to be standard, if 
\begin{enumerate}
	\item For $v \notin \Delta$ and $v \not = v_0$, we have  $\Phi_{v}'=1_{S(L_v)}\otimes 1_{L_{v,0}} \times 1_{L_{v,0}^*}$. 
\item At $v=v_0$, $\Phi_{1, v}'=1_{S(L_v, L_v^\vee)} \times 1_{L_{v,0}} \times 1_{(L_{v,0}^\vee)^*}$ and $\Phi_{2, v}'=1_{S(L_v, L_v^\vee)} \times 1_{L_{v,0}^\vee} \times 1_{(L_{v,0})^*}$ are standard test functions as in Definition (\ref{std function}).
\end{enumerate}
\end{definition}	
	
By the proven explicit Jacquet--Rallis transfers (Theorem \ref{proof of TC}) for $v \nmid \Delta$, the above definition makes sense and standard partial transfer always exists. 

Based on (\ref{key: DJ local-global}) and Proposition \ref{key: Int local-global}, we make the following comparisons for a finite non-split place $v \not \in \Delta$.  Let $\Phi'$ be a standard partial transfer of $\Phi$. 
We obtain the following key proposition.
\begin{proposition}\label{prop: global-semiglobal-local}
Consider a finite non-split place $v \not \in \Delta$ and $\xi \in F_0^\times$. Assume part $(i)$ of the arithmetic transfer conjecture \ref{conj: semi-Lie version ATC} for $L_v$ holds when $g \in \U(V^{(v)})(\alpha)$ and $(u,u)=\xi$. Then we have the semi-global identity
	\[
	2\partial \BJ_{\alpha, v}(\xi,\Phi'_i)= - \Int_{v}(\xi,\Phi_i) \in \BQ \log p.
	\]

\end{proposition}
\begin{proof}
	For matching orbits $(\gamma,u')\in [(S(V_0)(\alpha) \times V'_\xi)(F_0)]_\rs$ and 
	$[(\U(V^{(v)})(\alpha)\times V^{(v)}_\xi)(F_0)]_\rs$, we have
	\[
	\Orb((g,u), \Phi^v)=\Orb ((\gamma, u'), \Phi'^{v, \infty}).
	\]
	Part $(i)$ of the arithmetic transfer conjecture \ref{conj: semi-Lie version ATC} at $v$ for $(g,u)$ predicts
	\[
	\partial\Orb ((\gamma, u'), \Phi'_v)= - \Int_v(g, u) \log q_v.
	\]
	The proposition now follows from comparing the terms in equation (\ref{key: DJ local-global}) and the terms (\ref{Int v xi  Phi, good v}) and  (\ref{Int v xi  Phi, distinguished v0}) in Proposition \ref{key: Int local-global}.
\end{proof}

	\section{The end of proof}\label{section: the end}

	\subsection{The double induction method}\label{section: double induction}
	
	Recall $\psi$ is a fixed non-trivial additive character of $\BA_0$. For any finite place $v$ of $F_0$, let $\varpi_v$ be a uniformizer of $F_{0,v}$. Let $c_v$ be the conductor of $\psi$. By definition, $\psi_v$ is trivial on $\varpi_v^{-c_v}O_{F_{0,v}}$.   
	
	For a left $N^+(F_0)$-invariant  continuous function $f: \SL_2(\BA_0) \to \BC$ and $\xi \in F_0$, consider its $\xi$-th Fourier coefficient (\ref{defn: Fourier coeff}):
	\[
	W_{\xi, f}(h)=\int_{F_0 \backslash \BA_0}  f \left[ \begin{pmatrix} 1&b\\
		&1
	\end{pmatrix} h\right]\psi_{F_0}(-\xi b) db.
	\]
	Here we use left multiplication. There is a Fourier expansion (by an absolute convergent sum):
	\begin{align}\label{eq:def F exp}
		f(h)=\sum_{\xi\in F_0} W_{\xi,f}(h), \quad h\in  \SL_2(\BA_0).
	\end{align}

	\begin{lemma}\label{the double induction lemma}
		Let $B$ be a finite collection of finite places of $F_0$ with $v_0 \in B$. Let $(f_1, f_2)$ be a pair of continuous functions on $\SL_2(F_0) \backslash \SL_2(\BA_0)$ right invariant under some compact open subgroup $K_0 \subseteq \SL_2(\BA_0)$.  Assume
		
		\begin{enumerate}
			\item (Dual relation) We have $f_1\left(h \left(\begin{matrix}
				0 & \varpi_{v_0}^{-c_{v_0}} \\
				-\varpi_{v_0}^{c_{v_0}} & 0
			\end{matrix}\right) \right) = f_2(h), \quad \forall \, h \in \SL_2(\BA_0)$.
			\item (Support condition) $f_i$ are right invariant under 
			$$
			\prod_{v \in B, v \not=v_0} \diag\{\varpi_{v}^{-c_v},1\}\SL_2(O_{F_{0,v}}) \diag\{\varpi_{v}^{c_v},1\}.
			$$ For $v \in B, v \not=v_0$, set $k_{iv}:=-c_{v}$. Assume that there exists $k_{1v_0}, \, k_{2v_0} \in \BZ$ such that $f_i$ is right $\begin{pmatrix}
				1 & \varpi_{v_0}^{k_{iv_0}}O_{v_0} \\  0 & 1
			\end{pmatrix}$-invariant, and $k_{1v_0}+k_{2v_0} +2c_{v_0} \leq 1$.
			\item (Boundary condition) For all $\xi \in F_0$ with $v(\xi)=-k_{iv}-c_v$ for all $v \in B$ and any element $h_{\infty} \in \SL_2(F_{0,\infty})$, we have
			$$W_{\xi, f_i}(h_{\infty})=0. $$
			
		\end{enumerate}
		Then $f_1, f_2$ are both constant.
	\end{lemma}
	
	\begin{proof}
		Fix $i=1, 2$. Fix any element $b_{v_0} \in F_{0,v_0}$ such that $v_0(b_{v_0}) \geq k_{iv_0}-1$. Set $\wt B:=B - \{v_0\}$. Consider the function on $\SL_2(\BA_0)$:
		\[
		\wt f_{i}(h):=f_{i}(h\left(\begin{matrix}1 & b_{v_0}\\ &1\end{matrix}\right)_{v_0})-f_i(h).
		\]
		We claim that $\wt f_i$ is a constant. If $\wt B$ is empty, then $\wt{f}_i(h_{\infty}), h_{\infty} \in \SL_2(F_{0, \infty})$ is left invariant under $\left(\begin{matrix}
			1 & F_{0, \infty} \\
			0 & 1
		\end{matrix}\right)$ and right invariant under $K_{0,\infty}$, hence a constant. By strong approximation of $\SL_{2,F_0}$ at $\infty$, we see that $\wt{f}_i(h)$ is a constant. In general, we apply \cite[Lem 13.6]{AFL} to the pair $(\wt B, \wt f_{i})$. We check the assumptions in \emph{loc. cit.}: for any $v \in \wt B$, as $\SL_2(F_v)$ commutes with $\left(\begin{matrix}1 & b_{v_0}\\ &1\end{matrix}\right)_{v_0}$, so the invariance at $\wt B$ for $\wt {f}_i$ holds. Let $\xi \in F_0$ such that $v(\xi)=-c_{v}- k_{iv}=0$ for all $v \in \wt B$.  By condition $(2)$, $f_i$ is right $\begin{pmatrix}
			1 & \varpi_{v_0}^{k_{iv_0}}O_{v_0} \\  0 & 1
		\end{pmatrix}$-invariant.  If $v_0(\xi) \leq - k_{iv_0}-c_{v_0} -1$ then $\psi_{-\xi}$ is non-trivial on $\varpi_{v_0}^{k_{iv_0}}O_{v_0}$, which implies that  $W_{\xi, \wt f_i}(h_{\infty})=0$. By condition $(3)$, if $v_0(\xi)=-k_{iv_0}-c_{v_0}$ we have that $W_{\xi, f_i}(h_{\infty})=0$ hence $W_{\xi, \wt f_i}(h_{\infty})=0$. Moreover, if $v_0(\xi) \geq -v_0(b_{v_0})-c_{v_0}$ then $W_{ \xi, \wt f_i}(h_{\infty})=(\psi(\xi b_{v_0}) -1 ) W_{\xi, f_i}(h_{\infty})=0.$  As $v_0(b_{v_0}) \geq k_{iv_0}-1$, we see $W_{\xi, \wt f_i}(h_{\infty})=0$. We have checked the assumptions in \emph{loc. cit.}, hence $\wt{f}_i$ is a constant. 	
		
		Therefore, for any $\xi \in F_0^\times$, $W_{\xi, \wt{f}_i}=0$. Hence $W_{\xi, f_i}$ is right $\begin{pmatrix}
			1 & \varpi_{v_0}^{k_{iv_0}-1}O_{v_0} \\  0 & 1
		\end{pmatrix}$-invariant for $i=1,2$. We have 
		$$ W_{\xi,f_1} \left(h \left(\begin{matrix}
			0 & \varpi_{v_0}^{-c_{v_0}} \\
			-\varpi_{v_0}^{c_{v_0}} & 0
		\end{matrix}\right)_{v_0} \right) = W_{\xi, f_2}(h),
		$$ hence $W_{\xi, f_2}$ is right invariant under the subgroup of $\SL_2(F_{0, v_0})$ generated by $\begin{pmatrix}
			1 & 0 \\ \varpi^{k_{1v_0}-1+2c_{v_0}}_{v_0}O_{v_0} & 1
		\end{pmatrix}$ and $\begin{pmatrix}
			1 & \varpi_{v_0}^{k_{2v_0}-1}O_{v_0} \\  0 & 1
		\end{pmatrix}$.  As $k_{1v_0}+k_{2v_0}+2c_{v_0} \leq 1$, the subgroup is the whole group $\SL_2(F_{0, v_0})$: by a conjugacy we are reduced to the standard computation that $\begin{pmatrix}
			1 & 0 \\ \varpi^{-1}O_{v_0} & 1
		\end{pmatrix}$ and $\begin{pmatrix}
			1 & O_{v_0} \\  0 & 1
		\end{pmatrix}$ generates $\SL_2(F_{0, v_0})$. So $W_{\xi, f_2}(h)$ is right invariant under $\SL_2(F_{0,v_0})$ and therefore is a constant by strong approximation at $v_0$. It follows that, for all $\xi \in F_0^\times$, $W_{\xi, f_2}$ is a constant and hence must vanish (as $\psi$ is non-trivial). Therefore $f_2$ is a constant. As $f_1\left(h \left(\begin{matrix}
			0 & \varpi_{v_0}^{-c_{v_0}} \\
			-\varpi_{v_0}^{c_{v_0}} & 0
		\end{matrix}\right) \right) = f_2(h)$, $f_1$ is also a constant. The result follows.
	\end{proof}
	
	\begin{remark}
		In classical languages, consider a modular form $f$ of weight $k$ and level $\Gamma_0(N)$ and a prime $p|N$. If for all $(n,p)=1$, the Fourier coefficients $a_n(f)=0$ and $a_{n/p}(\begin{pmatrix}
			0 & 1 \\
			-1 & 0 
		\end{pmatrix}_p.f)=0$, then $f$ is a constant.
	\end{remark}

	\subsection{Globalization}
	
	Let $F_{v_0} / F_{0, v_0}$ be an unramified quadratic extension of $p$-adic local fields with residue fields $\BF_{q^2}/\BF_{q}$. Let $L_{v_0}$ be a vertex lattice of type $t_0$ in a $n$-dimensional $F_{v_0} / F_{0, v_0}$ hermitian space $V_{v_0}$. Denote by $V_{v_0}^{(v_0)}$ the nearby hermitian space of $V_{v_0}$. We have formulated the semi-Lie ATC \ref{conj: semi-Lie version ATC} for $L_{v_0}$ and any regular semisimple pair $(g_0, u_0) \in ( \U(V^{(v_0)}_{v_0} ) \times V_{v_0}^{(v_0)} )(F_{0, v_0})_\rs$. Now we give a proof assuming that $F_{0, v_0}$ is unramified over $\BQ_p$ if $0<t_0<n$.
	
	We can find a CM quadratic extension $F/F_0$ of a totally real field $F_0$ with a distinguished embedding $\varphi_0: F_0 \hookrightarrow \mathbb R$, a $n$-dimensional $F/F_0$-hermitian space $V$, a CM type $\Phi$ of $F$ such that
	\begin{itemize}
		\item $\Phi$ is unramified at $p$.
		\item There exists a place $v_0|p$ of $F_0$ inert in $F$ such that the completion of $F/F_0$ at $v_0$ is exactly $F_{v_0} / F_{0, v_0}$. All places $v \not = v_0$ of $F_0$ above $p$ are split  or inert in $F$. 
		\item $V$ is of signature $(n-1, 1)$ at $\varphi_0$, and $(n, 0)$ at all other real places of $F_0$. 
		\item The localization $V \otimes_{F} F_{v_0} $ is isomorphic to the hermitian space $V_{v_0}$ we start with. The space $V_p=\prod_{v|p} V_v$ contains a lattice $L_p = \prod_{v| p} L_v$ such that $L_v$ is self-dual for $v \not = v_0$ above $p$, and that $L_{v_0}$ a vertex lattice of type $t_0$ at $v_0|p$. 
	\end{itemize}
	
	If $0<t_0<n$, then we choose $F/ \BQ$ to be Galois hence the assumption in Theorem \ref{thm: almost global modularity} holds, i.e. the reflex field $E=F$. For example, we take  $F/F_0=\BQ(\zeta_m)/ \BQ(\zeta_m + \zeta_m^{-1}) $ with $m=q+1$. 
	
	Let $V^{(v_0)}$ be the nearby $F/F_0$ hermitian space of $V$ at $v_0$ so the localization of $V^{(v_0)}$ at $v_0$ is isomorphic to $V^{(v_0)}_{v_0}$. As in \cite[Section 10.1]{AFL2021}, by local constancy of intersection numbers \cite{AFLlocalconstant} (whose argument is quite general and applies to our Rapoport--Zink space $\CN_{L_{v_0}}$), we can find a pair $(g, u) \in ( \U(V^{(v_0)}) \times V^{v_0} )(F_0)_\rs $ such that
	
	\begin{itemize}
		\item $(g, u)$ is $v_0$-closely enough to $(g_0, u_0)$ such that two sides of Conjecture \ref{conj: semi-Lie version ATC} are the same for $(g, u)$ and $(g_0, u_0)$. The characteristic polynomial $\alpha$ of $g$ is irreducible over $F$.
		\item The norm $\xi_0= (u,  u)_{V^{v_0}}  \in F_0$ is non-zero.
		\item For all places $v \not = v_0$ of $F_0$ above $p$, the orbital integral for the $\U(V_v)$-orbit of $(g, u)$
		\[
		\Orb((g, u), 1_{\U(L_v)} \times 1_{L_v} ) \not =0. 
		\]
		Moreover,  if $v$ is inert in $F$, then $\alpha_v$ is of maximal order. 
	\end{itemize}
	
	Choose a finite collection $\Delta$ of finite places of $F_0$ such that $p \nmid \Delta$, i.e. all places $v$ of $F_0$ above $p$ are not in $\Delta$.  
	
	We can enlarge $\Delta$ (keeping $p \nmid \Delta$) to assume that $\Delta$ and $L$ are in the set up of \S \ref{section: distinguished v0} and $\Phi$ is unramfied at all $\ell \not \in \Delta$. Consider a level structure $\wt{K}$ for $L$ and $\Delta$. We consider the RSZ integral model 
	$$
	\CM_{\wt G, \wt K} \to \Spec O_E[\Delta^{-1}]
	$$ for the RSZ Shimura variety associated to $V$ with level for $\Delta$ and $L$. By \cite[Prop. 13.8]{AFL}, we can enlarge $\Delta$ and shrink $K_{G, \Delta}$ to find $\phi_{\Delta} \in \CS(V(F_{0, \Delta}))$ and $\phi_{CM, \Delta} \in \CS(K_{G, \Delta} \backslash \U(V)(F_{0, \Delta}) / K_{G, \Delta})$ such that the following holds:
	\begin{itemize}
		\item $L$ is self-dual and $\alpha(t)$ is a maximal order at any finite place $v \nmid \Delta$ and $v \not =v_0$. 
		\item For the function $\Phi^{v_0}=(\phi_{CM, \Delta}  \otimes 1_{K_G^{\Delta, v_0}}) \otimes (\phi_{ \Delta} \otimes 1_{\wh{L}^{\Delta, v_0}} ) \in \CS(\U(V(\BA_{0,f}^{v_0})) \times V(\BA_{0,f}^{v_0} )), $
		its orbital integral $\Orb((g_1, u_1), \Phi^{v_0})$ at any $(g_1, u_1) \in (\U(V^{(v_0)}) \times V^{(v_0)})(F_0)$ is zero, unless $(g_1,u_1)=(g,u)$ in $[(\U(V^{(v_0)}) \times V^{(v_0)})(F_0)]_{\rs}$. Moreover, $\Orb((g,u), \Phi^{v_0})$ is non-zero.
	\end{itemize}
	
Let $\Phi'$ be a standard partial transfer of $\Phi$ (Definition \ref{standard partial transfers}).
	
	\begin{theorem}\label{prop: local ATC is from global id of Fourier coeff derived at v0}
	\begin{enumerate}
		\item For $i=1,\,2$, part $(i)$ of the Conjecture \ref{conj: semi-Lie version ATC} for $(g_0, u_0)$ follows from the semi-global identity:
		\begin{equation}\label{key: DJ=Int v0 xi}
			2\partial \BJ_{\alpha, v_0}(\xi_0,\Phi'_i)= - \Int_{v_0}(\xi_0,\Phi_i) \in \BQ \log p.
		\end{equation}
	\item The semi-global identity (\ref{key: DJ=Int v0 xi}) is equivalent to the global identity
	\begin{equation}\label{key: DJ=Int global xi}
		2\partial \BJ_{\alpha}(\xi_0,\Phi'_i) + \Int^{\bf K-\bf B}(\xi_0, \Phi_i) = - \Int(\xi_0,\Phi_i) \in \BR_{\Delta}.
	\end{equation}
	\end{enumerate}	
	\end{theorem}
	\begin{proof}
	The first property, we use the comparison argument in Proposition \ref{prop: global-semiglobal-local}. As the orbital integrals
	$\Orb((g,u), \Phi^{v_0})=\Orb ((\gamma, u'), \Phi'^{v_0, \infty})$ are non-zero, we can divide it and deduce local AT identity (Conjecture \ref{conj: semi-Lie version ATC} ) from the semi-global identity (\ref{key: DJ=Int v0 xi}).
	
	For the second property, using the computation (\ref{Int v, archimedean}) in Proposition \ref{key: Int local-global} and the decomposition (\ref{eq: decom of Int xi}), we see 
	\[
		2\partial \BJ_{\alpha}(\xi_0,\Phi'_i) + \Int^{\bf K-\bf B}(\xi_0, \Phi_i) + \Int(\xi_0,\Phi_i) 
	\] 
is equal to ($\ell$ runs over prime numbers not lying in $\Delta$)
	\[
\sum_{v  \in \mathrm{Inert}(F/F_0) - \Delta } 2\partial \BJ_{\alpha, v}(\xi_0, \Phi'_i) + \Int_{v}(\xi_0,\Phi_i) \in \sum_{\ell \not \in \Delta} \BQ \log \ell \subseteq \BR_\Delta.
	\]
We use $\BQ$-linear independence of all $\log \ell$. We see the global identity is equivalent to the identities $(\ell)$ for all $\ell$ 
\[
(\ell), \, \, \,  \sum_{v  \in \mathrm{Inert}(F/F_0) - \Delta, v| \ell}  2\partial \BJ_{\alpha, v}(\xi_0, \Phi'_i) = - \Int_{v}(\xi_0,\Phi_i) = 0 \in \BQ \log \ell \subseteq \BR_\Delta. 
\]	
If $\ell \not =p$. Then the identity $(\ell)$ is known by the known AFL \cite{AFL2021} and Proposition \ref{prop: global-semiglobal-local}.  Similarly, the semi-global identity is known for all $v \not = v_0$ above $p$ by the known AFL \cite{AFL2021} and Proposition \ref{prop: global-semiglobal-local}. Hence the result follows.
	\end{proof}

	\subsection{Modification}	
	
	Possibly enlarging $\Delta$ (keeping $p \nmid \Delta$), we do the following modifications (similar to the proof of Theorem \ref{thm: global modularity}):
	
	\begin{enumerate}
		\item (After enlarging $\Delta$) We choose another irreducible conjugate self-reciprocal monic polynomial $\alpha_{m} \in O_F[1/\Delta][t]$ of degree $n$ such that $\alpha_{m}$ is a maximal order away from $\Delta$ and unramified at $v_0$. Choose $\mu_{0} \in K_{G, \Delta} \backslash G(F_{0, \Delta})/K_{G, \Delta}$ such that the cycle $\CC_{\alpha_m}:=\mathcal{CM}(\alpha_{m}, \mu_{0})$ is non-empty. From \S \ref{section: equidistribution}, there exists a rational number $\lambda_0 \in \BQ$, such that 
		\[
		\CC_{1}={}^\BL \mathcal{CM}(\alpha, \phi_{CM})-\lambda_0 \CC_{\alpha_m}
		\]
		is of degree $0$ over $\BC$. On the analytic side, the corresponding contribution is 
		$$
		2 \lambda_0 \partial \BJ_{\alpha_{m}} (\xi, \Phi_{im}')
		$$
		where $\Phi'_{im}$ are (relative to $\alpha_{m}$) Gaussian transfers for $\Phi_{im}:= (\Phi_i)^\Delta  \otimes (1_{\mu_0} \otimes \phi_\Delta )$. We can enlarge $\Delta$ such that $\Phi_{im}$ are also standard away from $\Delta$.
		By the known ATC (and AFL) for unramified maximal orders (Section \ref{section: ATC unram max order}), the local intersection number at $v_0$ for $\CC_{\alpha_m}$ can be computed directly to match the $v_0$-part of $2 \lambda_0 \partial \BJ_{\alpha_{m}} (\xi, \Phi_{im}')$ on the analytic side.
		\item By \S \ref{section: mod over BF}, for any place $\nu_0$ of $E$ over $v_0$ we can find a very special $1$-cycle $\CC_{\nu_0}$ in the basic locus of $\CM_{k_{\nu_0}}$ such that 
		$(X_j, \CC_1)_\CM=(X_j, \CC_{\nu_0})_\CM$ for all irreducible components $X_j$ of $\CM_{k_{\nu_0}}$.  On the analytic side, the corresponding contribution is
		$(\wh \CZ^{B} (\xi, \phi_i), \CC_{\nu_0})$ 
		which is $\xi$-th Fourier coefficient of a theta series $\Theta(h, \phi_{\CC}' )$  
		on $V^{(v_0)}$ hence modular by Theorem \ref{thm: int in basic locus = theta series}.
	\end{enumerate}
	
	For $i=1,2$, consider the $1$-cycle (independent of $i$)
	\begin{equation}\label{difference cycle}
		\CC^{\perp}:={}^\BL \mathcal{CM}(\alpha, \phi_{CM})-\lambda_0 \CC_{\alpha_m}-\sum_{\nu_0| v_0} \CC_{\nu_0} \in \CC_{1}(\CM).
	\end{equation}
	We define the difference function on the geometric side as the modified generating function on $h \in \SL_2(\BA_0)$:
	\begin{equation}
		\Int(h, \Phi_i)^\perp :=  \frac
		{1}{\tau(Z^{\BQ}) [E:F]}
		(Z(h, \phi_i), \CC^{\perp})^{\text{adm}},
	\end{equation}
	and the modified archimedean intersection number on $h \in \SL_2(\BA_0)$:
	\begin{equation}\label{archimedean int number?}
		\Int^{\bf K- \bf B}(h, \Phi_i)^\perp:= \frac{1}{\tau(Z^\BQ)[E:F]}{\sum_{\nu | \infty}} \CG^{\bf K - \bf B}_\nu (h, \phi_i) (\CC^\perp).
	\end{equation}

For $\xi \in F_0$, we may consider its $\xi$-th Fourier coefficient $\Int^{\bf K- \bf B}(\xi, \Phi_i)^\perp$. 

	We define the difference function on the analytic side:
	\begin{equation}\label{sum by sum BJ partial}
		\partial \BJ_{\rm hol, i}(h) :=2 \partial \BJ_{\alpha}(h, \Phi'_i) -  2 \lambda_0 \partial \BJ_{\alpha_{m}} (h, \Phi_{im}') + \sum_{\nu_0| v_0} \Theta(h, \phi'_{i, \CC_{\nu_0}}) + \Int^{\bf K- \bf B}(h, \Phi')^{\perp}. 
	\end{equation}

For $\xi \in F_0$, we may consider its $\xi$-th Fourier coefficient $\partial \BJ_{\rm hol, i}(\xi)$. 

By modularity on the generic fiber (Theorem \ref{modularity over generic fiber}) and Poisson summation formulas, we know $\Int^{\bf K- \bf B}(h, \Phi_i)^\perp$ and $\partial \BJ_{\rm hol, i}(h)$  are automorphic forms on $\SL_2(A_0)$ and $\partial \BJ_{\rm hol, i}$ is holomorphic. 

	For a compact open subgroup $K_{0, \Delta} \leq \SL_2(F_{0, \Delta}) $, consider the compact open subgroup $K_0 \subseteq \SL_2(\BA_{0, f})$ given by
	\[
	K_0= \diag\{c_{v_0}^{-1}, 1\}\Gamma(v_0) \diag\{c_{v_0}, 1\} \times K_0^{\Delta, v_0, \circ} \times K_{0, \Delta} \leq \SL_2(\BA_{0,f}).
	\] 
	where $K_0^{\Delta, v_0, \circ}:= \prod_{v \not \in \Delta, v \not = v_0} \diag\{c_v^{-1}, 1\}\SL_2(O_{F_0,v}) \diag\{c_v, 1\} $ and $\Gamma(v_0):=\Ker (\SL_2(O_{F_{0,v_0}}) \to \SL_2(k_{v_0}))$. Choose $K_{0, \Delta}$ to be sufficiently small so that the pair $\Phi_1$ is invariant under the Weil representation by $K_0$. Recall we use the standard additive character $\psi=\psi_{F_0}=\psi_\BQ \circ \tr_{F_0/\BQ}$ on $\BA_0$ for the Weil representation. The level $c_v$ of $\psi$ at $v$ is $0$ if $F_0/ \mathbb Q$ is unramified at $v$. In particular, by assumption $c_{v_0}=0$ if $0<t_0<n$.
	
	\begin{corollary}\label{key: int modularity}
		For $i=1, 2$, the generating function $\Int(h, \Phi_i)^\perp $ lies in $\CA_{\rm hol}(\SL_2(\BA_0), K_0, n)_{\ov \BQ} \otimes_{\ov \BQ} \BR_{\Delta, \ov \BQ}$. Moreover, if $c_{v_0}=0$ the dual relation holds:
		\[
		\Int(h\begin{pmatrix}
			0 & 1 \\ -1 & 0 
		\end{pmatrix}_{v_0}, \Phi_1)^\perp=\gamma_{V_{v_0}}\vol(L_{v_0}) \Int(h, \Phi_2)^\perp.
		\]
	\end{corollary}	
	\begin{proof}
		The first part follow from Theorem \ref{thm: almost global modularity}. If $c_{v_0}=0$, we know $\begin{pmatrix}
			0 & 1 \\ -1 & 0 
		\end{pmatrix}_{v_0}.1_{L_{v_0}}=\gamma_{V_{v_0}}\vol(L_{v_0})  1_{L_{v_0}^\vee}$ hence $\begin{pmatrix}
			0 & 1 \\ -1 & 0 
		\end{pmatrix}_{v_0}.\Phi_1=\gamma_{V_{v_0}}\vol(L_{v_0})  \Phi_2$. Then the dual relation follows from equation \ref{eq: dual relation on generic fiber}.
	\end{proof}		
	
	\begin{proposition}\label{key: DJ modularity}
		For $i=1, 2$, the difference function $\partial \BJ_{\rm hol, i}(h) $ lies in $\CA_{\rm hol}(\SL_2(\BA_0), K_0, n)$. Moreover, if $c_{v_0}=0$  the dual relation holds:
		\begin{equation}
			\partial \BJ_{\rm hol, 1}(h \begin{pmatrix}
				0 & 1 \\ -1 & 0 
			\end{pmatrix}_{v_0})=\gamma_{V_{v_0}}\vol(L_{v_0})   \partial \BJ_{\rm hol, 2}(h).
		\end{equation}
	\end{proposition}
	\begin{proof}
		We show the dual relation holds summand by summand in the decomposition (\ref{sum by sum BJ partial}). 
		
		The archimedean part $\Int^{\bf K - \bf B}(h, \Phi')^\perp$ follows from \cite[Prop. 10.5]{AFL2021}. We compute  $\gamma_{V_{v_0}}=(-1)^t$ and $\vol(L_{v_0})=q^{-t}_{v_0}$, hence  by Theorem \ref{thm: partial local modularity} (where $c_{v_0}=0$) we have $\begin{pmatrix}
			0 & 1 \\ -1 & 0 
		\end{pmatrix}_{v_0}. \phi'_{1, \CC_{\nu_0}}= \gamma_{V_{v_0}}\vol(L_{v_0})  \phi'_{2, \CC_{\nu_0}}.$
		So we have the dual relation for two theta series $\Theta(h, \phi'_{i, \CC_{\nu_0}})$ lying in $\CA_{\rm hol}(\SL_2(\BA_0), K_0, n)$. 
		
		The dual relations for $\partial \BJ_\alpha(h , \Phi_i')$ and $\partial \BJ_{\alpha_m}(h , \Phi_{im}')$ follows from the equivariant definition $\BJ_{\alpha}(h,\Phi',s)$ in equation (\ref{defn of BJ}). 
	\end{proof}

	\subsection{Proof by double induction}\label{section: the last}
	
	We may substract the global identity (\ref{key: DJ=Int global xi}) by known global identities.
	
	\begin{theorem}\label{thm: subtract local mod term}
		For $i=1,2$, the Conjecture \ref{conj: semi-Lie version ATC} for $(g_0, u_0)$ follows from the identity of Fourier coefficients
		\begin{equation}
			\partial \BJ_{\rm hol, i}(\xi_0) = - \Int (\xi_0, \Phi_i)^{\perp}. 
		\end{equation}
	\end{theorem}
	\begin{proof}
		We use Theorem \ref{prop: local ATC is from global id of Fourier coeff derived at v0}. Assume $\partial \BJ_{\rm hol, i}(\xi_0) = - \Int (\xi_0, \Phi_i)^{\perp}$ holds. We only need to show the global identity (\ref{key: DJ=Int global xi})
		\[
		2\partial \BJ_{\alpha}(\xi_0,\Phi'_i) + \Int^{\bf K-\bf B}(\xi_0, \Phi_i) = - \Int(\xi_0,\Phi_i) \in \BR_{\Delta}
		\]
	We apply  Proposition \ref{prop: global-semiglobal-local} and Theorem \ref{prop: local ATC is from global id of Fourier coeff derived at v0} to the maximal order $\alpha_m$. As we already prove the AT conjecture for unramified maximal orders in particular for $\alpha_m$ (Section \ref{section: ATC unram max order}), we have 
	  \[
	  2\partial \BJ_{\alpha_m}(\xi_0,\Phi'_{i,m}) + \Int^{\bf K-\bf B}(\xi_0, \Phi_{i, m}) = - \Int(\xi_0,\Phi_{i,m}) \in \BR_{\Delta}.
	  \]
	By the computation of the $\xi_0$-th Fourier coefficient $\Theta(\xi_0, \phi'_{i, \CC_{\nu_0}})$ in Theorem \ref{thm: int in basic locus = theta series}, we have
	\[
	\Theta(\xi_0, \phi'_{i, \CC_{\nu_0}})= (\wh \CZ^{\bf B} (\xi, \phi_i), \CC_{\nu_0}).
	\]
Add above known global terms together with respect to the decomposition of the analytic side (\ref{sum by sum BJ partial}) and the geometric side $\CC^\perp$
(\ref{difference cycle}), the proposition is proved.
	\end{proof}
	
	We now do inductions and assume the semi-Lie version Conjecture \ref{conj: semi-Lie version ATC} holds for all maximal parahoric levels for all dimensions (of the hermitian space) smaller than $n$. 
	
	\begin{theorem}\label{last theorem}
		For $i=1, 2$, we have the identity
		\[
		\partial \BJ_{\rm hol, i}(h) = - \Int (h, \Phi_i)^{\perp}. 
		\]
		in the function space $\CA_{\rm hol}(\SL_2(\BA_0), K_0, n)_{\ov \BQ} \otimes_{\ov \BQ} \BR_{\Delta, \ov \BQ}$.	In particular, we have $\partial \BJ_{\rm hol, i}(\xi_0) = - \Int (\xi_0, \Phi_i)^{\perp}$ . 
	\end{theorem}
	\begin{proof}
		
		By Corollary \ref{key: int modularity} and Proposition \ref{key: DJ modularity}, both sides are holomorphic modular forms.  For $i=1, 2$, denote the differences by $f_i(h):=\partial \BJ_{\rm hol, i}(h) + \Int (h, \Phi_i)^{\perp}$.  The theorem follows by applying the double induction lemma \ref{the double induction lemma} for the pair $(f_1, (\gamma_{V_{v_0}}\vol(L_{v_0})) f_2)$ and $B=\{v_0\}$ with $k_{1v_0}=-c_{v_0}, \, k_{2v_0}=-c_{v_0}+1$. For this, we check the assumptions in Lemma \ref{the double induction lemma}: 
		\begin{enumerate}
			\item we may assume that $F_{0v_0}/\mathbb Q_p$ is unramified so $c_{v_0}=0$ if $0<t_0<n$. For general $c_v$, it follows from how the local Weil representation change under the change of additive characters. The dual relation in Lemma \ref{the double induction lemma} holds by Corollary \ref{key: int modularity} and Proposition \ref{key: DJ modularity}.
			\item the valuation of $L_{v_0}$ and $L_{v_0}^\vee$ is $0$ and $-1$ respectively, so the support condition in Lemma  \ref{the double induction lemma} follows. 
		\end{enumerate} Now we only need to show the boundary condition in Lemma \ref{the double induction lemma}, namely  equalities of Fourier coefficients 
		\[
		\partial \BJ_{\rm hol, i}(\xi, h_{\infty}) = - \Int (\xi, h_{\infty}, \Phi_i)^{\perp}.
		\]
		for all $\xi \in F_{0}^\times$ such that $v_0(\xi)=0$ (resp. $v_0(\xi)=-1$) if $i=0$ (resp. $v_0(\xi)=-1$). Since both sides are holomorphic of the same weight $n$, we can assume $h_{\infty}=1$. It is sufficient to prove the following two claims:
		
		`\emph{Claim $\CZ$. For $\xi$ with $v_0(\xi) = 0$, we have $
			\partial \BJ_{\rm hol, 1}(\xi) = - \Int (\xi, \Phi_1)^{\perp} \in \BR_\Delta.$} 
		
		\emph{Claim $\CY$. For $\xi$ with $v_0(\xi) = -1$, we have $\partial \BJ_{\rm hol, 2}(\xi) = - \Int (\xi, \Phi_2)^{\perp} \in \BR_\Delta.$}
		
		Adding these local error terms at $\xi$ again, the Claim $\CZ$ follows if we can show the identity
		\[
		2\partial \BJ_{\alpha}(\xi,\Phi'_1) + \Int^{\bf K-\bf B}(\xi, \Phi_1) = - \Int(\xi,\Phi_1) \in \BR_\Delta.
		\]
		By the factorization and vanishing of archimedean terms, it is sufficient to show 
		\[
		2\partial \BJ_{\alpha, v}(\xi,\Phi'_1) + \Int^{\bf K-\bf B}(\xi, \Phi_1) = - \Int_{ v}(\xi,\Phi_1) \in \BR.
		\]
		holds for any finite place $v \not \in \Delta$ of $F_0$. By local-global decompositions on both sides (see Proposition \ref{key: Int local-global} and equation \ref{key: DJ local-global}), it is sufficient to prove the local identity
		\[
		\del\bigl ((\gamma,u'), \Phi'_v) =- \Int_v(g,u) \log q_v.
		\]
		We prove this. If $v =v_0$, then $v_0(\xi)=0$, hence the ATC identities $\del\bigl ((\gamma,u'), \Phi'_{iv_0}) =- \Int_{v_0}^i(g,u)$ holds by Theorem \ref{thm: semi-Lie and gp} and induction. If $v \not = v_0$, then due to our enlargement of $\Delta$ we know that $O_F[t]/(\alpha(t))$ is a maximal order at $v$. In this case the above identity $\del\bigl ((\gamma,u'), \Phi'_v) =- \Int_v(g,u) \log q_v$ is also known by the known maximal order AFLs in \cite[Cor. 9.1]{M-Thesis} (which is reproved for unramified maximal orders in \S \ref{section: ATC unram max order}).
		This finishes the proof of Claim  $\CZ$. The Claim $\CY$ is proved in the same way by applying the dual isomorphism. 
	\end{proof}	
	
	By Theorem \ref{thm: subtract local mod term} and Theorem \ref{last theorem},  the proof of arithmetic transfer identities for the vertex lattice $L_{v_0}$ (\ref{conj: semi-Lie version ATC}) is complete, assuming that $F_0$ is unramified over $\BQ_p$ if $0<t_0<n$.	
	
	\qed


\begin{thebibliography}{99}
		
		\bibitem{Atiyah}{M. Atiyah, \textit{On analytic surfaces with double points}, Proc. R. Soc. Lond. Ser. A 247, 237–244 (1958).}
		
		\bibitem{RBP19}{R. Beuzart-Plessis, \textit{a new proof of Jacquet-Rallis fundamental lemma}, Duke Math. J. 170(12) (2021), 2805-2814. }
		
		\bibitem{JBGreen}{J. Bruinier, \textit{Regularized theta lifts for orthogonal groups over totally real fields}, J. Reine Angew. Math. 672 (2012), 177–222.}
		
		\bibitem{BHKRY}{J. Bruinier, B. Howard, S. Kudla, M. Rapoport and T. Yang,  \textit{Modularity of generating series of divisors on unitary Shimura varieties}, Astérisque 421 (2020): 7-125.}
		
		\bibitem{Boumasmoud}{R. Boumasmoud, \textit{General Horizontal Norm Compatible Systems on unitary Shimura varieties}, \href{https://arxiv.org/abs/2106.12540}{\texttt{arXiv:2106.12540}}. }
		
		\bibitem{Cho18}{S. Cho, \textit{The basic locus of the unitary Shimura variety with parahoric level structure, and special cycles}, \href{https://arxiv.org/abs/1807.09997}{\texttt{ 	arXiv:1807.09997}} .}
		
		\bibitem{Deligne79}{P. Deligne. \textit{Vari\'et\'es de Shimura: interpr\'etation modulaire, et techniques de con struction de mod\'eles canoniques}. In \textit{Automorphic forms, representations and L-functions} (Proc. Sympos. Pure Math., Oregon State Univ., Corvallis, Ore., 1977), Part 2, Proc. Sympos. Pure Math., XXXIII, pages 247–289. Amer. Math. Soc., Prov idence, R.I., 1979.}
		
		
		
		\bibitem{GGP}{W. T. Gan, B. H. Gross, and D. Prasad, \textit{Symplectic local root numbers, central critical
				$L$-values, and restriction problems in the representation theory of
				classical groups},  Ast\'erisque \textbf{346} (2012), 1--109.}
		
		\bibitem{GS90}{H. Gillet and C. Soul\'e, \textit{Arithmetic intersection theory}, Inst. Hautes Etudes Sci. Publ. Math. 72 (1990), 93–174.}
		
		
		\bibitem{Gortz01}{U. G\"ortz, \textit{On the flatness of models of certain Shimura varieties of PEL-type}. Math. Ann. \textbf{321} (2001), no. 3, 689--727.} 
		
		\bibitem{ADLV-Coexter}{U. G\"ortz, X. He and S. Nie, \emph{Basic loci of Coxeter type with arbitrary parahoric level}, Canadian Journal of Mathematics (2022): 1-47. }
		
		\bibitem{HodgeNewtonBasic}{U. G\"ortz, X. He, and S. Nie. \textit{Fully hodge–newton decomposable
			shimura varieties}, Peking Mathematical Journal, 2(2):99–154, 2019. 102}
		
		\bibitem{Gross-canonical}{B. H. Gross, \textit{On canonical and quasi-canonical liftings}, Invent. Math. 84 (1986): 321-326. }
		
		\bibitem{GZ}{B. H. Gross and D. Zagier, \textit{Heegner points and derivatives of $L$-series}, Invent. Math. 84 (1986), no. 2, 225--320.}
		
		
		\bibitem{Hansen}{D. Hansen, \textit{On the supercuspidal cohomology of basic local Shimura varieties},  J. reine angew. Math., to appear.}
		
		\bibitem{IrredKR}{P. van Hoften, \textit{Mod p points on Shimura varieties of parahoric level (with an appendix by R. Zhou)}, \href{https://arxiv.org/abs/2010.10496v1}{\texttt{ arXiv:2010.10496v1}. } 	}
		
		\bibitem{Ho-CM}{B. Howard, \textit{Complex multiplication cycles and Kudla--Rapoport divisors}, Ann. of Math. (2) \textbf{176} (2012), no. 2, 1097--1171.}
		
		
		\bibitem{HMP2020}{B. Howard and K. Madapusi Pera. \textit{Arithmetic of Borcherds products}. Ast\'erisque, (421, Diviseurs arithm\'etiques sur les vari\'et\'es orthogonales et unitaires de Shimura):187–297, 2020. 47. }
		
		\bibitem{HSY20}{Q. He, C. Li, Y. Shi and T. Yang, \textit{A proof of the Kudla–Rapoport conjecture for Kramer models}, Inventiones mathematicae, 234(2), 721-817.}
		
		\bibitem{HR17-Axiom}{X. He and M. Rapoport, \textit{Stratifications in the reduction of Shimura varieties}, Manuscripta Math. 152 (2017), no. 3-4,	317–343.}
		
		\bibitem{ADLV-hyperspecial}{X. He, R. Zhou and Y. Zhu, \textit{Stabilizers of irreducible components of affine Deligne--Lusztig varieties}, Journal of the European Mathematical Society, To appear. }
		
		\bibitem{JR-GGP}{H. Jacquet and S. Rallis, \textit{On the Gross-Prasad conjecture for unitary groups}, in On Certain L-Functions, Clay Math. Proc. 13, Amer. Math. Soc., Providence, RI, 2011, pp. 205--264.}
		
		\bibitem{classicalherm}{R. Jacobowitz, \textit{Hermitian Forms Over Local Fields}. American Journal of Mathematics, 84(3), (1962), 441-465.}	
		
		\bibitem{Kottwitz-OrbGL3}{R. Kottwitz, \textit{Orbital integrals on $GL_3$}. American Journal of Mathematics, 102(2), (1980), 327–384.}
		
		\bibitem{Kudla97Duke}{S. Kudla, \textit{Algebraic cycles on Shimura varieties of orthogonal type}, Duke Math. J. 86 (1997), 39–78.}
		
		\bibitem{Kudla97Ann}{S. Kudla, \textit{Central derivatives of Eisenstein series and height pairings}, Ann. of Math. (2) 146 (1997), no. 3}.
		
		\bibitem{Kudla-Eisenstein}{S. Kudla, \textit{Special cycles and derivatives of Eisenstein series}, In Heegner points and Rankin L- series, volume 49 of Math. Sci. Res. Inst. Publ., pages 243–270. Cambridge Univ. Press, Cambridge, 2004}
		
		\bibitem{KR-height}{S. Kudla and M. Rapoport, \textit{Height pairings on Shimura curves and p-adic uniformization}. Invent. math., 142 (2000), pp. 153–222.	}
		
		\bibitem{KR-Drinfeld plane}{S. Kudla and M. Rapoport, \textit{An alternative description of the Drinfeld p-adic half-plane, Annales de l’Institut
				Fourier 64, no. 3 (2014)}, 1203–1228.}
		
		\bibitem{KR-local}{S. Kudla and M. Rapoport, \textit{Special cycles on unitary Shimura varieties I. Unramified local theory}, Invent. Math. \textbf{184} (2011), no. 3, 629--682.}
		
		\bibitem{KR-global}{S. Kudla and M. Rapoport, \textit{Special cycles on unitary Shimura varieties II: Global theory},  J. Reine Angew. Math. \textbf{697} (2014), 91--157.}
		
		\bibitem{KR-2004}{S. Kudla, M. Rapoport and T. Yang, \textit{ Derivatives of Eisenstein series and Faltings heights},  Compositio Mathematica, 140(4), (2004), 887-951. 
		}
		
		
		\bibitem{Liu-Fourier-Jacobi}{Y. Liu, \textit{Fourier-Jacobi cycles and arithmetic relative trace formula}, Cambridge Journal of Mathematics, 9 (2021) 1—147.}
		
		\bibitem{Liu-thesis}{Y. Liu, \textit{Arithmetic inner product formula for unitary groups}, ProQuest LLC, Ann Arbor, MI, 2012. Thesis (Ph.D.)–Columbia University.}
		
		\bibitem{LTXZZ}{Y. Liu, Y. Tian, L. Xiao, W. Zhang and X. Zhu, \textit{On the Beilinson--Bloch--Kato conjecture for Rankin--Selberg motives}, Inventiones mathematicae 228.1 (2022): 107-375. }
		
		\bibitem{Li-Liu2020}{C. Li and Y. Liu, \textit{Chow groups and L-derivatives of automorphic motives for unitary groups}.  Ann. of Math. 194 (2021) 817—901.}
		
		\bibitem{LZ19}{C. Li and W. Zhang, \textit{Kudla--Rapoport cycles and derivatives of local densities}, Journal of the American Mathematical Society 35.3 (2022): 705-797.}	
		
		\bibitem{AFL-linear}{Q. Li, \textit{An intersection number formula for CM cycles in Lubin-Tate towers}, Duke Math. J, to appear.}
		
		\bibitem{M-AFL}{A. Mihatsch, \textit{On the arithmetic fundamental lemma conjecture through Lie algebras}, Math. Z. \textbf{287} (2017), no. 1--2, 181--197.}
		
		\bibitem{M-Thesis}{A. Mihatsch, \textit{Relative unitary RZ-spaces and the arithmetic fundamental lemma}, Journal of the Institute of Mathematics of Jussieu. 2022;21(1):241-301.}
		
		
		\bibitem{AFL2021}{A. Mihatsch and W. Zhang, \textit{On the Arithmetic Fundamental Lemma conjecture over a general p-adic field}, J. Eur. Math. Soc. (2023). }
		
		\bibitem{AFLlocalconstant}{A. Mihatsch, \textit{Local Constancy of Intersection Numbers}, Algebra \& Number Theory 16.2 (2022): 505-519.}
		
		\bibitem{Milne}{J. S. Milne, \textit{Introdution to Shimura varieties}, Lecture notes, \href{https://www.jmilne.org/math/xnotes/svi.pdf}{https://www.jmilne.org/math/xnotes/svi.pdf}. }
		
		\bibitem{Morin-Strom}{K. A.  Morin-Strom . \textit{Witt’s Theorem for Modular Lattices}. American Journal of Mathematics, 101(6), (1979), 1181–1192.}
		
		\bibitem{Ngo-FL}{B. Chau Ngo, \textit{Le lemme fondamental pour les alg`ebres de Lie,} Publ. Math. Inst. Hautes Etudes Sci. 111 (2010), 1–169.}	
		
		\bibitem{Pappas-Rapoport}{G. Pappas and M. Rapoport, \textit{p-adic shtukas and the theory of global and local Shimura varieties}, \href{https://arxiv.org/abs/2106.08270}{\texttt{arxiv: 2106.08270}} }.
		
		\bibitem{Qiu-mod}{C. Qiu, \textit{Arithmetic modularity of special divisors and arithmetic mixed Siegel-Weil formula (with an appendix by Yujie Xu)}. \href{https://arxiv.org/abs/2204.13457}{\texttt{arXiv:2204.13457}} }.
		
		\bibitem{RSZ-Diagonalcycle}{M. Rapoport, B. Smithling, and W. Zhang, \textit{Arithmetic diagonal cycles on unitary Shimura varieties}, Compos. Math. 156 (2020), no. 9, 1745–1824. }
		
		\bibitem{RSZ-exotic}{M. Rapoport, B. Smithling, and W. Zhang, \textit{On the arithmetic transfer conjecture for exotic smooth formal moduli spaces}, Duke Math. J. \textbf{166} (2017), no. 12, 2183--2336.} 
		
		\bibitem{RSZ-regular}{M. Rapoport, B. Smithling, and W. Zhang, \textit{Regular formal moduli spaces and arithmetic transfer conjectures}, Math. Ann. \textbf{370} (2018), no.\ 3--4, 1079--1175.}
		
		\bibitem{RSZ-Shimura}{M. Rapoport, B. Smithling, and W. Zhang, \textit{On Shimura varieties for unitary groups}, Pure and Applied Mathematics Quarterly 17.2 (2021).}
		
		\bibitem{RZ96}{M. Rapoport and T. Zink, \textit{Period spaces for $p$-divisible groups}, Annals of Mathematics Studies, vol. \textbf{141}, Princeton University Press, Princeton, NJ, 1996.}
		
		\bibitem{RZ-Drinfeld space}{M. Rapoport and T. Zink, \textit{On the Drinfeld moduli problem of $p$-divisible groups},  Camb. J. Math. \textbf{5} (2017), no. 2, 229--279.}
		
		
		
		\bibitem{EKOR}{X. Shen. C-F. Yu. C. Zhang, \textit{EKOR strata for Shimura varieties with parahoric level structure}. Duke Math. J. 170 (14) 3111 - 3236.}
		
		\bibitem{SSTT19}{A. N. Shankar, A. Shankar, Y. Tang, and S. Tayou, \textit{Exceptional jumps of Picard ranks of reductions of K3 surfaces over number fields}, Forum of Mathematics, Pi. Vol. 10. Cambridge University Press, 2022.  }
	
	    \bibitem{StacksProject}{The Stacks Project Authors, Stacks Project, \href{http://stacks.math.columbia.edu}{http://stacks.math.columbia.edu}. }
		
		\bibitem{Stamm}{H. Stamm, \textit{On the reduction of the Hilbert-Blumenthal moduli scheme with $\Gamma_0(p)$ level structure}, Forum Mathematicum 9 (1997), 405–455.}
		
		
		\bibitem{San13}{S. Sankaran, \textit{Unitary cycles on Shimura curves and the Shimura lift I}, Documenta Math. 18 (2013), pp. 1403–1464. }
		
		\bibitem{San17}{S. Sankaran, \textit{Improper Intersections of Kudla-Rapoport divisors and Eisenstein series}, Journal of the Institute of Mathematics of Jussieu (2017), 16(5), 899-945. }
		
		\bibitem{SW20}{P. Scholze, J. Weinstein, Berkeley lectures on p-adic geometry. Ann. of Math. Studies, 207, Princeton
			University Press, Princeton, 2020.}
		
		\bibitem{Tits}{J. Tits, \textit{Reductive groups over local fields}, in Automorphic forms, representations and L-functions, Proc. Symp. Pure Math. Am. Math. Soc., Corvallis/Oregon 1977, Proc. Symp. Pure Math. 33 (1979), Part 1, 29–69.	}
		
		\bibitem{Vollaard05}{I. Vollaard, The supersingular locus of the Shimura variety for GU(1,s), Can. J. Math. 62 (2010), 668–720. 
		}
		
		
		\bibitem{VW11}{I. Vollaard and T. Wedhorn, The supersingular locus of the Shimura variety for GU(1,s) II, Invent. Math. 184 (2011), 591–627.}
		
		\bibitem{Yun-Duke}{Z. Yun, \textit{The fundamental lemma of Jacquet--Rallis in positive
				characteristics},  Duke Math. J. \textbf{156} (2011), no. 2, 167--228.}
		
		\bibitem{Basic-point}{C. Yu. \textit{Basic points in the moduli spaces of PEL-type}. MPIM-preprint 2005-113, 2005.}
		
		\bibitem{AFL-Invent}{W. Zhang, \textit{On arithmetic fundamental lemmas}, Invent. Math. \textbf{188} (2012), no. 1, 197--252.}
		
		\bibitem{AFL}{W. Zhang, \textit{Weil representation and Arithmetic Fundamental Lemma}, 
			Ann. of Math. (2) 193 (2021), no. 3, 863–978. }
		
		\bibitem{Zhang-GGP}{W. Zhang, \textit{Fourier transform and the global Gan--Gross--Prasad conjecture for unitary groups},  Ann. of Math. (2) 180 (2014), no. 3, 971--1049.}
		
		\bibitem{Zhang-RS}{W. Zhang, \textit{Automorphic periods and the central value of Rankin-Selberg $L$-function}, J. Amer. Math. Soc. \textbf{27} (2014), 541--612.}
		
		\bibitem{Zhang-ICM}{W. Zhang, \textit{Periods, cycles, and $L$-functions: a relative trace formula approach},  Proceedings of the International Congress of Mathematicians--Rio de Janeiro 2018. Vol. II. Invited lectures, 487--521, World Sci. Publ., Hackensack, NJ, 2018. }
		
		\bibitem{AFL-More}{W. Zhang, \textit{More Arithmetic Fundamental Lemma conjectures: the case of Bessel subgroups}, \href{https://arxiv.org/abs/2108.02086}{ \texttt{arXiv:2108.02086}.} }
		
		\bibitem{YZZ}{W. Zhang, X. Yuan and S. Zhang, \textit{The Gross-Kohnen-Zagier theorem over totally real fields}, Compositio Math. 145 (2009), no. 5, 1147-1162.}
		
		
		
		
	\end{thebibliography}
\end{document}